\documentclass[english]{smfbook}

\usepackage{amsfonts, amscd, amsthm, amssymb, amsmath, amsxtra, latexsym}
\usepackage{graphicx}
\usepackage{enumerate}
\usepackage{pb-diagram,pb-xy}

\usepackage[applemac]{inputenc}
\usepackage[cmtip,arrow]{xy}
\input xy
\xyoption{all}

\newtheorem{theorem}{Theorem}[chapter]
\newtheorem{lemma}[theorem]{Lemma}
\newtheorem{proposition}[theorem]{Proposition}
\newtheorem{corollary}[theorem]{Corollary}
\newtheorem*{Thm2}{Theorem}

\newtheorem*{problem}{Problem}

\theoremstyle{definition}
\newtheorem{definition}[theorem]{Definition}

\theoremstyle{remark}
\newtheorem{remark}[theorem]{Remark}

\newcommand{\mN}{\mathbb N}
\newcommand{\matH}{\mathbb H}
\newcommand{\mR}{\mathbb R}
\newcommand{\mZ}{\mathbb Z}

\newcommand{\G}{\Gamma}
\newcommand{\La}{\Lambda}
\newcommand{\p}{\prime}

\newcommand{\tilM}{\widetilde{M}}
\newcommand{\calM}{\mathcal{M}}
\newcommand{\calP}{\mathcal{P}}
\newcommand{\calC}{\mathcal{C}}
\newcommand{\calH}{\mathcal{H}}
\newcommand{\calG}{\mathcal{G}}
\newcommand{\calV}{\mathcal{V}}
\newcommand{\calE}{\mathcal{E}}
\newcommand{\mn}{{\rm Min}}
\newcommand{\rk}{{\rm rk}}
\newcommand{\out}{{\rm Out}}
\newcommand{\mcg}{{\rm MCG}}
\newcommand{\diff}{{\rm Diff}}

\numberwithin{section}{chapter}
\numberwithin{equation}{chapter}

\makeindex

\begin{document}

\frontmatter

\title[Rigidity of high dimensional graph manifolds]
    {Rigidity of high dimensional graph manifolds}

\alttitle[Rigidit\'e des vari\'et\'es graph\'ees de grande dimension]


\author{Roberto Frigerio}
\address{Dipartimento di Matematica - Largo B. Pontecorvo, 5 - 56127 Pisa, Italy}
\email{frigerio@dm.unipi.it}
\urladdr{http://www.dm.unipi.it/~frigerio/}
\thanks{}

\author{Jean-Fran\c cois Lafont}
\address{Department of Mathematics, The Ohio State University, 231 West 18th Avenue,
Columbus, OH 43210-1174, USA}
\email{jlafont@math.ohio-state.edu}
\urladdr{http://www.math.osu.edu/~jlafont/}
\thanks{}

\author{Alessandro Sisto}
\address{Mathematical Institute, 24-29 St Giles', Oxford, OX1 3LB, England, UK}
\email{sisto@maths.ox.ac.uk}
\urladdr{http://people.maths.ox.ac.uk/sisto/}
\thanks{}

\date{}



\begin{abstract}
We define the class of {\it high dimensional graph manifolds}. These are compact smooth manifolds
supporting a decomposition into finitely many pieces, each of which is diffeomorphic to the
product of a torus with a finite volume hyperbolic manifold with toric cusps. The various pieces 
are attached together via affine maps of the boundary tori. We require all the hyperbolic factors
in the pieces to have dimension $\geq 3$. Our main goal is to study
this class of graph manifolds from the viewpoint of rigidity theory.

We show that, in dimensions $\geq 6$, the Borel conjecture holds for our graph manifolds. We 
also show that smooth rigidity holds within the class: two graph manifolds are homotopy 
equivalent if and only if they are diffeomorphic. We introduce the notion of {\it irreducible}
graph manifolds. These form a subclass which has better coarse geometric properties, 
in that various subgroups can be shown to be quasi-isometrically embedded inside the
fundamental group.  
We establish some structure theory for finitely generated groups which are quasi-isometric
to the fundamental group of an irreducible graph manifold: any such group has a graph of 
groups splitting with strong constraints on the edge and vertex groups. Along the way, we
classify groups which are quasi-isometric to the product of a free abelian group and
a non-uniform lattice in $SO(n,1)$. We provide 
various examples of graph manifolds which do {\it not} support any locally CAT(0) metric.

Several of our results can be extended to allow pieces with hyperbolic surface factors.
We emphasize that, in dimension $=3$, our notion of graph manifold {\bf does not} coincide with
the classical graph manifolds. Rather, it is a class of $3$-manifolds that contains some (but not all)
classical graph $3$-manifolds (we don't allow general Seifert fibered pieces), as well as some 
non-graph 3-manifolds (we do allow hyperbolic pieces). 
\end{abstract}

\begin{altabstract}
Ce texte est consacr\'e \`a la d\'efinition et \`a l'\'etude syst\'ematique des {\it vari\'et\'es graph\'ees de grande dimension}.
Celles-ci sont des vari\'et\'es lisses, ayant une d\'ecomposition en un nombre fini de morceaux g\'eom\'etrique. Chaque 
morceau est diff\'eomorphe au produit d'un tore et d'une vari\'et\'e hyperbolique de volume fini dont tous
les bouts sont des tores. Les morceaux sont recoll\'es par des applications affines des tores 
qui en sont les bords. Nous exigeons que le facteur hyperbolique dans chaque morceau
soit de dimension $\geq 3$.
Notre but principal est d'\'etablir 
divers r\'esultats de rigidit\'e pour cette classe de vari\'et\'es graph\'ees.

Nous d\'emontrons, en dimension $\geq 6$, la conjecture de Borel pour les vari\'et\'es graph\'ees :
une vari\'et\'e quelconque est homotopiquement \'equivalente a une vari\'et\'e graph\'ee si et seulement si
elle est hom\'eomorphe a cette m\^eme vari\'et\'e graph\'ee.
Nous \'etablissons la rigidit\'e lisse pour la classe des vari\'et\'es graph\'ees :  deux vari\'et\'es 
graph\'ees sont homotopiquement \'equivalentes si et seulement si elles sont diff\'eomorphes. Du point
de vue de la g\'eom\'etrie \`a grande \'echelle, la distorsion des groupes
fondamentaux des morceaux dans le groupe fondamental de la vari\'et\'e graph\'ee joue un r\^ole 
essentiel. Nous introduisons la notion de {\it varit graphe irrductible}. Elles forment une sous-classe pour 
laquelle ces sous-groupes sont toujours non-distordus. Ceci nous permet d'analyser la structure des groupes 
quasi-isom\'etriques au groupe fondamental d'une vari\'et\'e graph\'ee irr\'eductible: un tel groupe a 
(virtuellement) une action sur un arbre, avec de fortes contraintes sur les stabilisateurs de sommets
et d'ar\^etes. Cet analyse comprend, entre autre, une classification des groupes quasi-isom\'etriques
au produit d'un groupe ab\'elien libre et d'un r\'eseau non-uniforme dans $SO(n,1)$. Nous pr\'esentons plusieurs
\'examples de vari\'et\'es graph\'ees qui n'admettent {\it aucune} m\'etrique locallement CAT(0).

Certains de nos r\'esultats s'appliquent aussi bien en pr\'esence de
morceaux ayant commes facteurs des surfaces hyperboliques.
Nous pr\'ecisons que, en dimension trois, notre notion de vari\'et\'e graph\'ee {\bf ne co\"incide pas} 
avec la notion classique de vari\'et\'e graph\'ee. Nos vari\'et\'es forment une classe comprenant
certaines (mais pas toutes) les vari\'et\'es graph\'ees classiques (nous excluons certaines sous-vari\'et\'es de Seifert), 
ainsi que des vari\'et\'es que ne sont pas des vari\'et\'es graph\'ees classiques (nous admettons des
morceaux purement hyperboliques). 

\end{altabstract}

\maketitle


\setcounter{page}{4}

\tableofcontents

%
%
%

\chapter*{Introduction}~\label{intro:sec}

In recent years, there has been an extensive amount of work done on proving rigidity results
for various classes of non-positively curved spaces. In this monograph, we are interested in
establishing similar rigidity theorems in the context of spaces which may not support any 
non-positively curved metrics.

To motivate our class of manifolds, we briefly recall some basic notions from $3$-manifold
topology. In the theory of $3$-manifolds, a central role is played by {\it Thurston's geometrization 
conjecture}, recently established by Perelman. Loosely speaking, this asserts that a closed 
$3$-manifold can be decomposed into pieces, each of which supports a {\it geometric structure}, 
i.e. a complete metric locally modelled on one of the eight $3$-dimensional geometries. 
When restricted to the class of $3$-manifolds which support a non-positively curved 
metric, the geometrization conjecture states that such a $3$-manifold contains a 
finite collection of pairwise disjoint, embedded $2$-tori, and each component of the 
complement is either hyperbolic (supports a metric modeled on $\mathbb H^3$) or 
is non-positively curved Seifert fibered (supports a metric modeled on $\mathbb H^2 
\times \mR$). In the case where there are no hyperbolic components, 
the $3$-manifold is an example of a {\it graph manifold}. The class of manifolds we consider 
are inspired by these notions.

\begin{definition}
We will say that a compact smooth $n$-manifold $M$, $n\geq 3$, is a {\it graph manifold}
provided that it can be constructed in the following way:
\begin{enumerate}
\item
For every $i=1,\ldots, r$, take a complete finite-volume
non-compact hyperbolic $n_i$-manifold $N_i$ 
with toric cusps, where $3\leq n_i\leq n$.
\item
Denote by $\overline{N}_i$ the manifold obtained by ``truncating the cusps'' of $N_i$, i.e.~by removing
from $N_i$ a horospherical neighbourhood of each cusp.
\item
Take the product $V_i=\overline{N}_i\times T^{n-n_i}$, where $T^k=(S^1)^k$ is the $k$-dimensional torus. 
\item
Fix a pairing of some boundary components of the $V_i$'s and glue the paired
boundary components using \emph{affine} diffeomorphisms of the boundary tori, so as to obtain a connected manifold of
dimension $n$ (see Section~\ref{construction:sec} for the precise definition of affine gluing in this context).
\end{enumerate}
Observe that $\partial M$ is either empty or consists of tori. The submanifolds $V_1,\ldots,V_r$ will be called
the \emph{pieces} of $M$. The manifold $\overline{N}_i$ is the \emph{base} of $V_i$, while every
subset of the form $\{\ast\}\times T^{n-n_i}\subseteq V_i$ is a \emph{fiber} of $V_i$. The boundary
tori which are identified together will be called the {\it internal walls} of $M$ (so any two distinct pieces in $M$
will be separated by a collection of internal walls), while the components of $\partial M$ will be called the \emph{boundary walls} of $M$.
\end{definition}

\vskip 5pt

Informally, our manifolds can be decomposed into pieces, each of which supports 
a finite-volume product metric locally modeled on some $\mathbb H^k \times \mR ^{n-k}$ ($k\geq 3$).

Our notion of generalized graph manifolds includes both the classical ``double'' of a finite volume 
hyperbolic manifold with toric cusps, as well as those twisted doubles of such manifolds (in the sense of 
Aravinda and Farrell \cite{ArFa}) that are obtained via affine gluings. 

A restriction that we have imposed on our graph manifolds is that all
pieces  have a base which is hyperbolic {\it of dimension $\geq 3$}. The
reason for this restriction is obvious: hyperbolic manifolds of
dimension $\geq 3$ exhibit a lot more rigidity than surfaces. 
However, some of our results extend also to the case when surfaces with boundary 
are allowed as bases of pieces. To allow for these, we introduce the following:

\begin{definition}
 For $n\geq 3$, an \emph{extended graph $n-$manifold} is a manifold built up from pieces as in the definition of graph manifold as well as \emph{surface pieces},
that is manifolds of the form $\Sigma \times T^{n-2}$ with $\Sigma$ non-compact, finite volume,
hyperbolic surface. Also, we require that each gluing does {\bf not} identify the fibers in adjacent surface pieces.
\end{definition}

Let us briefly comment about the last requirement described
in the above Definition. If we allowed gluings which identify the fibers
of adjacent surface pieces, then
the resulting decomposition into pieces of our extended graph manifold
would no longer be canonical, and some of our rigidity results (see e.g.~Theorem~\ref{iso-preserve:thm}) would no longer be true.
Indeed, within a surface piece $\Sigma \times T^{n-2}$, 
we can take any non-peripheral simple closed curve $\gamma \hookrightarrow \Sigma$ 
in the base surface, and cut the piece open along $\gamma \times T^{n-2}$. This
allows us to break up the original piece $\Sigma \times T^{n-2}$ into pieces 
$(\Sigma \setminus \gamma) \times T^{n-2}$ (which will either be two
pieces, or a single ``simpler'' piece, according to whether $\gamma$ separates
or not). Our additional requirement avoids this possibility. 
Note however that if one has
adjacent surface pieces with the property that
the gluing map matches up their fibers exactly, then it is not
possible to conclude that the two surface pieces
can be combined into a single surface piece (the resulting manifold could be a
non-trivial $S^1$-fiber bundle over a surface rather than just a product).

We emphasize that, restricting down to $3$-dimensions, our notion of (extended)
graph manifold  {\bf do not} coincide with the classical $3$-dimensional
graph manifolds.  For instance:
\begin{itemize}
\item we do not allow general finite volume quotients of $\mathbb H^2 \times \mR$, 
\item we allow purely hyperbolic pieces in our decompositions (i.e. the case where a
piece is just a truncated cusped hyperbolic $3$-manifold),
\item in the case of genuine graph manifolds, we do not allow pieces to be products of a hyperbolic surface with a circle. 
\end{itemize}

\vskip 10pt

Now our (extended) graph manifolds are ``built up'', in a relatively simple manner, from non-positively curved 
manifolds. If we know some property holds for non-positively curved manifolds, and hence for all
the pieces in our decomposition, we could expect it to hold for the (extended) graph manifold. This monograph
pursues this general philosophy, with a view towards establishing analogues of various rigidity
theorems for the class of (extended) graph manifolds. 

\smallskip 

In some special cases, the implementation of the strategy we have just described is quite plain. This is the case, for example, 
for purely hyperbolic graph manifolds, which we now define.
We say that a graph manifold is \emph{purely hyperbolic} if the fiber of each of its pieces
is trivial (i.e.~each piece is just a truncated hyperbolic manifold). 
Such manifolds enjoy additional nice properties, that are of great help in understanding
their geometry: for example, they support nonpositively curved Riemannian metrics (Theorem~\ref{purelyhyp:CAT0}), 
and their fundamental groups are relatively hyperbolic (Theorem~\ref{relhyp:thmintro}).
As a consequence, many of our results are much easier (and sometimes already known)
for purely hyperbolic graph manifolds. 
In order to support the reader's intuition of our arguments, in this introduction we pay a particular attention to this subclass of manifolds, pointing
out how some of our arguments could be shortened in the case of purely hyperbolic manifolds.

\bigskip 
 
 Let us now briefly describe the content of each Chapter.

\bigskip

Chapter \ref{quasiapp} starts out with a review of 
some basic notions: quasi-isometries, quasi-actions, and the Milnor-S\v varc Lemma.

\bigskip 

In Chapter \ref{construction:sec}, we introduce our (extended) graph manifolds, and establish some basic
general results.  A result by Leeb~\cite[Theorem 3.3]{leeb} ensures that every (extended) $3$-dimensional graph manifold containing at least one purely hyperbolic piece
supports a non-positively curved
Riemannian metric with totally geodesic boundary. 
A slight variation of Leeb's argument allows us to prove the following:

\begin{theorem}\label{purelyhyp:CAT0}
Let $M$ be a purely hyperbolic graph manifold. Then $M$ supports a nonpositively curved Riemannian metric with totally geodesic boundary.
\end{theorem}

On the other hand, in Section \ref{noncat0-easy:subsec} we provide a first family of examples of 
(extended) graph manifolds which {\it cannot} support any locally CAT(0)-metric. More precisely, for $n\geq 4$
we construct examples of $n$-dimensional (extended) graph manifolds $M$ where the fundamental group 
of the walls is {\it not} quasi-isometrically embedded in $\pi_1(M)$ (these examples
are genuine graph manifolds for $n\geq 5$). 
By the Flat Torus Theorem,  for these examples $\pi_1(M)$ cannot act via semisimple isometries on any CAT(0) space, so
$M$ cannot support 
a locally CAT(0)-metric. In fact, the fundamental groups of our examples contain distorted cyclic subgroups, so by the work of Haglund~\cite{Hag}
they cannot act properly on any (potentially infinite-dimensional) CAT(0) cube complex. This contrasts with the recent advances in $3$-manifold theory.

We also stress that our non-CAT(0) examples may contain purely hyperbolic pieces. 
Therefore, in contrast with Leeb's result in dimension 3, the hypothesis
of Theorem~\ref{purelyhyp:CAT0} cannot be weakened by replacing the condition that $M$ is purely hyperbolic with
the condition that it contains a purely hyperbolic piece. 

These first results already suggest that the geometry of generic graph manifolds may be more complicated than 
the one of purely hyperbolic
graph manifolds. As a consequence, more care is needed in our analysis when non-trivial fibers are present.

\bigskip

In Chapter \ref{toprigidity:sec}, we study the topology of our graph manifolds. Recall that the {\it Borel Conjecture}
states that if $M, M^\p$ are aspherical manifolds with isomorphic fundamental group, then
they are in fact homeomorphic.  If the manifold $M$ is assumed to support a Riemannian metric
of non-positive curvature and has dimension $\geq 5$, then the validity of the Borel Conjecture is
a celebrated result of Farrell-Jones.  Our next result establishes (Section~\ref{Borel:sec}):

\begin{theorem}[Topological rigidity]\label{toprigidity:thm}
Let $M$ be an (extended) graph manifold (possibly with boundary), of dimension $n\geq 6$. Assume $M^\prime$ is
an arbitrary topological manifold and $\rho: M^\prime \rightarrow M$ is a homotopy equivalence which 
restricts to a homeomorphism
$\rho|_{\partial M^\prime}: \partial M^\prime \rightarrow \partial M$ between the boundaries of the manifolds. 
Then $\rho$ is 
homotopic, rel $\partial$, to a homeomorphism $\bar \rho: M^\prime \rightarrow M$.
\end{theorem}


Recall that purely hyperbolic graph manifolds admit a nonpositively curved Riemannian metric. So if $M$ is purely 
hyperbolic, then Theorem~\ref{toprigidity:thm} follows from the result of Farrell and Jones mentioned above
(and holds even in dimension $5$).

Our Theorem~\ref{toprigidity:thm} is actually a special case of our more general Theorem \ref{Borel-general},
where we establish the Borel Conjecture for a broader class of manifolds.
Along the way, we also show that our (extended) graph manifolds are always aspherical (Section \ref{Asphere:sec}),
and have vanishing lower algebraic $K$-theory (Section \ref{algKth:sec}). 
We also point out that the Baum-Connes conjecture holds (Section \ref{BCC:sec}) and mention some 
well-known consequences. It is worth noting that, by work
of Ontaneda \cite[Theorem 1]{On}, there are examples of doubles of finite volume hyperbolic manifolds 
which support exotic PL-structures. As such, the conclusion of our Theorem \ref{toprigidity:thm} is optimal, 
since there are examples where no 
PL-homeomorphism (and hence, no diffeomorphism) exists between $M$ and $M^\p$. 

\bigskip

From the generalized Seifert-Van Kampen theorem, the fundamental group
$\G$ of one of our (extended) graph manifolds $M$ can be expressed as the fundamental group of a graph
of groups, with vertex groups given by the fundamental groups of the pieces, and edge groups
isomorphic to $\mZ^{n-1}$, where $n$ is the dimension of $M$. To further develop our analysis of (extended) graph manifolds, we would like
to ensure that reasonable maps between (extended) graph manifolds have to (essentially)
preserve the pieces. The following result, which is the main goal of Chapter~\ref{pieces-iso:sec},
is crucial:

\begin{theorem}[Isomorphisms preserve pieces]\label{iso-preserve:thm}
Let $M_1$, $M_2$ be a pair of (extended) graph manifolds and let $\G_i=\pi_1(M_i)$ be their respective
fundamental groups.  Let $\La_1 \leq \G_1$ be a subgroup conjugate to the fundamental
group of a piece in $M_1$, and $\varphi: \G_1\rightarrow \G_2$ be an isomorphism. Then 
$\varphi(\La_1)$ is conjugate to the fundamental group $\La_2 \leq \G_2$ of a piece in $M_2$.
\end{theorem}

A fairly straightforward consequence of this result is a necessary condition for two (extended) graph 
manifolds to have isomorphic fundamental groups (see also Theorem~\ref{preserve2:thm}):

\begin{corollary}\label{same-pieces}
Let $M, M^\p$ be a pair of (extended) graph manifolds. If $\varphi: \pi_1(M) \to \pi_1(M^\p)$ is an isomorphism,
then it induces a graph isomorphism between the associated graph of groups. Moreover, vertices
identified via this graph isomorphism must have associated vertex groups which are isomorphic.
\end{corollary}

It should not be difficult to prove that the description of $\pi_1(M)$ as the fundamental
group of the graph of groups corresponding to the decomposition of $M$ into pieces
provides a JSJ-decomposition of $\pi_1(M)$, in the sense of Fujiwara and Papasoglu \cite{FP}
(see also Dunwoody and Sageev \cite{DS}). As such, Theorem~\ref{iso-preserve:thm} and 
Corollary~\ref{same-pieces}
could probably be deduced from the uniqueness results proved
in~\cite{FP} (see also Forester \cite{Forester}, and Guirardel and Levitt \cite{GL}). However, the case
of (extended) graph manifolds is considerably easier than the general case treated in these other
papers, so we preferred to give complete and self-contained proofs of 
Theorem~\ref{iso-preserve:thm} and Corollary~\ref{same-pieces}.

\bigskip

In Chapter~\ref{smoothrig:sec}, we return to studying the topology of 
(extended) graph manifolds. 
Building on Theorem~\ref{iso-preserve:thm}, we prove the following:

\begin{theorem}[Smooth rigidity]\label{smrigidity:thm}
Let $M,M'$ be (extended) graph manifolds, and let
$\varphi\colon \pi_1(M)\to \pi_1(M')$ be a group isomorphism. 
Also assume that no boundary component of $M,M'$ lies in a surface piece
(of course, this condition is automatically satisfied if $M$ and $M'$ are 
genuine graph manifolds).
Then $\varphi$ 
is induced by a diffeomorphism $\psi\colon M\to M'$. 
\end{theorem}

In Theorem~\ref{smrigidity:thm}, the additional hypothesis 
preventing surface pieces to be adjacent to the boundary 
is necessary: if
$M=\Sigma\times S^1$ and
$M'=\Sigma'\times S^1$, where $\Sigma_1$ is a once-punctured torus and
$\Sigma'$ is a thrice-punctured sphere, then 
$\pi_1(M)\cong \pi_1(M')$, but $M$ and $M'$ are not diffeomorphic (in fact, they are not even homeomorphic).

Ontaneda \cite{On} had previously shown smooth rigidity within the class of doubles of 
finite volume hyperbolic manifolds. The proof of Theorem~\ref{smrigidity:thm} is easier if $M,M'$ 
are purely hyperbolic (as in Ontaneda's examples). In fact, in that case
Theorem~\ref{iso-preserve:thm}, 
together with Mostow Rigidity Theorem, ensures that the 
restriction of $\varphi$ to the fundamental groups of the pieces of $M$ is induced by
suitable diffeomorphisms between the pieces of $M$ and the pieces of $M'$. 
In presence of non-trivial fibers, the proof of this fact needs some more work (see the proof of Lemma~\ref{diffeo:pieces}).
Once this is established,
one has to carefully check that the diffeomorphisms between the pieces can be extended to a global
diffeomorphism between $M$ and $M'$, which indeed induces the fixed group isomorphism between the fundamental groups.

\bigskip

Next, for $M$ a closed smooth manifold, we denote by $\mcg (M)$ the \emph{mapping class group of} $M$,
i.e. the group of homotopy classes of diffeomorphisms of $M$ into itself.
Theorem~\ref{smrigidity:thm} easily implies the following corollary (see Section \ref{mcg:sec}):

\begin{corollary}\label{mcg:cor}
 Let $M$ be a closed graph manifold. Then, the group
$\mcg (M)$ is isomorphic to the group ${\rm Out} (\pi_1 (M))$ of the outer automorphisms
of $\pi_1 (M)$.
\end{corollary}

Using Corollary~\ref{mcg:cor}, it is easy to see that $\mcg (M)$ is often infinite. For example, 
this is always the case when considering doubles or twisted doubles (obtained via affine gluings)
of one-cusped hyperbolic manifolds with toric cusp (see Remarks~\ref{twist:rem}
and~\ref{mcg:rem}).

\bigskip

In Chapter~\ref{groups:sec} we
describe some group theoretic properties of fundamental groups of 
(extended) graph manifolds. 
In order to properly state our results, we need  to introduce some definitions. 

\begin{definition}
Let $M$ be an (extended) graph manifold, and $V^+, V^-$ a pair of adjacent (not necessarily distinct) pieces of $M$. 
We say that the two pieces have {\it transverse fibers} along the common internal wall $T$ provided that, under the gluing diffeomorphism 
$\psi: T^+ \rightarrow T^-$ of the paired boundary tori corresponding to $T$, the image of the fiber subgroup of $\pi_1(T^+)$ under $\psi_*$ intersects 
the fiber subgroup of $\pi_1(T^-)$ only in $\{0\}$ (in this case, we equivalently say that the gluing $\psi$ is \emph{transverse} along 
$T$).
This is equivalent to asking that the 
sum of the dimensions of the fibers of $T^+$ and $T^-$ is strictly less
than the dimension of $M$, and that
the image of every fiber of $T^+$
under $\psi$ is transverse to every fiber of $T^-$.
\end{definition}

\begin{definition}\label{irr:def}
An (extended) graph manifold is {\it irreducible} if every pair of adjacent
pieces has transverse fibers along every common internal wall.
\end{definition}

In the case of $1$-dimensional fibers, an (extended) graph manifold is irreducible
if and only if the $S^1$-bundle structure on each piece cannot be extended
to the union of adjacent pieces. 

Simple examples of irreducible graph manifolds include the doubles of truncated
finite volume hyperbolic manifolds with toric cusps, as well as the
twisted doubles of 
such manifolds obtained via affine gluings. More generally, every purely hyperbolic graph manifold
is  irreducible. Irreducible graph manifolds play an important role in our analysis. On the one hand,
they provide a much wider class than purely hyperbolic graph manifolds (for example, in contrast with 
Theorem~\ref{purelyhyp:CAT0}, some of them can fail to support nonpositively curved metrics - 
see Theorem~\ref{existence:thm}). On the other hand their geometry can still be understood
quite well in terms of the geometry of the pieces (see e.g.~Theorem~\ref{walls-ok}).

At the other extreme of irreducibility, 
it may happen that an (extended) graph manifold $M$ admits a toric bundle structure obtained by gluing the product structures defined on the pieces. In this case, $M$ is the total space of a fiber bundle
with base a graph manifold
of lower dimension. This observation motivates the following:

\begin{definition}\label{fibered:defn}
An (extended) graph manifold $M$ is \emph{fibered} if it is the total space of a 
smooth fiber bundle
$F\hookrightarrow M\to M'$, where the fiber $F$ is a $d$-dimensional torus, $d\geq 1$, and
$M'$ is an (extended) graph manifold.
\end{definition} 


\medskip

A natural question is whether the fundamental groups of (extended) graph manifolds
are relatively hyperbolic.
The notion of (strong) relative hyperbolicity first appeared in Gromov \cite{gro1}. 
The motivating example of a relatively hyperbolic group
is the fundamental group of a non-compact, finite volume, Riemannian manifold with
sectional curvature bounded above by some negative constant. Such a group is
relatively hyperbolic
with respect to the collection of cusp subgroups (see e.g. Farb \cite{farb}). 
Therefore, a  graph manifold consisting of a single piece with trivial toric fiber is relatively hyperbolic. Building on 
Dahmani's Combination Theorem \cite{Dahmani}, in Section~\ref{relhyp:sec} we extend this result as follows:
 
 \begin{theorem}\label{relhyp:thmintro}
 Assume the (extended) graph manifold $M$ has at least one piece with trivial toric fiber. 
Then $\pi_1(M)$ is relatively hyperbolic with respect to a finite family of proper subgroups.
 \end{theorem}
 
For example, the fundamental group of any purely hyperbolic graph manifold $M$ is relatively hyperbolic (in fact,
it is not difficult to show that in this case one may choose the fundamental groups of the walls of $M$ as a family
of peripheral subgroups, so $\pi_1(M)$ is toral relatively hyperbolic). 
However, at least in the case of irreducible graph manifolds, this is the only case in which $\pi_1(M)$ is relatively hyperbolic. 
In fact, in Section~\ref{thickness:sec} we prove:

\begin{theorem}\label{relhyp:thm2}
Let $M$ be an irreducible graph manifold. Then $\pi_1(M)$ is relatively hyperbolic with respect to a finite
family of proper subgroups if and only if $M$ contains at least one purely hyperbolic piece.
\end{theorem}

Our proof of Theorem~\ref{relhyp:thm2} is based on the study of the coarse geometric properties of the fundamental group
of irreducible graph manifolds, which is carried out in Chapters~\ref{strongirr:sec} and~\ref{preserve:sec}.
 
The notion of a hyperbolically embedded collection of subgroups has been recently introduced by
Dahmani, Guirardel, and Osin \cite{DGO},
and can be thought of as a generalization of
peripheral structures of relatively hyperbolic groups. One may wonder whether the fundamental
group of an (extended) graph manifold always contains a non-degenerate hyperbolically embedded subgroup. 
A very useful feature of irreducible (extended) graph
manifolds is that the action of the fundamental group on the associated Bass-Serre tree is
{\it acylindrical} (see Proposition~\ref{irr-acyl}).
In Chapter~\ref{groups:sec} we exploit (a refinement of) this result to prove the following:

\begin{theorem}\label{hypemb:thm}
 Let $M$ be an (extended) graph manifold, and suppose that
 $M$ contains an internal wall with transverse fibers. Then
  $\pi_1(M)$ contains a non-degenerate hyperbolically embedded subgroup.
\end{theorem}

Observe that the conclusion of Theorem~\ref{hypemb:thm} cannot hold in general: for example, if $M$ is the double of
a non-purely hyperbolic piece, then $\pi_1(M)$ has an infinite center, so it cannot
contain any non-degenerate hyperbolically embedded subgroup~\cite[Corollary 4.34]{DGO}.

\medskip


The following theorem describes other interesting algebraic properties of fundamental groups of
(extended) graph manifolds. It summarizes the results proved in
Propositions~\ref{UEG}, \ref{CH}, 
\ref{C-simple-manifolds}, \ref{SQ-u} and Corollaries~\ref{Kazhdan subgroups},  \ref{TA}, \ref{solvable:cor}.

\begin{theorem}\label{algebraic:thm}
Let $M$ be an (extended) graph manifold.
\begin{enumerate}
 \item 
If an arbitrary subgroup $H<\pi_1 (M)$ has Kazhdan's property (T), then
$H$ is the trivial subgroup.
\item
$\pi_1 (M)$ has uniformly exponential growth.
\item 
(Tits Alternative):
If $H<\pi_1 (M)$ is an arbitrary subgroup, then
either $H$ is solvable, or $H$ contains a non-abelian free group. Moreover, if $M$ is irreducible,
then every solvable subgroup of $M$ is abelian.
\item
Suppose that $\partial M= \emptyset$, and that $M$ 
contains a pair of adjacent pieces with transverse fibers.
Then $\pi_1 (M)$ is co-Hopfian.
\item
$\pi_1(M)$ is $C^\ast$-simple if and only if $M$ is not fibered.
\item
Suppose that at least one of the following conditions holds:
\begin{itemize}
\item
$M$ consists of a single piece without internal walls, or
\item
$M$ contains at least one separating internal wall, or
\item
$M$ contains at least one internal wall with transverse fibers. 
\end{itemize}
Then $\pi_1(M)$ is SQ-universal.
\end{enumerate}
\end{theorem}

Our proof of Theorem~\ref{algebraic:thm} is based on the study of the action of $\pi_1(M)$ on the
Bass-Serre tree corresponding to the decomposition of $M$ into pieces. 
Some of the statements of Theorem~\ref{algebraic:thm} are deduced from more general results concerning
fundamental groups of graph of groups.
For example, in Propositions~\ref{TA-graph}  and \ref{solvable}
we establish the Tits Alternative and the solvability of the word problem for 
wide classes of fundamental groups of graphs of groups.
Moreover, building on the results established in~\cite{DGO},
in Propositions~\ref{C-simple-graphs} and~\ref{SQ-u-graph} we characterize acylindrical
graphs of groups having respectively $C^*$-simple and SQ-universal fundamental groups.

\smallskip

If $M$ is purely hyperbolic, several points of Theorem~\ref{algebraic:thm} immediately follow from known results. In that case,
 $\pi_1(M)$ is toral relatively hyperbolic, so point (2) follows from Xie \cite{Xie}, point (5) from Arzhantseva and
 Minasyan \cite{AM} and point (6)
 from Arzhantseva, Minasyan, and Osin \cite{AMO}. Moreover, \cite[Theorem 8.2.F]{gro1} ensures that every finitely generated subgroup
 of $\pi_1(M)$ that does not contain a non-abelian free subgroup is either virtually cyclic, or contained in a parabolic subgroup (whence abelian, in our case).
Also observe that Mostow rigidity implies that fundamental groups of finite-volume hyperbolic manifolds are co-Hopfian, while
the fundamental groups of  tori are obviously non-co-Hopfian. As a consequence, 
in the case of purely hyperbolic manifolds our proof of point (4) may be simplified notably (see Proposition~\ref{CH}).

\medskip
 
Now recall that, by Corollary \ref{same-pieces}, to have any chance of 
having isomorphic fundamental groups, two graph manifolds would have to be built
up using the exact same pieces, and the gluings would have to identify the same collection
of boundary tori together. So the only possible variation lies in the choice of 
gluing maps used to identify the boundary tori together. In Section~\ref{groups:subsec},
we show how, in some cases, fixing the collection of pieces, we can still produce infinitely 
many non-isomorphic fundamental groups simply by varying the gluings between the 
common tori. The construction is flexible enough that we can even ensure that all 
the resulting graph manifolds are irreducible.

\bigskip

Starting from Chapter~\ref{strongirr:sec}, we shift our focus to coarse geometric properties of 
our graph manifolds. Let us first observe that the coarse geometry of (the universal covering of) surface pieces is very different from the coarse geometry of non-surface pieces: namely, surface pieces contain many more quasi-flats of maximal dimension, and this implies for example that there cannot exist a coarse-geometric characterization
of the boundary components of a surface piece. In order to avoid the resulting
complications, we restrict our attention to genuine graph manifolds.

As we mentioned earlier, there exist examples of $n$-dimensional graph manifolds $M$
with the property that certain walls $T\subset M$ have fundamental groups
$\pi_1(T) \cong \mZ ^{n-1} \hookrightarrow \pi_1(M)$ which are {\it not} quasi-isometrically
embedded. 
As one might expect, the presence of such walls causes serious difficulties
when trying to study the coarse geometry of $M$. 

If $M$ is purely hyperbolic, then $\pi_1(M)$
is hyperbolic relative to the fundamental groups of walls, and so these walls {\it cannot} be distorted
(this also follows from Theorem~\ref{purelyhyp:CAT0} and the Flat Torus Theorem).
It follows easily that, in the purely hyperbolic case, the fundamental group
of every fiber and of every piece of $M$ is quasi-isometrically embedded.
However, restricting our attention to the class of
purely hyperbolic manifolds would be much too limiting. As we hinted above, irreducible manifolds
provide the right class of manifolds to work with, as they satisfy the important:

\begin{theorem}\label{walls-ok}
Let $M$ be an irreducible graph manifold. Then the fundamental group of every fiber,
wall, and piece, is quasi-isometrically embedded in $\pi_1(M)$.
\end{theorem}

The proof of this result occupies the bulk of Chapter 7 (see in particular
Theorem~\ref{quasi-isom:thm} and Corollary~\ref{fibre:cor}). 

\bigskip

In Chapter \ref{preserve:sec}, we start analyzing quasi-isometries between fundamental
groups of irreducible graph manifolds. By studying the asymptotic cone of the universal 
cover of $M$, we are able to show:

\begin{theorem}[QI's preserve pieces of irreducible graph manifolds]\label{qi-preserve:thm}
Let $M_1$, $M_2$ be a pair of irreducible graph manifolds, and $\G_i=\pi_1(M_i)$ their respective
fundamental groups.  Let $\La_1 \leq \G_1$ be a subgroup conjugate to the fundamental
group of a piece in $M_1$, and $\varphi: \G_1\rightarrow \G_2$ be a quasi-isometry. 
Then, the set $\varphi(\La_1)$ is within finite Hausdorff distance from
a conjugate of $\La_2 \leq \G_2$, where $\La_2$ is the fundamental group of a piece  in $M_2$.
\end{theorem}

The key step in the proof of Theorem~\ref{qi-preserve:thm} consists in showing that 
fundamental groups of walls of irreducible graph manifolds are quasi-preserved by
quasi-isometries. 
In the case of purely hyperbolic graph manifolds, this readily follows from the results
on quasi-isometries between relatively hyperbolic groups established in~\cite[Theorem 1.7]{dru}. 
Namely, Drutu and Sapir proved that, if $G$ is relatively hyperbolic, then
$G$ is asymptotically tree-graded with respect to the family of its peripheral subgroups,
and they used this result to show that a quasi-isometric copy in $G$ of any unconstricted group 
(e.g.~any free abelian of rank $\geq 2$) is close to a peripheral subgroup. 
In the case of a (not necessarily purely hyperbolic) irreducible graph manifold
 $M$, we are still able to give an asymptotic characterization of wall subgroups of $\pi_1(M)$, but the presence of non-trivial fibers
 makes the geometry of the asymtotic cone of $\pi_1(M)$ quite complicated, so more care is needed. 

\smallskip

Since pieces are essentially mapped to pieces under quasi-isometries, our next goal is
to understand the behavior of groups quasi-isometric to the fundamental group of a piece.
This is the subject of Chapter \ref{product:sec}, where we establish:

\begin{theorem}[QI-rigidity of pieces]\label{product:thm}
Let $N$ be a complete finite-volume hyperbolic $m$-manifold, $m\geq 3$, and
let $\Gamma$ be a finitely generated group quasi-isometric to
$\pi_1 (N)\times\mZ^d$, $d\geq 0$. 
Then there exists a finite-index subgroup $\Gamma'$ of $\Gamma$,
a finite-sheeted covering $N'$ of $N$, a group $\Delta$  and a finite group $F$ 
such that the following short exact sequences hold:
\begin{equation*}
\xymatrix{
1\ar[r] &\mZ^d \ar[r]^j & \Gamma' \ar[r] & \Delta \ar[r] & 1,\\
}
\end{equation*}
\begin{equation*}
\xymatrix{
1\ar[r] & F \ar[r] & \Delta\ar[r] & \pi_1 (N')\ar[r] & 1 .
}
\end{equation*}
Moreover, $j(\mZ^d)$ is contained in the center of $\Gamma'$.
In other words, $\Gamma'$ is a central extension by $\mZ^d$ 
of a finite extension of $\pi_1 (N')$.
\end{theorem}

In the case of purely hyperbolic pieces, i.e.~when $d=0$, Theorem~\ref{product:thm} is proved
by Schwartz \cite{schw}.
Note that the analogous result in the setting where $N$ is compact has been
established by Kleiner and Leeb \cite{klelee}.
A consequence of this result is that we can determine when two pieces have 
quasi-isometric fundamental group: their fibers must be of the same dimension, 
while their bases must be commensurable. 

\bigskip

In Chapter \ref{qirigidity:sec}, we study groups quasi-isometric to an irreducible
graph manifold, and show that they must exhibit a graph of groups structure which
closely resembles that of a graph manifold (compare with the work of Mosher, Sageev,
and Whyte \cite{MSW1}, \cite{MSW2}, and Papasoglu \cite{Pap}). 
Once Theorems~\ref{qi-preserve:thm} and~\ref{product:thm} are established, to deduce Theorem~\ref{qirigidity:thm} 
it is sufficient to ensure that a quasi-action on the universal cover of an irreducible graph manifold yields a genuine action on the Bass-Serre tree, and this follows
quite easily from the fact that walls and pieces are quasi-preserved. 

\begin{theorem}\label{qirigidity:thm}
Let $M$ be an irreducible graph $n$-manifold obtained by gluing the pieces
$V_i=\overline{N}_i\times T^{d_i}$, 
$i=1,\ldots, k$. Let $\Gamma$ be a group quasi-isometric
to $\pi_1 (M)$. Then either $\Gamma$ itself or a subgroup of $\Gamma$ of index two is isomorphic
to the fundamental group of a graph of groups satisfying the following conditions:
\begin{itemize}
\item
every edge group contains $\mZ^{n-1}$ as a subgroup
of finite index;
\item
for every vertex group $\G_v$ there exist $i\in\{1,\ldots, k\}$,
a finite-sheeted covering $N'$ of $N_i$ and a finite-index subgroup
$\G'_v$ of $\G_v$ that fits into the exact sequences
$$
\xymatrix{
1\ar[r] &\mZ^{d_i} \ar[r]^j & \Gamma_v' \ar[r] & \Delta \ar[r]& 1,\\
}
$$
$$
\xymatrix{
1\ar[r] & F \ar[r] & \Delta\ar[r] & \pi_1 (N')\ar[r] & 1 ,
}
$$
where $F$ is a finite group, and $j(\mZ^{d_i})$ is contained
in the center of $\Gamma'_v$.
\end{itemize}
\end{theorem}

\bigskip

As we mentioned at the beginning of this introduction, many of our rigidity results are inspired by
corresponding results in the theory of non-positively curved spaces and groups. 
We have already mentioned the fact that fundamental groups of irreducible graph manifolds are 
not relatively hyperbolic in general. 
We say that a group is CAT(0)
if it acts properly via semisimple isometries on a complete CAT(0) space (see Section~\ref{strirr:sub}). 
In Chapter \ref{construction2:sec} we show that fundamental groups of irreducible graph manifolds are not CAT(0) in general:

\begin{theorem}\label{existence:thm}
In each dimension $n\geq 4$, there are infinitely many examples of $n$-dimensional
irreducible graph manifolds having a non-CAT(0) fundamental group.
In particular, there exist infinitely many irreducible graph $n$-manifolds
which do not support any locally CAT(0) metric.
\end{theorem}

Finally, in Chapter \ref{open:sec}, we provide some concluding
remarks, and propose various open problems suggested by our work.

\include{Acknowledgments}

\mainmatter

\part{Graph manifolds: topological and algebraic properties}

%
%
%

\chapter{Quasi-isometries and quasi-actions}\label{quasiapp}

In this chapter we fix some notations we will extensively use in the rest of this monograph. 
We also list some well-known results about quasi-isometries and quasi-actions,
providing a proof for the strengthened version of Milnor-Svarc's Lemma
described in Lemma~\ref{milsv+:lem}. Such a result is probably well-known to experts,
but we did not find an appropriate reference for it in the literature. 

Let $(X,d),(Y,d')$ be metric spaces and $k\geq 1$, $c\geq 0$ be real numbers. 
A (not necessarily continuous) map $f\colon X\to Y$
is a $(k,c)$-\emph{quasi-isometric embedding} if for every $p,q\in X$ the following inequalities hold:
$$
\frac{d(p,q)}{k}-c\leq d'(f(p),f(q))\leq k\cdot d(p,q)+c.
$$
Moreover, 
a $(k,c)$-quasi-isometric embedding $f$ is a $(k,c)$-\emph{quasi-isometry} if there exists
a $(k,c)$-quasi-isometric embedding $g\colon Y\to X$ such that $d'(f(g(y)),y)\leq c$, 
$d(g(f(x)),x)\leq c$ for every $x\in X$, $y\in Y$. Such a map $g$ is called a \emph{quasi-inverse}
of $f$. It is easily seen that 
a $(k,c)$-quasi-isometric embedding $f\colon X\to Y$ is a $(k',c')$-quasi-isometry for some $k'\geq 1$, 
$c'\geq 0$ if and only if its image is
$r$-dense for some $r\geq 0$, \emph{i.e.}~if every point in $Y$ is at distance at most $r$ from
some point in $f(X)$ (and in this case $k',c'$ only depend on $k,c,r$).

\section{The quasi-isometry type of a group}
If $\Gamma$ is a group endowed with a finite system of generators $S$ such that $S=S^{-1}$, the \emph{Cayley graph}
$C_S (\Gamma)$ of $\Gamma$ is the geodesic graph
defined as follows: $C_S(\Gamma)$ has $\Gamma$ as set of vertices, 
two vertices $g,g'\in C_S(\Gamma)$
are joined by an edge if and only if $g^{-1}g'$ lies in $S$, and every edge has unitary length.
It is very easy to show that different finite sets of generators for the same group define quasi-isometric 
Cayley graphs, so every finitely generated group is endowed with a metric which is well-defined
up to quasi-isometry. 

\begin{remark}\label{easyrem}
Suppose $i\colon \Gamma_1\to \Gamma_2$, $j\colon \Gamma_2\to \Gamma_3$ 
are injective group homomorphisms between finitely generated groups, and let
$S_i$ be a finite system of generators for $\Gamma_i$, $i=1,2,3$. We may enlarge $S_2$ and $S_3$ in such a way that
 $i(S_1)\subseteq S_2$, $j(S_2)\subseteq S_3$. Under this assumption, both $i$ and $j$
are $1$-Lipschitz embeddings with respect to the word metrics defined via the $S_i$'s. 
Using this fact, it is not hard to show that
the composition $j\circ i$ is a quasi-isometric embedding if and only if both $i$ and $j$ are
quasi-isometric embeddings. 
\end{remark}

\section{The Milnor-Svarc Lemma}
The following fundamental result shows how the quasi-isometry type of a group is related to
the quasi-isometry type of a metric space on which the group acts geometrically.
A geodesic metric space $X$ is \emph{proper} if every closed ball in $X$ is compact.
An isometric action $\Gamma\times X\to X$ of a group $\Gamma$ on a metric space $X$ is
\emph{proper} if for every compact subset $K\subseteq X$
the set $\{g\in\Gamma\, |\, g\cdot K\cap K\neq \emptyset \}$ is finite, and \emph{cocompact}
if $X/\Gamma$ is compact.

\begin{theorem}[Milnor-Svarc Lemma]\label{milsv}
Suppose $\Gamma$ acts by isometries, properly and cocompactly on a proper
geodesic space $X$. Then $\Gamma$ is finitely generated and quasi-isometric to $X$, a quasi-isometry
being given by the map
$$
\psi\colon\Gamma\to X,\qquad \psi(\gamma)=\gamma(x_0),
$$
where $x_0\in X$ is any basepoint.
\end{theorem} 

As a corollary, if $M$ is a compact Riemannian manifold with Riemannian universal covering $\tilM$,
then the fundamental group of $M$ is quasi-isometric to $\tilM$. A proof of this result can be found in
\cite[Chapter I.8.19]{bri}, and we will prove a slightly more general version of the Lemma in the next section.

\section{From quasi-isometries to quasi-actions}\label{quasiact:sub}
Suppose $(X,d)$ is a geodesic metric space, let ${\rm QI} (X)$ be the set of quasi-isometries of $X$ into itself, and let $\Gamma$ be a group. 
For $k\geq 1$, 
a $k$-\emph{quasi-action} of $\Gamma$ on $X$
is a map $h\colon \Gamma\to {\rm QI} (X)$ such that the following conditions hold:
\begin{enumerate}
\item
$h(\gamma)$ is a $(k,k)$-quasi-isometry with $k$-dense image 
for every $\gamma\in\Gamma$;
\item
$d(h(1)(x),x)\leq k$ for every $x\in X$;
\item
the composition $h(\gamma_1)\circ h(\gamma_2)$ is at distance bounded by $k$ from the quasi-isometry
$h(\gamma_1 \gamma_2)$, \emph{i.e.}
$$
d\big(h(\gamma_1 \gamma_2)(x),h(\gamma_1)(h(\gamma_2)(x))\big)\leq k\quad {\rm for\ every}\ x\in X,\ 
\gamma_1,\gamma_2\in\Gamma .
$$
\end{enumerate}
A $k$-quasi-action $h$ as above is $k'$-\emph{cobounded} if every orbit of $\Gamma$ in $X$ is $k'$-dense.
A (cobounded) quasi-action is a map which is a ($k'$-cobounded) $k$-quasi-action for some $k,k'\geq 1$. 
Throughout the whole paper, by an abuse of notation, when $h$ is a quasi-action as above
we do not distinguish between $\gamma$ and $h(\gamma)$.

\begin{remark}\label{basepoint:rem}
If $h$ is a $k$-quasi-action as above, then 
for every $\gamma\in \Gamma$, 
$x_0,x_1,p\in X$ we have
$$ 
d(\gamma (x_1),p)\leq d(\gamma (x_1),\gamma (x_0))+d(\gamma (x_0),p)\leq k d(x_0,x_1) +k +d (\gamma (x_0),p) .
$$
Using this inequality, it is not difficult to show that if there exists a $k'$-dense orbit of $\Gamma$ in $X$, then
$h$ is $k''$-cobounded for some $k''$ (possibly larger than $k'$).
\end{remark}

Suppose $M$ is a geodesic metric space with metric universal covering $\tilM$, let $\Gamma$ be a finitely
generated group and suppose we are given a quasi-isometry
$\tilde{\varphi}\colon \Gamma\to\pi_1(M)$. We now briefly recall the well-known fact that
$\tilde{\varphi}$ naturally induces a cobounded quasi-action of $\Gamma$ on $\tilM$.

Let $\varphi\colon \Gamma\to\tilM$ be a fixed quasi-isometry provided by Milnor-Svarc's Lemma, and
let $\psi\colon \tilM\to\Gamma$ be a quasi-inverse of $\varphi$. 
For each $\gamma\in\Gamma$ we define a map $h(\gamma)\colon \tilM\to \tilM$ by setting
$$h(\gamma)(x)=\varphi(\gamma\cdot\psi(x))\qquad {\rm for\ every}\ x\in\tilM .$$
Since $h(1)=\varphi\circ\psi$, the map $h(1)$ is at finite distance from the identity of $\tilM$.
The left multiplication by a fixed element of $\Gamma$ defines an isometry
of any Cayley graph of $\Gamma$, so each $h(\gamma)$ is the composition of three quasi-isometries 
with fixed constants. In particular, it is a quasi-isometry and
its quasi-isometry constants can be bounded by a universal constant which only depends on
$\varphi$ and $\psi$, and is therefore independent of $\gamma$. As such, we have that
for every $\gamma\in\Gamma$ the map $h(\gamma)$ is a $(k,k)$-quasi-isometry
with $k$-dense image, where $k$ is some fixed uniform constant.
Moreover, it is easily seen that for each $\gamma_1, \gamma_2$, 
$h(\gamma_1\gamma_2)$ is at a finite distance 
(bounded independently of $\gamma_1,\gamma_2$) from $h(\gamma_1)\circ h(\gamma_2)$, that is, $h$ defines a quasi-action. 
Since every $\Gamma$-orbit in $\Gamma$ is $1$-dense, the quasi-action $h$ is clearly cobounded.

In Chapters~\ref{product:sec} and~\ref{qirigidity:sec} we need the following 
strengthened version of Milnor-Svarc's Lemma.

\begin{lemma}\label{milsv+:lem}
Let $X$ be a geodesic space with basepoint $x_0$, and let $\Gamma$ be a group.
Let $h\colon \Gamma\to {\rm QI} (X)$ be a cobounded quasi-action 
of $\Gamma$ on $X$, 
and suppose that 
for each $r>0$, the set $\{\gamma\in \Gamma\,|\, \gamma(B(x_0,r))\cap B(x_0,r)\neq \emptyset\}$ is finite.
Then $\Gamma$ is finitely generated and the map 
$\varphi\colon \Gamma\to X$ defined by $\varphi(\gamma)=\gamma(x_0)$ is a quasi-isometry.
\end{lemma}
\begin{proof}
The usual proof of Milnor-Svarc's Lemma works in this case too, up to minor changes. We will closely 
follow~\cite[Chapter I.8.19]{bri}. Suppose that $h$ is a $k$-cobounded $k$-quasi-action, and let us 
first prove that the finite set
$$\mathcal{A}=\{\gamma\in \Gamma\, |\, \gamma(B(x_0,2k^2+5k)\cap B(x_0,2k^2+5k)\neq \emptyset\}$$
generates $\Gamma$.
Fix $\gamma\in \Gamma$ and consider a geodesic $\alpha:[0,1]\to X$ joining $x_0$ with $\gamma(x_0)$. 
If $n\in\mathbb{N}$ is such that $d(x_0,\gamma (x_0))\leq n\leq d(x_0,\gamma (x_0))+1$, we can choose 
$0=t_0<\dots<t_n=1$ in such a way $d(\alpha(t_i),\alpha(t_{i+1}))\leq 1$ for each $i$. For each $t_i$ pick 
$\gamma_i$ so that $d(\alpha(t_i),\gamma_i(x_0))\leq k$, with $\gamma_0=1$ and $\gamma_n=\gamma$, 
and observe that $d(\gamma_i (x_0),\gamma_{i+1} (x_0))\leq 2k+1$ for $i=0,\ldots,n-1$. Since
\begin{align*}
d(x_0,(\gamma_i^{-1}\gamma_{i+1})(x_0)) & \leq d(\gamma_i^{-1}(\gamma_i (x_0)),\gamma_i^{-1}(\gamma_{i+1} (x_0)))+3k\\
& \leq kd(\gamma_i(x_0),\gamma_{i+1}(x_0) ) +4k \\
& \leq k (2k+1)+4k 
\end{align*}
we see that  $\gamma_i^{-1}\gamma_{i+1}\in\mathcal{A}$. This tells us that
$$\gamma=\gamma_0(\gamma_0^{-1}\gamma_1)\dots(\gamma_{n-1}^{-1}\gamma_n)$$
is a product of at most $d(x_0,\gamma(x_0))+1$ elements of $\mathcal{A}$. But $\gamma$ was 
chosen arbitrarily, so $\mathcal{A}$ is indeed a generating set for $\Gamma$.

Moreover, if $d_\mathcal{A}$ is the word metric with respect to $\mathcal{A}$, we have 
$d_\mathcal{A}(1,\gamma)\leq d(x_0,\gamma(x_0))+1$, and
for every $\gamma,\gamma'\in\Gamma$ we have
\begin{align*}
d_\mathcal{A}(\gamma,\gamma') = d_\mathcal{A}(1,\gamma^{-1}\gamma') &\leq 
d(x_0,(\gamma^{-1}\gamma')(x_0))+1\\
& \leq  d(\gamma^{-1} (\gamma (x_0)), \gamma^{-1}(\gamma' (x_0)))+3k+1\\
&\leq kd(\gamma(x_0),\gamma'(x_0))+4k+1
\end{align*}
which is one of the two inequalities needed to prove that $\varphi$ is a quasi-isometric embedding. 
For the reverse inequality, we first establish a useful inequality. For an arbitrary pair of elements 
$\gamma_1,\gamma_2$ in $\Gamma$, we have
the estimate:
\begin{align*}
d(\gamma_1(x_0),\gamma_2(x_0)) &= d\big(\gamma_1 (x_0),(\gamma_1\gamma_1^{-1})(\gamma_2 (x_0))\big)+k \\
& \leq d\Big(\gamma_1 (x_0),\gamma_1 \big(\gamma_1^{-1}(\gamma_2 (x_0))\big)\Big)+2k \\
& \leq  kd\big(x_0,\gamma_1^{-1}(\gamma_2(x_0))\big)+3k\\
& \leq  kd\big( x_0, (\gamma_1^{-1}\gamma_2)(x_0)\big)+k^2+3k \\
\end{align*}
Choose $\mu$ so that $d(x_0,a(x_0))\leq\mu$ 
for each $a\in \mathcal{A}$. Given any two elements $\gamma,\gamma'\in\Gamma$, let $n=d_\mathcal{A}(\gamma,\gamma')$
and write $\gamma^{-1}\gamma'=a_1\dots a_n$, 
where $a_i\in \mathcal{A}$. Set $g_0=1$, $g_i=a_1\dots a_i$, $i=1,\ldots,n$, so that
$g_n=\gamma^{-1}\gamma'$. From the above inequality, we see that $d(g_i(x_0),g_{i+1}(x_0))\leq k\mu+k^2+3k$ 
for every $i=0,\ldots,n-1$. Combining this estimate with the above inequality, we finally obtain
\begin{align*}
d(\gamma(x_0),\gamma'(x_0)) & \leq kd(x_0,g_n(x_0))+k^2+3k \\
& \leq k\Big( \sum_{i=1}^n d\big(g_{i-1}(x_0),g_{i}(x_0)\big)\Big)+k^2+3k \\
& \leq k(k\mu+k^2+3k) d_\mathcal{A}(\gamma,\gamma')+k^2+3k.
\end{align*}
We have thus proved that $\varphi$ is a quasi-isometric embedding,
and the fact that $h$ is cobounded now implies that it is in fact a quasi-isometry.
\end{proof}

%
%
%

\chapter{Generalized graph manifolds}\label{construction:sec}

Let us introduce the precise definition of high dimensional graph manifold.
Fix $n\geq 3$, $k\in \mN$ and $n_i\in\mN$ with $3\leq n_i\leq n$, and for every $i=1,\ldots, k$ let $N_i$ be a complete finite-volume
non-compact hyperbolic $n_i$-manifold 
with toric cusps. It is well-known that each cusp of $N_i$ supports a canonical smooth foliation by closed 
tori, which defines in turn a diffeomorphism between the cusp and $T^{n_i-1}\times [0,\infty)$, where
$T^{n_i-1}=\mR^{n_i-1} /\mZ^{n_i-1}$ is the standard torus. Moreover, the restriction of the hyperbolic metric to each leaf
of the foliation induces a flat metric on each torus, and there is a canonical affine diffeomorphism between any such two leaves.

We now ``truncate'' the cusps of $N_i$ by setting
$\overline{N}_i=N_i\setminus \cup_{j=1}^{a_i}
T_j^{n_i-1}\times (4,\infty)$, where $T_j^{n_i-1}\times [0,\infty)$, $j=1,\ldots,a_i$ are the cusps
of $V_i$. 
If $V_i=\overline{N}_i\times T^{n-n_i}$, then $V_i$ is a well-defined smooth manifold
with boundary, and as mentioned above the boundary of $V_i$ is endowed with a well-defined affine structure.
Moreover, the boundary of $V_i$ admits a collar which is canonically foliated by affine tori.

Let now $\mathcal{B}$ be a subset of the set of boundary components of the $V_i$'s, and suppose
that a pairing of the boundary components in $\mathcal{B}$ is fixed. 
We can construct a smooth manifold $M$ by gluing the $V_i$'s 
along affine diffeomorphisms between the paired tori in $\mathcal{B}$: the smooth
manifold $M$ obtained in this way is what we call a \emph{graph $n$-manifold}.
The manifolds $V_1,\ldots,V_k$ (which will be often considered as subsets of $M$ itself)
are called the \emph{pieces} of $M$. For every $i$, we say that $N_i$ (or $\overline{N}_i$) 
is the \emph{base} of $V_i$, while if $p\in\overline{N}_i$, then 
the set $\{p\}\times T^{n-n_i}\subseteq V_i$ is a \emph{fiber} of $V_i$. Abusing terminology,
we will sometimes also refer to $T^{n-n_i}$ as the fiber of $V_i$.
The toric hypersurfaces of $M$ corresponding to the tori in $\mathcal{B}$
will be called the {\it internal walls} of $M$ (so any two distinct pieces in $M$
will be separated by a collection of internal walls), while the components of $\partial M$ will be called the \emph{boundary walls} of $M$.
We say that $M$ is \emph{purely hyperbolic} if the fiber of every piece of $M$ is trivial, i.e.~if pieces of $M$
are just truncated complete finite-volume hyperbolic $n$-manifolds with toric cusps.

Observe that $M$ is closed (\emph{i.e.}~$\partial M=\emptyset$) if and only if
$\mathcal{B}$ coincides with the whole set of boundary components of the $V_i$'s.

\begin{remark}
The product
of an affine torus with a truncated hyperbolic manifold with toric cusps
provides the simplest example of graph manifold with non-empty boundary.
The quasi-isometry type of the fundamental group of
such a manifold will be studied in detail in Chapter~\ref{product:sec}.
\end{remark}

\begin{remark}
The simplest examples of closed graph manifolds are closed purely hyperbolic graph manifolds. Therefore, it makes sense to compare our rigidity results with the analogous
results described in~\cite{On} (for doubles of cusped hyperbolic manifolds), in~\cite{ArFa} (for twisted doubles of cusped hyperbolic manifolds),
and in~\cite{N} (for manifolds obtained by gluing locally symmetric negatively curved manifolds with deleted cusps).
\end{remark}

A restriction that we have imposed on our graph manifolds is that all
pieces  have a base which is hyperbolic {\it of dimension $\geq 3$}. The
reason for this restriction is obvious: hyperbolic manifolds of
dimension $\geq 3$ exhibit a lot more rigidity than surface groups. 
In fact, some of our results extend to a more general case, namely when surfaces with boundary are allowed as bases of pieces.

\begin{definition}
 For $n\geq 3$, an \emph{extended graph $n-$manifold} is a manifold built up from pieces as in the definition of graph manifold as well as \emph{surface pieces},
that is manifolds of the form $\overline{\Sigma} \times T^{n-2}$ with $\Sigma$ non-compact, finite volume,
hyperbolic surface. Also, we require that each gluing does not identify the fibers in adjacent surface pieces.
\end{definition}

Let us briefly comment about the last requirement described
in the above Definition. If we allowed gluings which identify the fibers
of adjacent surface pieces, then
the resulting decomposition into pieces of our extended graph manifold
would no longer be canonical. Indeed, within a surface piece $\Sigma \times T^{n-2}$, 
we can take any non-peripheral simple closed curve $\gamma \hookrightarrow \Sigma$ 
in the base surface, and cut the piece open along $\gamma \times T^{n-2}$. This
allows us to break up the original piece $\Sigma \times T^{n-2}$ into pieces 
$(\Sigma \setminus \gamma) \times T^{n-2}$ (which will either be two
pieces, or a single ``simpler'' piece, according to whether $\gamma$ separates
or not). Our additional requirement avoids this possibility. 
Note however that if one has
adjacent surface pieces with the property that
the gluing map matches up their fibers exactly, then it is not
possible to conclude that the two surface pieces
can be combined into a single surface piece (the resulting manifold could be a
non-trivial $S^1$-fiber bundle over a surface rather than just a product).

\begin{remark}\label{deepness:rem}
Let $N$ be the base of a piece of an (extended) graph manifold, and
suppose that $\overline{N}$ and $\overline{N}'$ are obtained as above by deleting from 
$N$ horospherical cusp neighbourhoods of possibly different ``heights''. Then, there exists a diffeomorphism
between $\overline{N}$ and $\overline{N}'$ which is coherent with the identification of $\partial \overline{N}$
and $\partial \overline{N}'$ induced by the canonical foliations of the cusps of $N$. In particular,
the diffeomorphism type of an (extended) graph manifold $M$ does not depend on the choice of the height of the cusps 
removed from the hyperbolic factors of the pieces into which $M$ decomposes.
\end{remark}

\begin{remark}\label{aff-diff:rem}
It is proved in~\cite{HsWa} that, if $n\geq 5$, then any diffeomorphism between affine $n$-dimensional tori is  
$C^0$-isotopic to an affine diffeomorphism. As a consequence, for $n\geq 6$, if we allow also non-affine gluings,
then we do not obtain new homeomorphism
types of (extended) graph manifolds. On the other hand, as showed in~\cite{ArFa}, requiring the gluings to be affine is necessary 
for getting smooth rigidity results as in our Theorem~\ref{smrigidity:thm} (i.e. non-affine gluings can give rise to new 
diffeomorphism types of manifolds).
\end{remark}

\section{Putting a metric on (extended) graph manifolds}\label{initial:subsec} 
By construction,
each hypersurface in $M$ corresponding to a boundary torus of some $V_i$ is
either a boundary component of $M$, or
admits a canonical smooth bicollar in $M$ diffeomorphic to $T^{n-1}\times [-3,3]$, 
which is obtained by gluing, according to the pairing of the boundary components
in $\mathcal{B}$, some subsets of the form $\partial V_i\times [1,4]$,
where $\partial V_i$ is canonically identified with $\partial V_i\times \{4\}$. 

In what follows, we will say that a point $p\in T^{n-1}\times \{-3\}$
is \emph{tied} to $q\in T^{n-1}\times \{3\}$ if $p=(x,-3)$, $q=(x,3)$
for some $x\in T^{n-1}$, \emph{i.e.}~if $p,q$ have the same ``toric'' component 
in the product space $T^{n-1}\times [-3,3]\subseteq M$.

The following
lemma shows how one can put on $M$ a Riemannian metric which somewhat extends the 
product metrics defined on the $V_i$'s.

\begin{lemma}\label{glumetric:lem}
Consider $A_1=T^{k}\times [-3,0]$ and $A_2=T^{k}\times [0,3]$,
each equipped with a Riemmanian metric $g_i$, and let $B_1=T^k\times
[-3,-2]$,
$B_2=T^k\times [2,3]$. Then there exists a Riemmanian metric
on $A=T^{k}\times[-3,3]$ such that $g|_{B_i}=g_i|_{B_i}$, $i=1,2$.
\end{lemma}
\begin{proof}
Let $\rho:[-3,3]\to[-3,3]$ be an odd $C^\infty$ function such that:
\begin{enumerate}
\item
$\rho|_{[2,3]}=id$,
\item
$\rho([1,2])=[0,2]$,
\item
$\rho|_{[0,1]}=0$.
\end{enumerate}

Also, let $\delta:[-1/2,1/2]\to[0,1]$ be an increasing
$C^\infty$ function which is constantly 0 (resp. 1) in a neighborhood of
-1/2 (resp. 1/2) and is strictly positive in $[0,1/2]$.
We can define $g$ as follows:

$$g(p,x)=\begin{aligned}
\left\{
\begin{array}{c c}
g_1(p,\rho(x))&\mathrm{for\ }x\in[-3,-1/2]\\
\delta(-x)g_1(p,0)+\delta(x)g_2(p,0)&\mathrm{for\ }x\in[-1/2,1/2]\\
g_2(p,\rho(x))&\mathrm{for\ }x\in[1/2,3]
\end{array}
\right.
\end{aligned}
$$
for all $p\in T^{k}, x\in [-3,3]$.
\end{proof}

 

From Lemma~\ref{glumetric:lem}
we get the following:

\begin{corollary}\label{glumetric:cor}
Suppose $M$ is an (extended) graph manifold, and let $U\subseteq M$ be the union of the bicollars of the internal walls of $M$.
Then $M$ admits a Riemannian metric $g$ which extends the restriction to $M\setminus U$
of the product metrics originally defined
on the pieces of $M$. 
\end{corollary}

\section{Purely hyperbolic graph manifolds are nonpositively curved}

In the case of purely hyperbolic manifolds, much more can be proved. In that case, the negatively curved Riemannian metrics defined
on the pieces of $M$ can be glued together into a non-positively curved Riemannian metric on the whole of $M$:

\begin{theorem}\label{CAT0:leeb}
Let $M$ be a purely hyperbolic graph manifold. Then $M$ supports a nonpositively curved Riemannian metric for which each component
of $\partial M$ is totally geodesic and flat.
\end{theorem}
\begin{proof}
Our proof is based on the Claim below, which deals with the extension of flat metrics on the boundary of a piece $V$ of $M$ to nonpositively curved metrics
on $V$. The Claim
 provides the $n$-dimensional analogue of~\cite[Proposition 2.3]{leeb}
(indeed, if $M$ is $3$-dimensional, then the theorem readily follows from~\cite[Theorem 3.3]{leeb}). 
Actually, the proof of~\cite[Proposition 2.3]{leeb} already works in any dimension. However, we prefer to recall it here with full details, both for the sake of completeness, and because
in higher dimensions a more precise statement is needed, which takes into account the fact that distinct affine structures on the boundary of $V$
may be non-equivalent via diffeomorphisms of $V$. 
In fact, the main result of~\cite{ArFa} implies that the theorem would be false if we allowed non-affine gluings between the pieces of $M$.

\smallskip\noindent

{\bf Claim:}
Let $V$ be a piece of $M$,  let $h$ be a flat metric on $\partial V$, and assume that the affine structure induced by $h$ on $\partial V$
coincides with the affine structure induced on $\partial V$ by the hyperbolic structure of $V$.  Then $h$  
extends to a nonpositively curved Riemannian metric on $V$, which is flat in a collar of $\partial V$.
\smallskip

Let $g$ be the original hyperbolic metric on the hyperbolic manifold $N=V\cup (\partial V\times [0,\infty))$. 
The cusps of $N$ are identified with the product $\partial V\times [0,\infty)$. On this set, the metric $g$ is isometric to
 a warped product metric
 $$
 e^{-2t}g_{\partial} + dt^2\ ,
 $$
 where $g_\partial$ is the flat metric induced by $g$ on $\partial V$. 
 It is now sufficient to modify $g$ into a smooth metric on $V\cup (\partial V\times [0,T_1])$ which coincides
 with $h$ on $\partial V\times \{T_1\}$ (up to the obvious identification between $\partial V$ and $\partial V\times \{T_1\}$), and is
 a product in $\partial V\times [T_1,T_2]$ for some $0<T_1<T_2$. In fact, after identifying $V$ with $V\cup (\partial V\times [0,T_2])$
 via a diffeomorphism which is affine on every boundary component, such a metric satisfies the properties described in the Claim.

Since the affine structures induced by $g_\partial$ and by $h$ on $\partial V$ coincide, the Spectral Theorem ensures that the tangent bundle of $\partial V$
(endowed with the flat structure induced by these coincident affine structures) 
admits a parallel frame which is orthonormal for $g_\partial$ and orthogonal for $h$. In other words, 
on every component of $\partial V$ we may choose local coordinates
$x_1,\ldots,x_{n-1}$ such that 
$$
g_\partial=dx_1^2+\ldots+dx_{n-1}^2\, ,\quad
h=a_1^2dx_1^2+\ldots+a_{n-1}^2dx_{n-1}^2\ ,
$$
where $a_i>0$ for every $i$.
To interpolate between
the conformal types of $g_\partial$ and $h$, we put on $\partial V\times [0,\infty)$ the metric
$$
e^{-2t}\, \left(\sum_{i=1}^{n-1} (\phi+(1-\phi)a_i)^2dx_i^2\right)\ ,
$$
where $\phi\colon [0,\infty)\to [0,1]$ is smooth, equal
to 1 in a neighborhood of 0 and equal to 0 in a neighborhood of $\infty$.

If the first and the second derivative of $\phi$ are small with respect to the $a_i$'s, then each
$e^{-t}(\phi +(1-\phi)a_i)$ is strictly monotonically descreasing and convex, and this implies in turn that the 
above metric is negatively curved.
Hence we can find a complete negatively curved metric on $V\cup (\partial V\times [0,\infty))$ 
which is negatively curved, and isometric to 
the warped product metric
$$
e^{-2t}h +dt^2
$$
on $\partial V\times [T_0,\infty)$ for a suitably chosen $T_0>0$. 
We now replace the factor $e^{-2t}$ 
by a convex and monotonically decreasing function $\psi\colon [T_0,\infty)\to \mathbb{R}^+$
which coincides with $e^{-2t}$ 
in a neighborhood
of $T_0$ and is constant in $[T_1,\infty)$.
The curvature of the resulting
complete metric is nonpositive because $\psi$ is convex. After rescaling, this metric is negatively curved, and
isometric to the product $h+ dt^2$ on $\partial V\times [T_1,T_2]$ for every $T_2>T_1$.
 This concludes the proof of the Claim.

\smallskip

Recall now that, by definition of graph manifold, the gluings defining $M$ are affine. Therefore,
every internal wall $T$ of $M$ can be endowed with a flat Riemannian metric $h_T$
whose induced affine structure coincides with the affine structures induced on $T$
by the hyperbolic structures on the two pieces adjacent to $T$.
Let now $V$ be a piece of $M$.
The Claim allows us to replace
the metric on $V$ with a nonpositively curved smooth metric which coincides with a product
metric in a neighborhhod of $T$ in $V$. By construction, such metrics on the pieces of $M$ glue into
a globally defined smooth nonpositively curved metric on $M$, which is totally geodesic and flat on each component of $\partial M$.
\end{proof}

In dimension 3, Leeb proved that an (extended) graph manifold supports a nonpositively curved Riemannian metric provided that 
it contains at least one purely hyperbolic piece~\cite[Theorem 3.3]{leeb}. However, Leeb's result does not extend to higher dimensions
(see Remark~\ref{nogeneral:rem}).

\section{$\pi_1(M)$ as the fundamental group of a graph of groups}\label{graphofgroups}
The decomposition of an (extended) graph $n$-manifold $M$ into pieces $V_1,\ldots,V_k$ induces on $\pi_1 (M)$
the structure of the fundamental group of a graph of groups $\mathcal{G}_M$ (see~\cite{serre} for the definition and some basic results on the fundamental group of a graph of groups). 
More precisely, let $\mathcal{G}_M$ be the graph of groups that describes
the decomposition of $M$ into the $V_i$'s, in such a way that
every vertex group is the fundamental group of the corresponding piece,
every edge group is isomorphic to $\mZ^{n-1}$, and the homomorphism
of every edge group into the group of an adjacent vertex is induced by the inclusion
of the corresponding boundary component of $V_i$ into $V_i$.
Then we have an isomorphism $\pi_1 (M)\cong \pi_1 (\mathcal{G}_M)$
(see e.g.~\cite{SW} for full details). 

Recall that cusps of hyperbolic manifolds are $\pi_1$-injective, so
every boundary component of $V_i$ is $\pi_1$-injective in $V_i$. This implies
that every piece (hence every boundary component of a piece) is $\pi_1$-injective
in $M$.

For later reference, we point out the following lemma, which can be easily deduced
from~\cite[Lemma D.2.3]{BePe}:

\begin{lemma}\label{basicprop:lem}
Let $N$ be a complete finite-volume hyperbolic $n$-manifold, $n\geq 3$. 
\begin{enumerate}
\item
Suppose that the cusps of $N$ are toric, and
that $\gamma$ is a non-trivial element of
$\pi_1 (N)$. Then, the centralizer of $\gamma$ in $\pi_1 (N)$ is free abelian.
\item
The center of $\pi_1 (N)$ is trivial.
\end{enumerate}
\end{lemma}

The following remark is an immediate consequence of Lemma~\ref{basicprop:lem}-(2).

\begin{remark}\label{center:rmk}
If $N$ is a complete finite-volume hyperbolic $n$-manifold and $d$ is a natural
number, then
the center of $\pi_1 (N)\times \mZ^d$ is given by $\{1\}\times\mZ^d$.
Therefore, if $V_i\cong\overline{N}_i\times T^d$ is a piece of $M$ and $p_i\colon V_i\to N_i$
is the natural projection, then the center of $\pi_1 (V_i)$ coincides with
$\ker (p_i)_\ast$. 
\end{remark}

\begin{definition}\label{fiberdef}
Let $V_i$ be a piece of $M$. Then the center of $\pi_1 (V_i)$ is called the 
\emph{fiber subgroup} of $\pi_1 (V_i)$. If $T$ is a component of $\partial V_i$,
we call \emph{fiber subgroup} of $\pi_1 (T)$ the intersection of $\pi_1 (T)$
with the fiber subgroup of $\pi_1 (V_i)$.
\end{definition}

\section{The universal cover of $M$ as a tree of spaces}\label{univ:subsec}
In this subsection we begin our analysis of the metric structure of the universal covering
$\tilM$ of $M$. We will be mainly interested in the study of the quasi-isometric
properties of $\tilM$. 

\begin{definition}\label{tree-of-spaces:def}
A \emph{tree of spaces} $(X,p,T)$ is a topological space $X$ equipped with 
a map $p$ on a (simplicial, but possibly not locally finite) tree $T$ with the following property: 
for any edge $e$ in $T$ and $t$ in the internal part ${e}^\circ$ of $e$, 
if $X_e=p^{-1} (t)$ then 
$p^{-1} ({e}^\circ)$ is homeomorphic to $X_e\times (0,1)$.
\end{definition}

\begin{definition}
Suppose $(X,p,T)$ is a tree of spaces where $X$ is a Riemannian manifold.  
An \emph{internal wall} of $X$ 
is the closure of the preimage under $p$ of the interior of an edge of $T$;
a \emph{boundary wall} of $X$ is simply a connected component
of $\partial X$. If $W$ is a (boundary or internal) wall of $X$,
 we will denote by $d_W$
the path metric induced on $W$ by the restriction to $W$ of the Riemannian structure of $X$.
A \emph{chamber} $C\subseteq X$ is the preimage under $p$ of a vertex
of $T$;  we will denote by $d_C$
the path metric induced on $C$ by the restriction to $C$ of the Riemannian structure
of $X$. Two distinct chambers of $X$ are \emph{adjacent} if the corresponding vertices
of $T$ are joined by an edge, while a wall $W$ is \emph{adjacent} to the chamber $C$
if $W\cap C\neq\emptyset$
(if $W$ is internal, then $W$ is adjacent to $C$ if and only if
the vertex corresponding to $C$ is an endpoint of the edge corresponding to $W$, while
if $W$ is a boundary wall, then $W$ is adjacent  to $C$ if and only if $W\subseteq C$).
\end{definition}

Let us now come back to our (extended) graph $n$-manifold $M$.
If $\dim N_i=n_i$, the universal covering of $\overline{N}_i$ is isometric
(as a Riemannian manifold)
to the complement $B_i$ in $\mathbb{H}^{n_i}$ of an equivariant family of open disjoint horoballs. 
Following Schwartz, we say that $B_i$ 
is a \emph{neutered space}. In the rest of this monograph, we will extensively use several features 
of neutered spaces (see for example Proposition~\ref{osin:prop} or Section~\ref{treegr:subsec},
where we will deduce asymptotic properties
of such spaces from the well-know fact that they are \emph{relatively hyperbolic} in the metric sense).

Since the fundamental group of each $\overline{N}_i$ and each $V_i$ injects
in the fundamental group of $\pi_1 (M)$, the universal coverings
$\widetilde{V}_i=B_i\times \mR^{n-n_i}$ 
embed into $\tilM$. 
Putting together this observation and  
Corollary~\ref{glumetric:cor} we get the following:

\begin{corollary}\label{treeofspaces:cor}
$M$ admits a Riemmanian metric such that $\tilM$ can be turned into a tree of spaces such that:
\begin{enumerate}
\item
If $C$ is a chamber of $\tilM$, then $(C,d_C)$ is isometric (as a Riemannian manifold) to $B\times \mR^k$,
where $B$ is a neutered space in $\mathbb{H}^{n-k}$.
\item
If $W$ is an internal wall of $\tilM$, then $W$ is diffeomorphic to $\mR^{n-1}\times [-1,1]$.
\item
If $W$ is a boundary wall of $\tilM$, then $W$ is isometric (as a Riemannian manifold) to $\mR^{n-1}$.
\end{enumerate}
\end{corollary}

We will call $B$ the \emph{base} of $C$, and $F=\mR^k$ the \emph{fiber} of $C$.
If $\pi_B\colon C\to B$, $\pi_F\colon C\to \mR^k$ are the natural projections, 
we will abuse the terminology, and also refer to a subset $F\subseteq C$ of the form 
$F=\pi_B^{-1}(x_0)$, where $x_0$ is a point in $B$, as a \emph{fiber} of $C$. A fiber of 
$\tilM$ is a fiber of some chamber of $\tilM$.

If $x,y\in C$, we denote by $d_B (x,y)$ the distance (with respect
to the path metric of $B$) between $\pi_B(x)$ and $\pi_B (y)$, and by
$d_F(x,y)$ the distance between $\pi_F(x)$ and $\pi_F(y)$
(so by construction $d_C^2=d_B^2+d_F^2$).

If $(\tilM,p,T)$ is the tree of spaces described in Corollary~\ref{treeofspaces:cor},
we refer to $T$ as to the \emph{Bass-Serre tree} of $\pi_1 (M)$ (with respect to the
isomorphism $\pi_1 (M)\cong\pi_1 (\mathcal{G}_M)$, or to the decomposition of $M$ into the $V_i$'s). The action of $\pi_1 (M)$
on $\tilM$ induces an action of $\pi_1 (M)$ on $T$.
By the very definitions,
(every conjugate of) the fundamental group of a piece (resp.~of a paired 
boundary component of a piece) coincides with the stabilizer of a vertex
(resp.~of an edge) of $T$, and vice versa. Also recall that the fiber subgroup
is normal (even central) in the fundamental group of a piece, so
it is well-defined as a subgroup of a vertex stabilizer.

\begin{lemma}\label{acyl-lem-0}
Let $M$ be an (extended) graph manifold and let $T$ be the Bass-Serre tree corresponding
to the decomposition of $M$ into pieces. Also set $G=\pi_1(M)$, and for every vertex $v$ (resp.~edge $e$) of $T$ denote by $G_v$ (resp.~$G_e$) the stabilizer of $v$ (resp.~$e$)
in $G$.
\begin{enumerate}
\item If $v$ is a vertex of $T$, then $v$ is the unique vertex fixed by $G_v$.
\item 
Let $W_1,W_2$ be distinct walls of $\widetilde{M}$, and let $v$ be a vertex
of $T$ such that any path joining $W_1$ with $W_2$ must intersect
the chamber corresponding to $v$. If
$g\in G$ is such that
$g(W_i)=W_i$ for $i=1,2$, then $g$ belongs to the fiber subgroup of $G_v$.
\item 
Let $W$ be a wall of $\widetilde{M}$, and denote by $H$ the (set-wise) stabilizer
of $W$ in $G$. Then $W$ is the unique wall
which is stabilized by $H$. 
\end{enumerate}
\end{lemma}
\begin{proof}
(1): If $G_v$ fixes another vertex $v'\neq v$, then
it fixes an edge $e$ exiting from $v$. This implies that
$G_v$ is contained in the stabilizer of an edge, which is clearly impossible
since edge stabilizers are abelian.

(2): Let $\widetilde{V}\subseteq \tilM$ be the chamber corresponding to 
$v$, and denote by $V$ the piece of $M$ corresponding to $\widetilde{V}$.
Our hypothesis implies that there exist two connected
components $Z_1,Z_2$ of $\partial \widetilde{V}$ such that
$g(\widetilde{V})=\widetilde{V}$, $g(Z_1)=Z_1$ and $g(Z_2)=Z_2$.
In particular we have $g\in G_v$.

Let us fix an identification of $G_v$ with 
$\pi_1(V)=\pi_1(N)\times \mZ^k$, where $N$ is the base of $V$.
Also denote by $\rho\colon G\to \pi_1(N)$ the projection map,
and recall that $\pi_1(N)$ acts on the universal covering
$\widehat {\mathbb H}^{n-k}$ of $N$, which
is a copy of hyperbolic space with a suitable $\pi_1(N)$-equivariant family of (open) horoballs removed. 
The boundary components of $\widetilde{V}$ are in natural bijection
with the boundary components of $\widehat {\mathbb H}^{n-k}$, and the action of
$g\in G$ on the components of $\partial \widetilde{V}$ can be detected
by looking at the action of $\rho(g)$ on the set of connected components of
$\partial (\widehat {\mathbb H}^{n-k})$. 
Therefore, 
$\rho(g)$ leaves two boundary components
of $\widehat {\mathbb H}^{n-k}$ invariant. This implies that $g$ pointwise fixes the unique minimal geodesic joining these boundary components. But the action
of $\pi_1(N)$ on $\widehat {\mathbb H}^{n-k}$ is free, so $\rho(g)=e$.
This means that $g$ belongs to the fiber subgroup of $G_v$, and concludes
the proof of point (2).

(3): 
Notice that $H$ is free abelian of rank $n-1$, where $n$ is the dimension of $M$.
Suppose that $H$ stabilizes a wall $W'\neq W$. By point (2), $H$  is contained
in the fiber subgroup of $G_v$ for some vertex $v$ of $T$. But the rank of $H$ is 
strictly
bigger than the rank of the fiber subgroup of $G_v$, a contradiction.
\end{proof}

\begin{lemma}\label{conj:lemma}
Set $G=\pi_1 (M)$. 
Let $V_1,V_2$ be pieces of $M$ and $T_i$ a component of $\partial V_i$, $i=1,2$.
Let $G_i< \pi_1 (M)$ (resp.~$H_i<\pi_1 (M)$)
be (any conjugate of) the fundamental group of $V_i$ (resp.~of $T_i$). 
Then:
\begin{enumerate}
 \item 
The normalizer of $G_1$ in $G$ is equal to $G_1$.
\item
If $G_1$ is conjugate to $G_2$ in $G$,
then $V_1=V_2$.
\item
The normalizer of $H_1$ in $G$ is equal to $H_1$.
\item
If $H_2$ is conjugate to $H_1$ in $G$,
then $T_1=T_2$ in $M$.
\item
If $g\in G$ is such that $G_1 \cap g G_1 g^{-1}\supseteq H_1$, then either $g \in G_1$ or
$V_1$ is glued to itself along $T_1$ in $M$.
\end{enumerate}
\end{lemma}
\begin{proof}
Let us consider the action of $G$ on the Bass-Serre tree $T$ corresponding
to the decomposition of $M$ into pieces. 

(1): By Lemma~\ref{acyl-lem-0}, there exists a unique
vertex $v_1$ such that $G_1=G_{v_1}$. If $g$ normalizes $G_1$, then
$G_1$ fixes $g(v_1)$, so $g(v_1)=v_1$ and $g\in G_1$.

(2): Let $v_1,v_2$ be the vertices of $T$ fixed respectively by $G_1,G_2$,
and suppose that there exists $g\in G$ such that $g G_1 g^{-1}=G_2$.
Then $G_{1}$ fixes both $v_2$ and $g(v_1)$, so $v_2=g(v_1)$.
Therefore, the covering automorphism $g\colon \widetilde{M}\to\widetilde{M}$
sends a chamber covering $V_1$ onto a chamber covering $V_2$, and $V_1=V_2$.

(3): 
By Lemma~\ref{acyl-lem-0}, there exists a unique
wall $W$ such that $G_1$ is the stabilizer of $W$ in $G$.
If $g$ normalizes $G_1$, then
$G_1$ stabilizes $g(W)$, so $g(W)=W$ and $g\in G_1$.

(4): 
Let $W_1,W_2$ be the walls of $T$ stabilized respectively by $G_1,G_2$
(see Lemma~\ref{acyl-lem-0}),
and suppose that there exists $g\in G$ such that $g G_1 g^{-1}=G_2$.
Then $G_{1}$ fixes both $W_2$ and $g(W_1)$, so $W_2=g(W_1)$.
Therefore, the covering automorphism $g\colon \widetilde{M}\to\widetilde{M}$
sends a wall covering $T_1$ onto a wall covering $T_2$, and $T_1=T_2$
in $M$.

(5): Let us suppose that $g\notin G_1$, and prove that $V_1$ is glued to itself along $T_1$.
Let $v_1,v_1'$ be the vertices of $T$ associated to $G_1$, $g G_1 g^{-1}$ respectively.
Since $g\notin G_1$ we have $v_1'\neq v_1$.
The assumption $G_1\cap g G_1 g^{-1}\supseteq H_1$ implies
that every element of $H_1$ fixes every edge of the injective path
joining $v_1$ with $v_1'$. Equivalently, if $C,C'$ are the chambers
corresponding to $v_1,v_1'$, then $g$ stabilizes every wall
which separates $C$ from $C'$. By Lemma~\ref{acyl-lem-0},
this implies that $C$ is adjacent to $C'$ along a wall stabilized
by $H_1$, whence the conclusion.
\end{proof}

\section{Basic metric properties of $\tilM$}
In this subsection we collect several metric properties of $\tilM$
that we will extensively use in the following chapters in order to study the quasi-isometry
type of the fundamental group of an (extended) graph manifold.

Recall from Corollary~\ref{treeofspaces:cor} that,
if $C$ is a chamber of $\tilM$, then $(C,d_C)$ is isometric to
the product of a neutered space with a Euclidean space.
An elementary application of Milnor-Svarc Lemma (see~Theorem~\ref{milsv}) implies the following:

\begin{lemma}\label{easywall:lem}
If $W$ is a wall of $\tilM$, then $(W,d_W)$ is quasi-isometric to $\mR^{n-1}$.
\end{lemma}

Also recall that $d$ denotes the distance associated to the Riemannian structure of $\widetilde{M}$.
For every $r\geq 0$ and $X\subseteq \widetilde{M}$, we denote by $N_r (X)\subseteq \widetilde{M}$
the $r$-neighbourhood of $X$ in $\widetilde{M}$, with respect to the metric $d$.

\begin{lemma}\label{comparedist:lem}
If $C$ is a chamber of $\tilM$,
then there exists a function $g:\mR^+\to\mR^+$ such that $g(t)$ tends to $+\infty$ as $t$ 
tends to $+\infty$ and $d(x,y)\geq g(d_{C}(x,y))$ for each $x,y\in C$.
\end{lemma}
\begin{proof}
By quasi-homogeneity of $C$ it is enough to prove the statement for a fixed $x$. 
Let us observe that $d$ and $d_C$ induce the same topology on $C$.
Take any 
sequence $\{y_i\}$ of points such that $d_{C}(x,y_i)$ tends to $+\infty$. Since $\tilM$ is proper,
if the $d(x,y_i)$'s are bounded, then up to passing to a subsequence we 
can suppose $\lim_{i\to\infty} y_i=y$ for some $y\in\tilM$.
But $C$ is closed in $\tilM$, so we have $y\in C$. It is easily seen that this contradicts $d_C(x,y_i)\to +\infty$.
\end{proof}

\begin{lemma}\label{prefacile2:lem}
Let $W_1,W_2$ be walls of $\tilM$,
and suppose that there exists $r\in\mR^+$ such that $W_1\subseteq N_r(W_2)$.
Then $W_1=W_2$. In particular, distinct walls of $\tilM$ lie at infinite Hausdorff
distance from each other.
\end{lemma}

\begin{proof}
Considering the realization of $\tilM$ as a tree of spaces,
one can easily reduce to the case that $W_1$ and $W_2$ are adjacent to the same chamber $C$.
By Lemma~\ref{comparedist:lem}, up to increasing $r$ we may assume that
$W_1$ is contained in the $r$-neighbourhood of $W_2$ with respect to
the path distance $d_C$ of $C$.

Let $C=B\times\mR^k$ be the decomposition of $C$ into the product of a neutered
space and a Euclidean space. Then, $W_1$ and $W_2$ project onto two horospheres
$O_1,O_2$ of $B\subseteq \mathbb{H}^{n-k}$, and $O_1$ is contained
in the $r$-neighbourhood of $O_2$ with respect to the distance
$d_B$. Now, the distance $d_B$ is bounded below by the restriction of the hyperbolic distance
$d_{\mathbb{H}}$ of $\mathbb{H}^{n-k}$, so $O_1$ is contained in the 
$r$-neighbourhood of $O_2$ with respect to the distance
$d_{\mathbb{H}}$. This forces $O_1=O_2$, whence $W_1=W_2$.
\end{proof}

\begin{corollary}\label{prefacile3:lem}
 Let $W$ (resp.~$C_1,C_2$) be a wall (resp.~two chambers) 
of $\tilM$. Then:
\begin{enumerate}
 \item 
if $W\subseteq N_r (C_1)$ for some $r\geq 0$,
then $W$ is adjacent to $C_1$;
\item
if $C_1\subseteq N_r (C_2)$ for some $r\geq 0$, then $C_1=C_2$; 
in particular, the Hausdorff distance between distinct chambers of $\tilM$
is infinite.
\end{enumerate}
\end{corollary}
\begin{proof}
 (1) By considering the realization of $\tilM$ as a tree of spaces, it is immediate
to realize that $W$ is contained in the $r$-neighbourhood of a wall
adjacent to $C_1$, so $W$ is adjacent to $C_1$ by Lemma~\ref{prefacile2:lem}.

(2) Suppose $W,W'$ are distinct walls both adjacent to $C_1$. Then,
by point~(1) they are adjacent also to $C_2$, and this forces
$C_1=C_2$.
\end{proof}

In order to study the quasi-isometry type of $\tilM$, it would be very useful
to know that the inclusions of walls and chambers are quasi-isometric embeddings.
However, this is not true in general, as it is shown in the proof of Proposition~\ref{notqi} below, where
we exploit this fact for constructing (extended) graph manifolds which do not support any CAT(0) metric.

In Chapter~\ref{strongirr:sec} we will define the class of \emph{irreducible} graph manifolds,
and we will prove that walls and chambers are quasi-isometrically embedded in the universal 
covering of an irreducible graph manifold. 

\section{Examples not supporting any locally CAT(0) metric}\label{noncat0-easy:subsec}
In this section we construct (extended) graph manifolds which do not support any locally CAT(0) metric. 
The construction described here is easy, and it is based on a straightforward application of the 
Flat Torus Theorem (see \emph{e.g.}~\cite[Chapter II.7]{bri}).
As mentioned in the Introduction, however, there are reasons for being interested in 
\emph{irreducible} graph manifolds (see Chapter \ref{strongirr:sec}). It turns out that providing examples 
of irreducible graph 
manifolds which do not support any locally CAT(0) metric is much harder. We will discuss this issue 
in detail in Chapter~\ref{construction2:sec}.

\begin{proposition}\label{notqi}
Let $n\geq 2$, and take a  
hyperbolic $n$-manifold $N$ with at least two cusps. We suppose as usual that every
cusp of $N$ is toric.  For $i=1,2$, let $N_i=N$ and $V_i=\overline{N}_i\times T^2$.
Then, we can glue the pieces $V_1$ and $V_2$ in such a way that the
resulting (extended) graph manifold $M$ does not support any CAT(0) metric.
\end{proposition}
\begin{proof}
Let $A,A'$ be two distinct boundary tori of $\overline{N}$, and let
$A_i\times T^2$, $A'_i\times T^2$ be the corresponding boundary tori of $V_i$.
We now glue $V_1$ to $V_2$ as follows:
$A_1\times T^2$ is glued to $A_2\times T^2$ with the identity, where $A_1,A_2$ are indentified
with $A$; $A'_1\times T^2$ is glued to $A'_2\times T^2$ by an affine map $\varphi$ 
such that $\varphi_\ast\colon \pi_1(A'_1\times T^2)\to \pi_1 (A'_2\times T^2)$
is given by $\varphi_\ast (\overline{a},c,d)=(\overline{a},c,c+d)$, where 
$\overline{a}\in\mZ^{n-1}$ and 
we are identifying 
$A'_i$ with $A'$, and $\pi_1 (A'_i\times T^2)=\pi_1 (A')\times \pi_1 (T^2)$
with $\mZ^{n-1}\oplus \mZ^2=\mZ^{n+1}$.

Let $M$ be the (extended) graph manifold obtained by the gluings just described. It is readily seen
that the natural projections $V_i\to\overline{N}_i$ define a projection $q\colon M\to D\overline{N}$,
where $D\overline{N}$ is the double of the natural compactification of $N$.
The map $q$ is a locally trivial fiber-bundle with fibers homeomorphic
to $T^2$. If $\gamma$ is the support of any simple curve joining the two boundary components
of $\overline{N}$, then the double $\alpha$ of $\gamma$ defines a simple loop in $D\overline{N}$.
Let $L=q^{-1} (\alpha)$. It is easily seen that 
$$
\pi_1 (L)\cong \langle x,y,z\, |\, yz=zy, xy=yzx, xz=zx\rangle \cong \mZ^2\rtimes _\psi \mZ,
$$ 
where if $x$ generates $\mZ$ we have $\psi (x)(y,z)=(y,y+z)$.
Moreover, if $L'$ is the intersection of $L$ with one component $Y$ of $\partial V_1=\partial V_2\subseteq 
M$, then $L'\cong T^2$, and $i\colon L'\to L$ induces 
an injective homomorphism $i_\ast\colon \pi_1 (L')\to \pi_1 (L)$ with 
$i_\ast (\pi_1 (L'))=\langle y,z\rangle$. It is well-known
(see \emph{e.g.}~\cite[III.$\Gamma$.4.17]{bri}) 
that $i_\ast$ is \emph{not}
a quasi-isometric embedding, so the inclusion of $\pi_1 (L')$ into $\pi_1 (M)$
is \emph{not} a quasi-isometric embedding (see Remark~\ref{easyrem}). 

On the other hand, 
since the inclusion $\pi_1 (L')\hookrightarrow \pi_1 (Y)$ is a quasi-isometric embedding,
if the inclusion $\pi_1 (Y)\hookrightarrow \pi_1 (M)$ were a quasi-isometric embedding,
then by Remark~\ref{easyrem} the inclusion $\pi_1(L')\hookrightarrow \pi_1 (M)$ would also be quasi-isometric,
while we have just proved that this is not the case.
Therefore, the inclusion
$\pi_1 (Y)\hookrightarrow \pi_1 (M)$ is also not a quasi-isometric embedding, and by the Milnor-Svarc Lemma,
this implies that there exist walls of $\tilM$ which are not quasi-isometrically
embedded in $\tilM$.

As a consequence, $M$ cannot support any locally CAT(0) metric: in fact,
due to Milnor-Svarc Lemma and the Flat Torus Theorem (see \emph{e.g.}~\cite[pg. 475]{bri}), 
if a compact manifold $M$ supports a locally CAT(0) metric
and $H<\pi_1 (M)$ is isomorphic to $\mZ^r$ for some $r\geq 1$, then
$H$ is necessarily quasi-isometrically embedded in $\pi_1 (M)$.
\end{proof}

We can exploit Proposition~\ref{notqi} to prove a portion of Theorem~\ref{existence:thm} in any dimension $n\geq 4$. Indeed, for 
every $n\geq 3$, there exists a cusped hyperbolic $n$-manifold with at least two cusps,
and whose cusps are all toric
(in fact, such manifolds fall into infinitely many distinct
commensurability classes, see \cite{MRS}).
Applying 
Proposition~\ref{notqi} and the rigidity results
proved in Chapters~\ref{pieces-iso:sec} and~\ref{smoothrig:sec}, we immediately deduce:

\begin{corollary}
For every $n\geq 4$, there exist infinitely many $n$-dimensional (extended) graph manifolds
which do \emph{not} support any locally CAT(0) metric.
\end{corollary}

\begin{remark}\label{nogeneral:rem}
 Let $n\geq 2$, and let us take two hyperbolic $n$-manifolds $N_1$, $N_2$ with more than two cusps,
and whose cusps are all toric. Also take an $(n+2)$-hyperbolic manifold $N_3$ with at least one cusp,
and whose cusps are all toric. (Such manifolds exist by~\cite{MRS}.) If $V_i=\overline{N}_i\times T^2$, $i=1,2$,
then we can glue $V_1$ to $V_2$ as described in Proposition~\ref{notqi}, thus getting an (extended) graph manifold $M$
such that $\pi_1(M)$ contains a subgroup isomorphic to $\mathbb{Z}^2$ which is not quasi-isometrically embedded.
We can now glue $\overline{N}_3$ to $M$, thus getting an (extended) graph manifold $M'$
such that $\pi_1(M')$ again contains a subgroup isomorphic to $\mathbb{Z}^2$ which is not quasi-isometrically embedded.
As a consequence $M'$, while containing a purely hyperbolic piece, does not support any locally CAT(0) metric.
This shows that~\cite[Theorem 3.3]{leeb} may not be extended to higher dimensions. 
\end{remark}




%




%



%


\chapter{Topological rigidity}\label{toprigidity:sec}

In this chapter, we will establish various topological results for (extended) graph manifolds. The main goal will be to
establish Theorem \ref{toprigidity:thm}, which we restate here for the reader's convenience.

\begin{Thm2}[Topological Rigidity] 

Let $M$ be an (extended) graph manifold (possibly with boundary), of dimension $n\geq 6$. Assume $M^\prime$ is
an arbitrary topological manifold and $\rho: M^\prime \rightarrow M$ is a homotopy equivalence which restricts to a 
homeomorphism $\rho|_{\partial M^\prime}: \partial M^\prime \rightarrow \partial M$ between the boundaries 
of the manifolds. Then $\rho$ is homotopic, rel $\partial$, to a homeomorphism $\bar \rho: M^\prime \rightarrow M$.

\end{Thm2}

This result will be deduced as a special case of a more general result. For a compact topological manifold $M$, we
will call a finite family $\{N_i\}$ of embedded codimension one submanifolds in the interior of $M$ 
a {\it topological 
decomposition} if each $N_i$ has a product neighborhood $E_i \cong N_i\times (-1,1)$, and
the submanifolds are all pairwise disjoint. The {\it complexity} of the decomposition will be
the size of the family $\{N_i\}$. Given a topologically embedded codimension one submanifold 
$N\hookrightarrow M$ with a product neighborhood, 
the open manifold $M\setminus N$ has two ends, which can each be compactified by adding a copy of $N$. We
will say that the resulting manifold with boundary is obtained from $M$ {\it by cutting along $N$}. Note that if we
have a topological decomposition $\{N_i\}$ of $M$, then cutting along one of the $N_i$ yields a new 
topological manifold
$M^\prime$, with a topological decomposition of complexity one less. As the process of cutting decreases
the complexity, this allows us to use inductive arguments in our proofs. 

If $\{N_i\}$ is a topological decomposition of the manifold $M$, we can repeatedly cut along the $N_i$ until
we obtain a manifold $M^\prime$ with an empty topological decomposition (i.e. complexity zero). Each connected
component $M_j$ of $M^\prime$ will be called a {\it piece}, and each $N_i$ will be called a {\it wall}. Note that $M$ can
be reconstructed from its pieces, by performing repeated gluings along the walls. Observe also that our high
dimensional (extended) graph manifolds obviously come equipped with a topological decomposition, 
given by letting $\{N_i\}$
consist of all its internal walls (in the graph manifold sense). The Borel conjecture for (extended) graph manifolds is 
then a consequence of the following more general result.

\begin{theorem}[Topological Rigidity - general case]\label{Borel-general}

Let $M$ be a compact manifold of dimension $n\geq 6$, with a 
topological decomposition $\{N_i\}$. 
Assume the following conditions hold:
\begin{enumerate}[(i)]
\item each of the pieces $\{M_j\}$ and each of the walls $\{N_i\}$ are aspherical,
\item each of the pieces $\{M_j\}$ and each of the walls $\{N_i\}$ satisfy the Borel Conjecture,
\item each of the inclusions $N_i\hookrightarrow M_j$ is $\pi_1$-injective, 
\item each of the inclusions $\pi_1(N_i)\hookrightarrow \pi_1(M_j)$ is square-root-closed,
\item the rings $\mathbb Z\pi_1(N_i)$ are all regular coherent, and
\item $Wh_k\big(\mathbb Z\pi_1(M_j)\big) =0$ for $k\leq 1$, and likewise for $\pi_1(N_i)$.
\end{enumerate}
Then the manifold $M$ also satisfies the Borel Conjecture.
\end{theorem}

In Section \ref{Asphere:sec}, we start by discussing asphericity of our (extended) graph manifolds. In Section
\ref{algKth:sec}, we establish vanishing results for the lower algebraic $K$-theory. In Section \ref{Borel:sec},
we will prove Theorem \ref{Borel-general}, and in Section \ref{proof-Borel-graphmanifold:sec}, we will 
use it to deduce Theorem \ref{toprigidity:thm}. Finally, in Section \ref{BCC:sec}, we point out that
the Baum-Connes Conjecture also holds for the (extended) graph manifolds, and mention some consequences.


\section{Contractible universal cover}\label{Asphere:sec}

A basic result in metric geometry implies that the universal cover of a closed CAT(0) manifold is
contractible, and hence that any such manifold is aspherical. We establish the analogue:

\begin{lemma}

Let $M$ be a compact topological manifold, with a 
topological decomposition $\{N_i\}$. Assume that each of the pieces $M_j$ and each of the
walls $N_i$ are aspherical, and that each inclusion $N_i \hookrightarrow M_j$ is $\pi_1$-injective.
Then $M$ is aspherical.

\end{lemma}

\begin{proof}

We argue by induction on the complexity $k$ of the topological decomposition of $M$. If $k=0$, then $M$ is formed 
from a single piece. By hypothesis, the piece is aspherical, which establishes the base case for our induction.

Now assume $M$ has topological decomposition of complexity $k>0$, and that the result holds whenever we
have such a topological decomposition of complexity $<k$. Let 
$N_i$ be an arbitrary wall in $M$, and cut $M$ open along $N_i$. There are
two cases to consider, according to whether $N_i$ separates $M$ into two components or not. We deal with the
case where $W$ separates $M$ into $M^\prime$ and $M^{\prime \prime}$ (the other case uses a similar reasoning).
The manifolds $M^\prime$, $M^{\prime \prime}$ come equipped with a topological decomposition
of complexity $<k$. The inductive hypothesis now ensures that
they are both aspherical. 

So $M$ is obtained by gluing together the two aspherical spaces $M^\prime$ and $M^{\prime \prime}$ along a 
common aspherical
subspace $N_i$. A result of Whitehead \cite{Wh} now asserts that $M$ is also aspherical, {\it provided} that each of the
inclusions $N_i\hookrightarrow M^\prime$, $N_i\hookrightarrow M^{\prime \prime}$ is $\pi_1$-injective. But this 
follows from the fact that all the walls lie $\pi_1$-injectively in the adjacent pieces. This completes the inductive step, 
and establishes the Lemma.

\end{proof}

Let us now specialize to the case of (extended) graph manifolds. We have that each piece $M_j$ is homeomorphic 
to the 
product $\overline{N} \times T^k$ where $\overline{N}$ is a finite volume hyperbolic manifolds with cusps cut off, 
and $T^k$ is a torus. Since both factors are aspherical, and a product of aspherical manifolds is aspherical, we 
see that the pieces are aspherical. Each wall is homeomorphic to a torus $T^{n-1}$, hence is also aspherical. 
Moreover, we know (see Section~\ref{graphofgroups}) that the embedding of a wall into a piece is always 
$\pi_1$-injective. So an immediate consequence of the Lemma is:

\begin{corollary}\label{Aspherical}
If $M$ is an (extended) graph manifold (possibly with boundary), then $M$ is aspherical.
\end{corollary}

\section{Lower algebraic K-theory}\label{algKth:sec}

In the field of high-dimensional topology, some of the most important invariants of a manifold $M$
are the (lower) algebraic $K$-groups of the integral group ring of the fundamental group. Obstructions
to various natural problems often reside in these groups, and in some cases, all elements in the
group can be realized as such obstructions. As a result, it is of some interest to obtain vanishing results for
the lower $K$-groups. We will focus on the following covariant functors:

\begin{itemize}

\item the {\it Whitehead group} of $M$, $Wh\big(\pi_1(M)\big)$, which is 
a quotient of the group $K_1\big(\mZ [\pi_1(M)]\big)$,

\item the reduced $K_0$-group, $\tilde K_0\big(\mZ [\pi_1(M)]\big)$, and

\item the lower $K$-groups, $K_i\big(\mZ [\pi_1(M)]\big)$ with $i\leq -1$.

\end{itemize}
To simplify notation, we define the functors $Wh_i$ (for
$i\leq 1$) from the category of groups to the category of abelian groups as follows:
$$Wh_i(\G) := 
\begin{cases}
Wh(\G) & i=1 \\
\tilde K_0\big(\mZ [\G]\big) & i=0 \\
K_i\big(\mZ [\G]\big) & i\leq -1 \\
\end{cases}
$$
Recall that a ring $R$ is said to be {\it regular
coherent} provided every finitely generated $R$-module has a finite-length resolution by finitely generated 
projective $R$-modules. The following Lemma is an immediate consequence of work of Waldhausen.

\begin{lemma} \label{WhVanish-general}
Let $\mathcal G$ be a graph of groups, with vertex groups $G_j$ and edge groups $H_k$, and let
$\Gamma$ denote the fundamental group $\pi_1(\mathcal G)$. Assume that
we have $Wh_i(G_j)=0$ and $Wh_i(H_k)=0$ for all $i\leq 1$ and all $j,k$. If the rings $\mathbb Z[H_k]$ are
all regular coherent, then $Wh_i(\Gamma)=0$ for all $i\leq 1$.
\end{lemma}

\begin{proof}

We proceed by induction on the number $k$ of edges in the graph of groups $\mathcal G$. If $k=0$, 
then $\Gamma \cong G$, where $G$ is the (single) vertex group in $\mathcal G$. By hypothesis, we 
have $Wh_i(\Gamma)=Wh_i(G) = 0$ for all $i\leq 1$. So we may now assume that $k>0$. Pick an 
arbitrary edge $e$ in $\mathcal G$, and consider the induced 
splitting of the group $\G$. There are two cases to consider: 

\begin{enumerate}

\item if the edge separates the graph $\mathcal G$ into two components, then $\G = \G_1*_H \G_2$ is an 
amalgamation of two groups $\G_1$, $\G_2$ over a subgroup $H$.

\item if the edge does {\it not} separate, then $\G = \G^\prime*_H$ is isomorphic to an HNN extension of 
$\G^\prime$ over a subgroup $H$.

\end{enumerate}

Moreover, $H$ is the group associated to the edge $e$, and $\G^\prime, \G_1, \G_2$ are fundamental groups of 
graphs of groups with $<k$ edges (and which satisfy the hypotheses of this Lemma). By induction, the
$Wh_i$ functors ($i\leq 1$) vanish on the groups $\G^\prime, \G_1, \G_2$. We explain Case (1) in detail, as 
the argument for Case (2) is completely analogous.

\vskip 10pt

Waldhausen has established \cite{Wa1}, \cite{Wa2} (see also Bartels and L\"uck \cite{BL}
and Connolly and Prassidis \cite{CP})
a Mayer-Vietoris type sequence for 
the functors $Wh_i$ of an amalgamation $\G=\G_1*_H \G_2$ (or of an amalgamation $\G=\G^\prime*_H$). 
Waldhausen's sequence requires an ``adjustment term'' to 
$Wh_i(\G)$, and takes the form:
\begin{equation}
\hskip -0.8in \ldots \rightarrow Wh_i(H) \rightarrow Wh_i(\G_1)\oplus Wh_i(\G_2) \rightarrow Wh_i(\G)/Nil_i
\end{equation}

$$
\hskip 1.2in  \rightarrow Wh_{i-1}(H) \rightarrow Wh_{i-1}(\G_1)\oplus Wh_{i-1}(\G_2) \rightarrow \ldots
$$
In the above sequence, the adjustment terms $Nil_i$ are called the {\it Waldhausen Nil-groups} 
associated to the amalgamation $\G_1*_H \G_2$.

For our specific amalgamation, the inductive hypothesis ensures that the 
terms involving the $\G_i$ and the $H$ all vanish. 
Hence the Waldhausen long exact sequence gives us an isomorphism
$Wh_i(\G) \cong Nil_i$ for $i\leq 1$. Now the Waldhausen Nil-groups for a general amalgamation are 
extremely difficult to compute. However, when the amalgamating subgroup $H$ has the property that
its integral group ring $\mZ [H]$ is regular coherent,
Waldhausen has shown that the Nil-groups all vanish (see \cite[Theorem 4]{Wa1}). 
This gives us $Wh_i(\G) \cong Nil_i =0$ for $i\leq 1$, concluding 
the inductive step in Case (1). In Case (2), we can apply an identical argument to the analogous long exact
sequence for $\G =\G^\prime *_H$:

\begin{equation}
 \ldots \rightarrow Wh_i(H) \rightarrow Wh_i(\G^\prime) \rightarrow Wh_i(\G)/Nil_i \rightarrow Wh_{i-1}(H) 
 \rightarrow Wh_{i-1}(\G^\prime) \rightarrow \ldots
\end{equation}

This completes the proof of the proposition.

\end{proof}

Next let us specialize to the case of (extended) graph manifolds.
As discussed in Section~\ref{graphofgroups}, the fundamental group of $M$ coincides 
with the fundamental group of a graph of groups. 
The vertex groups are the fundamental groups
of manifolds with boundary, whose interiors are homeomorphic
to the product of a finite volume hyperbolic manifold with a torus.
For such manifolds, Farrell
and Jones \cite{FJ2} established the vanishing of the $Wh_i$ functors ($i\leq 1$).
The edge groups are fundamental groups of codimension one tori. When $M$ is a closed manifold of 
non-positive sectional curvature of dimension
$n\geq 5$, it follows from work of Farrell and Jones \cite{FJ1} that $Wh_i\big(\pi_1(M)\big)=0$ for all $i\leq 1$. 
As a special case, $Wh_i(\mZ ^k)$ vanishes for $i\leq 1$, $k\geq 5$ (in fact, using work of 
Bass, Heller, and Swan \cite{BHS} one can establish this for all $k$).
 Moreover, it is an old result of Hall \cite{Ha} that the integral group ring of finitely generated 
free abelian groups are regular coherent. Applying the previous Lemma, we can immediately conclude:

\begin{corollary}[Lower $K$-groups vanish] \label{WhVanish}

Let $M$ be a (extended) graph manifold (possibly with boundary) and $\G = \pi_1(M)$. 
Then we have that $Wh_i(\G)=0$ for all $i\leq 1$.

\end{corollary}


\section{Topological rigidity - the general case}\label{Borel:sec}

Having established our preliminary results, we now turn to showing Theorem \ref{Borel-general}. 
We start with a compact topological manifold $M$, of dimension $\geq 6$, equipped with a 
topological decomposition $\{N_i\}$, and satisfying the following conditions:
\begin{enumerate}[(i)]
\item each of the pieces $\{M_j\}$ and each of the walls $\{N_i\}$ are aspherical,
\item each of the pieces $\{M_j\}$ and each of the walls $\{N_i\}$ satisfy the Borel Conjecture,
\item each of the inclusions $N_i\hookrightarrow M_j$ is $\pi_1$-injective, 
\item each of the inclusions $\pi_1(N_i)\hookrightarrow \pi_1(M_j)$ is square-root-closed,
\item the rings $\mathbb Z\pi_1(N_i)$ are all regular coherent, and
\item $Wh_k\big(\mathbb Z\pi_1(M_j)\big) =0$ for $k\leq 1$, and likewise for $\pi_1(N_i)$.
\end{enumerate}
Moreover, we have a homotopy equivalence 
$\rho: M^\prime \rightarrow M$ where $M^\prime$ is an arbitrary topological manifold, and $\rho$ restricts to a
homeomorphism from $\partial M^\prime$ to $\partial M$. Our goal is to find a homeomorphism
$\bar \rho: M^\prime \rightarrow M$ homotopic to $\rho$ (rel $\partial$).

The proof of the theorem will proceed by induction on $k$, the number of walls in the topological decomposition 
of the manifold $M$. The base case for our induction, $k=0$, corresponds to the case where $M$ consists of
a single piece $M_j$, and the theorem follows immediately from condition (ii).
So we may now assume that $k>0$, and choose an arbitrary  
wall $N$ from the topological decomposition of $M$. Recall
that this wall $N$ is a topologically embedded codimension one submanifold, and that
the embeddings $N\hookrightarrow M$ extends to an embedding $N\times (-1, 1) \hookrightarrow
M$, with the wall corresponding
to the subset $N \times \{0\}$. We may also assume that this neighborhood is disjoint from any of the other 
walls in the topological decomposition of $M$. As a first step, we want to homotope the homotopy equivalence 
$\rho$ to a continuous map $f:M^\prime \rightarrow M$ having the property that $f$ is topologically transverse to $N$.

\vskip 10pt

Since transversality in the topological category might 
not be familiar to most readers, we briefly recall some aspects of the theory.
Milnor developed in \cite{Mi} a bundle theory for the topological category. A {\it microbundle} over a 
space $B$ consists of a triple $\mathfrak{X}:=(E, i, j)$, where $E$ is a space, 
$i: B\rightarrow E$ and $j: E\rightarrow B$ are a
pair of maps with $j\circ i \equiv \text{Id}_B$ ($i$ is called the injection, $j$ is called the projection). 
Additionally, this triple must satisfy a {\it local triviality} condition:
around each point $p\in B$, there should exist open neighborhoods $p\in U$, $i(p) \in V$ satisfying $i(U)\subset V$
and $j(V)\subset U$, and a homeomorphism $\phi:V\rightarrow U\times \mathbb R^n$ so that the following diagram 
commutes:
$$
\xymatrix{
U\times \mathbb R^n  \ar[r] ^-{p_1} & U \\
U  \ar[u]^{\text{Id}\times \{0\}} \ar[r]_{i|_U} & V \ar[ul]_\phi \ar[u] _{j|_V}\\}
$$
Two bundles $\mathfrak{X}_1, \mathfrak{X}_2$ over $B$ are considered isomorphic if, after passing to smaller 
neighborhoods of the sets $i_1(B) \subset V_1^\prime$ and $i_2(B)\subset V_2^\prime$, one has a homeomorphism
$\psi: V_1' \rightarrow V_2'$ with the property that $i_2' = \psi \circ i_1'$ and $j_1'=j_2'\circ \psi$. In other words, 
one only cares about the local behavior near the subset $i(B)$ in $E$.

If one has a topological submanifold $N$ inside an ambient manifold $M$, we say the submanifold has a {\it normal
microbundle} $\mathfrak{n}= (E, i, j)$ if the space $E$ is a neighborhood of $N$ inside $M$, and $i$ is the
obvious inclusion of $N$ into $E$. A map 
$f: M'\rightarrow M$ is said to be {\it topologically transverse} to the bundle $\mathfrak{n}$ if it satisfies:
\begin{itemize}
\item $N' :=f^{-1}(N)$ is a topological submanifold inside $M'$,
\item the submanifold $N'$ has a normal microbundle $\mathfrak{n}' =(E', i', j')$, and
\item $f$ restricts to a topological microbundle map $f|_{E'}: E' \rightarrow E$ (i.e. restricts to an open topological embedding of each
fiber of $\mathfrak{n}'$ into a corresponding fiber of $\mathfrak{n}$).
\end{itemize}
A fundamental result of Kirby and Siebenmann is that one can always homotope a map to be transverse to a given
normal microbundle for a submanifold in the target (see 
\cite[Essay III, Theorem 1.1, pg. 85]{kirby-siebenmann}, 
along with Quinn \cite[Theorem 2.4.1]{Q} for the remaining cases). Moreover, the homotopy can be chosen
to have support in an arbitrarily small neighborhood of the preimage of $N$ (assuming all manifolds involved
are compact).

\vskip 10pt

Now returning to the proof of the Borel Conjecture, we observe that, by hypothesis, the wall $N$ comes equipped
with a canonical normal microbundle $\mathfrak{n}$, whose total space is given by the product neighborhood 
homeomorphic to $N\times (-1,1)$. Applying Kirby-Seibenmann, we know that one can homotope $\rho$ to a map
$f$ which is a topologically transverse to $\mathfrak{n}$. We would like to further ensure that the
resulting topologically transverse continuous map $f:M^\prime \rightarrow M$ have the additional property that
(a) $f$ restricts to a homotopy equivalence 
$f|_{f^{-1}(N)}:f^{-1}(N) \rightarrow N$, and (b) $f$ restricts to a homotopy equivalence from 
$M^\prime \setminus f^{-1}(N)$ to $M \setminus N$.
This question was studied by Cappell \cite{Ca}, who showed that there are two further 
obstructions to being able to do this: 

\begin{itemize}

\item an element in a suitable quotient group of $Wh\big(\pi_1(M)\big)$, and

\item an element in a group $\text{UNil}$ defined by Cappell, which depends on the decomposition of $\pi_1(M)$ as
an amalgamation over $\pi_1(N)$ (or on the expression of $\pi_1 (M)$ as an HNN-extension over $\pi_1 (N)$).

\end{itemize}

In view of our hypotheses, conditions (v) and (vi) allows us to appeal to Lemma \ref{WhVanish-general}, which ensures
that the first obstruction must vanish. To deal with the second obstruction, we use a result of Cappell \cite{Ca}
showing that the $\text{UNil}$ group vanishes provided the subgroup $\pi_1(N)$ is {\it square-root closed} in the 
group $\pi_1(M)$. Recall that a subgroup $H\leq G$ is $n$-root closed provided that for $g\in G$, $g^n\in H$ forces
$g\in H$. But it is a general result that, for a graph of groups, root closure of the edge groups in the adjacent vertex
groups implies that the edge group is root closed in the fundamental group of the graph of groups
(see the proof of Lemma \ref{root free}). Using our topological decomposition, we can realize $\pi_1(M)$ as the 
fundamental group of a graph of groups, with edge groups the $\pi_1(N_i)$ and vertex groups the  
$\pi_1(M_j)$. Our hypothesis (iv) then ensures that $\pi_1(N)$ is square-root closed in $\pi_1(M)$, and hence
forces Cappell's secondary obstruction in the $\text{UNIl}$ group to also vanish (as the later group is trivial).

\vskip 10pt

So from Cappell's work, we have now succeeded in homotoping the homotopy equivalence $\rho$ to a map $f$
with the property that $f$ is topologically transverse to the normal microbundle $\mathfrak{n}$ of $N$.  Moreover, the homotopy can be chosen to have support in a small neighborhood of $\rho^{-1}(N)$, and in particular, we have
that $f$ coincides with $\rho$ on $\partial M'$. Let $N^\prime = f^{-1}\big(N\big)$, an 
$(n-1)$-dimensional submanifold of $M^\prime$. By transversality, $N'$ has a neighborhood $E'$ 
which forms the total space of a normal microbundle $\mathfrak{n}'$ over $N'$, and the map $f$ induces
a topological microbundle map from $E'$ into the product neighborhood $E\cong N\times (-1,1)$ of $N$.
Since $N$ separates the product neighborhood $E$ into two components, $N'$ must likewise
separate its neighborhood $E'$ into two components. This forces the microbundle $\mathfrak{n}'$ to be isomorphic
to the trivial $1$-dimensional microbundle $N' \times \mathbb R$ over $N'$, so after possibly shrinking the 
neighborhood $E'$, we can assume that $E'$ is also a product neighborhood homeomorphic to $N' \times (-1,1)$. 
By further restricting the total spaces
of the microbundles $\mathfrak{n}'$ and $\mathfrak{n}$, we can assume that the restriction 
$f: E' \rightarrow E$ to the product neighborhood $E' \cong N' \times (-1,1)$ takes the form 
$f(x, t) = (f_N(x), t)$, where $f_N: N^\prime \rightarrow N$ denotes the restriction
of $f$ to $N^\prime$. 

We know from Cappell's property (a) that the map $f_N: N^\prime \rightarrow N$ is a homotopy equivalence. 
By assumption (ii), the manifold $N$ 
satisfies the Borel Conjecture, so there exists a homotopy $F: N^\prime \times [0,1/2] \rightarrow N$ where
$F|_{N^\prime \times \{1/2\}} \equiv f_N$ and $F|_{N^\prime \times \{0\}}: N^\prime \rightarrow N$ 
is a homeomorphism. Inserting this homotopy into the map $f$, we obtain a new map 
$\hat f: M^\prime \rightarrow M$ defined via:
$$
{\hat f}(x) := 
\begin{cases}
f(x) & x\in M^\prime \setminus \big(N^\prime \times [-1/2, 1/2]\big)\\
F(x,|t|) & x\in N^\prime , t\in [-1/2, 1/2] \\
\end{cases}
$$

Now consider cutting $M$ open along the submanifold $N$. There are two possibilities, according to 
whether the complement of the wall has one or two connected components. We focus on the first case, since the second 
case is completely analogous. We now have a new manifold $M_0:= M \setminus N$ with two open 
ends, and we denote by $\bar M$ the obvious compactification  of $M_0$ obtained by closing off each end by attaching 
a copy of $N$. The compact manifold $\bar M$ inherits a topological decomposition, with one fewer wall 
than the topological decomposition of $M$, but with two additional boundary components. Likewise, we can cut 
$M^\prime$ open along the submanifold $N^{\prime}$, resulting in a manifold 
$M^\prime_0 = M^\prime \setminus N^{\prime}$ with two open ends, and corresponding manifold 
with boundary $\bar M^\prime$ obtained from $M^\prime_0$ by compactifying both ends with a copy of 
$N^\prime$. Now the map 
$\hat f$ induces a map, which we denote $g_o$, from $M^\prime_0$ to $M_0$. From the specific form
of $f$ in the vicinity of the submanifold $N^\prime \subset M^\prime$, we see that $g_0$ obviously extends to 
a map $g: \bar M^\prime \rightarrow \bar M$ between the compactifications, which induces a homeomorphism
between the compactifying set $\bar M^\prime \setminus M^\prime _0$ (two copies of $N^{n-1}$)
and the compactifying set $\bar M \setminus M_0$ (two copies of $N$). By Cappell's property (b), $g_0$ is
a homotopy equivalence, and since we have obvious homotopy equivalences $\bar M^\prime \simeq M^\prime _0$ 
and $\bar M \simeq M_0$, we conclude that $g$ is also a homotopy equivalence. 

We now have that $\bar M$ is a manifold with a topological decomposition having $<k$ walls, 
and a homotopy equivalence $g:\bar M^\prime
\rightarrow \bar M$ which restricts to a homeomorphism from $\partial \bar M^\prime$ to $\partial \bar M$. From the
inductive hypothesis, we see that the map $g$ is homotopic, rel $\partial$, to a homeomorphism. Since the homotopy
leaves the boundaries unchanged, we can lift the homotopy, via the obvious ``re-gluing'' of boundary components, 
to a homotopy from $\hat f: M^\prime \rightarrow M$ to a new map $\bar \rho: M^\prime \rightarrow M$. Moreover, it is 
immediate that the map $\bar \rho$ is a homeomorphism, completing the inductive step, and concluding the proof of 
our Theorem \ref{Borel-general}.


\section{Topological rigidity - (extended) graph manifolds}\label{proof-Borel-graphmanifold:sec}

\vskip 10pt

In the last section, we proved Theorem \ref{Borel-general}, establishing the Borel Conjecture for a broad 
class of manifolds.
We now proceed to prove Theorem \ref{toprigidity:thm}, by checking that our (extended) graph manifolds satisfy all the 
hypotheses of Theorem \ref{Borel-general}. We need to verify the following
six conditions:
\begin{enumerate}[(i)]
\item each of the chambers $\{C_j\}$ and each of the walls $\{W_i\}$ are aspherical,
\item each of the chambers and walls satisfy the Borel Conjecture, 
\item each of the inclusions $W_i\hookrightarrow C_j$ are $\pi_1$-injective, 
\item each of the inclusions $\pi_1(W_i)\hookrightarrow \pi_1(C_j)$ are square-root-closed,
\item the rings $\mathbb Z\pi_1(W_i)$ are all regular coherent, and
\item $Wh_k\big(\mathbb Z\pi_1(C_j)\big) =0$ for $k\leq 1$, and likewise for $\pi_1(W_i)$.
\end{enumerate}
Conditions (i) and (iii) have already been verified (see the paragraph preceding Corollary \ref{Aspherical}), as have
conditions (v) and (vi) (see the paragraph preceding Corollary \ref{WhVanish}). The fact that the chambers
and walls satisfy the Borel Conjecture is due to Farrell and Jones (see \cite{FJ1}, \cite{FJ2}), so condition (ii) holds.
We verify condition (iv).

\begin{lemma}
If $C$ is a chamber in an (extended) graph manifold, and $W$ is any adjacent wall, then $\pi_1(W)$ is 
square-root closed inside $\pi_1(C)$.
\end{lemma}

\begin{proof}
From the product structure
$C = \overline{N}  \times T^k$ on the chambers, we have that $\pi_1(C)$ splits as a product $\pi_1(\overline{N})
\times \mZ ^k$, 
where $\overline{N}$ is a suitable finite volume hyperbolic manifold with cusps cut off, and the $\mZ^k$ comes from the 
torus factor. $W$ is a boundary component of $C$, hence splits as $\pi_1(Y) \times \mZ^k$, where $Y\subset 
\overline{N}$
is a boundary component of $\overline{N}$. It is immediate from the definition that $\pi_1(W)$ is square-root closed
in $\pi_1(C)$ if and only if $\pi_1(Y)$ is square-root closed in $\pi_1(\overline{N})$.

Using the induced action of $\pi_1(\overline{N})$ on the neutered space $B$ (see Section \ref{univ:subsec}), we 
can identify $\pi_1(Y)$
with the stabilizer of a boundary horosphere component $\tilde Y$ in $B$. Now assume that 
$g\in \pi_1(\overline{N})$ satisfies 
$g^2\in \pi_1(Y)$, but $g\not \in \pi_1(Y)$. Then $g^2$ maps $\tilde Y$ to itself, but $g$ maps $\tilde Y$ to some other
boundary component $\tilde Y^\prime \neq \tilde Y$, i.e. $g$ interchanges the two horospheres $\tilde Y^\prime$ and 
$\tilde Y$. Since these two horospheres are centered at different points at infinity, there is a unique minimal length
geodesic segment $\eta$ joining $\tilde Y^\prime$ to $\tilde Y$. But $g$ acts isometrically, and interchanges the
two horospheres, hence must leave $\eta$ invariant. This forces $g$ to fix the midpoint of $\eta$, contradicting the
fact that the $\pi_1(\overline{N})$ action on $B$ is free. We conclude that every $\pi_1(W)$ is square-root closed in each
adjacent $\pi_1(C)$.

\end{proof}

This completes the proof of Theorem
\ref{toprigidity:thm}, establishing the Borel Conjecture for (extended) graph manifolds.

\vskip 5pt

\begin{remark} \label{Borel-extend:rmk}
Nguyen Phan~\cite{N} introduced the class
of {\it cusp decomposable} manifolds. These manifolds are defined in a manner similar to our graph manifolds,
but have pieces which are homeomorphic to finite volume negatively curved locally symmetric spaces with the 
cusps truncated.
The walls are homeomorphic to infra-nil manifolds. It is
straightforward to check that these pieces and walls satisfy conditions (i)-(vi) in our generalized
Theorem. As such, the Borel Conjecture also holds for the class of cusp decomposable manifolds.

\end{remark}


\section{Baum-Connes Conjecture and consequences}\label{BCC:sec}

We conclude this chapter by discussing the Baum-Connes conjecture for fundamental groups of (extended) 
graph manifolds.
Recall that to any group $G$, one can associate it's reduced group $C^*$-algebra $C^*_r(G)$ (see Section
\ref{C-simple:sec} for the definition). For a torsion-free group, the {\it Baum-Connes Conjecture} predicts that 
the complex $K$-homology of the classifying space $BG$ coincides with the topological $K$-theory of 
$C^*_r(G)$. 
For a thorough discussion of this subject, we refer the reader to the book \cite{MV} or the survey article 
\cite{LuR}.
We will actually establish a somewhat stronger result known as the Baum-Connes conjecture
with coefficients (the latter has better inheritance properties).

A group $G$ is {\it a-T-menable} (or {\it Haagerup}) if one can find an affine isometric action of $G$ on some 
Hilbert space $\mathcal H$
with the property that for any point $x\in \mathcal H$ and bounded set $B\subset \mathcal H$, only finitely many 
group elements map $x$ into $B$. This notion is extensively discussed in the book  \cite{CCJJV}. 

\begin{lemma}

Let $\mathcal G$ be a graph of groups, with vertex groups $G_i$, and let
$\Gamma$ denote the fundamental group $\pi_1(\mathcal G)$. If all vertex groups $G_i$ are a-T-menable,
then $\Gamma$ satisfies the Baum-Connes Conjecture with coefficients.

\end{lemma}

\begin{proof}
Groups which are a-T-menable satisfy the 
Baum-Connes Conjecture with coefficients (see Higson and Kasparov \cite[Thm. 1.1]{HK}), and
if a graph of groups has vertex groups satisfying the Baum-Connes Conjecture with coefficients,
so does the fundamental group of the graph of groups (by work of Oyono-Oyono, see \cite[Thm. 1.1]{O-O}).
\end{proof}

Fundamental 
groups of finite volume hyperbolic manifolds are examples of a-T-menable groups. Extensions of a-T-menable
groups by amenable groups are still a-T-menable (see \cite[Ex. 6.1.6]{CCJJV}). This tells us that the fundamental 
groups of pieces in our graph manifolds are always a-T-menable. So we obtain the immediate:

\begin{corollary}[Baum-Connes conjecture]

For $M$ an (extended) graph manifold (possibly with boundary), $\pi_1(M)$ satisfies the Baum-Connes conjecture (with
coefficients).

\end{corollary}

A nice feature of the Baum-Connes conjecture is that it is known to imply several other well-known
conjectures. We explicitly mention three of these consequences which
may be of general interest. Throughout the rest of this section, we let $G$ denote the fundamental group
of an arbitrary (extended) graph manifold.

\begin{corollary}[Idempotent conjectures] The Kadison
Conjecture holds: the reduced $C^*$-algebra $C^*_r(G)$ has no non-zero idempotents.
As a consequence, the Kaplansky Conjecture also holds: 
the group algebra $\mathbb Q G$ has no non-zero idempotents.
\end{corollary}

\begin{corollary}[Gromov-Lawson-Rosenberg conjecture]  
Let $W$ be a closed, connected, smooth, spin manifold with $\pi_1(W) \cong G$. If
$W$ supports a Riemannian metric of positive scalar curvature, then the higher $\hat A$-genera
of $W$ all vanish.
\end{corollary}

\begin{corollary}[Zero-in-the-Spectrum conjecture]

Let $M$ be an (extended) graph manifold, equipped with an arbitrary Riemannian metric. Then there exists
some $p\geq 0$ so that zero lies in the spectrum of the Laplace-Beltrami operator $\Delta _p$ acting
on square-integrable complex valued $p$-forms on $\tilde M$ (the universal cover of $M$). 
\end{corollary}

\chapter{Isomorphisms preserve pieces}\label{pieces-iso:sec}
This chapter is devoted to the proof of Theorem~\ref{iso-preserve:thm}.
We recall the statement here for convenience.

\begin{Thm2}
Let $M_1$, $M_2$ be a pair of (extended) graph manifolds and let $\G_i=\pi_1(M_i)$ be their respective
fundamental groups.  Let $\La_1 \leq \G_1$ be a subgroup conjugate to the fundamental
group of a piece in $M_1$, and $\varphi\colon \G_1\rightarrow \G_2$ be an isomorphism. Then $\varphi(\La_1)$ is conjugate to the fundamental group $\La_2
\leq \G_2$ of a piece in $M_2$.
\end{Thm2}

Let us briefly describe the strategy of our proof. 
It is sufficient to provide a group-theoretic characterization of
fundamental groups of pieces for a generic (extended) $n$-dimensional graph manifold $M$.
We study the action of the fundamental group of $M$ on the Bass-Serre tree associated
to the decomposition of $M$ into pieces. We first describe the maximal subgroups
of $\pi_1(M)$ which are isomorphic to $\mZ^{n-1}$. In the case when $M$
is a graph manifold, these subgroups are just (conjugates of) the fundamental groups
of the boundary components of the pieces of $M$. From the point of view of the geometry
of $\widetilde{M}$, this implies
that, if $M$ is a graph manifold, then the stabilizers of walls of $\widetilde{M}$ admit an easy algebraic characterization.

In the general case things get more complicated, because
the fundamental groups of surface pieces contain many maximal
abelian subgroups of rank $n-1$. However, this fact allows us to provide
a group-theoretic characterization of (conjugates of) fundamental groups
of surface pieces. The algebraic description of the fundamental groups
of non-surface pieces requires more work, and it is based on the study
of the coarse geometry of non-surface chambers in $\widetilde{M}$. 
We first provide a coarse-geometric characterization of the fundamental groups 
of the boundary components of such pieces (as mentioned above, a much easier 
algebraic description of these subgroups is available in the case of graph manifolds). Since 
any chamber is coarsely approximated by the adjacent walls and
every group isomorphism is  a quasi-isometry, 
via Milnor-Svarc Lemma this implies
that every group isomorphism between the fundamental groups
of (extended) graph manifolds  quasi-preserves the fundamental
groups of non-surface pieces. Finally, a standard trick 
allows us to show that fundamental
groups of non-surface pieces are indeed preserved (rather than only quasi-preserved)
by any isomorphism, and this concludes the proof.



As already mentioned, in the case of graph manifolds the proof of Theorem~\ref{iso-preserve:thm} may be significantly simplified. We refer the reader to Remark~\ref{simplification} for a brief description of the shortcuts
that are available in that case.

\section{Some properties of wall stabilizers}
Let $M$ be an (extended) $n$-dimensional graph manifold. We set $\G=\pi_1(M)$, and we denote
by $T$ the Bass-Serre tree associated to the decomposition of $M$ into pieces.
Edge (resp.~vertex) stabilizers correspond to stabilizers
of internal walls (resp.~chambers) of $\widetilde{M}$. 
If $e$ is an edge (resp.~$v$ a vertex) of $T$, then we denote
by $\G_e$ (resp.~$\G_v$) the stabilizer of $e$ (resp.~$v$) in $\G$. 
In order
to prove Theorem~\ref{iso-preserve:thm} we need to provide a group-theoretic
characterization of vertex stabilizers of $T$.

We denote by $F(\G)$ the collection of maximal subgroups of $\Gamma$ which are isomorphic to $\mZ^{n-1}$ (the symbol $F(\G)$ is meant to
suggest that elements in $F(\G)$ behave somewhat like
$(n-1)$-dimensional flats -- note however that subgroups in $F(\G)$
may  be distorted in $\G$).
We will see in Corollary~\ref{char:graph:cor} that, in the case
when $M$ is a graph manifold, wall stabilizers
are exactly the elements of $F(\G)$. 
Unfortunately, this is not true when surface pieces are allowed.

\begin{lemma}\label{maximal:sub:lemma}
 A subgroup $H$ of $\G$ belongs to $F(\G)$ if and only if
it is a maximal subgroup isomorphic to $\mZ^{n-1}$ of the stabilizer
$\G_v$ of a vertex $v$ of $T$. 
\end{lemma}
\begin{proof}
Let $H\in F(\G)$. As $H$ is a finitely generated nilpotent group, a standard result about groups acting on a tree (see~\cite[Proposition 6.5.27]{serre}) guarantees that if $H$ does {\it not} stabilize a vertex, then there exists a geodesic $\gamma$ in $T$ that is invariant under the action of $H$. So we only need to prove that there is no such geodesic.
\par

The stabilizer $Stab(\gamma)$ of any geodesic has a subgroup 
$Fix(\gamma)$, with quotient isomorphic to either $1$, $\mathbb Z/2$,
$\mathbb Z$, or $\mathbb D_\infty$. So if $H=Stab(\gamma)$, then the subgroup
$Fix(\gamma) \leq H \cong \mZ^{n-1}$ is abstractly isomorphic to either 
(i) $\mZ^{n-1}$ or (ii) $\mZ^{n-2}$. From the properties of the action on the Bass-Serre tree, we know
that the subgroup which fixes a pair of adjacent edges, when thought of as a
subgroup of the common vertex group, is contained in the corresponding fiber
subgroup (see Lemma \ref{acyl-lem}). Since these 
fiber subgroups have rank $\leq n-2$, we see that (i) cannot occur. 

To see that (ii) cannot occur, we note that this would force {\it all} vertices on
the geodesic $\gamma$ to correspond to surface pieces. But we assumed that surface pieces have
fiber subgroups whose intersection has rank $\leq n-3$. Since $H$ would have
to be contained in this intersection, we again obtain a contradiction. This rules out
case (ii), thus showing that $H$ is contained in the stabilizer
$\G_v$ of a vertex $v$ of $T$. Being maximal among the
subgroups of $\G$ that are isomorphic to $\mZ^{n-1}$, the subgroup $H$ is 
maximal among the
subgroups of $\G_v$ that are isomorphic to $\mZ^{n-1}$.

Let now $v$ be a vertex of $T$, and suppose that $H$ is 
maximal among the
subgroups of $\G_v$ that are isomorphic to $\mZ^{n-1}$. Take a subgroup $H'<\G$
isomorphic to $\mZ^{n-1}$ and containing $H$. We distinguish two cases.

If $H$ fixes a vertex $w\neq v$, then
it fixes an edge $e$ of $T$ exiting from $v$.  Using that $H$
is maximal  among
the subgroups of $\G_v$ that are isomorphic to $\mZ^{n-1}$, it is readily seen that
$H$ coincides with the stabilizer of $e$. Moreover, $H'$ is contained in the normalizer
of $H$, so Lemma~\ref{conj:lemma} implies that $H'=H$. 

If $v$ is the unique vertex fixed by $H$, then $H'$ also fixes $v$, so
$H'$ is contained in $\G_v$, and $H'=H$ by maximality of $H$ among
the subgroups of $\G_v$ isomorphic to $\mZ^{n-1}$. 

In any case, we have shown that $H\in F(\G)$, and this concludes the proof.
\end{proof}

\begin{lemma}\label{first-char}
 If $H<\G$ is a wall stabilizer, then $H\in F(\G)$. On the other hand,
if $H\in F(\G)$, then:
\begin{enumerate}
 \item either $H$ is a wall stabilizer, or
\item there exists a unique vertex $v$ of $T$ which is fixed by $H$, and
this vertex corresponds to a surface piece of $M$.
\end{enumerate}
\end{lemma}
\begin{proof}
 It is immediate to check that a wall stabilizer is a maximal subgroup isomorphic to $\mZ^{n-1}$ of the stabilizer of a vertex of $T$, so the first statement
follows from Lemma~\ref{maximal:sub:lemma}.

Assume now that  $H\in F(\G)$ is not a wall stabilizer. 
In order to conclude we need to show that $H$ satisfies condition (2) of the statement.

By Lemma~\ref{maximal:sub:lemma}
we know that $H$ is contained in the stabilizer of
a vertex $v$ of $T$. Moreover, $v$ is the unique vertex fixed by $H$, 
because otherwise
$H$ would fix an edge of $T$, and by maximality it would coincide
with a wall stabilizer. 

Suppose by contradiction that the piece $V$ of $M$ corresponding to $v$
is not a surface piece.
We denote by
$N$ and $T^k$ respectively the hyperbolic and the toric factor of $V$, so that
$H$ is contained in a conjugate of $\pi_1(N\times T^k)<\pi_1(M)$.
For our purposes, we can safely assume $H<\pi_1(N\times T^k)$.
 Since $V$ is not a surface piece, we have $k\leq n-3$. The projection of $H$ on $\pi_1(N)$ is an abelian group of rank at least $n-k-1\geq 2$, and it is therefore contained in a cusp subgroup. By maximality,
this implies that  $H$ is a wall stabilizer, a contradiction. 
\end{proof}

The previous Lemma shows that, in the case when $M$ 
is a graph manifold, wall stabilizers admit an easy group-theoretic characterization:

\begin{corollary}\label{char:graph:cor}
 Suppose that $M$ is a graph manifold and let $H$ be a subgroup
of $\G$. Then $H\in F(\G)$ if and only if $H$ is a wall stabilizer.
\end{corollary}

\begin{remark}\label{simplification}
 As mentioned in the introduction of the Chapter, Corollary~\ref{char:graph:cor}
allows us to simplify the proof of Theorem~\ref{iso-preserve:thm} in the case of graph manifolds. In fact, the reader who is not interersted in the extended case
can safely skip all the material preceding Section~\ref{coarse:section},
with the exception of Lemma~\ref{avoid-hor} (where ${\rm Fl}(\overline{H})$,
for $H\in F(\G)$, may be replaced by the wall stabilized by $H$ -- see Corollary~\ref{char:graph:cor}).
\end{remark}

It will be useful to deepen our understanding
of groups in $F(\G)$  in the general case of extended graph manifolds. To this aim
we point out the following
geometric description
of elements in $F(\G)$.

\begin{lemma}
\label{fmstab}
 Each $H\in F(\G)$ is contained is the stabilizer of a flat ${\rm Fl}(H)$
in some chamber $C$ (with respect to the CAT(0) metric of $C$). Such flat is unique up to bounded Hausdorff distance. Moreover we may choose ${\rm Fl}(H)$ in such a way that:
\begin{itemize}
 \item either ${\rm Fl}(H)$ is a boundary component of $C$, 
\item or $C$ is a surface chamber, and ${\rm Fl}(H)=\gamma\times \mR^{n-2}$, 
where $\gamma$ is a geodesic contained in the base of $C$. 
\end{itemize}
\end{lemma}

\begin{proof}
If $H$ is a wall stabilizer, then the conclusion is clear, so 
by Lemma \ref{first-char} we may suppose that $H$ is contained in the stabilizer of a surface chamber $C$.
Let $V=\Sigma\times T^{n-2}$ be the piece of $M$ which is covered by $C$,
and let us fix an identification of $\pi_1(V)=\pi_1(\Sigma)\times \mZ^{n-2}$
with the stabilizer of $C$. It is immediate to check that, being a maximal
abelian subgroup of $\pi_1(V)$ of rank $n-1$, the group
$H$ decomposes as a product $H=\langle \alpha\rangle \times \mZ^{n-2}$,
where $\alpha$ is an indivisible element of $\pi_1(\Sigma)$.
We can now apply the Flat Torus Theorem, and notice that the flat 
associated to $H$ splits as claimed because of the product structure of $H$. 
\end{proof}


\section{Characterizing surface pieces}

\begin{lemma}
 \label{n-2int}
 Let $H_1,H_2\in F(\G)$ and suppose that $H_1\cap H_2=K$ is an abelian group
of rank $n-2$. Then:
\begin{enumerate}
 \item There exists a unique vertex $v$ of $T$ which is fixed by $H_1\cup H_2$. 
\item 
The vertex
$v$ corresponds to a surface piece $V$ of $M$, and $K$ coincides with the fiber
subgroup of (a conjugate of) $\pi_1(V)$.
\item
Let $\G_v$ be
the stabilizer of $v$ in $\G$. Then
$$
\G_v\, =\, \bigcup_{H\in F(\G),\, H\supseteq K} H\ .
$$ 
\end{enumerate}
\end{lemma}
\begin{proof}
(1): We know from Lemma~\ref{first-char} that there exist vertices $v_i$, $i=1,2$
such that $v_i$ is fixed by $H_i$. Let us denote by $T_K\subseteq T$ the
subset of $T$ fixed by $K$. It is well-known that $T_K$ is a subtree of $T$.
Moreover, we have $\{v_1,v_2\}\subseteq T_K$. We claim that the diameter
of $T_K$ is at most 2. In fact, if this is not the case, then there exist
3 consecutive edges $e_1,e_2,e_3$ of $T$ which are fixed by $K$. 
By Lemma~\ref{acyl-lem}, this implies that $K$ is contained in the fiber subgroups
of the stabilizers of two adjacent vertices of $T$.
But this contradicts the fact that adjacent surface pieces have
fiber subgroups whose intersection has rank $\leq n-3$. We have thus shown that 
the diameter of $T_K$ is at most 2. 
Observe now that the group $\langle H_1,H_2\rangle$ generated by $H_1\cup H_2$
centralizes $K$, so $\langle H_1,H_2\rangle$ leaves $T_K$ 
invariant. Since $T_K$ is bounded, this implies
that $\langle H_1,H_2\rangle$ fixes a point $v\in T_K$ (see e.g.~\cite[Corollary 2.8]{bri}). But $\G$ acts on $T$ without inversions, so we may assume that
$v$ is a vertex of $T$. 

In order to conclude the proof of (1) we are left to show that, if $v'$ is a vertex fixed
by $H_1\cup H_2$, then $v'=v$. However, if $v'\neq v$, then $\langle H_1\cup H_2\rangle$ fixes an edge $e$ of $T$, so $\langle H_1\cup H_2\rangle$
is abelian of rank at most $n-1$. By maximality, this forces
$H_1=\langle H_1\cup H_2\rangle=H_2$, against the fact that
the rank of $K=H_1\cap H_2$ is equal
to $n-2$.

(2): Let $V=N\times T^k$ be the piece of $M$ corresponding to $v$, 
where $N$ and $T^k$ denote respectively the hyperbolic base and the toric
factor of $V$.
The stabilizer $\G_v$ of $v$ in $\G$ is isomorphic to $\pi_1(V)=\pi_1(N)\times \pi_1(T^k)$, and the maximal abelian subgroups
of rank $n-1$ of $\pi_1(N)\times \pi_1(T^k)$ are given by the products
$J\times \mZ^k$, where $J$ varies among the maximal abelian subgroups
of $\pi_1(N)$ of rank $n-1-k$. Two such products intersect in the fiber
subgroup $\{1\}\times \mZ^k$, which has rank equal to $n-2$ if and only
if $k=1$, i.e.~if and only if $V$ is a surface piece.

(3): 
We first prove the inclusion $\subseteq$.
Let $V=\Sigma\times T^{n-2}$ be the piece of $M$ corresponding to $v$, 
where $\Sigma$ and $T^{n-2}$ denote respectively the hyperbolic base and the toric
factor of $V$, and let us fix an identification between $\G_v$ and $\pi_1(V)=\pi_1(\Sigma)\times \mZ^{n-2}$. Every element of $\pi_1(V)$ lies 
in a subgroup $J$ of the form $\langle \alpha\rangle \times \mZ^{n-2}$, where
$\alpha$ is an indivisible element of $\pi_1(\Sigma)$, and it is immediate
to check that $J$ is a maximal
abelian subgroup of rank $n-1$ in $\pi_1(V)$. 
Then the conclusion follows from Lemma~\ref{maximal:sub:lemma}.

Let now $H\in F(\G)$ be such that $K\subseteq H$. Then $H$
leaves invariant the subtree $T_K$ introduced in the proof of point (1).
As a consequence, $H$ fixes $v$, i.e.~$H\subseteq \G_v$, and this concludes
the proof of the Lemma.
\end{proof}

\begin{corollary}\label{surface-pieces-preserved}
 Let $\Lambda$ be a subgroup of $\G$. Then $H$ is (conjugate to) the fundamental
group of a surface piece of $M$ if and only if the following condition holds:
there exist elements $H_1,H_2\in F(\G)$ such that $K=H_1\cap H_2$ has rank $n-2$,
and $$
\Lambda =\, \bigcup_{H\in F(\G),\, H\supseteq K} H\ .
$$ 
\end{corollary}

The previous corollary implies that any isomorphism between the fundamental
groups
of extended graph manifolds must preserve (up to conjugacy)
the fundamental groups of surface pieces. In order to prove the same result
for the fundamental groups of non-surface pieces we need to develop the study of the
coarse geometry of chambers.

\section{Further properties of wall stabilizers}


\begin{lemma}
\label{avoid-hor}
 Let $C$ be a non-surface chamber in $\widetilde{M}$. Let $W,W'$ be distinct walls 
adjacent to $C$, let $H,H'$ be the stabilizers of $W,W'$, and
take $\overline{H}\in F(\G)\setminus \{H,H'\}$.
Then for every $D\geq 0$ there exist points 
$w\in W\cap C$, $w'\in W'\cap C$ which are joined by a path $\gamma\colon [0,l]\to C$
which avoids $N_D({\rm Fl}(\overline{H}))$. 
\end{lemma}

\begin{proof}
Let us set $F={\rm Fl}(\overline{H})$. Since $\overline{H}\notin \{H,H'\}$
we may suppose that $F$ is disjoint from $C\cup W\cup W'$.

We first assume that $F$ lies in the connected
component of $\widetilde{M}\setminus (W\cup W')$ containing $C$.
Then
there exists a wall $\overline{W}\neq W,W'$ adjacent to $C$
such that every path joining $F$ to $W\cup W'$
must pass through $\overline{W}$. Then any path which joins $W$ to $W'$
and avoids $N_D(\overline{W})$ also avoids $N_D(F)$. Therefore,
by Lemma~\ref{comparedist:lem}, if $D'$ is any given constant, then
it is sufficient to construct a path $\gamma$ joining $W$ to $W'$ and such that
$d_C(\gamma (t),\overline{W})\geq D'$ for every $t\in [0,l]$, where
$d_C$ is the path metric of $C$.

If $\pi\colon C\to B$
is the projection of the chamber $C$ on its base, then $\pi (W\cap C)=O$, 
$\pi (W'\cap C)=O'$ and 
$\pi (\overline{W}\cap C)=\overline{O}$ for distinct horospheres $O,O',\overline{O}$ of the neutered space $B\subseteq \mathbb{H}^k$. 
Let us fix an identification of $\mathbb{H}^k$ with the half-space model, in such a way that
$\overline{O}$ corresponds to a horosphere centered at the point at infinity. 
Since $k\geq 3$, it is now easy to show that for every sufficiently small $\varepsilon>0$ it is possible
to join a point in $O$ with a point in $O'$ by a rectifiable path supported
on the intersection of $B$ with the Euclidean horizontal hyperplane at height $\varepsilon$. 
In fact, this intersection is (homeomorphic to) $\mR^{k-1}$ with 
a countable family of open disjoint balls removed (recall that $k-1\geq 2$). Let 
$\gamma\colon [0,l] \to C$ be a lift to $C$ of such a path. 
It is clear that $d_{C} (\gamma (t),\overline{W})\geq D'(\varepsilon)$ for every $t\in [0,l]$,
where $D'(\varepsilon)$ tends to $+\infty$ as $\varepsilon$ tends to $0$.
This concludes the proof in the case when
$F$ lies in the connected
component of $\widetilde{M}\setminus (W\cup W')$ containing $C$.

Let us now suppose that $F$ and $C$ lie in distinct connected components of $\widetilde{M}\setminus (W\cup W')$. Then either every path joining $W$ with $F$ must
pass through $W'$, or every path joining $W'$ with $F$ must
pass through $W$. We assume that the second case holds, the first case
being symmetric. Under our assumptions, there exists a chamber $\overline{C}\neq C$
which is adjacent to $W$. We choose a fiber of $\overline{C}\cap W$,
and denote by $\overline{P}$ the corresponding affine subspace of 
$W\cap C$.
It is not difficult to show that there exists $D'\geq 0$ such that any path
in $C$ avoiding $N_{D'} (\overline{P})$ also avoids $N_D(F)$.
Therefore, by Lemma~\ref{comparedist:lem}, if $D''$ is any given constant, then
it is sufficient
to construct a path $\gamma$ joining $W$ to $W'$ and such that
$d_C(\gamma (t),\overline{P})\geq D''$ for every $t\in [0,l]$, where
$d_C$ is the path metric of $C$.

If $P$ is the fiber of $C$, we can choose $v\in P$ so $B\times\{v\}\subseteq C$ intersects $P$ in a proper subset of $W\cap (B\times\{v\})$. 
If $D''$ is large enough, then
any path in $B\times\{v\}$ that connects $W$ to $W'$ avoiding $N_{D''}(P\cap (B\times\{v\}))$  also avoids $N_{D''}(P)$, so we want to find such a path. Let us fix an identification of $B\times\{v\}$ with a neutered space $\hat{\mathbb{H}}^k$ in the half-space model of $\mathbb{H}^k$ (where $k\geq 3$) in such a way that $W\cap (B\times\{v\})$ corresponds to the horosphere $O$ centered at infinity. Let $O'$ be the horosphere corresponding to $W'\cap (B\times\{v\})$. For any neighborhood of finite radius $N$ of $P\cap (B\times\{v\})$ there is $\epsilon$ so that $N$ does not intersect the Euclidean horizontal hyperplane $H\subseteq \mathbb{H}^k$ at height $\varepsilon$, and also there is a path $\gamma$ in $\hat{\mathbb{H}}^k$ connecting $O$ to $H$ and not intersecting $N$. As $H\cap \hat{\mathbb{H}}^k$ is connected (see above) and $H$ intersects $O'$, we can then concatenate $\gamma$ and a path in $H\cap \hat{\mathbb{H}}^k$ whose final point is in $O'$, as required.
\end{proof}

For $H\in F(\G)$ we set 
$$\mathcal{I}(H)=\{H'\in F(\G)\, |\, {\rm rk}(H\cap H')\geq n-2\}\ .$$

\begin{proposition}\label{charwalls1}
Let $H\in F(\G)$. Then
$H$ is the stabilizer of a wall which is adjacent to at least one non-surface chamber
if and only if the following condition holds: 
\begin{itemize}
\item[(*)]
There exists
$H'\in F(\G)\setminus \mathcal{I}(H)$ such that, for every $H''\in F(\G)\setminus \{H,H'\}$
and $D\geq 0$, there exists a path joining $H$ to $H'$ which avoids $N_D(H'')$.
\end{itemize}
\end{proposition}
\begin{proof}
Suppose that $H$ is the stabilizer of a wall $W$ which is adjacent
to the non-surface chamber $C$. Let $W'$ be any other wall of $C$, and denote
by $H'$ the stabilizer of $W'$.
Since $C$ is not a surface chamber we have $H'\notin \mathcal{I}(H)$.
 We now take an arbitrary element
$H''\in F(\G)\setminus \{H,H'\}$, and we denote by
${\rm Fl}(H'')$ the flat associated to $H''$
by Lemma~\ref{fmstab}.
By Lemma~\ref{avoid-hor}, 
for every $D'\geq 0$ we may join $W$ to $W'$ by a path in $C$ which avoids
$N_{D'}({\rm Fl}(H''))$. By Milnor-Svarc Lemma, up to suitably choosing the constant $D'$, this path translates into a path in the Cayley
graph of $\G$ which joins $H$ to $H'$ and avoids $N_{D}(H'')$.

We have thus shown that, if
$H$ is the stabilizer of a wall which is adjacent to at least one non-surface chamber,
then condition (*) holds.

We now suppose that 
$H\in F(\G)$ is 
not the stabilizer of a wall which is adjacent to at least one non-surface chamber,
and show that condition (*) does not hold.
Let ${\rm Fl}(H)$ be the flat associated
to $H$ by Lemma~\ref{fmstab}. Our assumption on $H$ implies that one of the following possibilities holds:
\begin{itemize}
 \item[(a)]
either ${\rm Fl}(H)$ is contained in a wall $W$ which is not adjacent
to any non-surface chamber,
\item[(b)]
or ${\rm Fl}(H)$ is contained in a surface chamber $C$, and it is not at finite Hausdorff
distance from any boundary component of $C$.
\end{itemize}
In case (a) we denote by $\widehat{C}$ the union of the chambers
which are adjacent to $W$ (so $\widehat{C}$ is a surface chamber
if $W={\rm Fl}(H)$ is a boundary wall, and the union of two surface chambers
otherwise).
In case (b) we simply set $\widehat{C}=C$. 

Let us now take $H'\in F(\G)\setminus \mathcal{I}(H)$, and let 
${\rm Fl}(H')$ be the flat associated to $H'$. Since
$H'\notin \mathcal{I}(H)$ the flat ${\rm Fl}(H')$
is disjoint from $\widehat{C}$. Therefore, there exists a boundary component
$W''$ of $\widehat{C}$ such that every path joining ${\rm Fl}(H')$ with ${\rm Fl}(H)$
must pass through $W''$. Let $H''$ be the stabilizer of $W''$.
By construction $H''\in \mathcal{I}(H)$, so $H''\neq H'$. Moreover,
${\rm Fl}(H)$ is not at finite Hausdorff distance from $W''$, so
$H''\neq H$. Via Milnor-Svarc Lemma, the fact that every path joining ${\rm Fl}(H')$ with 
${\rm Fl}(H)$
must pass through $W''$ implies that there exists $D\geq 0$ such that 
every path joining $H'$ with $H$
must intersect $N_D(H'')$, so condition (*) is violated. 
\end{proof}

\begin{proposition}\label{charwalls2}
Let $H\in F(\G)$. Then
$H$ is the stabilizer of a wall which is adjacent to two surface chambers
if and only if 
there exist elements
$H_1,H_2$ of $\mathcal{I}(H)$ such that the rank of 
$H_1\cap H_2$ is strictly less than $n-2$.
\end{proposition}
\begin{proof}
Suppose that there exist elements
$H_1,H_2$ of $\mathcal{I}(H)$ such that the rank of 
$H_1\cap H_2$ is strictly less than $n-2$. 
Of course $H,H_1,H_2$ are pairwise distinct, so
by Lemma~\ref{n-2int} there exists
a unique vertex $v_i$ of $T$ which is stabilized by $H\cup H_i$. Moreover,
$v_i$ corresponds to a surface piece of $M$.
If $v_1=v_2=v$, then Lemma~\ref{n-2int} implies that $H\cap H_i$ coincides
with the fiber subgroup of $\G_{v}$, 
so $H_1\cap H_2$ contains an abelian subgroup of rank $n-2$,
a contradiction. 
Observe now that $H$ fixes the geodesic path $\gamma$ in $T$ joining $v_1$ to $v_2$. 
If the length of $\gamma$ is at least two, then Lemma~\ref{acyl-lem} implies
that $H$ is contained in the fiber subgroup of a vertex stabilizer,
which is impossible since ${\rm rk}\, H=n-1$. We have thus shown that $v_1$ and $v_2$
are joined by an edge $e$ of $T$. Moreover, $H$ stabilizes $e$, so by maximality
$H=\G_e$. Therefore, $H$ is the stabilizer of  
a wall which is adjacent to two surface chambers.

Suppose now that $H$ is the stabilizer of a wall $W$ which is adjacent to
the surface chambers $C_1,C_2$. By Lemma~\ref{n-2int} there exist
subgroups $H_1,H_2\in \mathcal{I}(H)$ such that $H_i$ stabilizes $C_i$,
and $H_i\cap H$ coincides with the fiber subgroup of the stabilizer of $C_i$.
But the stabilizers of adjacent surface chambers have
fiber subgroups whose intersection has rank $\leq n-3$,
so ${\rk}\, H_1\cap H_2\leq n-3$, and this concludes the proof of the Proposition.
\end{proof}

\section{Isomorphisms quasi-preserve non-surface pieces}\label{coarse:section}
In this section we show that fundamental groups of pieces are coarsely preserved by isomorphisms, and actually certain quasi-isometries as well. 

Let us come back to the notation of the statement of Theorem~\ref{iso-preserve:thm},
i.e.~let $\varphi\colon \pi_1(\G_1)\to\pi_1(\G_2)$ be an isomorphism between the fundamental groups of the (extended) graph manifolds $M_1,M_2$.

\begin{definition}\label{proper:defn}
Let $\tilM_1$, $\tilM_2$ be the universal coverings of $M_1$, $M_2$. 
We say that a wall of $\widetilde{M}_i$ is \emph{proper} if it is not a boundary
wall adjacent to a surface chamber.
\end{definition}

The group-theoretic characterizations
of stabilizers of proper walls provided by Propositions~\ref{charwalls1} and~\ref{charwalls2}
implies the following:

\begin{corollary}\label{iso-preserve-walls}
Let $H$ be a subgroup of $\G_1$. Then $H$ is the stabilizer of a proper wall $W$ of $\widetilde{M}_1$ if and only if $\varphi(H)$ is the stabilizer
of a proper wall of $\widetilde{M}_2$.
\end{corollary}

We observe that the stabilizers of non-proper walls may not be preserved by isomorphism: this phenomenon occurs, for example, for any group isomorphism
between $\pi_1(M_1)$ and $\pi_1(M_2)$, where $M_1=\Sigma_1\times S^1$
$M_2=\Sigma_2\times S^1$, and $\Sigma_1$ is a once-punctured torus, while
$\Sigma_2$ is a thrice-punctured sphere.

Let us now come back to the setting of Theorem~\ref{iso-preserve:thm}.
By Milnor-Svarc's Lemma, the isomorphism $\varphi\colon \pi_1(M_1)\to\pi_1(M_2)$ induces
a $(k,c)$-quasi-isometry $f\colon\tilM_1\to \tilM_2$. 
Corollary~\ref{iso-preserve-walls}, together with
the fact that the $\pi_1(M_i)$-orbits of the walls of $\tilM_i$
are in finite number, implies that there exists a constant $\lambda>0$ such that
 for every proper wall $W_1\subseteq \tilM_1$ the set 
$f(W_1)$ 
is at Hausdorff distance bounded by $\lambda$ from a proper wall $W_2\subseteq \tilM_2$
(the wall $W_2$ is unique in view of Lemma~\ref{prefacile2:lem}). 

The following result plays an important role in the proof both of Theorem~\ref{iso-preserve:thm} and 
of Theorem~\ref{qi-preserve:thm}.

\begin{proposition}\label{useful:prop}
Let $f\colon \tilM_1\to\tilM_2$ be a $(k,c)-$quasi-isometry
and let $g$ be a quasi-inverse of $f$. 
Suppose that one of the following conditions hold:
\begin{enumerate}
 \item either $M_1,M_2$ do not have surface pieces (i.e.~they are graph manifolds), and there exists $\lambda$ with 
the property that, for each wall $W_1$ of $\tilM_1$, 
there exists a wall $W_2$ of $\tilM_2$ with
the Hausdorff distance between $f(W_1)$ and $W_2$ bounded by $\lambda$;
also assume
that, up to switching the roles of $W_1$ and $W_2$, the same
property also holds for $g$;
\item or at least one $M_i$ contains at least one surface piece, and the quasi-isometry $f$ is induced by an isomorphism
between $\G_1$ and $\G_2$.
\end{enumerate}
Then there exists a universal constant $H$ with the property that, for every non-surface chamber $C_1\subseteq \tilM_1$,
there exists a unique non-surface chamber $C_2\subseteq \tilM_2$ with the Hausdorff distance between $f(C_1)$ 
and $C_2$ bounded by $H$. Moreover, if $W_1$ is a wall adjacent to $C_1$ then $f(W_1)$
lies at finite Hausdorff distance from a wall $W_2$ adjacent to $C_2$.
\end{proposition}
\begin{proof}
Let us fix a non-surface chamber $C_1$ of $\tilM_1$, and let
$W_1,W'_1$ be walls adjacent to $C_1$.
Condition (1) in the statement (in the case when $M_1,M_2$ are graph manifolds)
or condition (2) and Proposition~\ref{charwalls1} imply
that
there exist proper walls $W_2,W'_2$ of $\widetilde{M}_2$ such that $f(W_1)$ and $f(W'_1)$
lie within finite Hausdorff distance respectively from
$W_2$ and $W'_2$ (such walls are uniquely determined -- see Lemma~\ref{prefacile2:lem}). 
We first prove that a non-surface chamber $C_2$ exists such that
$W_2$ and $W'_2$ are both adjacent to $C_2$.
  
Suppose by contradiction 
that there exists a wall $P_2\subseteq \tilM_2$ such 
that $P_2\neq W_2,W'_2$, and every continuous path connecting $W_2,W'_2$ 
intersects $P_2$. Then $P_2$ is proper, so 
there exists a wall $P_1\subseteq \tilM_1$ such that
$f(P_1)$ is at Hausdorff distance at most $\lambda$ from $P_2$ (just take $P_1$ to be
the wall at bounded distance from $g(P_2)$). Now it is not difficult to show that,
since $f$ and $g$ are quasi-isometries, the fact that $P_2$ separates $W_2$ from $W'_2$
implies that $P_1$ coarsely separates $W_1$ from $W'_1$: in other words, there exists
$D>0$ such that every path joining $W_1$ with $W_1'$ must intersect $N_D(P_1)$.
However, this contradicts  Lemma~\ref{avoid-hor}. We have thus shown
that $W_2$ and $W'_2$ are adjacent to a chamber $C_2$ of $\widetilde{M}_2$.

Suppose by contradiction that $C_2$ is not a surface chamber.
Then there exists a subgroup $H_2\in F(\G_2)$
such that ${\rm Fl}(H_2)$ separates $W_2$ from $W_2'$, and it is not at finite Hausdorff
distance from $W_2$ nor from $W_2'$. Since we are supposing that $f$ is induced by
an isomorphism $\G_1\cong \G_2$, this means that 
there exists a subgroup $H_1\in F(\G_1)$
such that ${\rm Fl}(H_1)$ coarsely separates $W_1$ from $W_1'$.
But this contradicts Lemma~\ref{avoid-hor}, so $C_2$ cannot be a surface chamber.

Let us now prove that $C_2$ lies at a universally bounded Hausdorff distance
from $f(C_1)$.
Since walls are $h$-dense in  
$\tilM$ for some $h>0$, for every $p_1\in C_1$ there exists $p'_1\in W_1$ with $d(p_1,p_1')\leq h$, where
$W_1$ is a wall adjacent to $C_1$. Then 
$$d(f(p_1),C_2)\leq d(f(p_1),f(p_1'))+d(f(p_1'),C_2)\leq kh+c+\lambda.$$
This tells us that $f(C_1)$ is contained in the $(kh+c+\lambda)$-neighbourhood of $C_2$.
Let $g$ be the quasi-inverse of $f$.
The same argument shows that $g(C_2)$ is contained in the $(kh+c+\lambda)$-neighbourhood of 
some non-surface chamber $C_1'$, and Lemma~\ref{prefacile3:lem} implies that $C_1'=C_1$. Now, if $q_2\in C_2$ we have 
$d(q_2,f(g(q_2)))\leq c$, and there exists $q_1\in C_1$ with $d(g(q_2), q_1)\leq kh+c+\lambda$. 
We now can estimate the distance 
\begin{align*}
d(q_2,f(q_1)) &\leq d(q_2,f(g(q_2)))+d(f(g(q_2)),f(q_1))\\
& \leq c+kd(g(q_2),q_1)+c \\
& \leq 2c+k(kh+c+\lambda).
\end{align*} 
So we can set $H=k^2h+(k+2)c+k\lambda$, and we are done. Finally, the uniqueness of $C_2$ is a 
consequence of Lemma~\ref{prefacile3:lem}.
\end{proof}

\section{Isomorphisms preserve pieces}\label{iso-preserve:sub}

We are ready to establish Theorem~\ref{iso-preserve:thm}. 

\begin{proof}
Let $\Lambda_1<\G_1$ be the fundamental group of a piece $V_1$
of $M_1$. If $V_1$ is a surface piece, then 
Corollary~\ref{surface-pieces-preserved} ensures that
$\varphi(\Lambda_1)$
is (conjugated to) the fundamental group of a non-surface piece
of $M_2$.

Otherwise, by Proposition~\ref{useful:prop}
and the Milnor-Svarc Lemma, 
the Hausdorff distance between $\varphi(\Lambda_1)$ and $g\Lambda_2 g^{-1}$ is bounded
by $H$ for some fundamental group of a non-surface piece $\Lambda_2<\Gamma_2$
and some $g\in\Gamma_2$.
Up to conjugation, and increasing $H$ by $d(g,id)$, we may assume $g=id$. 

A standard argument now allows us to prove that $\varphi(\Lambda_1)=\Lambda_2$.
In fact, if $h\in\Lambda_1$ we have that 
$$
\varphi (h)\cdot \varphi (\Lambda_1)=\varphi (h\cdot \Lambda_1)=\varphi (\Lambda_1).
$$
Since $\varphi(\Lambda_1)$ is at bounded Hausdorff distance from $\Lambda_2$, this implies that $\varphi (h)\cdot \Lambda_2$
is at bounded Hausdorff distance from $\Lambda_2$. By Milnor-Svarc's Lemma, if
$C_2$ is the chamber of $\tilM_2$ that is fixed by $\Lambda_2$, 
then the chamber $\varphi(h)(C_2)$
is at finite Hausdorff distance from $C_2$. By Lemma~\ref{prefacile3:lem}
this implies in turn that $\varphi(h)(C_2)=C_2$, so
$\varphi (h)\in \Lambda_2$, and 
$\varphi (\Lambda_1)\subseteq \Lambda_2$.
Finally, 
since $\varphi^{-1}$ is a quasi-inverse of $\varphi$, we have that $\varphi^{-1} (\Lambda_2)$ stays at finite
distance from $\Lambda_1$. The above argument again shows that $\varphi^{-1}(\Lambda_2)\subseteq \Lambda_1$.
We conclude that $\varphi (\Lambda_1)=\Lambda_2$, completing the proof of Theorem~\ref{iso-preserve:thm}.
\end{proof}

\vskip 10pt

Putting together 
Theorem~\ref{iso-preserve:thm} and Lemma~\ref{conj:lemma} one can easily refine
the statement of Theorem~\ref{iso-preserve:thm} as follows:

\begin{theorem}\label{preserve2:thm}
Let $M$, $M'$ be a pair of graph manifolds
which decompose into pieces $V_1,\ldots,V_h$, and $V'_1,\ldots,V'_k$ respectively. 
Suppose that $\varphi\colon \pi_1 (M)\to\pi_1 (M')$ is an isomorphism.
Then $h=k$ and, up to reordering the indices, for every
$i=1,\ldots,h$ the image of $\pi_1 (V_i)$
under $\varphi$ coincides with a conjugate of $\pi_1 (V'_i)$.
Moreover, with this choice of indices 
$V_i$ is adjacent to $V_j$ if and only if $V'_i$ is adjacent to $V'_j$.
\end{theorem}

\chapter{Smooth rigidity}\label{smoothrig:sec}

This chapter is devoted to the proof of Theorem~\ref{smrigidity:thm}, which we recall here for the convenience of the reader:
\begin{Thm2}
Let $M,M'$ be (extended) graph manifolds, and let
$\varphi\colon \pi_1(M)\to \pi_1(M')$ be a group isomorphism. Suppose that the boundaries of $M,M'$ do not intersect any 
surface piece. Then $\varphi$ is induced by a diffeomorphism $\psi\colon M\to M'$. 
\end{Thm2}

It will be clear from our construction that the diffeomorphism $\psi$ of the above theorem can be chosen
in such a way that $\psi|_{\partial M}\colon \partial M\to\partial M'$ is an affine diffeomorphism. As a corollary, we obtain that every group isomorphism
between the fundamental groups of two graph manifolds is realized by a diffeomorphism.
Notice that, when dealing with extended graph manifolds, the additional hypothesis 
preventing surface pieces to be adjacent to the boundary 
is necessary: if
$M=\Sigma\times S^1$ and
$M'=\Sigma'\times S^1$, where $\Sigma_1$ is a once-punctured torus and
$\Sigma'$ is a thrice-punctured sphere, then 
$\pi_1(M)\cong \pi_1(M')$, but $M$ and $M'$ are not diffeomorphic (in fact, they are not even homeomorphic).

\section{Rigidly decomposable pairs}

In this section we single out the hypotheses on $M,M'$ that we need in the proof of Theorem \ref{smrigidity:thm}.

\begin{definition}
 Let $M,M'$ be smooth manifolds and let $\varphi:\pi_1(M)\to\pi_1(M')$ be a group isomorphism. We say that $(M,M',\varphi)$ is \emph{rigidly decomposable} if
\begin{enumerate}
 \item $M,M'$ are obtained by gluing submanifolds with toric $\pi_1$-injective boundary, called \emph{pieces}, using affine diffeomorphisms of pairs of boundary components,
 \item $\varphi$ preserves conjugacy classes of fundamental groups of the pieces,
 \item the restriction of $\varphi$ to the fundamental group of any piece $P$ of $M$ is induced by a diffeomorphism from $P$ to the corresponding piece of $M'$ which restricts to affine diffeomorphisms between corresponding boundary components,
 \item the normalizer in $M$ and $M'$ of (a conjugate of) the fundamental group of any piece coincides the fundamental group of the piece,
 \item fundamental groups of distinct boundary components of pieces are not conjugate to each other.
\end{enumerate}
\end{definition}

\begin{remark}
 Notice that pieces are automatically $\pi_1$-injective because of the $\pi_1$-injectivity of their boundary components.
\end{remark}

\begin{theorem}
\label{rigiddec:thm}
 If $(M,M',\phi)$ is rigidly decomposable then there exists a diffeomorphism $\psi:M\to M'$ inducing $\phi$ at the level of fundamental groups.
\end{theorem}

The proof of Theorem~\ref{rigiddec:thm} is deferred to Section~\ref{proofrigiddec}.
We first check that the theorem applies to (extended) graph manifolds.

Since we will need to be careful about some well-known, but somewhat subtle,
details of the theory of fundamental groups, we recall here some basic facts. If $f\colon M\to N$
is a continuous map between path connected spaces, then
$f$ induces a homomorphism $f_\ast \colon \pi_1 (M)\to \pi_1 (N)$ which is well-defined
up to conjugacy (in $\pi_1 (N)$). This is due to the fact that, for $x_0,x_1\in M$, $x_0\neq x_1$, 
the identification of $\pi_1 (M,x_0)$
with $\pi_1 (M,x_1)$ is canonical up to conjugacy, and the same holds when choosing different
basepoints in $N$. If $\varphi\colon \pi_1 (M)\to \pi_1 (N)$ is a homomorphism,
we will say that $\varphi$ is induced by $f$ if for some (and hence every)
choice of basepoints $x_0\in M$, $y_0\in N$ the homomorphism
$f_\ast\colon \pi_1 (M,x_0)\to\pi_1 (N,y_0)$ is equal to $\varphi$, up to conjugacy
by an element of $\pi_1 (N)$ (by the discussion above, this notion is indeed well-defined).
Also observe that if $V$ is a path connected subset of $M$ and $i\colon V\hookrightarrow M$
is the inclusion, then we can define $i_\ast (\pi_1 (V))$ as a subgroup
of $\pi_1 (M)$, well-defined up to conjugacy. When saying that $\pi_1 (V)$ is a 
subgroup of $\pi_1 (M)$, we will be implicitly choosing a preferred representative
among the conjugate subgroups representing the conjugacy class of $\pi_1 (V)$: this amounts
to choosing a basepoint in $V$, a basepoint in $M$ and a path joining these basepoints.

In order to deduce Theorem \ref{smrigidity:thm} from Theorem \ref{rigiddec:thm} we need to show that if $\varphi:\pi_1(M)\to \pi_1(M')$ is an isomorphism between fundamental groups of (extended) graph manifolds then $(M,M',\varphi)$ is rigidly decomposable,
provided that no component of $\partial M$ and $\partial M'$ is contained in a surface piece of $M$ and $M'$.

\begin{proposition}\label{rigidly:prop}
 Let $M,M'$ be extended graph manifolds so that their boundaries do not intersect any 
surface piece and let $\varphi:\pi_1(M)\to \pi_1(M')$ be an isomorphism. Then $(M,M',\varphi)$ is rigidly decomposable.
\end{proposition}

This Section is devoted to the proof
of Proposition~\ref{rigidly:prop}. Item 1 
in the definition of rigidly decomposable
follows from the definition of graph manifold. Item 2 is the content of Theorem Theorem \ref{iso-preserve:thm}. Lemma \ref{conj:lemma}-(2)-(3) implies items 4 and 5.

Item 3, as we are about to explain, is ultimately a consequence of Mostow rigidity (and, in the extended case, of the fact that outer automorphisms of fundamental groups of surfaces which preserve the conjugacy 
classes of the fundamental groups of the boundary components are induced by diffeomorphisms, see \emph{e.g.}~\cite[Theorem 8.8]{farbbook}).

The isomorphism $\varphi$ establishes
a bijection between the (conjugacy classes of the) fundamental groups of the pieces
of $M$ and $M'$.
Let 
$N_1,\ldots,N_h$ (resp.~$N'_1,\ldots,N'_h$) be the (truncated) cusped hyperbolic manifolds
such that $V_i=N_i\times T^{a_i}$ (resp.~$V'_i=N'_i\times T^{b_i}$) 
are the pieces of $M$ (resp.~of $M'$), $i=1,\ldots,h$. From now on, 
for every $i=1,\ldots,h$, we fix an identification of $\pi_1 (V_i)$ (resp.~of $\pi_1 (V'_j)$)
with a distinguished subgroup of $\pi_1 (M)$ (resp.~of $\pi_1 (M')$).
As mentioned above, such an identification
depends on the choice of one basepoint for $M,M'$ and for each piece, and suitable
paths connecting the basepoint of the ambient manifolds with the basepoints
of their pieces. We also fix $g_i\in\pi_1 (M')$ such that
$\varphi (\pi_1 (V_i))=g_i\pi_1 (V'_i)g_i^{-1}$ for every $i=1,\ldots,h$.

We now formulate item 3 as a Lemma for future reference.

\begin{lemma}\label{diffeo:pieces}
 For $i=1,\ldots,h$ there exists a diffeomorphism
$\psi_i\colon V_i\to V'_i$ which induces the isomorphism
$g\mapsto g_i^{-1} \varphi (g) g_i$ between $\pi_1 (V_i)$ and $\pi_1 (V'_i)$, and
restricts to an affine diffeomorphism
of $\partial V_i$ onto $\partial V'_i$.
\end{lemma}
\begin{proof}
Set $V=V_i$, $V'=V'_i$, $N=N_i$, $N'=N'_i$. 
The center of $\pi_1 (V)$ is equal to the fundamental group of its toric factor
(see Remark~\ref{center:rmk}), so $\pi_1 (N)$ is just the quotient
of $\pi_1 (V)$ by its center, and the same holds true for $\pi_1 (N')$.
We have in particular $V=N\times T^a$, $V'=N'\times T^a$ for the same  $a\in\mN$,
so $\pi_1 (V)$ (resp.~$\pi_1 (V')$)
is canonically isomorphic to $\pi_1 (N)\times \mZ^a$ (resp.~$\pi_1 (N')\times\mZ^a$),  
and the isomorphism $\varphi_i\colon \pi_1 (V)\to \pi_1 (V')$ 
defined by $\varphi_i (g)=g_i^{-1}\varphi (g) g_i$ for every $g\in \pi_1 (V)$
induces an isomorphism 
$\theta\colon \pi_1 (N)\to\pi_1 (N')$. 
Henceforth, we identify $T^a$ with the quotient of $\mR^a$ by the standard
action of $\mZ^a$, \emph{i.e.}~we fix an identification of $\pi_1 (T^a)$ with
$\mZ^a\subseteq \mR^a$ (since $\pi_1 (T^a)$ is abelian, we do not need to 
worry about choice of basepoints). Then the isomorphism
$\varphi_i\colon  \pi_1 (N)\times \mZ^a\to\pi_1 (N')\times \mZ^a$ has the form
$\varphi_i (g,v)=(\theta (g), f(g,v))$, where $f\colon \pi_1 (N)\times \mZ^a\to \mZ^a$
is a homomorphism.
If $\beta\colon\mZ^a\to\mZ^a$, $\alpha\colon\pi_1 (N)\to\mZ^a$ are defined
by $\beta (v)=f(1,v)$ and $\alpha (g)=f(g,0)$, we have that
$$
\varphi_i (g,v)=(\theta (g), \alpha (g)+ \beta (v))\quad\ {\rm for\ every}\ g\in \pi_1 (N),\,
v\in \mZ^a .
$$
Moreover, since $\varphi_i$ is an isomorphism, we have that $\alpha$ is a homomorphism and
$\beta$ is an automorphism. Any automorphism 
of $\pi_1(T^a)$ is induced by an affine diffeomorphism of
$T^a$ onto itself, so in order to construct the required diffeomorphism
$\psi\colon V\to V'$ inducing $\varphi_i$ it is not restrictive to assume
that $\beta (v)=v$ for every $v\in\mZ^a$.

Let us now fix identifications $\pi_1(N)\cong \Gamma< {\rm Isom} (\mathbb{H}^l)$,
$\pi_1(N')\cong \Gamma' < {\rm Isom} (\mathbb{H}^l)$,
$N=B/\Gamma$, $N'=B'/\Gamma'$, where $B,B'\subseteq \mathbb{H}^l$ are the 
neutered spaces providing the universal coverings of $N,N'$. 
For later purposes,  we will denote by $p_1\colon B\to N$ the covering
map just introduced.
In the case of non-surface pieces, Mostow rigidity provides an isometry (whence in particular a diffeomorphism)
$\widetilde{\kappa}\colon \mathbb{H}^l\to\mathbb{H}^l$ such that
$\widetilde\kappa (g\cdot x)=\theta (g)\cdot \widetilde\kappa (x)$ for every $g\in \Gamma$,
$x\in\mathbb{H}^l$. Up to changing the choice of the horospherical
sections defining $N$ as the truncation of a cusped hyperbolic manifold,
we may also suppose that $\widetilde\kappa (B)=B'$ (see Remark~\ref{deepness:rem}). 

Let us now consider surface pieces instead. Not every isomorphism between fundamental groups of surfaces with boundary is induced by a diffeomorphism. However, we claim that our isomorphism, $\theta$, preserves the conjugacy classes of the fundamental groups of the boundary components. Assuming this for the moment, we can use the fact that outer automorphisms of fundamental groups of surfaces which preserve the conjugacy 
classes of the fundamental groups of the boundary components are induced by diffeomorphisms, as we recalled at the beginning of the proof of the Proposition. So, there exists a diffeomorphism from $N$ to $N'$ inducing $\theta$ at the level of fundamental groups. We now regard both $N$ and $N'$ as punctured surfaces with chopped-off cusps, and notice that the said diffeomorphism extends to a diffeomorphism $\kappa$ of the punctured surfaces. By construction, lifting $\kappa$ to the universal covers we get a diffeomorphism $\widetilde{\kappa}\colon \mathbb{H}^2\to\mathbb{H}^2$ which by construction satisfies $\widetilde\kappa (B)=B'$ and $\widetilde\kappa (g\cdot x)=\theta (g)\cdot \widetilde\kappa (x)$ for every $g\in \Gamma$, $x\in\mathbb{H}^2$. The set-ups both in the case of surface and non-surface pieces are thus identical.

We are only left to show that $\theta$ preserves the conjugacy classes of the fundamental groups of the boundary components. But this just follows from the fact that the isomorphism $\varphi$ preserves wall stabilizers by Corollary \ref{iso-preserve-walls}
(notice that the assumption that the boundaries of $M,M'$ do not intersect any surface piece
implies that the walls of $\widetilde{M}$ and $\widetilde{M'}$ are all proper, according to Definition~\ref{proper:defn}).
In particular, we have that wall stabilizers are also preserved by $\varphi_i\colon  \pi_1 (N)\times \mZ^a\to\pi_1 (N')\times \mZ^a$, and $\theta$ is obtained just projecting $\varphi_i$ on the first factors.

We now establish the following:
\smallskip

\noindent {\bf Claim:} There exists a smooth function $\widetilde\eta\colon B\to \mR^a$
such that $\widetilde\eta (g\cdot x)=\widetilde\eta (x)+\alpha (g)$ for every 
$x\in B$, $g\in \Gamma$. 

\smallskip

In fact, let $\Gamma$ act on $B\times \mR^a$ by setting
$g\cdot (x,v)=(g\cdot x, v+\alpha (g))$, denote by $Y$ the quotient space
and let $p_2\colon B\times \mR^a\to Y$ be the natural projection. 
Since $N$ is canonically identified with the quotient of $B$ by the
action of $\Gamma$, we have a canonical projection
$p_3\colon Y\to N$, which defines a natural structure of 
flat affine fiber bundle. More precisely, $Y$ is the total space of
a flat fiber bundle with fiber $\mR^a$ and structural group
given by the group of integer translations of $\mR^a$.
In particular, every fiber of $p_3$ inherits a well-defined 
affine structure, so it is possible to define affine combinations
of points in a fiber. 
Exploiting this fact, we can use a suitable partition of unity
to glue local sections of $p_3$ into a global smooth section 
$s\colon N\to Y$.
\[
\xymatrix{ B\times \mR^a \ar[r]^>>>>>{p_2} & Y \ar[d]_{p_3}\\
B \ar[r]^{p_1} & N \ar@/_/[u]_{s}
}
\]

We now define $\eta$ as follows. Let us take $x\in B$. For every
$v\in\mR^a$ we have $p_3 (p_2 (x,v))=p_1 (x)$. Moreover, by construction
$p_2 (x,v)=p_2 (x,w)$ if and only if $v=w$. As a consequence,
there exists a unique $\widetilde\eta (x)\in\mR^a$ such that $p_2 (x,\widetilde{\eta} (x))=s(p_1 (x))$.
Since $p_1,p_2,s$ are smooth, $\widetilde\eta$ is also smooth. Moreover, 
for $x\in B$ and $g\in\Gamma$ we have:
\begin{align*}
p_2 (g\cdot x, \widetilde\eta (x)+\alpha (g)) & = p_2 (g\cdot (x,\widetilde\eta (x)))= p_2 (x,\widetilde\eta (x))\\
& = s(p_1 (x))=s (p_1 (g\cdot x)).\\
\end{align*}
The first equality is due to the definition of the $\Gamma$-action on $B\times \mR^a$. The
second and fourth equality are immediate from the definition of the quotient maps $p_2$ and
$p_1$ respectively. The third equality follows from the choice of $\tilde \eta$ (see previous 
paragraph). Finally, comparing the first and last term, we see that $\widetilde\eta (x)+\alpha (g)$
satisfies the defining property for the point $\widetilde\eta (g\cdot x)$, so by uniqueness
we obtain $\widetilde\eta (g\cdot x)=\widetilde\eta (x)+\alpha (g)$, and the {\bf Claim} is proved.

\smallskip

We now return to the proof of the Lemma. Define the map $\widetilde{\psi}\colon B\times\mR^a\to 
B'\times\mR^a$ via $\widetilde{\psi} (x,v)=(\widetilde\kappa (x), v+\widetilde\eta (x))$. Of course
$\widetilde{\psi}$ is a diffeomorphism. Moreover, for every
$(x,v)\in B\times\mR^a$ and $(g,w)\in \Gamma\times \mZ^a\cong \pi_1 (V)$, we have

\begin{align*}
\widetilde{\psi} ((g,w)\cdot (x,v))&= \widetilde{\psi}(g\cdot x,v+w)\\ 
&=(\widetilde\kappa (g\cdot x), v+w+\widetilde\eta (g\cdot x))\\
&=(\theta(g)\cdot \widetilde\kappa (x), v+w+\widetilde\eta (x)+\alpha (g)) \\ 
&=(\theta (g),w+\alpha (g))\cdot (\widetilde\kappa (x),v+\widetilde\eta (x))\\
&=(\theta (g),w+\alpha (g))\cdot\widetilde{\psi} (x,v)\\
\end{align*}
so $\widetilde{\psi}$ defines a diffeomorphism $\overline\psi\colon V\to V'$
inducing the isomorphism $\varphi$ at the level of fundamental groups. Now let
$\kappa\colon N\to N'$ be the diffeomorphism induced by $\widetilde\kappa$, $H$ be a component
of $\partial N$, and set $H'=\kappa(H)\subseteq \partial N'$. By construction, 
the restriction of $\overline\psi$  to the component
$H\times T^a$ of $\partial V$ has the form
$$
H\times T^a\to H'\times T^a,\qquad (x,v)\mapsto (\kappa(x), v+\overline{\eta} (x))
$$
for some smooth $\overline{\eta}\colon H\to T^a$. Recall that $H$ is affinely
diffeomorphic to a torus $T^b$, and that every 
map between affine tori is homotopic to an affine map, so $\overline{\eta}$ 
is homotopic to an affine map $\eta\colon H\to T^a$. Using this homotopy, we 
modify $\overline\psi$ in a collar of $H\times T^a$ in order to get a diffeomorphism 
$\psi\colon V\to V'$ whose restriction to $H\times T^a$ has the form 
$(x,v)\mapsto (\kappa(x), v+{\eta} (x))$. After repeating this procedure
for every component of $\partial V$ we are left with the desired diffeomorphism $\psi$.
\end{proof}

We have thus shown that $(M,M',\varphi)$ is rigidly decomposable, thus concluding the proof
of Proposition~\ref{rigidly:prop}.

\section{Dehn twists}
In this section we define Dehn twists. Only the definition of Dehn twist is strictly needed for the proof of Theorem \ref{rigiddec:thm}, but we will also provide some motivation and make some side remarks.

Let us first of all point out the issue we have to deal with. In order to establish smooth rigidity for graph manifolds, one would like to
glue the diffeomorphisms $\psi_i\colon V_i\to V'_i$ provided by Lemma~\ref{diffeo:pieces}
into a diffeomorphism $\psi \colon M\to M'$. In order to make this strategy work,
we have to take care of two issues. First, to define $\psi$ we have to check that
if $V_i$ and $V_j$ share a boundary component $H$, then $\psi_i$ and $\psi_j$
coincide on $H$. Once this has been established, we have to ensure that
the obtained $\psi$ induces the isomorphism $\varphi\colon \pi_1 (M)\to \pi_1 (M')$
fixed at the beginning of the section. The following remark, which is essentially 
due - in a different context - to Nguyen Phan~\cite{N}, shows that the issues just discussed
may really hide some subtleties.

\begin{remark}\label{twist:rem}
 Suppose $M=M'$ is a graph manifold obtained by gluing
two pieces $V_1$, $V_2$ along their unique boundary component $H=V_1\cap V_2\subseteq M$. 
Fix a basepoint
$x_0\in H$, and set $G_1=\pi_1 (V_1,x_0)$, $G_2=\pi_1 (V_2,x_0)$,
$K=\pi_1 (H,x_0)$. The group $\pi_1 (M,x_0)$ is canonically identified with the amalgamated product $G=G_1\ast_K G_2$, where we are considering $K$ as a subgroup
of $G_1$ and $G_2$ via the natural (injective) maps induced by the inclusions
$H\hookrightarrow V_1$, $H\hookrightarrow V_2$. Let us take $g_0\in K\setminus \{1\}$.
Since $K$ is abelian, there exists a unique isomorphism $\varphi\colon G\to G$
such that
$$
\varphi (g)=\left\{\begin{array}{cl} 
g & {\rm if}\ g\in G_1\\
g_0 g g_0^{-1} & {\rm if}\ g\in G_2 
\end{array}\right. .
$$
It is easy to see that, in this special case, the construction described in 
Lemma~\ref{diffeo:pieces} leads to diffeomorphisms $\psi_1\colon V_1\to V_1$,
$\psi_2\colon V_2\to V_2$ which can be chosen to equal the identity
on $V_1$, $V_2$ respectively. In particular, since $M$ and $M'$ are obtained 
by gluing $V_1$ and $V_2$ exactly in the same way, 
no issue about the possibility of defining $\psi$ arises. However, if we chose naively
to glue $\psi_1$ and $\psi_2$ simply by requiring that $\psi |_{V_i}=\psi_i$,
we would obtain $\psi={\rm Id}_M$. But this contradicts the fact that, when the
element $g_0$ is chosen appropriately, $\varphi$ may define a non-trivial outer 
automorphism of $G$ (of infinite order), see Lemma~\ref{InfOrderDehnTwist}
below.
\end{remark}

The previous remark motivates the following:

\begin{definition}
 Suppose $M$ is a manifold as in the definition of rigid decomposability (e.g. a graph manifold), and let $V_1,V_2$ be pieces 
of $M$ glued to each other along a common toric component
$H$ of $\partial V_1$ and $\partial V_2$. Let $h$ be a fixed element
of $\pi_1 (H)$ (since $\pi_1 (H)$ is abelian, this is independent of basepoints). 
The \emph{Dehn twist $t_h$ along $h$} is the diffeomorphism
$t_h\colon M\to M$ which is defined as follows. 

By construction, $H$ admits a collar $U$ in $M$ which is 
canonically foliated by tori (see Chapter~\ref{construction:sec}). In particular, $U$
is affinely diffeomorphic
to $T^{n-1}\times [-1,1]$, where $T^{n-1}=\mR^{n-1}/\mZ^{n-1}$ is the standard
affine $(n-1)$-torus,  and $\pi_1 (H)$ is canonically identified with the group
$\mZ^{n-1}$ of the automorphisms of the covering $\pi\colon \mR^{n-1}\to T^{n-1}$.
Let now $l\colon [-1,1]\to [0,1]$ be a smooth function such that
$l|_{[-1,-1+\varepsilon)}=0$, $l|_{(1-\varepsilon,1]}=1$ and set
$$\widetilde{t}_h\colon \mR^{n-1}\times [-1,1]\to \mR^{n-1}\times [-1,1],\quad
\widetilde{t}_h (v,s)=(v+l(s)\cdot h, s).$$
 The map $\widetilde{t}_h$ 
is $\mZ^{n-1}$-equivariant, so defines a diffeomorphism
$\widehat{t}_h\colon T^{n-1}\times [-1,1]\to T^{n-1}\times [-1,1]$ which is the identity
in a neighbourhood of $T^{n-1}\times \{-1,1\}$. We now define 
$t_h\colon M\to M$ as the diffeomorphism of $M$ such that
$t_h|_U=\widehat{t}_h$, $t_h|_{M\setminus U}={\rm Id}_{M\setminus U}$.
\end{definition}

Next we show how Dehn twists can be used to give elements of 
infinite order in the outer automorphism group of graph manifolds.

\begin{lemma}\label{InfOrderDehnTwist}
Let $M$ be a graph manifold, with $G= \pi_1(M)$. Assume $V_1, V_2$ are adjacent
pieces of $M$ glued together along a common toric component $H$, with 
$G_i:= \pi_1(V_i)$ and $K:= \pi_1(H)$. 
Let $F_i \leq G_i$ be the subgroups corresponding to the fibers in $V_i$, 
and set $F= F_1\cdot F_2 \leq K$ to be the subgroup generated
by the two fiber subgroups. If $h\in K$ is chosen so that $\langle h \rangle
\cap F = \{e\}$, then we have that the associated Dehn twist $\varphi := t_h$ 
has infinite order in $Out(G)$.
\end{lemma}

\begin{proof}
Suppose by way of contradiction
that for some $k\geq 1$ the automorphism $\varphi^k$ is equal to an internal 
automorphism of $G$, \emph{i.e.}~that there exists $\overline{g}\in G$ such that 
$\varphi^k (g)=\overline{g} g \overline{g}^{-1}$ for every $g\in G$. We have in 
particular $\overline{g} g \overline{g}^{-1} =g$ for every $g\in G_1$. By 
Lemma~\ref{conj:lemma}-(3), this implies that $\overline g$ belongs to $G_1$, 
whence to the center of $G_1$, which coincides with the fiber subgroup $F_1$ of 
$G_1$ (see Remark~\ref{center:rmk}). We conclude the conjugating element 
$\overline g$ satisfies $\overline g \in F_1$.

Similarly, for every $g\in G_2$
we have $\overline{g} g \overline{g}^{-1}=h^k g h^{-k}$. Rewriting, we obtain
$(h^{-k}\overline{g}) g (h^{-k}\overline{g})^{-1}=g$,
forcing $h^{-k}\overline{g}$ to lie in the fiber subgroup $F_2$ of $G_2$, and hence
$$h^{-k} \in \overline{g}^{-1} \cdot F_2 \subset F_1\cdot F_2 = F.$$
But this contradicts the fact that $\langle h \rangle
\cap F = \{e\}$. We conclude that $\varphi^k$ is not internal for every $k\geq 1$,
as desired.
\end{proof}

It is clear that the group automorphism described in Remark~\ref{twist:rem}
is induced by a Dehn twist. As a result, Dehn twists arise naturally as basic ingredients
when trying to ``patch together'' diffeomorphisms $\psi_i: V_i \rightarrow V_i^\prime$ 
between individual pieces into a globally defined 
diffeomorphism $\psi: M\rightarrow M^\prime$.

\section{Proof of Theorem \ref{rigiddec:thm}}\label{proofrigiddec}
Let $(M,M',\varphi)$ be rigidly decomposable. Let $V_i$ (resp.~$V'_i$) 
be the pieces of $M$ (resp.~of $M'$), for $i=1,\ldots,h$. We fix an identification of $\pi_1 (V_i)$ (resp.~of $\pi_1 (V'_j)$)
with a subgroup of $\pi_1 (M)$ (resp.~of $\pi_1 (M')$).
As we already remarked, such an identification
depends on the choice of one basepoint for $M,M'$ and for each piece, as well as suitable
paths connecting the basepoint of the $M,M'$ with the basepoints
of their pieces.
We also choose $g_i\in\pi_1 (M')$ such that
$\varphi (\pi_1 (V_i))=g_i\pi_1 (V'_i)g_i^{-1}$ for every $i=1,\ldots,h$.

By hypothesis, for $i=1,\ldots,h$ there exists a diffeomorphism
$\psi_i\colon V_i\to V'_i$ which induces the isomorphism
$g\mapsto g_i^{-1} \varphi (g) g_i$ between $\pi_1 (V_i)$ and $\pi_1 (V'_i)$, and
restricts to an affine diffeomorphism
of $\partial V_i$ onto $\partial V'_i$.

In order to construct $\psi\colon M\to M'$ that induces $\phi$, let us consider
a piece $V_i$ of $M$, a component $H_i$ of $\partial V_i$, and let
$V_j$ be the piece of $M$ adjacent to $V_i$ along $H_i$ (we allow the case $i=j$). 
Denote by $H_j$ the component of $V_j$ which is identified 
to $H_i$ in $M$, and by $H\subseteq M$ the image of $H_i$ and $H_j$
in $M$. 
We fix identifications of $\pi_1 (H_i)$ with a subgroup $K_i$ of $\pi_1 (V_i)$
and of $\pi_1 (H_j)$ with a subgroup $K_j$ of $\pi_1 (V_j)$ (as usual, this amounts to choosing a basepoint
in $H$ and paths joining this basepoint with the fixed basepoints of $V_i$ and $V_j$).
Via the fixed identifications of $\pi_1 (V_i)$ and $\pi_1 (V_j)$ with
subgroups of $\pi_1 (M)$, the groups $K_i$ and $K_j$ are identified
with conjugated subgroups of $\pi_1 (M)$, and this implies that 
the subgroups $\varphi (K_i)$, $\varphi (K_j)$ are conjugated
in $\pi_1 (M')$. By item 5 in the definition of rigid decomposability, this implies that $\psi_i (H_i)=H'_i$ is glued
in $M'$ to $\psi_j (H_j)=H'_j$.  

Denote by $\alpha\colon H_i\to H_j$ and $\alpha'\colon H'_i\to
H'_j$ the gluing maps which enter into the definition of $M$ and $M'$.
We now show that the diagram
\begin{equation}\label{glu:eq}
\xymatrix{
H_i \ar[r]^{\psi_i} \ar[d]^{\alpha} & H'_i \ar[d]^{\alpha'}\\
H_j \ar[r]^{\psi_j} & H'_j
}
\end{equation}
commutes, up to homotopy. In fact, recall that there exist $g_i,g_j\in \pi_1 (M')$
such that $(\psi_i)_\ast (g)=g_i^{-1} \varphi(g) g_i$ for every $g\in H_i$, 
$(\psi_j)_\ast (g)=g_j^{-1} \varphi(g) g_j$ for every $g\in H_j$. Moreover,
we can choose identifications $\pi_1 (H'_i)\cong K'_i< g_i \pi_1 (V'_i)g_i^{-1}$,
$\pi_1 (H'_j)\cong K'_j< g_j \pi_1 (V'_j)g_j^{-1}$ in such a way that
the isomorphisms $\alpha_\ast\colon K_i\to K_j$, $\alpha'_\ast\colon K'_i\to K'_j$
are induced by conjugations by an element of $\pi_1(M),\pi_1(M')$ respectively. 

It follows that there exists $h\in \pi_1(M')$ such that
$\alpha'_\ast ((\psi_i)_\ast (g))=h (\psi_j)_\ast (\alpha_\ast (g))h^{-1}$
for every $g\in K_i$. By item 4 in the definition of rigid decomposability, this implies that
$h\in K'_j$, and this implies in turn that the diagram above commutes, up to homotopy.
In order to properly define $\psi$, 
we now need to modify $\psi_i$ and $\psi_j$ in a neighbourhood of $H_i$ and $H_j$,
also taking care of the fact that $\psi$ has eventually to induce 
the fixed isomorphism $\varphi\colon \pi_1 (M)\to \pi_1 (M')$.

Being homotopic \emph{affine} diffeomorphisms of $T^{n-1}$, the diffeomorphisms
$\alpha'\circ \psi_i$ and $\psi_j\circ \alpha$
are in fact isotopic, and this implies that $\psi_i$ can be modified in a
collar of $H_i$ in order to make diagram~\eqref{glu:eq} commute.
This ensures that the maps $\psi_i$, $\psi_j$ can be glued into
a diffeomorphism $\widehat{\psi}\colon V_i\cup_{\alpha} V_j \to V'_i\cup_{\alpha'} V'_j$.
As pointed out above, 
we are now granted that an element $h\in K'_j$
exists such that
$\alpha'_\ast ((\psi_i)_\ast (g))=h (\psi_j)_\ast (\alpha_\ast (g))h^{-1}$
for every $g\in K_i$. Observe that $h$ uniquely identifies an element
of $\pi_1 (H')$. It is now easily seen that if $\psi_0 \colon V_i\cup_{\alpha} V_j \to V'_i\cup_{\alpha'} V'_j$ is obtained by composing $\widehat{\psi}$ with a Dehn
twist along $H'$ relative to $h$ (or to $-h$), then 
$\psi_0$ induces on $\pi_1 (V_i\cup_{\alpha} V_j)$ the restriction of $\varphi$.

We can apply the procedure just described along any boundary component
of any piece of $M$, eventually obtaining the desired diffeomorphism
$\psi\colon M\to M'$ inducing $\varphi$. 
\qed

\section{Mapping class group}\label{mcg:sec}
Let $M$ be a  closed graph manifold. We recall that
$\mcg (M)$ is the mapping class group of $M$, \emph{i.e.}
the group of \emph{homotopy} classes of diffeomorphisms
of $M$ onto itself. We also denote by $\out (\pi_1 (M))$ the group
of outer automorphisms of $\pi_1 (M)$. Every diffeomorphism
of $M$ induces an isomorphism of $\pi_1 (M)$, which is well-defined
up to conjugacy. Since homotopic diffeomorphisms induce
conjugate isomorphisms, there exists a well-defined map 
$$\eta\colon \mcg (M)\to \out (\pi_1 (M)),$$
which is clearly a group homomorphism.

\begin{theorem}\label{mcg:thm}
Let $M$ be a closed graph manifold. Then
the map $\eta\colon \mcg (M)\to \out (\pi_1 (M))$
is a group isomorphism.
\end{theorem}
\begin{proof}
The fact that $M$ is aspherical (see Lemma~\ref{Aspherical}) easily implies that 
$\eta$ is injective, while surjectivity of $\eta$ is just a restatement of 
Theorem~\ref{smrigidity:thm}.
\end{proof}

\begin{remark}\label{mcg:rem}
Remark~\ref{twist:rem} provides some evidence that the
mapping class group of $M$ should always be infinite: in fact, 
Dehn twists generate an
abelian subgroup of $\mcg (M)$, and with some effort one could probably show that
such a subgroup is never finite. 
\end{remark}

\begin{remark}\label{Waldhausen}
A celebrated result
due to Waldhausen~\cite{Wa3} shows that Theorem~\ref{mcg:thm}
also holds in the case of classical closed $3$-dimensional graph manifolds which either decompose
into the union of at least two Seifert pieces, or do not consist of a single ``small'' Seifert manifold (for example, if $M=S^3$ then of course $\out(\pi_1(M))=\{1\}$, while $\mcg (M)$
has two elements). Observe however that Seifert pieces that are homeomorphic to the product
$\Sigma\times S^1$, where $\Sigma$ is a hyperbolic punctured surface, are never small.

In the case of classical graph manifolds with boundary,
Theorem~\ref{mcg:thm} still holds,
provided that we replace the group
$\out(\pi_1(M))$ with the group of the conjugacy classes of isomorphisms 
which preserve the peripheral structure of $\pi_1 (M)$ (one says that an isomorphism
of $\pi_1(M)$ preserves its peripheral structure if it sends the subgroup
corresponding to a boundary component of $M$ into the subgroup corresponding
to a maybe different boundary component of $M$, up to conjugacy). 

It is not difficult to show that Lemma~\ref{InfOrderDehnTwist} may be adapted to construct
big abelian subgroups of $\out(\pi_1 (M))$ also in the case of classical graph manifolds, 
so one expects that $\mcg (M)$ should be infinite for generic $3$-dimensional graph manifolds.
\end{remark}

%
%
%

\chapter{Algebraic properties}\label{groups:sec}
The aim of this chapter is the study of fundamental groups of (extended)
graph manifolds
(and of their subgroups) with respect to some classical properties of abstract groups.
The decomposition of an (extended) graph manifold $M$ into pieces induces a description of
$\pi_1(M)$ as the fundamental group of a graph of groups $\mathcal{G}$, and our arguments will often
exploit the study of the action of $\pi_1(M)$ on the Bass-Serre tree associated to $\mathcal{G}$.
Therefore, at the beginning of the Chapter we recall some useful results from Bass-Serre theory
(we refer the reader e.g.~to \cite{serre} for more 
background). Whenever possible, we state our results in the context of fundamental groups of graph of groups
whose vertex groups satisfy suitable conditions, and deduce as corollaries the corresponding properties 
of fundamental groups of (extended) graph manifolds.


We first study the action of the fundamental group of an (extended) graph manifold on its Bass-Serre tree. We show
that one can detect whether the (extended) graph manifold $M$ is irreducible or fibered by looking at the action of $\pi_1 (M)$ on its Bass-Serre tree: namely, in Propositions~\ref{irr-acyl} and \ref{weakly:char} we prove that the action is acylindrical (resp.~faithful)
if and only if $M$ is irreducible (resp.~non-fibered). We refer the reader to the Introduction for the definition of transverse gluings, and of
irreducible and fibered (extended) graph manifolds.

In Section~\ref{relhyp:sec} we  show that the fundamental group of an
(extended) graph manifold $M$ is relatively hyperbolic with respect to a finite family of proper subgroups,
provided that at least one piece of $M$ is purely hyperbolic (Proposition~\ref{exception}). This condition  is probably also necessary, and indeed it is in the case
of irreducible (extended) graph manifolds, which is discussed in Chapter~\ref{preserve:sec}. On the contrary, in Proposition~\ref{hyp-emb-elements} we show that the fundamental
group of an (extended) graph manifold contains hyperbolically embedded subgroups, provided that the manifold
has an internal wall with transverse fibers.

Then, we show that the fundamental groups of (extended) graph manifolds contain no non-trivial Kazhdan groups (Corollary~\ref{Kazhdan subgroups}), have universal 
exponential growth (Proposition~\ref{UEG}), and we
establish
the Tits alternative (Corollary~\ref{TA}).

In Proposition~\ref{CH} we prove that, if $M$ 
contains a pair of adjacent pieces with transverse fibers, then the fundamental group of $M$ is 
co-Hopfian. 
Then we prove that $\pi_1(M)$ is $C^*$-simple if and only if $M$ is not fibered (Proposition~\ref{C-simple-manifolds})
and that $\pi_1(M)$ is SQ-universal provided that $M$ contains an internal wall
which either is disconnecting, or has transverse fibers (Proposition~\ref{SQ-u}).
Our proofs of Propositions~\ref{C-simple-manifolds} and \ref{SQ-u}
exploit 
results from~\cite{dlH-Pr} and from the recent preprint~\cite{DGO}, and also provide a characterization
of $C^*$-simple and SQ-universal fundamental groups of acylindrical graphs of groups
(see Propositions~\ref{C-simple-graphs} and \ref{SQ-u-graph}).

Building on results from the subsequent Chapter~\ref{strongirr:sec},
in Section~\ref{word:sec} we show that
the word problem for $\pi_1(M)$ is always solvable for irreducible graph manifolds.
Finally, in the last section, we study how the choice of the gluing between pieces can affect the fundamental group of graph manifolds.

\section{Graphs of groups and groups acting on trees}\label{tree-theory}
Let $\mathcal{G}$ be a finite graph of groups based on the the graph $\Gamma$.
Following~\cite{dlH-Pr}, we say that an edge $e$ of $\mathcal{G}$ is \emph{trivial} if it has
distinct endpoints and at least one of the two monomorphisms associated to $e$
is an isomorphism. The graph of groups $\mathcal{G}$ is \emph{reduced} if 
no edge of $\mathcal{G}$ is trivial.
If $\mathcal{G}$ 
is not reduced, then one can define a graph $\overline{\Gamma}$ obtained from $\Gamma$
by collapsing a trivial edge to a point, and a new 
graph of groups
$\mathcal{G}'$ based on $\overline{\Gamma}$, in such a way that the fundamental
group of $\mathcal{G}'$ is canonically isomorphic to the fundamental group of
$\mathcal{G}$. Therefore, every finite graph of groups may be simplified
into a reduced one without altering its fundamental group (see~e.g.~\cite[Proposition 2.4]{Ba1}). 

We say that $\mathcal{G}$ is \emph{non-trivial}
if it is reduced and based on a graph with at least one edge. The graph of groups
$\mathcal{G}$ is \emph{degenerate} if one of the following possibilities holds:
\begin{enumerate}
\item either $\Gamma$ is a segment with two vertices, and the indices of the unique edge group
in the two vertex groups are both equal to $2$ (in this case the fundamental group of
$\mathcal{G}$ is an amalgamated product of its vertex groups), or
\item $\Gamma$ is a loop with one vertex, and both the monomorphisms associated to
the edge are isomorphisms (in this case the fundamental group of $\mathcal{G}$
is a semidirect product of the vertex group with $\mZ$).
\end{enumerate}

Finally, we say that $\mathcal{G}$ is \emph{exceptional} if it is degenerate and the edge
group of $\mathcal{G}$ is trivial. If $\mathcal{G}$ is exceptional, then
$\pi_1(\mathcal{G})$ is isomorphic either to $\mZ_2 * \mZ_2\cong \mZ\rtimes \mZ_2$ (in 
the case of the amalgamated product) 
or to $\mZ$ (in the case of the semidirect product).

Any graph of groups $\mathcal{G}$ with fundamental group $G$ determines a tree $T$, called
the \emph{Bass-Serre tree of $\mathcal{G}$}, on which $G$ acts by isometries.
It is well-known that $G$ acts on $T$ without inversions: if an element $g\in G$ and and edge $e$ of $T$
are such that $g(e)=e$, then $g$ does not interchange the vertices of $e$.

Let now $G$ be a group acting without inversions on a tree $T$.
It is well-known that the subset of fixed points of $T$ under the action of $G$
is a subtree, that will be denoted by $T_G$. Conversely,
if $T'$ is a subtree of $T$ (for example, a vertex or an edge),
then $G_{T'}$ is the subgroup of those elements of $G$ that pointwise fix $T'$.
An old result by Tits implies that every $g\in G$ is either \emph{elliptic}, if it fixes a vertex of $T$,
or \emph{hyperbolic}, if there exists a $g$-invariant subtree $T'$ of $T$ which is isomorphic to the real line, and is such
that $g$ acts on $T'$ as a non-trivial translation. The action of $G$ on $T$ is \emph{faithful} if $G_T=\{1\}$, and it is \emph{minimal} if $T$ does not contain any $G$-invariant proper subtree.
Following Delzant \cite{Del}, we say that the action of $G$ is {\it $K$-acylindrical} 
if there exists a constant $K$, such that any element
which pointwise fixes any path in $T$ of length $\geq K$ is automatically trivial. 
The action of $G$ is acylindrical if it is $K$-acylindrical for some
$K\geq 0$. When we say that a graph of groups $\mathcal{G}$ is faithful, minimal or acylindrical, we understand
that the action of $\pi_1(\mathcal{G})$ on the Bass-Serre tree of $\mathcal{G}$ is respectively faithful,
minimal or acylindrical.

The following Lemma collects some elementary results about the Bass-Serre tree
of a finite graph of groups.

\begin{lemma}\label{reduced}
 Let $\mathcal{G}$ be a non-trivial finite graph of groups,
let $G$ be the
fundamental group of $\mathcal{G}$,
and let $T$ be the Bass-Serre tree of $\mathcal{G}$. Then:
\begin{enumerate}
 \item $G$ contains at least one hyperbolic element (in particular, $T$ has infinite diameter).
\item The action of $G$ on $T$ is minimal.
\item The tree $T$ is isomorphic to the real line if and only if
$\mathcal{G}$ is degenerate.
\item If $\mathcal{G}$ is non-degenerate, then $G$ contains a free non-abelian subgroup.
\item $\mathcal{G}$ is exceptional if and only if it is degenerate and acylidrical.
\item If $\mathcal{G}$ is faithful, then $G$ does not contain any 
non-trivial finite normal subgroup.
\end{enumerate}
\end{lemma}
\begin{proof}
In order to prove point (1) it is sufficient to show that there exists an element
of $G$ which does not fix any vertex of $G$. But $\mathcal{G}$ is non-trivial,
so it is not a filtering tree of groups, according to the terminology used in~\cite{Ba1}.
Therefore, point (1) is a consequence of~\cite[Proposition 3.7]{Ba1}.

Points (2), (3) and (4) are proved respectively in ~\cite[Proposition 7.12]{Ba},
\cite[Proposition 18]{dlH-Pr}
and \cite[Theorem 6.1]{Ba1}.

Let us prove (5). By definition, if $\mathcal{G}$ is exceptional then it is
degenerate, and it is easy to check that the two exceptional graphs of groups
are acylindrical. Let us now suppose that $\mathcal{G}$ is degenerate. Then $T$ is the real line, so an automorphism of $T$ fixing an edge acts as the identity of $T$. Therefore,
if $\mathcal{G}$ is also acylindrical, then the stabilizer of any edge of $T$
is trivial. This means that the edge group of $\mathcal{G}$ is trivial, so $\mathcal{G}$
is exceptional, and point (5) is proved.

Suppose now that $G$ contains a non-trivial finite normal subgroup $N$. Being finite, $N$ fixes a vertex of $T$, so the fixed subtree $T_N\subseteq T$ is non-empty.
Since $N$ is normal in $G$, the subtree $T_N$ is $G$-invariant, so by minimality 
$T_N=T$, and $N$ lies in the kernel of the action of $G$ on $T$. This proves point (6).
\end{proof}

It is shown in~\cite{DGO} that (with very few exceptions)
a group acting acylindrically on 
a Gromov hyperbolic space contains hyperbolically embedded subgroups. Therefore, we have the following result
(we refer the reader to~\cite{DGO} for the definition of non-degenerate hyperbolically embedded subgroup):

\begin{proposition}[Hyperbolically embedded subgroups]\label{DGO-prop}
 Let $\mathcal{G}$ be a non-trivial non-exceptional acylindrical graph of groups
with fundamental group $G$. Then 
$G$ contains a non-degenerate hyperbolically embedded subgroup.
\end{proposition}
\begin{proof}
 Let $T$ be the Bass-Serre tree of $\mathcal{G}$. The notion of acylindrical action
used in~\cite{DGO} is taken from~\cite{Bow}, and makes sense in the  context of group actions on Gromov hyperbolic spaces. However, as observed in~\cite{Bow},
in the particular case of trees our acylindrical actions are acylindrical
also in the sense of~\cite{DGO}. 

Let us consider the action of $G$ on $T$. By Lemma~\ref{reduced} $G$ contains
a hyperbolic element $h$.
As observed in~\cite[Remark 6.2]{DGO}, the acylindricity of the action of $G$ on $T$
implies that 
$h$ satisfies the weak proper discontinuity condition defined by Bestvina and Fujiwara in~\cite{BF}.
Therefore, if $E(h)$ is the unique maximal elementary subgroup of $G$
containing $h$, then $E(h)$ is hyperbolically embedded in $G$ (see~\cite[Lemma 6.5 and Theorem 6.8]{DGO}). In order to conclude we need to show that $E(h)$ does not coincide
with the whole of $G$. However, the infinite cyclic subgroup
generated by $h$ is of finite index in $E(h)$. Moreover, since $\mathcal{G}$ is non-degenerate, Lemma~\ref{reduced} implies that $G$ contains a free non-abelian subgroup. In particular, $G$ is not virtually cyclic, so $E(h)\neq G$, and we are done.

An alternative proof of the Proposition follows from the results contained in the recent
preprint~\cite{osin-pre}, where it is shown that the class of groups containing
 a proper infinite hyperbolically embedded subgroup coincides with the class of groups
admitting a non-elementary acylindrical action on a Gromov hyperbolic
space.
\end{proof}

\section{The graph of groups associated to an (extended) graph manifold}\label{tree-graph-theory}
Let $M$ be an (extended) graph manifold. The decomposition of $M$ into pieces determines
a description of $\pi_1(M)$ as the fundamental group of a graph of groups
$\mathcal{G}$, where vertex groups of $\mathcal{G}$ correspond to fundamental
groups of the pieces of $M$, and edge groups correspond to fundamental
groups of the internal walls of $M$. Edge groups have infinite index in the adjacent vertex groups,
so $\mathcal{G}$ is always reduced. Therefore, if $M$ contains at least one internal wall, then 
$\mathcal{G}$ is non-trivial and non-degenerate. 
In this Section we establish
some useful properties of the action of $\pi_1(M)$
on the Bass-Serre tree associated to $\mathcal{G}$. 

Before stating our first lemma, recall that the stabilizer of any vertex of the Bass-Serre tree $T$ of $\mathcal{G}$ is identified with the fundamental group $G_i$ of a piece
of $M$, and that the fiber subgroup of $G_i$ coincides with the center of $G_i$ (see Remark \ref{center:rmk}). As a consequence,
we may speak without ambiguities about the fiber subgroup of the stabilizer of any vertex of $T$.

\begin{lemma}\label{acyl-lem}
Let $M$ be an (extended) graph manifold and let $T$ be the Bass-Serre tree corresponding
to the decomposition of $M$ into pieces. 
\begin{enumerate}
 \item 
Let $e_1,e_2$ be distinct edges of $T$ sharing
the common vertex $v$, and suppose that the element $g\in\pi_1(M)$ is such that
$g(e_i)=e_i$ for $i=1,2$. Then $g$ belongs to the fiber subgroup of $G_v$.
\item
Let $\mathcal{P}$ be a path in $T$ of length three, let $e_1, e_2, e_3$ be the three consecutive edges in $\mathcal{P}$, and suppose
that there exists a non-trivial element $g\in G$ such that
$g(e_i)=e_i$ for $i=1,2,3$. Then the gluing corresponding to the edge $e_2$
is not transverse.
\end{enumerate}
\end{lemma}
\begin{proof}
Point (1) is an immediate consequence of Lemma~\ref{acyl-lem-0}--(2).

(2): Let $v_1, v_2$ the two intermediate 
vertices of the path $\mathcal{P}$, and let $G_i$
be the stabilizer of $v_i$ in $G$. By point (1) the element
$g$ belongs both to the fiber subgroup of $G_1$ and to the fiber subgroup
of $G_2$. Since $g$ is non-trivial, this implies that the gluing corresponding to the edge $e_2$ joining $v_1$ with $v_2$ is not transverse.
\end{proof}

We are now ready to provide a characterization of irreducibility in terms of the action of $\pi_1(M)$ on its Bass-Serre tree.
The ``if'' part of the following result was suggested by the anonymous referee.

\begin{proposition}[Irreducible $\Longleftrightarrow$ Acylindrical]\label{irr-acyl}
 Let $M$ be an (extended) graph manifold containing at least one internal wall. We denote by $G$ the fundamental group of $M$,
and by $T$ the Bass-Serre tree associated to the decomposition of $M$ into pieces. 
Then $M$ is irreducible if and only if
the action of $G$ on $T$ is acylindrical.
\end{proposition}
\begin{proof}
If $M$ is irreducible, then Lemma~\ref{acyl-lem} implies
that the graph of groups corresponding to the decomposition of $M$
into pieces is $3$-acylindrical.


On the other hand, let us suppose that $M$ is not irreducible. Then 
there exists a non-transverse gluing $\psi$ between two
(possibly non-distinct)
adjacent pieces $V_{i_1}$, $V_{i_2}$ of $M$. This gluing 
determines an infinite abelian subgroup $H$ of 
$G$ that 
acts trivially
on the Bass-Serre tree associated to the amalgamation (or HNN-extension)
of $\pi_1(V_{i_1})$
and $\pi_1(V_{i_2})$ corresponding to $\psi$.
This tree contains a bi-infinite geodesic, and equivariantly embeds into the Bass-Serre tree $T$ of the ambient
graph of groups. In particular, $T$ contains a bi-infinite geodesic admitting an infinite
pointwise stabilizer, so the action of $\pi_1(M)$ on $T$ is not acylindrical.
\end{proof}

We can construct an acylindrical action of $\pi_1(M)$ on a tree under the weaker hypothesis that 
$M$ contains at least one internal wall with transverse fibers. To this aim we introduce the following construction.

The decomposition of $M$ into pieces determines a description
of $\pi_1(M)$ as the fundamental group of a graph of groups
$\mathcal{G}$ based on the finite graph $\Gamma$. 
Let us choose an internal wall $H$ of $M$, which corresponds to the edge
$\mathcal{E}$ of $\Gamma$, and let
$M'=M\setminus N(H)$, where $N(H)$ is an open regular
neighborhood of $H$ in $M$.
So $M'$ is either an (extended) graph manifold
(if $\Gamma\setminus \mathcal{E}$ is connected) or the disjoint union
of two (extended) graph manifolds (if $\Gamma\setminus \mathcal{E}$ is disconnected).
We consider the graph of groups $\mathcal{G}'$ associated to the realization of $M$
as a gluing of $M'$ along the boundary components arising from the cut along $H$.
Then $\mathcal{G}'$ is based on the graph $\Gamma'$
obtained by collapsing to points all the edges in $\Gamma\setminus \mathcal{E}$.
By definition, the fundamental group of $\mathcal{G}'$ is still isomorphic to $\pi_1(M)$.
Of course, $\mathcal{G}'$ is non-degenerate, and it represents a realization of $G$
as an amalgamated product (if $M'$ is disconnected) or as an HNN-extension (if $M'$ is connected). We say that $\mathcal{G}'$ is obtained by collapsing $\mathcal{G}$ 
outside $\mathcal{E}$. 

\begin{lemma}\label{addedlemma}
 Let $M$ be an (extended) graph manifold containing at least one internal wall, let $\mathcal{G}$ be the graph of groups
corresponding to the decomposition of $M$ into pieces, and let $\mathcal{G}'$
be the graph of groups obtained by collapsing $\mathcal{G}$ outside
the edge $\mathcal{E}$. Then:
\begin{enumerate}
 \item If $\mathcal{E}$ corresponds to an internal wall of $M$ with transverse
fibers, then $\mathcal{G}'$ is $3$-acylindrical.
\item $\mathcal{G}$ is faithful if and only if $\mathcal{G}'$ is faithful.
\end{enumerate}
\end{lemma}
\begin{proof}
Let $M'=M\setminus N(H)$, where $H$ is  
the internal wall corresponding to the edge
$\mathcal{E}$, and $N(H)$ is an open regular neighborhood of $H$ in $M$.
We call \emph{big chamber} a connected component
of the preimage of $M'$ in $\tilM$. Of course, a big chamber is just the union
of a (usually infinite) number of chambers of $\tilM$. Moreover, the Bass-Serre $T'$
associated to $\mathcal{G}'$ is dual to the decomposition of $\tilM$ into big chambers, and
the inclusion of chambers into big chambers
induces a surjective $G$-equivariant map $\rho\colon T\to T'$. 

We say that en edge ${e}$ of $T$ is \emph{special} if it corresponds to a preimage of $H$ in $\tilM$,
or, equivalently, if $\rho({e})$ is an edge of $T'$, i.e.~$\rho$ does not collapse
${e}$ to a vertex. The union of the special edges of $T$ is $G$-invariant, 
and $\rho$ establishes a bijection
between the set of special edges of $T$
and the set of edges of $T'$.
Therefore, if ${e}$ is a special edge of $T$ and 
$g\in G$ is such that $g(\rho({e}))=\rho({e})$, then
$g({e})={e}$. 

(1):
We take a path $\mathcal{P}'$ of length three in $T'$ with endpoints $v_0'$ and $v_3'$, we
denote by
$e_1',e_2',e_3'$ be the three consecutive edges in $\mathcal{P}'$, and we take
$g\in G$ such that $g(e_i')=e_i'$ for $i=1,2,3$. 
Let 
${e}_i$ be the special edge of $T$ such that $\rho(e_i)=e'_i$, let 
$v_0$ (resp.~$v_3$) be the vertex of $e_1$ (resp.~of $e_3$) 
such that $\rho(v_i)=v_i'$, $i=0,3$,
and let $\gamma\subseteq T$ be the geodesic joining $v_0$ with $v_3$. 
The discussion above shows that $g(e_i)=e_i$ for $i=1,2,3$, so $g(v_0)=v_0$ and
$g(v_3)=v_3$, hence $g(\gamma)=\gamma$. Moreover,
if $v_0,v_3$
were in the same connected component of $T\setminus e_2$, then
$\rho(v_0)=v_0'$ and $\rho(v_3)=v_3'$ would be in the same connected component of $T'\setminus e_2'$,
a contradiction. This implies that $\gamma$ contains 
a path of length three which is fixed by $g$ and has $e_2$
as intermediate edge. But $e_2$ is special, so 
the gluing corresponding to $e_2$ is transverse, and 
$g$ must be the identity by Lemma~\ref{acyl-lem}. 
We have thus
shown that the action of $G$ on $T'$ is $3$-acylindrical.

(2): Let $K,K'$ be the kernels of the action of $G$ on $T,T'$ respectively.
We will show that $K=K'$. Recall that $\rho$ establishes a $G$-equivariant
bijection between the set of special edges of $T$ and the set of edges of $T'$.
Therefore, $K'$ is just the group of those
elements of $G$ that fix every special edge of $T$. This already proves the inclusion
$K\subseteq K'$. Let us now show that $K'\subseteq K$.
Let $T_{K'}\subseteq T$ be the fixed subtree of $K'$.
Every special edge is contained in $T_{K'}$, so $T_{K'}$ is non-empty.
But $K'$ is normal
in $G$, so $T_{K'}$ is $G$-invariant. The minimality of the action of $G$ on $T$
now implies that $T_{K'}=T$, i.e.~that
$K'\subseteq K$.
\end{proof}

The following result is an immediate consequence of point (1) of Lemma~\ref{addedlemma}.

\begin{proposition}\label{acyl-new}
 Let $M$ be an (extended) graph manifold, 
 and suppose that $M$ contains an internal wall with transverse fibers. Then
$G$ admits a realization either as a non-degenerate acylindrical amalgamated product or as a non-degenerate acylindrical HNN-extension.
\end{proposition}

By looking at the action of $\pi_1(M)$ on
its Bass-Serre tree, one can also establish whether $M$ is fibered or not.

\begin{proposition}[Fibered $\Longleftrightarrow$ Non-faithful]\label{weakly:char}
Let $M$ be an (extended) graph manifold containing at least one internal wall, let $\mathcal{G}$ be the graph of groups associated to the decomposition of $M$ into pieces, and set
$G=\pi_1(\mathcal{G})=\pi_1(M)$. Then the following conditions are equivalent:
\begin{enumerate}
\item
$M$ is fibered.
\item
$G$ contains a non-trivial normal abelian subgroup.
\item
The action of $G$ on the Bass-Serre tree of $\mathcal{G}$ is not faithful.
\item
If $\mathcal{G}'$ is obtained by collapsing $\mathcal{G}$ outside an edge, then the action
of $G$ on the Bass-Serre tree of $\mathcal{G}'$ is not faithful.
\item
Let $V_1,\ldots,V_s$ be the pieces of $M$, let $G_i=\pi_1(V_i)=\pi_1(B_i)\times\mZ^{k_i}$, and let $\{1\}\times \mZ^{k_i}$
be the fiber subgroup of $G_i$. Then one can choose a distinguished subgroup $F_i$ of the fiber subgroup of $G_i$ 
for every $i=1,\ldots,s$,
in such a way that each gluing involved in the construction of $M$ identifies the distinguished subgroups
of the fundamental groups of the corresponding adjacent pieces.
\end{enumerate}
\end{proposition}
\begin{proof}
Point (2) of Lemma~\ref{addedlemma} implies that points (3) and (4) are equivalent,
so it is sufficient to prove the chain of implications (1) $\Longrightarrow$ (2) $\Longrightarrow$ (3)
$\Longrightarrow$ (5) $\Longrightarrow$ (1).

\smallskip

(1) $\Rightarrow$ (2): Suppose that $M$ is the total space of a fiber bundle $F\hookrightarrow M\to M'$, where the fiber
$F$ is a $d$-dimensional torus, $d\geq 1$, and $M'$ is an (extended) graph manifold.  
We consider the following portion of
the exact sequence for fibrations in homotopy: 
$$
\pi_2(M')\to \pi_1 (F)\to \pi_1(M)\to \pi_1(M')\ . 
$$
Since $M'$ is aspherical, we have $\pi_2(M')=0$. Therefore, $\pi_1(F)$ injects onto a non-trivial abelian normal subgroup
of $\pi_1(M)$.

\smallskip

(2) $\Rightarrow$ (3):
It is sufficient to show that every
normal abelian subgroup of $G$ is contained in the kernel of the action of $G$ on $T$. Since $N$ is abelian, by \cite[page 65, Proposition 27]{serre} 
either $N$ fixes a vertex of $T$, or there exists a unique line $L$ in $T$ which is left invariant by the action of $N$. In the first case, the fixed subtree $T_N$
is non-empty.
But $N$ is normal in $G$, so $T_N$ is $G$-invariant, and $T_N=T$ 
by minimality of the action of $G$.
Therefore, $N$ is contained in the kernel of the action of $G$ on $T$. 
In the second case, take $g\in G$ and consider the line $g(L)$. Using that $N$ is normal in $G$ it is easily checked that
$g(L)$ is also $N$-invariant, so $g(L)=L$. We have thus shown that $L$ is $G$-invariant, so $L=T$ by minimality of the action of $G$ on $T$.
But this implies that $\mathcal{G}$ is degenerate, a contradiction.

 \smallskip

(3) $\Rightarrow$ (5): Let $T$ be the Bass-Serre tree of $\mathcal{G}$ and let $N$ be the kernel of the action of $G$ on $T$.
Let $v$ be a vertex of $T$, let $G_v$ be the stabilizer of $v$ in $T$ and take $g\in N$.
Of course $N<G_v$, so $g\in G_v$. Since $g$ fixes all the edges exiting from $v$, 
Lemma~\ref{acyl-lem} implies that 
$g$ belongs to the fiber subgroup of $G_v$.
We have thus shown that $N$ is contained in the fiber subgroup of every vertex stabilizer. If $V$ is a piece of $M$, then $\pi_1(V)$ is identified
with a vertex stabilizer $G_{v}$, so we may consider $N$ as a
subgroup
of $\pi_1(V)$ (since $N$ is normal, no ambiguities arise from the choice of $v$
and of
the identification $\pi_1(V)\cong G_{v}$). So we may choose $N$ as a 
distinguished subgroup
of the fiber subgroup of $\pi_1(V)$. It is now obvious from the definition
of $N$ that each gluing involved in the construction of $M$ identifies the distinguished subgroups
of the corresponding adjacent pieces.

\smallskip

(5) $\Rightarrow$ (1): First observe that $F_i$ is contained in a unique subgroup $F_i'<\{1\}\times \mZ^{k_i}<G_i$ such that the index
$[F_i':F_i]$ is finite and the quotient $(\{1\}\times \mZ^{k_i})/F'_i$ is torsion-free. 
The gluing maps between the pieces
preserve the $F_i$, and this readily implies that they also preserve the $F_i'$, so we may assume that $F_i=F'_i$ for every $i$. Also observe that
the $F_i$ share all the same rank, say $d\geq 1$.
Let us now consider the piece $V_i$ of $M$. 
Since $(\{1\}\times \mZ^{k_i})/F_i$ is torsion-free,
the fiber subgroup $\mZ^{k_i}$ of $\pi_1(V)$ decomposes as a direct sum
$\mZ^{k_i}=\mZ^{k_i-d}\oplus F_i$. We consider the corresponding
decomposition 
$$
V_i=N_i\times T^{k_i}\cong (N_i\times  T^{k_i-d})\times T^d=W_i\times T^d\ ,
$$
where $N_i$ is the base of $V_i$. We call \emph{small fiber of $V_i$} a subset of the form $\{*\}\times T^d\subseteq W_i\times T^d=V_i$. 
Recall now that 
the gluing maps are affine, and preserve the distinguished subgroups of the fundamental groups of the pieces. This readily implies
that the gluing maps
identify the small fibers of adjacent pieces. Therefore, the product structures $V_i=W_i\times T^d$
may be glued into a global structure of 
fiber bundle on $M$ with fiber $T^d$. If $M\to M'$ is the associated projection, then $M'$ is obtained by gluing the $W_i$ via affine gluings,
so $M'$ is an (extended) graph manifold, and $M$ is fibered.


\end{proof}

From condition (5) of the previous Proposition we obtain:

\begin{corollary}
Suppose that $M$ contains an internal wall with transverse fibers. Then $M$ is not fibered. 
\end{corollary}

For later reference we point out also the following:

\begin{lemma}\label{pure:lemma}
Suppose that $M$ consists of a single piece $V$ without internal walls. Then, $M$ is fibered if and only if 
$V$ is not purely hyperbolic.
\end{lemma}
\begin{proof}
Of course, if $V$ is not purely hyperbolic, then $M$ is fibered. On the other hand, just as in the proof of the implication
(1) $\Rightarrow$ (2) of Proposition~\ref{weakly:char}, one can see that, if $M$ is fibered, then $\pi_1(M)=\pi_1(V)$
contains a non-trivial abelian normal subgroup. But this implies that $\pi_1(V)$ cannot be isomorphic to the fundamental
group of a complete finite-volume hyperbolic manifold, so $V$ cannot be purely hyperbolic.
\end{proof}

We conclude the section with a technical lemma that will prove useful later.
If $G$ is a group acting on a tree $T$, then we say that the action is \emph{without reflections}
if there do no exist an element $g\in G$ and distinct edges $e_1,e_2$ of $T$ sharing a common vertex such that
$g(e_1)=e_2$ and $g(e_2)=e_1$.

\begin{lemma}\label{root free}
 Let $M$ be an (extended) graph manifold with at least one internal wall.
Let $\pi_1(M)=G$, let $\mathcal{G}$ be the graph of groups
corresponding to the decomposition of $M$ into pieces, and let $T$
be the Bass-Serre tree of $\mathcal{G}$. Let also $\mathcal{G}'$
be a graph of groups obtained by collapsing $\mathcal{G}$ outside an edge,
and let $T'$ be the Bass-Serre tree of $\mathcal{G}'$.
Then:
\begin{enumerate}
\item If $e$ is an edge of $T$, and $g\in G$ is such that
$g^n\in G_e$ for some $n\geq 1$, then $g\in G_e$.
\item If $e'$ is an edge of $T'$, and $g\in G$ is such that
$g^n\in G_{e'}$ for some $n\geq 1$, then $g\in G_{e'}$.
\item
$G$ acts on $T$ without reflections. 
\end{enumerate}
\end{lemma}
\begin{proof}
(1): 
Suppose by contradiction that there exists $g\in G$ such that
$g^n(e)=e$ for some $n\geq 1$ but $g(e)\neq e$. Then the subgroup generated by $g$ admits a finite orbit
in $T$, so it fixes a vertex $v$ of $T$. Consider the geodesics
$\gamma,\gamma'$ connecting $v$ with $e, g(e)$ respectively, and let
$v'$ be the last vertex  in $\gamma\cap \gamma'$. Then $g(v')=v'$, and
there exist distinct edges $e_1,e_2$ exiting from $v'$ such that $g^n(e_1)=e_1$,
$g^n(e_2)=e_2$, and $g(e_1)=e_2$. Let us now look at the stabilizer $G_{v'}$
of $v'$ in $G$. The element $g^n$ fixes $v'$ and two distinct edges exiting from
$v'$, so it belongs to the fiber subgroup of $G_{v'}$ (see Lemma~\ref{acyl-lem}).
Since $g\in G_{v'}$, this easily implies that
also $g$ belongs to the fiber subgroup
of $G_{v'}$. As a consequence, $g$ fixes all the edges exiting
from $v'$, a contradiction since $g(e_1)=e_2\neq e_1$.

Point (2) is an immediate consequence of point (1), since the stabilizer of an edge of $T'$
coincides with the stabilizer of an edge of $T$.

(3): Let $e_1,e_2$ be distinct edges of $T$ sharing the vertex $v$, and suppose
that $g\in G$ is such that $g(e_1)=e_2$, $g(e_2)=e_1$. Then $g^2(e_i)=e_i$
for $i=1,2$, so $g(e_i)=e_i$ by point (1), a contradiction.
\end{proof}


\section{Relative hyperbolicity and hyperbolically embedded subgroups}\label{relhyp:sec}
The fundamental group of a purely hyperbolic piece of a graph manifold provides the typical example
of relatively hyperbolic group. The following result shows that there exist more complicated (extended) graph manifolds
with relatively hyperbolic fundamental group:

\begin{proposition}\label{exception}
Assume the (extended) graph manifold $M$ has at least one purely hyperbolic piece. 
Then $\pi_1(M)$ is relatively hyperbolic with respect to a finite family of proper subgroups.
\end{proposition}
\begin{proof}
Let $M$ be an (extended) graph manifold containing the purely hyperbolic piece $M_0$.
We define a finite family $\mathcal{P}(M)$ of subgroups of $\pi_1(M)$ as follows.
Let $M_1\cup\ldots\cup M_k$ be the (obvious compatifications of the)
connected components of $M\setminus M_0$, and for every $i=1,\ldots,k$ let
$G_i$ be the image of $\pi_1(M_i)$ into $\pi_1(M)$. 
Moreover, let $T_1,\ldots,T_h$ be the connected components of $\partial M\cap \partial M_0$ and
let $P_j$ be the image of $\pi_1(T_j)$ in $\pi_1(M)$ for $j=1,\ldots,h$. 
Note that $G_i$ and $P_j$ are well-defined only up to conjugation, but this won't be relevant to our purposes. We set
$$
\mathcal{P}(M)=\{G_1,\ldots,G_k,P_1,\ldots,P_h\}\ ,
$$
and we claim that $\pi_1(M)$ is relatively hyperbolic with respect to the family
of subgroups $\mathcal{P}(M)$.

We proceed by induction on the number $c(M)$ of boundary components
of $M_0$ that are \emph{not} boundary components of $M$.
If $c(M)=0$, then $M=M_0$, 
and the conclusion follows from the fact that the fundamental group
of a complete finite-volume hyperbolic manifold is relatively hyperbolic
with respect to its cusp subgroups~\cite{farb}.

Let now $H$ be an internal wall of $M$ corresponding to
a boundary component of $M_0$, and consider the manifold $M'$
obtained by removing from $M$ a regular open neighborhood of $H$.
Then $M'$ is an (extended) graph manifold containing at least one purely hyperbolic piece,
and $c(M')<c(M)$.
Therefore, we may assume that $\pi_1(M')$ is relatively hyperbolic with respect
to the family of subgroups $\mathcal{P}(M')$. Then,
Dahmani's combination Theorem implies that $\pi_1(M)$ is relatively hyperbolic with respect
to $\mathcal{P}(M)$
(see~\cite[Theorem 0.1 -- (2)]{Dahmani} if $M'$ is disconnected, 
and~\cite[Theorem 0.1 -- (3')]{Dahmani}
if $M'$ is connected).
\end{proof}

It seems likely that the condition described in Proposition~\ref{exception} is not only sufficient, but also necessary
for an (extended) graph manifold to have a relatively hyperbolic fundamental group. This is the case in the context of irreducible graph
manifolds, that will be discussed
in Chapter~\ref{preserve:sec} (see Proposition~\ref{thick:prop}). On the contrary, the fundamental group of an (extended) graph manifold often contains hyperbolically embedded
subgroups. In fact, 
putting together Proposition~\ref{acyl-new} with Proposition~\ref{DGO-prop} we get
the following:

\begin{proposition}\label{hyp-emb-elements}
 Let $M$ be an (extended) graph manifold containing an internal wall with transverse fibers.
Then $\pi_1(M)$ contains a non-degenerate hyperbolically embedded subgroup.
\end{proposition}

\section{Kazhdan subgroups}
In this Section we show how we can completely classify the subgroups of $\pi_1(M)$
which are Kazhdan (we refer the reader to \cite{BdlHV} for a comprehensive introduction to Kazhdan groups). At the 
other extreme, one has amenable subgroups, which will be analyzed in the next section.

\vskip 10pt

\begin{proposition}\label{Kazhdan subgroups-pre}
Let $\mathcal{G}$ be a graph of groups with fundamental group $G$, and 
suppose that no vertex group
of $\mathcal{G}$ contains a non-trivial subgroup which satisfies
Kazhdan's property (T). Then no non-trivial subgroup of $G$
satisfies Kazhdan's property (T).
\end{proposition}
\begin{proof}
Let $T$ be the Bass-Serre tree associated to $\mathcal{G}$. Being a subgroup of $G$, the group $H$ acts on $T$.
Kazhdan groups are known to have Serre's property (FA), i.e. any action on a tree has a globally fixed point
(see \cite[Section 2.3]{BdlHV}). We conclude that $H$ 
must fix a vertex in $T$, and hence is isomorphic to a subgroup of a vertex group of $\mathcal{G}$.
Our assumptions now imply that $H=\{1\}$, and this concludes the proof.
\end{proof}

\begin{corollary}[Kazhdan subgroups of (extended) graph manifold groups]\label{Kazhdan subgroups}
Let $M$ be an (extended) graph manifold, and $H\leq \pi_1(M)$ an arbitrary subgroup.  If $H$ has Kazhdan's property (T), then
$H$ has to be the trivial group.
\end{corollary}

\begin{proof}
By Proposition~\ref{Kazhdan subgroups-pre}, it is sufficient to show that, if $H$ is an arbitrary subgroup
of $\pi_1(V) \cong \pi_1(N) \times \mathbb Z^k$, where $N$ is a non-compact, finite volume 
hyperbolic manifold, and $H$ has  Kazhdan's property (T), then
$H=\{1\}$.

Looking at the image of $H$ inside the
factor $\pi_1(N)$, we get an induced action of $H$ on hyperbolic space. But any action of a Kazhdan group on hyperbolic space must have a 
global fixed point (see \cite[Section 2.6]{BdlHV}). Since $\pi_1(N)$ acts freely on hyperbolic space, we conclude that $H$ 
must lie in the kernel of the natural
projection $\pi_1(V)\rightarrow \pi_1(N)$, i.e. must be entirely contained in the $\mathbb Z^k$ factor. Finally, 
the only subgroup of $\mathbb Z^k$ that has Kazhdan's property (T) is the trivial group, concluding the proof.
\end{proof}

\vskip 10pt

By~\cite{BDS}, there are finitely many conjugacy classes of homomorphisms from a Kazhdan group into a mapping class
group. With respect to this issue, the behaviour of $\pi_1 (M)$ is similar. In fact, 
as the homomorphic image of a Kazhdan group is Kazhdan, an immediate consequence of the previous Lemma is
the following:

\begin{corollary}
Let $M$ be an (extended) graph manifold. Then,
there are no non-trivial homomorphisms from a Kazhdan group to $\pi_1(M)$.
\end{corollary}

\section{Uniformly exponential growth}
We now consider the notion of {\it growth} of a group $G$. Fixing a finite,
symmetric generating set $S$, one considers the Cayley graph $C_{S}(G)$ of $G$ with respect to the generating set 
$S$. Recall that the graph $C_{S}(G)$ is viewed as
a metric space by setting every edge to have length one. For any positive real number $r$, we can look at the ball of radius $r$
in $C_{S}(G)$ centered at the identity element, and let $N_{S}(r)$ count the number of vertices lying within that ball. 
The group has {\it exponential growth} provided there exists a real number $\lambda_{S}>1$ with the property 
$N_{S}(r) \geq \lambda_{S}^r$. The property of having exponential growth is a quasi-isometry invariant, hence does
not depend on the choice of generating set $S$, though the specific constant $\lambda_S$ does depend on the choice
of generating set. It is easy to see that any group which contains a free subgroup (such as
the fundamental groups of our (extended) graph manifolds) automatically has exponential growth. 
The more sophisticated notion of {\it uniform exponential growth} has been the subject of recent work. A group $G$ has 
uniform exponential growth if there exists a $\lambda >1$ with the property that, for every finite symmetric generating set $S$,
we have $N_{S}(r) \geq \lambda^r$. The point here is that the constant $\lambda$ is independent of the generating set $S$.
Non-elementary Gromov hyperbolic groups are known to have uniform exponential growth (see Koubi \cite{Kou}), 
while CAT(0) groups might not even have exponential growth (as the example of $\mathbb Z^n$ shows). In our situation, an
easy argument shows:

\vskip 10pt

\begin{proposition}\label{UEG}
If $M$ is an (extended) graph manifold, then $\pi_1(M)$ has uniform exponential growth.
\end{proposition}

\begin{proof}
Bucher and de la Harpe \cite{Bu-dlH} have analyzed uniform exponential growth for groups which split as an amalgam (or as an HNN 
extension). It follows immediately from their work that if the graph of groups description of $\pi_1(M)$ does {\it not} reduce to 
a single vertex, then $\pi_1(M)$ has uniform exponential growth. So we merely need to consider the remaining case, where $M$ has
a single piece. In this case, $\pi_1(M)$ splits as a product $\pi_1(V) \times \mathbb Z^k$, where $V$ is a non-compact, finite volume hyperbolic
manifold. But projecting onto the first factor, we see that $\pi_1(M)$ surjects onto a group of uniform exponential growth (by work of 
Eskin, Mozes, and Oh \cite{EMO}). It follows that $\pi_1(M)$ also has uniform exponential growth, concluding the proof of the Proposition.
\end{proof}

\vskip 10pt

Recall that given a Riemannian metric $g$ on a compact manifold $M$, the {\it volume growth entropy} of the metric
is defined to be the limit 
$$h_{vol}(M,g):=\lim _{r\rightarrow \infty} \frac{1}{r}\log \big(Vol_{\tilde g}(B(r))\big)$$
where $B(r)$ is the ball of radius $r$ centered
at a fixed point in the universal cover $(\widetilde M, \tilde g)$ with the pull-back metric from $(M,g)$. Work of
Manning \cite{Ma} shows that the {\it topological entropy} $h_{top}(M,g)$ of the geodesic flow on the unit tangent bundle of $M$
satisfies the inequality $h_{top}(M,g) \geq h_{vol}(M,g)$. An immediate consequence of uniform exponential growth is the:

\vskip 10pt

\begin{corollary}
For $M$ an (extended) graph manifold, there exists a real number $\delta _M>0$ with the property that for {\it any} Riemannian
metric $g$ on $M$, normalized to have diameter equal to one, we have the inequality 
$h_{top}(M,g)\geq h_{vol}(M,g)\geq \delta_M>0$.
\end{corollary}

\section{The Tits Alternative}
We now show that  the fundamental group of an (extended) graph manifold satisfies 
a strong version of the Tits Alternative. If $G$ is a group, we denote by $G^{(1)}$
the subgroup of commutators, and we inductively define $G^{(n)}$ by setting
$G^{(n+1)}=[G^{(n)},G^{(n)}]$. Recall that a group $G$ is \emph{solvable} if
$G^{(n)}=\{1\}$ for some $n\in\mathbb{N}$. The least $n$ such that 
$G^{(n)}=\{1\}$ is the \emph{derived length} of $G$. If 
$$
0 \to H \to G\to G/H \to 0
$$
is an exact sequence of groups, and $H, G/H$ are both solvable, then also $G$ is solvable,
and the derived length of $G$ is at most the sum of the derived lengths of $H$ and $G/H$.
We refer the reader to Section~\ref{tree-graph-theory} for the definition of action without reflections.

\begin{proposition}\label{TA-graph}
Let $\mathcal{G}$ be a finite graph of groups with fundamental group $G$
and Bass-Serre tree $T$, and suppose that 
any arbitrary subgroup of any vertex group either contains
a non-abelian free group, or is abelian. Let $H$ be an arbitrary subgroup of $G$. Then
either:
\begin{itemize}
 \item $H$ is sovable, or
\item
$H$ contains a non-abelian free group.
\end{itemize}
Moreover, if $H$ is solvable, then:
\begin{enumerate}
\item
the derived length of $H$ is at most $3$;
\item
if the action of $G$ on $T$ is acylindrical or without reflections, then
the derived length of $H$ is at most $2$.
\item
if the action of $G$ on $T$ is acylindrical
and without reflections, then $H$ is abelian.
\end{enumerate}
\end{proposition}
\begin{proof}
Let $T$ be the Bass-Serre tree associated to $\mathcal{G}$.
We first prove that, if $K$ is a subgroup of $G$ that consists solely of elliptic elements and does not contain any non-abelian free group,
then $K$ is abelian. If $K$ fixes a vertex of $T$, then  
this is just our hypothesis.
Otherwise, it is proved in  \cite[Proposition 3.7]{Ba1} that
there exists an infinite path in $T$, say with vertices $v_0,\ldots,v_n,\ldots$,
such that $K_{v_i}\subseteq K_{v_{i+1}}$ for every $i$, and $K=\cup_{i\geq 0} K_{v_i}$.
In particular, for every $i$
the group $K_{v_i}$ does not contain a non-abelian free group, so
our hypothesis implies that $K_{v_i}$ is abelian. But any ascending union of abelian
groups is abelian, so $K$ is itself abelian.

Let us now come back to our arbitrary subgroup $H$ of $G$, and let us suppose
that $H$ does not contain a non-abelian free group. By the discussion above, we may also assume that $H$ contains an element acting hyperbolically on $T$.
By~\cite[Section 2]{PV}, these conditions imply that there are two possibilities
for $H$:
\begin{enumerate}
\item $H$ is a subgroup of ${\rm Stab}(\gamma)$, where $\gamma\subset T$ is a geodesic, or
\item $H$ is a subgroup of ${\rm Stab}(\mathcal E)$, where $\mathcal E$ is an end of $T$.
\end{enumerate} 
In each of these cases, we need to show that $H$ is solvable, and estimate the derived length of $H$.


Let us consider case (1). 
Since $H$ leaves $\gamma$ invariant, we can define ${\rm Isom}_H(\gamma)$
as the image of $H$ in the group of isometries of $\gamma$
(which we can identify with $\mathbb R$).
If we
denote by $H_\gamma$ is the subgroup of $H$ which pointwise fixes $\gamma$, then we get the exact sequence:
$$0\rightarrow H_\gamma \rightarrow H \rightarrow {\rm Isom}_H(\gamma)\rightarrow 0\ .$$
The group $H_\gamma$ fixes
any given vertex of $\gamma$, hence can be identified with a subgroup of a vertex 
group of $\mathcal{G}$. Since $H_\gamma$ does not contain any non-abelian free group,
this implies that $H_\gamma$ is abelian.
Also observe that, if the action of $G$ on $T$ is acylindrical, then 
$H_\gamma$ 
reduces to the identity.

On the other hand, the group ${\rm Isom}_H(\gamma)$
is a subgroup of the group of simplicial automorphisms of $\mathbb R$ (with the standard simplicial structure), hence
is either $1, \mathbb Z_2, \mathbb Z$, or the infinite dihedral group $\mathbb D_\infty$. In all cases, we see that 
${\rm Isom}_H(\gamma)$ is solvable of derived length not bigger than 2. 
Also observe that,
if $G$ acts on $T$ without reflections, then ${\rm Isom}_H (\gamma)$ is 
necessarily abelian.

From the short exact sequence, we deduce that $H$ 
is solvable. Moreover, its derived length is at most 3 in general. If $G$ acts on $T$ acylindrically or without inversions, then
the derived length of $H$ is at most 2, while if $G$ acts on $T$ acylindrically \emph{and} without inversions,
then $H$ is abelian. This concludes the proof in case (1).

Let us now consider case (2).
To analyze this case, we consider the {\it relative translation length} map. Given an 
end $\mathcal E$ of a tree $T$, and any pair of vertices $v,w\in T$,
there are unique unit speed geodesic rays $\gamma_v, \gamma_w \subset T$ originating at $v,w$, and exiting into 
the end $\mathcal E$. One then defines the distance of the points {\it relative to $\mathcal E$} to be the integer 
$d_{\mathcal E}(v,w) := \lim_{t\to \infty} d\big(\gamma_v(t), \gamma_w(t)\big)$. The relative translation length of an element 
$g\in {\rm Stab}(\mathcal E)$ is defined to be the integer $\tau(g):=\inf_v d_{\mathcal E}\big(v,g(v)\big)$. A basic property 
of the relative translation length is that it defines a homomorphism $\tau: {\rm Stab}(\mathcal E) \rightarrow \mathbb Z$
(see e.g. \cite[Lemme 4]{PV}).
So our group $H$ fits into a short exact sequence
$$0\rightarrow H_0\rightarrow H \rightarrow \mathbb Z\rightarrow 0$$
where $H_0 = H\cap \ker (\tau)$. 
But 
every element in $H_0$ has to be elliptic, so the discussion
at the beginning of the proof implies that $H_0$ is abelian. This implies that $H$ is solvable of derived length at most $2$. Also observe that every element
of $H_0$ fixes a geodesic ray exiting into the end $\mathcal{E}$. Therefore,
if the action of $G$ is acylindrical, then $H_0=\{1\}$, which implies that $H$ is abelian.
\end{proof}

\begin{corollary}\label{TA}
Let $M$ be an (extended) graph manifold, and $H$ be an arbitrary subgroup
of $\pi_1(M)$. Then either:
\begin{itemize}
\item $H$ is solvable of derived length at most 2, or
\item $H$ contains a non-abelian free group.
\end{itemize}
If $M$ is irreducible, then either:
\begin{itemize}
\item $H$ is abelian, or
\item $H$ contains a non-abelian free group.
\end{itemize}
\end{corollary}
\begin{proof}
Let $T$ be the Bass-Serre tree associated to the decomposition of $M$ into pieces.
Recall that the action of $G$ on $T$ is acylindrical if and only if $M$ is irreducible
(Proposition~\ref{irr-acyl}). Moreover, Lemma~\ref{root free} implies that $G$ acts on $T$ without reflections.
Therefore, by Proposition~\ref{TA-graph}
it is sufficient to prove that, if $H$ is a subgroup of the fundamental group
$\pi_1(N)\times \mZ^d$ of a piece of $M$, and $H$ does not contain any non-abelian free group, then $H$ is abelian. Let $\overline{H}$ be the projection of $H$ onto $\pi_1(N)$,
and recall that $\pi_1(N)$ acts by isometries on the hyperbolic space $\mathbb{H}^{n-k}$. 
Every non-elementary discrete group of isometries 
of $\mathbb{H}^{n-k}$ contains a non-abelian free group (see e.g.~\cite[page 616, Exercise 15]{Ratcliffe}), so $\overline{H}$ must be elementary.
Moreover, $\overline{H}$
does not contain any elliptic element, so $\overline{H}$ is contained either in an infinite cyclic hyperbolic subgroup of $\pi_1(N)$
or in a parabolic subgroup of $\pi_1(N)$. But the cusps of $N$ are toric, so the parabolic subgroups of $\pi_1(N)$ are abelian.
We have thus proved that $\overline{H}$ is abelian, so $H$ is itself abelian, and we are done.
\end{proof}

\begin{remark}
As a consequence of the Flat Torus Theorem, if a solvable group $G$ acts properly via semisimple isometries
on a CAT(0) space, then $G$ is virtually abelian. This fact provides a useful obstruction for a group to be the fundamental group
of a compact locally CAT(0) space. Corollary~\ref{TA} implies that this obstruction is never effective for irreducible graph manifolds. In fact,
the construction of irreducible graph manifolds that do not support any locally CAT(0) metric described in Chapter~\ref{construction2:sec}  is based on a more 
sophisticated use of the Flat Torus Theorem.
\end{remark}

\vskip 10pt

Since a group which contains a non-abelian free subgroup is automatically non-amenable, Corollary~\ref{TA} 
implies the following:

\begin{corollary}
Let $M$ be an (extended) graph manifold and let  $H\leq \pi_1(M)$ be an amenable subgroup. Then $H$ is solvable.
If, in addition, $M$ is irreducible, then $H$ is abelian.
\end{corollary}

\section{Co-Hopf property}

\begin{proposition}\label{CH}
Let $M$ be an (extended) graph manifold, with $\partial M=\emptyset$, and  assume that 
$M$ 
contains a pair of adjacent pieces with transverse fibers.
Then the fundamental group $\pi_1(M)$ 
is co-Hopfian, i.e. every injective homomorphism $\phi: \pi_1(M)\hookrightarrow \pi_1(M)$ is automatically an isomorphism.
\end{proposition}

\begin{proof}
Let $G:= \pi_1(M)$. Using $\phi$, we can identify $\phi (G)$ with a subgroup of $G$, and our goal is
to show the index $[G: \phi(G)]$ must be equal to one. A standard argument shows that $[G: \phi(G)]$
must be finite, for if it wasn't, then we would have two manifold models for a $K(G, 1)$: the compact manifold $M$, and its non-compact
cover $\widehat M$ corresponding to the infinite index subgroup $\phi(G) \leq G$. Using these models to compute the 
top dimensional group cohomology of $G$ with $\mathbb Z /2$-coefficients gives:
$$\mathbb Z/2 = H^n(M^n; \mathbb Z/2) \cong H^n(G;\mathbb{Z}/2) \cong H^n(\widehat M; \mathbb Z/2) =0,$$
a contradiction.

Now assume the index is some finite number $[G: \phi(G)] = k$, which we would like to show is equal to $1$. 
We consider again the covering map $\pi\colon \widehat M\to M$ associated to the subgroup $\phi(G)$.
Observe that $\widehat{M}$ is itself an (extended) graph manifold, whose pieces are just the connected components
of the preimages under $\pi$ of the pieces of $M$.
By smooth rigidity, the isomorphism $\phi\colon \pi_1(M)\to \pi_1(\widehat{M})$ is realized by a diffeomorphism
$f\colon M\to \widehat{M}$ which induces a bijection between the set of pieces of $M$ and the set of pieces of $\widehat{M}$
(see Theorem~\ref{preserve2:thm}). This can only happen if, under our covering map
$\pi$, each piece of $M$ lifts to a single piece in $\widehat M$.
Let now $g:=\pi\circ f\colon M\to M$. 
The map $g$ permutes the pieces of $M$, so there exists $s\in\mathbb{N}$ such that $g^s(V)=V$ for every piece $V$ of $M$. 
By construction,
the map $g^s\colon M\to M$ is a $k^s$ degree covering, and restricts
to a degree $k^s$ covering $g^s|_V\colon V\to V$ for every piece $V$
of $M$. Therefore, if we set $\psi:=\phi^s$, then, up to conjugation,
we have $\psi\big(\pi_1(V)\big) \subset \pi_1(V)$ and 
$\big[ \pi_1(V): \psi\big(\pi_1(V)\big)\big] = k^s$
for each piece $V$ of $M$.

Let us now fix an arbitrary
piece $V$ in $M$, and let $V$ be homeomorphic to $N\times T^d$, where as usual $N$ is a non-compact finite volume hyperbolic 
manifold  and $T^d$ is a $d$-dimensional torus. The group $\Lambda:=\pi_1(V)$ is isomorphic to 
$\pi_1(N) \times \mathbb Z^d$, and $\psi$ restricts to give us an injective map from this group to itself.

We now analyze the possible injective maps from $\Lambda=\pi_1(N)\times \mathbb Z^d$ into itself (this is similar to 
the analysis in Lemma \ref{diffeo:pieces}). 
Let $\rho: \Lambda \rightarrow \pi_1(N)$ be the natural projection onto the $\pi_1(N)$ factor. 
As a first step, we consider the effect of $\psi$ on the $\mathbb Z^d$ factor in $\Lambda$, and show that its image must
be contained in the $\mathbb Z^d$ factor. Look at the image of $\psi(\mathbb Z^d) \leq \Lambda$ under
the $\rho$ map. The group $\rho\big(\psi(\mathbb Z^d)\big)$ is a free abelian subgroup of $\pi_1(N)$, and our goal is 
to show it is trivial. Since the $\mathbb Z^d$ factor is the center of the group $\Lambda$ (see Remark~\ref{center:rmk}), we see that
all of $\rho \big(\psi(\Lambda)\big)$ is contained in the centralizer of $\rho \big(\psi(\mathbb Z^d)\big)$. But inside the 
group $\pi_1(N)$, the centralizer of any {\it non-trivial} free abelian subgroup is itself free abelian (see Lemma~\ref{basicprop:lem}). 
This implies that $\rho \big(\psi(\mathbb Z^d)\big)$ is indeed trivial, because otherwise 
the preimage of its centralizer under $\rho$ should also be free abelian, but should contain an embedded copy 
 $\psi(\Lambda)$ of the non-abelian group $\Lambda$. Since $\rho \big(\psi(\mathbb Z^d)\big)$ is indeed trivial,
we conclude that $\psi(\mathbb Z^d)\leq \ker (\rho) \cong \mathbb Z^d$. In other words, we
have just established that the map $\psi$ embeds the $\mathbb Z^d$ factor into itself.

Next, let us see how the map $\psi$ behaves on the $\pi_1(N)$ factor, by again considering the composition with $\rho$. 
From the discussion in the previous paragraph, we have that $\rho \big( \psi(\Lambda) \big) = \rho \big( \psi(\pi_1(N))\big)$.
Since $\psi(\Lambda)$ has finite index in $\Lambda$, the same holds for any homomorphic image, giving us that 
$\rho \big( \psi(\pi_1(N))\big)$ has finite index in $\pi_1(N)$. But the group $\pi_1(N)$ is known to be {\it cofinitely Hopfian} 
(see \cite[Prop. 4.2]{BGHM}),
i.e. any homomorphism $\pi_1(N)\rightarrow \pi_1(N)$ whose image has finite index is automatically an isomorphism. 
We conclude that the composite $\rho \circ \psi$ maps $\pi_1(N)$ isomorphically onto $\pi_1(N)$. Summarizing our
discussion so far, in terms of the two factors in the group $\Lambda$, we can decompose the morphism $\psi$ as
$\psi(g, v) = \big( \phi(g), \nu(g) + Lv\big)$, where $\phi\in Aut \big(\pi_1(N)\big)$, $\nu\in Hom \big(\pi_1(N), \mathbb Z^d\big)$,
and $L$ is a $d\times d$ matrix with integral entries and non-vanishing determinant.

To calculate the index of $\psi(\Lambda)$ in $\Lambda$, 
consider the automorphism $\widehat {\psi} \in Aut(\Lambda)$ defined via $\widehat {\psi} (g, v) = \big(g, -\nu (\phi^{-1}(g))+v\big)$.
An easy computation shows that $\big(\widehat {\psi} \circ \psi \big)(g, v) =\big(\phi (g), Lv\big)$, allowing us to see that the index is
$$k^s = \big[ \Lambda : \psi (\Lambda)\big]=\big[ \Lambda : \widehat {\psi} \big( \psi (\Lambda)\big)\big]  
= \big[\mathbb Z^d : L(\mathbb Z ^d)\big]= |\det (L)|.$$
This formalizes the statement that the degree $k^s$ cover $g^s\colon M \rightarrow M$ comes from unfolding
the torus factors in each piece of $M$ (along with sliding the base over the fiber, which has no affect on the degree).

Finally, let us return to our manifold $M$, and exploit the hypothesis on transverse fibers. Let $V_1, V_2$ be the pair of adjacent 
pieces with transverse fibers along the common torus $T$. The torus $T$ corresponds to a $\mathbb Z^{n-1}$ subgroup of $G$,
and the two pieces give splittings of this group into direct sums $F_1^s\oplus B_1^{n-1-s} = \mathbb Z^{n-1}= F_2^t\oplus B_2^{n-1-t}$, 
where $F_i$ are the fiber subgroups and $B_i$ are the base subgroups. The homomorphism $\psi$ takes 
$\mathbb Z^{n-1}$ into itself, and by the analysis above, we can compute the index in two possible ways:
$$|\det(L_1)| = \big[F_1 : \psi(F_1)\big] = \big[\mathbb Z^{n-1} : \psi(\mathbb Z^{n-1})\big] = \big[F_2 : \psi(F_2)\big] = |\det (L_2)|$$
where $L_i$ is a matrix representing the $\psi$ action on $F_i$. Therefore, we get $k^s=|\det (L_1)|=|\det (L_2)|=|\det (\widehat{L})|$,
where $\widehat{L}$ is a matrix representing the $\psi$ action on $\mathbb{Z}^{n-1}$. We will now show that this forces
$k^s=1$, whence the conclusion.

Since we have transverse fibers, we have $F_1\cap F_2= \{0\}$. Let us denote by $K$ the subgroup $F_1\oplus F_2\subseteq\mathbb{Z}^{n-1}$,
and let us set $J=\{v\in\mathbb{Z}^{n-1}\, |\, mv\in K\ {\rm for\ some}\ m\in\mathbb{Z}\}$. Of course, $K$ is a finite index subgroup of $J$, and 
the $\psi$-invariance of $K$ implies that also $J$ is $\psi$-invariant. Our choices also ensure that the quotient
group $ \mathbb{Z}^{n-1}/J$ is free abelian.
Since $\psi$ is injective, the following equalities hold:
$$
\big[J:K\big]\, \big[K:\psi (K)\big]=\big[J:\psi(K)\big]=\big[J:\psi(J)\big]\, \big[\psi(J):\psi(K)\big]=
\big[J:\psi(J)\big]\, \big[J:K\big]\, .
$$
This tells us that 
$$\big[J:\psi(J)\big]=\big[K:\psi(K)\big]=\big| \det (L_1)\cdot \det (L_2)\big| =k^{2s}.$$ 
Moreover, $\psi$ induces a homomorphism
$\overline{\psi}\colon \mathbb{Z}^{n-1}/J \to  \mathbb{Z}^{n-1}/J$, and we have of course
$\det (\widehat{L})=\det (L_J) \cdot \det (\overline{L})$, where $L_J$ and $\overline{L}$ are matrices representing
$\psi |_J$ and $\overline{\psi}$ respectively. Since $\det (\overline{L})\geq 1$, we finally get
$$
k^{2s}=\big[J:\psi(J)\big]=| \det (L_J) | \leq | \det (\widehat{L}) | =\big[\mathbb{Z}^{n-1}:\psi(\mathbb{Z}^{n-1})\big]=k^s\, .
$$
We conclude from this inequality that $k=1$, giving us that  $[G : \phi(G)]=k=1$, as desired.
\end{proof}

\begin{remark}
In Proposition~\ref{CH}, one cannot remove the assumption that $M$ contains a pair of adjacent pieces with transverse fibers.
In fact, if $N$ is any hyperbolic manifold with toric cusps and $d\geq 1$, then the fundamental groups of 
the graph manifolds $\overline{N}\times T^d$
and $D\overline{N}\times T^d$, where $D\overline{N}$ is the double of $\overline{N}$, are not co-Hopfian. It would be interesting
to understand whether Proposition~\ref{CH} still holds under the weaker hypothesis that $M$ be non-fibered.
\end{remark}

\begin{remark}
Most arguments proving that the fundamental group of a closed manifold is co-Hopfian usually involve
invariants which are multiplicative under coverings. Two such invariants which are commonly used are the Euler characteristic $\chi$,
and the simplicial volume. But in the case where every piece in our (extended) graph manifold has non-trivial fiber, both these invariants vanish. In fact, if $V=\overline{N}\times T^d$, $d> 0$, is a piece with non-trivial fiber,
then $\chi (V)=\chi (\overline{N})\times \chi (T^d)=0$. Moreover, the pair $(V,\partial V)$ admits a self-map of degree greater than
one, and this easily implies that the (relative) simplicial volume of $V$ vanishes.
Suppose now that a compact manifold $M$ is obtained by gluing a (maybe disconnected) $M'$ along pairs of $\pi_1$-injective toric boundary components.
Since the Euler characteristic of the torus is zero we have $\chi (M)=\chi (M')$, while the amenability
of $\mZ^d$ and Gromov additivity Theorem~\cite{gro-bddcohom} 
(see also~\cite{ku} and~\cite{BBFIPP})
imply that the (relative) simplicial volumes of $M$ and $M'$ coincide.
Together with an obvious inductive argument, this implies that $\chi (M)=\| M\|=0$ for an
(extended)
graph manifold, provided all its pieces have non-trivial fibers.  

Conversely, the (relative) simplicial volume is additive with respect to gluings along
$\pi_1$-injective tori, and it never vanishes on a cusped hyperbolic manifold. So if there
is a single piece in $M$ which is purely hyperbolic (i.e. has trivial fiber), then $||M||>0$. 
Similarly, the Euler characteristic of an even dimensional cusped hyperbolic manifold is
never zero, so a similar conclusion holds. We summarize this discussion in the following:

\end{remark}

\begin{proposition}\label{simp-vol}
Let $M$ be an (extended) graph manifold. Then 
\begin{enumerate}
\item $||M||= 0$ if and only if every piece in $M$ has non-trivial fibers, and
\item if $M$ is even dimensional, then $\chi(M)=0$ if and only if every piece of $M$ has non-trivial fibers.
\end{enumerate}
\end{proposition}

\section{$C^\ast$-simplicity of acylindrical graphs of groups}\label{C-simple:sec}
Recall that to any countable discrete group $G$, one can associate $C^*_r(G)$, its {\it reduced $C^*$-algebra}. 
This algebra is obtained by looking at the action $g\mapsto \lambda_g$ of $G$ on the Hilbert space $l^2(G)$ of square summable 
complex-valued functions on $G$, given by the left regular representation:
$$\lambda_g \cdot f(h) = f\big(g^{-1}h\big) \hskip 20pt g, h \in G, \hskip 20pt f\in l^2(G)\ .$$
The algebra $C^*_r(G)$ is defined to be the operator norm closure of the linear span of the operators $\lambda_g$ inside the
space $B\big(l^2(G)\big)$ of bounded linear operators on $l^2(G)$. The algebra $C^*_r(G)$ encodes various analytic properties
of the group $G$, and features prominently in the Baum-Connes conjecture. 
A group $G$ is said to be {\it $C^*$-simple} if the algebra $C^*_r(G)$ is a simple algebra, i.e.~has no proper two-sided ideals. 
We refer the interested reader to the survey paper by de la Harpe \cite{dlH} for an extensive discussion of this notion. The following result may be deduced
from~\cite{DGO}, and characterizes acylindrical
graphs of groups having a $C^*$-simple fundamental group. 

\vskip 10pt

\begin{proposition}\label{C-simple-graphs}
Let $G$ be the fundamental group of a non-trivial acylindrical graph of groups $\mathcal{G}$. Then $G$ is $C^*$-simple if and only if $\mathcal{G}$ is not exceptional.
\end{proposition}
\begin{proof}
If $\mathcal{G}$ is exceptional, then $G$ is virtually abelian, whence amenable.
As a consequence, $G$ is not $C^*$-simple (see e.g.~\cite{dlH}). Therefore, we are left to show that
$G$ is $C^*$-simple provided that $\mathcal{G}$ is not exceptional.

However, if $\mathcal{G}$ is not exceptional, then Proposition~\ref{DGO-prop}
implies that
$G$ contains a non-degenerate hyperbolically embedded subgroup.
Moreover, since $T$ has infinite diameter, any acylindrical action on $T$
is faithful, so
Lemma~\ref{reduced} guarantees that $G$ does not contain any non-trivial finite normal
subgroup. These conditions allow us to apply~\cite[Theorem 2.32]{DGO}, which
concludes the proof of the Proposition.
\end{proof}

\begin{remark}
Proposition~\ref{C-simple-graphs} could be probably deduced also from the results proved in~\cite{dlH-Pr},
which in fact imply $C^*$-simplicity of amalgamated products and HNN-extensions under a weaker hypothesis than acylindricity
(see the proof of Proposition~\ref{C-simple-manifolds} below). However,
some work would be required to reduce the case of generic graphs of groups to the case of one-edged graphs of groups.
\end{remark}

Propositions~\ref{acyl-new} and~\ref{C-simple-graphs} already imply
that $\pi_1(M)$ is $C^*$-simple, provided that $M$ is an (extended) graph manifold
with at least one internal wall with transverse fibers.
In the following proposition we improve this result and give a complete
characterization of (extended) graph manifolds with $C^*$-simple fundamental group.

\begin{proposition}[Non-fibered $\Longleftrightarrow$ $C^*$-simple]\label{C-simple-manifolds}
Let $M$ be an (extended) graph manifold. Then $\pi_1(M)$ is $C^*$-simple if and only if
$M$ is not fibered.
\end{proposition}
\begin{proof}
Let us first suppose that $\pi_1(M)$ is $C^*$-simple. It is well-known that a $C^*$-simple group
cannot contain non-trivial amenable normal subgroups (see e.g.~\cite{dlH}). 
If $M$ consists of a single piece, this implies that $M$ is purely hyperbolic, whence non-fibered (see Lemma~\ref{pure:lemma}). Otherwise, we may apply Proposition~\ref{weakly:char}, and conclude again that
$M$ is not fibered.

\smallskip

Let us now turn to the converse implication.
A criterion for $C^*$-simplicity was discovered by Powers \cite{Pow}, who showed that the free group on two generators is
$C^*$-simple. In fact, 
as observed e.g.~in \cite{dlH-new}, Powers' argument applies to 
every group belonging to the class of \emph{Powers group}, as defined in~\cite{dlH-new}.
Our argument exploits
some criteria for a countable group to be Powers that are described in~\cite{dlH-Pr}.

Suppose first that $M$ contains at least one internal wall, let $\mathcal{G}$ be the graph
of groups corresponding to the decomposition of $M$ into pieces, 
and let $\mathcal{G}'$
be a graph of groups obtained by collapsing $\mathcal{G}$ outside an edge of $\mathcal{G}$.
Let us denote by $G$ the fundamental group of $M$, and by $T'$ the Bass-Serre tree 
associated to $\mathcal{G}'$.
Of course we have $G=\pi_1(\mathcal{G}')$, and the graph of groups $\mathcal{G}'$
describes $G$ as an amalgamated product or an HNN-extension. The edge groups of $\mathcal{G}'$ have infinite index in the adjacent vertex groups, so the main result
of~\cite{dlH-Pr} ensures that $G$ is $C^*$-simple, provided that the following condition holds:

\begin{itemize}
 \item[(*)] There exists $k\in\mathbb{N}$ such that, if $e$ is an edge of $T'$
and  $g\in G$ pointwise fixes the $k$-neighborhood $N_k(e)$
of $e$ in $T'$, then $g=1$ in $G$.
\end{itemize}

So we are left to show that, if condition (*) does not  hold, then $M$ is fibered.
Suppose that for every $k\in\mathbb{N}$ there exist a non-trivial element 
$g_k\in G$ and an edge $e_k$ of $T'$
such that $g_k$ pointwise fixes $N_k(e_k)$. Recall that $\mathcal{G}'$
has only one edge, so, up to conjugating $g_k$, we may choose an edge
$e$ of $T'$ such that
$e_k=e$ for every $k$. Let us denote by $G_k$ the subgroup
of $G$ fixing pointwise $N_k(e)$. Then $G_0$ is isomorphic to a subgroup
of the edge group of $\mathcal{G}'$, so it is finitely generated free abelian.
Moreover, for every $k$ we have $G_{k+1}\subseteq G_{k}$. We are now going to show that
the sequence of groups $\{G_k,\, k\in\mathbb{N}\}$ stabilizes after a finite number of steps. Being a finitely generated abelian
group, every $G_k$ has a well-defined rank. Of course, the sequence of the ranks of the groups $G_k$ is eventually constant, so
it is sufficient to show that,
if $G_{k+1}$ is a finite index subgroup of $G_k$, then $G_{k+1}=G_k$.
Let $g\in G_{k}$. Since $G_{k+1}$ has finite index in $G_k$, there exists
$n\geq 1$ such that $g^n\in G_{k+1}$. Therefore, if $\overline{e}$ is an edge
in $N_{k+1}(e)$, then $g^n(\overline{e})=\overline{e}$. By Lemma~\ref{root free},
this implies that $g(\overline{e})=\overline{e}$. We have thus shown that every element
of $G_{k}$ fixes every edge of $N_{k+1}(e)$, so $G_{k}=G_{k+1}$.

Let now $k_0$  be such that $G_k=G_{k_0}$ for every $k\geq k_0$. Then the element
$g_{k_0}$ fixes the whole of $T$, so $g_{k_0}$ is a non-trivial element
of the kernel of the action of $G$ on $T$. By Proposition~\ref{weakly:char}, this implies that $M$ is fibered.

Let us now consider the case when $M$ consists of a single piece.
Being non-fibered, $M$ consists of a purely hyperbolic piece.
Therefore,
$\pi_1(M)$ is a non-elementary, relatively hyperbolic group. For these
groups, Arzhantseva and Minasyan \cite{AM} have shown that being $C^*$-simple is equivalent to having no
non-trivial finite normal subgroup. Since $\pi_1(M)$ is torsion-free, this latter condition is automatically
satisfied, and hence $\pi_1(M)$ is indeed $C^*$-simple. 
\end{proof}

\section{SQ-universality}\label{sq:sec}
Recall that a group $G$ is {\it SQ-universal} if every countable group can be embedded into a quotient of $G$. It is proved in \cite{DGO} that a group $G$ containing
a non-degenerate hyperbolically embedded subgroup is SQ-universal. Together with
Proposition~\ref{DGO-prop}, this readily implies the following:

\begin{proposition}\label{SQ-u-graph}
 Let $\mathcal{G}$ be a non-trivial acylindrical graph of groups,
and let $G$ be the fundamental group of $\mathcal{G}$. Then $G$ is SQ-universal
if and only if $\mathcal{G}$ is not exceptional. 
\end{proposition}
\begin{proof}
 If $\mathcal{G}$ is not exceptional, then $G$ is SQ-universal by Proposition~\ref{DGO-prop} and~\cite[Theorem 2.30]{DGO}. If $\mathcal{G}$
is exceptional, then $G$ is virtually abelian, so it cannot be SQ-universal.
\end{proof}

Our next result provides sufficient conditions under which the fundamental
group of an (extended) graph manifold is SQ-universal.

\begin{proposition}\label{SQ-u}
Let $M$ be an (extended) graph manifold, and
assume that at least one of the following conditions holds:
\begin{enumerate}
\item
$M$ consists of a single piece without internal walls, or
\item
$M$ contains at least one separating internal wall, or
\item
$M$ contains at least one internal wall with transverse fibers. 
\end{enumerate}
Then $\pi_1(M)$ is SQ-universal.
\end{proposition}
\begin{proof}
(1):
If $M$ consists of a single piece, then $\pi_1(M) \cong \pi_1(V) \times \mathbb Z^k$, where $V$ is a finite volume hyperbolic manifold
of dimension $\geq 3$. Since $\pi_1(V)$ is a non-elementary (properly) relatively hyperbolic group, work of Arzhantseva, Minasyan 
and Osin~\cite{AMO} implies that $\pi_1(V)$ is SQ-universal. Since $\pi_1(M)$ surjects onto a SQ-universal group, it is itself SQ-universal.

\medskip

(2): We first recall that
Lyndon \& Schupp \cite{LySc} provide some criterions under which an amalgamation or HNN-extension is 
SQ-universal.
For a group $A$, define a {\it blocking set} for a subgroup $C \leq A$ to be a pair of distinct elements 
$\{x,y\} \subset A \setminus C$ with 
the property that all the intersections $x^{\pm 1} C y^{\pm 1} \cap C = \{1\}$. Then 
\cite[pg. 289, Theorem V.11.3]{LySc} establishes that, if the subgroup $C$ is blocked inside $A$, the amalgamation 
$G=A*_CB$ is SQ-universal. 

We now verify that the conditions for SQ-universality are fulfilled for the amalgamations that arise
in case (2) of our statement. In this case,
the group $\pi_1(M)$ splits as an amalgamation over
$C:=\mathbb Z^{n-1}$, with the two vertex groups $A, B$ themselves fundamental groups of (extended) graph manifolds (with fewer 
pieces than $M$). Since the amalgamating subgroup $\mathbb Z^{n-1}$ is contained in a piece, it is sufficient to show that
a blocking set exists within the fundamental group of that piece. By projecting onto the first factor, the group 
$\pi_1(V)\times \mathbb Z^k$ acts on 
$\widehat  {\mathbb H}^{n-k}$, a copy of hyperbolic space with a suitable 
$\pi_1(V)$-equivariant collection of (open) horoballs removed. The subgroup $C=\mathbb Z^{n-1}$ can then be identified
with the subgroup that leaves invariant a fixed boundary horosphere ${\mathcal H} \subset \partial (\widehat {\mathbb H} ^{n-k})$.
In this context, the blocking condition requires us to find two elements $x,y\in \pi_1(V) \setminus C$ with the property that
$x^{\pm 1} C y^{\pm 1} \cap C = \{1\}$, which is equivalent to $(x^{\pm 1} C y^{\pm 1}) \cdot \mathcal H \neq \mathcal H$.
The $\pi_1(V)$ action on $\widehat {\mathbb H}^{n-k}$ is via isometries, so it is sufficient to show that we can find elements
$x, y$ having the property that the following sets of distances satisfy:
$$\big \{ d(x^{\pm 1} \cdot \mathcal H, \mathcal H)\big \} \cap 
\big \{ d(y^{\pm 1} \cdot \mathcal H, \mathcal H)\big \} = \emptyset .$$

Now pick $x\in \pi_1(V)$ stabilizing some horosphere $\mathcal H^\prime$ (distinct from $\mathcal H$). Then we know that
$x$ does not leave any other horosphere invariant, so $d(x^{\pm 1}\cdot \mathcal H, \mathcal H)>0$.
Moreover, taking large powers of $x$, we can find an $n$ for which the two real numbers 
$d(x^{\pm n} \cdot \mathcal H, \mathcal H)$ are as large as we want. In particular, there exists a sufficiently large 
$n\in\mathbb{N}$ such that, for $y:= x^n$, the distance $d(y^{\pm 1}\cdot \mathcal H, \mathcal H)$ exceeds the distances 
$d(x^{\pm 1}\cdot \mathcal H, \mathcal H)$. 

By the discussion in the previous paragraph, this implies that $\{x,y\}$ form a blocking set for the $\mathbb Z^{n-k-1}$ 
subgroup in $\pi_1(V)$ corresponding to the stabilizer of the horosphere $\mathcal H$. Taking the product with any
element in the $\mathbb Z^k$ factor gives a blocking set for the subgroup $\mathbb Z^{n-1}$ inside $\pi_1(V)\times 
\mathbb Z^k$. This completes the verification of SQ-universality in 
case (2).

\medskip

(3): Suppose now that $M$ contains at least one internal wall with transverse
fibers. Then Proposition~\ref{acyl-new} implies that $\pi_1(M)$ is the fundamental
group of a non-exceptional acylindrical graph of groups,
so the conclusion follows from Proposition~\ref{SQ-u-graph}.
\end{proof}


\section{Solvable word problem}\label{word:sec}
We now shift our attention to an algorithmic question. Given a finite presentation of a group $G$, the {\it word problem} asks
whether there exists an algorithm for deciding whether or not two words $w_1, w_2$ in the generators represent the same 
element in the group $G$. Building on work of Dehn, who resolved the case where $G$ is a surface group, we know 
that this question is equivalent to the presentation having a recursive Dehn function (see Gersten \cite{Ger}). It is possible to formulate this condition in terms of the coarse geometry of $G$, and this approach would be probably quite convenient
to study the solvability of the word problem for fundamental groups of graph of groups
(see also Remark~\ref{filling:rem} below).
However, in this Section we prefer to develop more geometric arguments, that
may be applied to the 
study of fundamental
groups of Riemannian manifolds.
In fact, in the case where
the group $G$ is the fundamental group of a compact connected Riemannian manifold (possibly with boundary), a consequence of 
the well known Filling Theorem (see e.g. Burillo and Taback \cite{BT}) is that 
the word problem for $G$ is solvable
if and only if the 2-dimensional filling function for the universal cover $\widetilde M$ has a recursive upper bound. As we will require 
this in our arguments, we remind the reader of the definition of the 2-dimensional filling function:
$$Area_{M}(L) := \sup _c \inf _D \big \{ Area(D) \hskip 5pt | \hskip 5pt D: \mathbb D^2 \rightarrow \widetilde M, \hskip 5pt D|_{\partial 
\mathbb D^2} = c, \hskip 5pt L(c) \leq L \big \}\ . $$
In other words, we find a minimal area spanning disk for each curve, and try to maximize this area over all curves of length
$\leq L$. We are now ready to show:

\vskip 10pt

\begin{proposition}\label{solvable}
Let $M$ be a compact manifold, and assume that $M$ contains an embedded finite family of pairwise disjoint $2$-sided smooth 
submanifolds $N_i$, cutting $M$ into a finite collection of connected open submanifolds $M_j$ (denote by $\bar M_j$ their closure). 
Moreover, assume this 
decomposition has the following properties:
\begin{enumerate}[(a)]
\item each inclusion $N_i \hookrightarrow \bar M_j$, and $\bar M_j\hookrightarrow M$ is $\pi_1$-injective,
\item each $\pi_1(N_i)$ is a quasi-isometrically embedded subgroup of $\pi_1(M)$, and
\item each $\pi_1(M_j)$ has solvable word problem.
\end{enumerate}
Then the group $\pi_1(M)$ also has solvable word problem.
\end{proposition}

\vskip 10pt

\begin{proof}
To show that $\pi_1(M)$ has solvable word problem, we need to find a recursive function $F:\mathbb N \rightarrow \mathbb N$
having the property that, if $\gamma : S^1\rightarrow \widetilde M$ is any closed curve of length $\leq n$, one can find a bounding
disk $H: \mathbb D^2\rightarrow \widetilde M$ with area $\leq F(n)$. This will be achieved by giving a construction for finding 
a bounding disk, and verifying that the resulting areas are bounded above by a recursive function.

From hypothesis (a), $\pi_1(M)$ is the fundamental group of a graph of groups $\mathcal G$, with 
vertex groups isomorphic to 
the various $\pi_1(M_j)$, and edge groups isomorphic to the various $\pi_1(N_i)$. Let $T$ denote the associated Bass-Serre tree.
Take closed tubular neighborhoods $\widehat{N}_i \supset N_i$ be  of the various $N_i$, chosen small enough so as to be pairwise disjoint. 
Let $\widehat{M}_j \supset M_j$ be 
the manifold with boundary obtained by taking the union of $\widehat{M}_j$ with all of the various $\widehat{N}_i$ (ranging over all $N_i$ that
occur as boundary components of $\widehat{M}_j$). The inclusion $\bar M_j \subset \widehat{M}_j$ is clearly a $\pi_1$-isomorphism.

Next, let us construct a map from $M$ to the graph $\mathcal G$. This is achieved by mapping each $\widehat{N}_i \cong N_i \times [-1,1]$ 
to the edge labelled by the corresponding $\pi_1(N_i)$, by first collapsing $\widehat{N}_i$ onto the interval factor $[-1,1]$,
and then identifying the interval with the edge. Finally, each connected component of the complement $M \setminus \bigcup \widehat{N}_i$
is entirely contained inside one of the submanifolds $M_j$; we map the component to the vertex $v_j\in \mathcal G$ whose label
is $\pi_1(M_j)$. This map lifts to an equivariant map $\Phi: \widetilde M \rightarrow T$, which we will use to analyze
the behavior of a closed loop $\gamma : S^1\rightarrow \widetilde M$. Note that $\Phi$ is essentially the map 
defining the ``tree of spaces'' structure on $\widetilde M$, see Section \ref{univ:subsec} (particularly the discussion
around Definition \ref{tree-of-spaces:def}).

Our analysis of the loop $\gamma$ will start by associating a {\it type} to each point in $S^1$, i.e.~by defining a map from $S^1$ to the vertex set of $T$. Using the map $\Phi\circ \gamma$, we first assign the type of any point lying in the
pre-image of a vertex $v\in T$ to be that same vertex. We now need to discuss how to extend this map to points in the preimage of 
an open edge $e^\circ \subset T$ (i.e. $e^\circ$ excludes the two endpoints of $e$). Each connected component of the
pre-image of $e^\circ$ is either the whole $S^1$, or an open interval $U = (a,b)$ in the circle, which inherits an orientation
from the ambient $S^1$. 
In the first case, we choose an endpoint $v$ of $e$, and we simply
establish that every point of $S^1$ has type equal to $v$.
Otherwise, the two endpoints of the interval $U=(a,b)$ either (i) map to the same vertex $v$ in $T$, 
or (ii) map to distinct vertices $v,w$ in $T$. In case (i), we define the type of that interval to be the vertex $v$. In case (ii),
taking into account the orientation on the interval, we can talk of an ``initial vertex" $\Phi\big(\gamma(a)\big)=v$, and a 
``terminal vertex" $\Phi\big(\gamma(b)\big)=w$. The 
restriction of $\gamma$ to $U = (a,b)$ maps into a subset $\widehat{N}_i$. Let $t\in (a, b)$ be the largest $t$ so that
$\gamma(t) \in N_i$. Then we define the type of the points in $(a, t]$ to be $v$, and the type of the points in $(t, b)$ to be $w$.
By construction, we have that the type function $\rho: S^1\rightarrow Vert(T)$ takes on values contained in the image of
$\Phi\circ \gamma(S^1)$, and hence only assumes {\it finitely many} values (as the latter set is compact).

Let us now fix a vertex $v$ of $T$.
Having defined the type function $\rho: S^1\rightarrow Vert(T)$ associated to the closed loop $\gamma$, we now have that 
either $\rho^{-1} (v)$ is equal to
the whole $S^1$, or the preimage $\rho^{-1}(v)$ satisfies the following properties: 
\begin{enumerate}
\item each connected component of $\rho^{-1}(v)$ is a half-open interval $(a_k, b_k] \subset S^1$, and there are finitely
many such components,
\item there exists a fixed connected lift $\widetilde {\widehat{M}_j}$ of some $\widehat{M}_j$ with the property that the restriction of 
$\gamma$ to each connected component $(a_k,b_k]$ has image $\alpha_k$ contained entirely inside $\widetilde {\widehat {M}_j}$,
\item the point $\gamma(a_k)$ lies on the lift $W_k$ of some $N_j$, and the point $\gamma(b_k)$ lies on the lift $W^\prime_{k}$ of some (possibly distinct) $N_{j^\prime}$, and
\item if one considers the cyclically ordered collection of intervals $(a_k,b_k]$ along the circle $S^1$, then we have that 
$W^\prime _k = W_{k+1}$.
\end{enumerate}
Except for the fact that there are finitely many components in $\rho^{-1}(v)$ (which will be justified later), the four properties 
stated above follow immediately from the definition of the type function $\rho$. 
Let us concentrate on the case when $\rho^{-1}(v)$ is not the whole $S^1$,
the case when $\rho$ is constant being much easier.
We proceed to construct a bounding disk for $\gamma$, where $\gamma$ has length $\leq L \in \mathbb N$, 
and to estimate the resulting area. This will be achieved by 
first expressing $\gamma$ as a concatenation of loops $\gamma_v$, where $v$ ranges over all the (finitely many) types 
associated to the loop $\gamma$. The bounding disk for $\gamma$ will be obtained by concatenating the bounding disks
for the $\gamma_v$.

So let $v\in Vert(T)$ lie in the range of the type function, and consider the connected lift $\widetilde {\widehat {M}_j}$ given
by property (2). Each $W_k$ appearing in property (3) is a connected lift of one of the $N_i$. From hypothesis (a), 
$W_k$ is a copy of the universal cover of $N_i$,
and from hypothesis (b), the inclusion $W_k \hookrightarrow \widetilde M$ is a quasi-isometric embedding. As there are
only finitely many such $N_i$ in $M$, we can choose constants $C,K \in \mathbb N$ so that all the inclusions 
$W_k \hookrightarrow \widetilde M$ are $(C,K)$-quasi-isometries. The two 
points $\gamma(b_{k-1}) \in W_{k-1}^\prime$ and  $\gamma(a_k)\in W_k$ are contained in the same $W_k$ by property (4);
let $\beta_k$ be a minimal length curve in $W_k$ joining them together. The distance between these two points is clearly
$\leq L$ in $\widetilde M$, so as measured inside the submanifold 
$W_k$, their distance is $\leq CL+K$. Define the loop $\gamma_v$ by cyclically concatenating 
$\alpha_1* \beta_1* \alpha_2 *\beta_2 * \cdots *\alpha_r * \beta_r$.
Since each of the $\beta_i$ has length $\leq CL+K$, while the union of the $\alpha_i$ has length $\leq L$ (being a subpath
of the loop $\gamma$), we can estimate the total length of $\gamma_v$ to be $\leq r\cdot (CL+K) + L \in \mathbb N$.

So to complete our estimate on the length of $\gamma_v$, we need to estimate the integer $r$ (this will also justify the 
``finitely many''
in property (1) above). For any of the intervals $U=(a_k, b_k]\subset S^1$ in $\rho^{-1}(v)$, the type of the point $a_k$ is a 
vertex $w$ which is adjacent to $v$. Correspondingly, there is another subinterval $V\subset S^1$, consisting of points
of type $w$, which satisfies $V\cap \bar U= \{a_k\}$. Moreover, there exists a small neighborhood $[a_k-\epsilon, a_k+\delta]
\subset V\cup U$ whose image under $\gamma$ lies entirely in a connected lift $\widetilde {\widehat{N}_i}$ of some $\widehat{N}_i$,
and whose endpoints map to opposite boundary components of $\widetilde {\widehat{N}_i}\cong \widetilde N_i \times [-1,1]$.
For each of the $\widehat{N}_i \subset M$, we let $\lambda _i >0$ denote the minimal distance between the two
boundary components of $\widehat{N}_i \cong N_i \times [-1,1]$. Since there are only finitely many such $\widehat{N}_i$, we can
find a $\lambda \in \mathbb N$ so that $1/\lambda \leq \min \{\lambda_i\}$. We have seen above that to each connected 
component inside each of the sets $\rho^{-1}(v)$ (where $v\in Vert(T)$), we can associate a subpath of $\gamma$ 
contained inside a connected lift of one of the $\widehat{N}_i$, which moreover connects opposite boundary components of the lift. 
These paths are pairwise disjoint, and from the discussion above, have length $\geq 1/\lambda$. We conclude that the total
number of such paths is bounded above by $\lambda \cdot L\in \mathbb N$. In particular, this gives us the upper bound
$\lambda \cdot L$ for:
\begin{itemize}
\item the number $r$ of connected components in $\rho^{-1}(v)$, for any $v\in Vert(T)$, and
\item the total number of vertices $v\in Vert(T)$ for which $\rho^{-1}(v)$ is non-empty.
\end{itemize}
Combining this with our estimate above, we see that the total length of $\gamma_v$ is bounded above by 
the natural number $\lambda CL^2+\lambda KL + L$.

From hypothesis (a), the space $\widetilde {\widehat {M}_j}$ can be identified with the universal cover of 
${\widehat {M}_j}$. From hypothesis (c), $\pi_1(M_j)$ has solvable word problem, and hence the $2$-dimensional
filling function $Area_{\widehat M_j}$ on $\widetilde {\widehat {M}_j}$ has a recursive upper bound 
$F_j:\mathbb N \rightarrow \mathbb{N}$. Observe
that there are only finitely many $\widehat M_j$ inside the manifold $M$, hence we can choose a single 
recursive $F: \mathbb N\rightarrow \mathbb{N}$ which serves as a common upper bound for {\it all} the $2$-dimensional
filling functions for the $\widetilde {\widehat {M}_j}$ (for instance, take $F= \sum F_j$). Then we can find a bounding
disk for $\gamma_v$ whose area is $\leq F(\lambda CL^2+\lambda KL + L)$. Finding such a bounding disk for
each of the vertices $v$ in the range of the type map $\rho$, we obtain a bounding disk for the original curve $\gamma$.
As we know that there are $\leq \lambda \cdot L$ vertices in the range of $\rho$, we conclude that the original
curve $\gamma$ has a bounding disk of total area
$$\leq \lambda \cdot L \cdot F\big(\lambda CL^2+\lambda KL + L\big)$$

Finally, we recall that the class of recursive functions is closed under composition as well as elementary arithmetic 
operations, and hence the function $$G(L):= \lambda \cdot L \cdot F\big(\lambda CL^2+\lambda KL + L\big)$$ provides
the desired recursive upper bound for the function $Area_{M}$. From the Filling Theorem \cite{BT}, we conclude that
$\pi_1(M)$ has a recursive Dehn function, and hence that the word problem is solvable for $\pi_1(M)$.

\end{proof}

\vskip 10pt

Note that the obvious decomposition of a graph manifold into pieces satisfies property (a) in 
the statement of the previous Proposition. Moreover, since all the pieces support a locally CAT(0) metric, their fundamental 
groups have solvable word problem (see for instance Bridson and Haefliger \cite[Section 3.$\Gamma$, Theorem 1.4]{bri}), so 
property (c) always holds. Finally, 
the main result of Chapter~\ref{strongirr:sec}
guarantees that, if the graph manifold is assumed to be irreducible, then properties (b) also holds
(see Theorem~\ref{quasi-isom:thm}). 
This gives us the immediate:

\vskip 10pt

\begin{corollary}[Irreducible $\Rightarrow$ solvable word problem]\label{solvable:cor}
For $M$ an irreducible graph manifold, the fundamental group $\pi_1(M)$ has solvable word problem.
\end{corollary}

\begin{remark}\label{filling:rem}
(1) The above proposition doesn't seem to appear in the literature, though it is no doubt well-known
to experts. Indeed, estimates for the Dehn function of a
free product with amalgam (or HNN-extension) in terms of the Dehn functions of the vertex groups along with estimates
of the {\it relative distortion} of the edge group inside the vertex groups first seems to have been studied in the (unpublished)
thesis of A. Bernasconi \cite{Be}. See also the stronger estimates recently obtained by Arzhantseva and 
Osin \cite{AO}.

\noindent (2) The argument given in the proposition shows that, assuming all vertex groups have solvable word problem,
the complexity of the word problem for the fundamental group of a graph of groups is closely related to the distortion
of the edge/vertex groups in the ambient group (see also the discussion in Farb \cite{farb1}). In fact, one can weaken hypothesis
(c) in the statement of the proposition by instead requiring the distortion of each $\pi_1(N_i)$ inside $\pi_1(M)$ to be bounded
above by a recursive function (generalizing the linear bound one has in the special case of a QI-embedding). The same 
argument works to show that $\pi_1(M)$ still has solvable word problem. 
\end{remark}

\section{Gluings and isomorphism type}\label{groups:subsec}
In this final section, we consider the question of when the fundamental groups of a pair of 
graph manifolds are isomorphic. Let us first recall that, by Theorem~\ref{preserve2:thm},
a pair $M_1,M_2$ of graph manifolds can have isomorphic fundamental groups only if there
is a bijection between the pieces of $M_1$ and the pieces of $M_2$, having the property
that the bijection respects the fundamental groups of the pieces.  This implies that the
only possible freedom occurs in the {\it gluing maps}, telling us how the various pieces are
glued to each other.

For the sake of simplicity, we will only treat the case when the pieces involved 
are constructed starting from cusped hyperbolic manifolds of a fixed dimension $n\geq 3$
and toric fibers of a fixed dimension $k\leq n-2$.
Let us fix a finite \emph{directed} graph $\calG$, that is
a finite connected CW-complex of dimension one with an orientation attached
to every edge, and let $\calV$, $\calE$ be the sets of vertices and edges of $\calG$. 
As usual, the valency of a vertex $v$ of $\calG$ is the total number of germs of edges
starting or ending at $v$. For each $v\in \calV$ with valency $h$
let $N_v$ be a (truncated) cusped hyperbolic $n$-manifold with at least $h$ cusps, 
and set $V_v=N_v\times T^k$. We define $G_v=\pi_1 (V_v)=\pi_1 (N_v)\times \mZ^k$,
and we associate to every germ of edge starting or ending at $v$ a subgroup
$H_{e,v}$ of $G_v$, in such a way that the following conditions hold:
\begin{itemize}
 \item 
each $H_{e,v}$ is (a fixed representative in the conjugacy class of)
the fundamental group of a 
boundary component
of $V_v$; 
\item
$H_{e,v}$ is not conjugated to $H_{e',v}$ whenever
$e\neq e'$, i.e. subgroups corresponding to different edges
with an endpoint in $v$ 
are associated to different boundary components of $V_v$.
\end{itemize}
As a consequence, every $H_{e,v}$ is isomorphic to $\mZ^{n+k-1}$. 
The graph $\calG$ and the groups $G_v$, $H_{e,v}$ determine what we call
a \emph{pregraph} of groups. 

For every $e\in E$ let now $v_-(e),v_+(e)\in\calV$ be respectively
the starting point and the ending point of $e$. A \emph{gluing pattern}
for $\calG$ is a collection 
of group isomorphisms $\Phi=\{\varphi_e\colon H_{e,v_-(e)}\to H_{e,v_+(e)},\ e\in \calE\}$.
We say that $\Phi$ is irreducible if for every $e\in\calE$ the fiber subgroup
of $H_{e,v_+(e)}$ intersects trivially the image of the fiber subgroup of $H_{e,v_-(e)}$ via
$\varphi_e$.
Of course, every gluing pattern for $\calG$ defines a 
graph of groups $(\calG,\Phi)$, which has in turn a well-defined fundamental group 
$\pi_1 (\calG,\Phi)$, according
to the Bass-Serre theory. We say that $(\calG,\Phi)$ is 
supported by $\calG$, and is irreducible if $\Phi$ is.

Let $\calM(\calG)$ be the set of diffeomorphism classes of
graph manifolds obtained by gluing the pieces $V_{v}$, $v\in\calV$ 
according to the pairing of the boundary components encoded by the edges
of $\calG$.
It follows by Theorem~\ref{smrigidity:thm} that the isomorphism classes
of
fundamental groups of (irreducible) graph of groups supported by $\calG$
coincide with the isomorphism classes of fundamental groups of 
(irreducible) manifolds
in $\calM(\calG)$. 

\begin{remark}
The assumption $k\leq n-2$ on the dimensions of toric and hyperbolic factors
of the pieces
will play a crucial
role in the proof of Theorem~\ref{infinitelymany} below. Note however that 
there could not exist irreducible gluing patterns for $\calG$ if
the dimension of
the toric factors of the pieces exceeded the dimension of the hyperbolic factors.
Moreover,
it seems reasonable (and the proof of Theorem~\ref{infinitelymany}
strongly suggests) that an analogue of Theorem~\ref{infinitelymany}
could also hold when different pieces have toric factors of variable
dimensions, provided that such dimensions are sufficiently small.
\end{remark}

The main result of this section is the following:

\begin{theorem}\label{infinitelymany}
Suppose that $\calG$ has at least two vertices. Then,
there exist infinitely
many irreducible graphs of groups supported by $\calG$
with mutually non-isomorphic fundamental groups. Equivalently, there
exist infinitely many diffeomorphism classes of
irreducible manifolds in $\calM(\calG)$.
\end{theorem}
\begin{proof}
An automorphism of a pregraph of groups is a combinatorial automorphism $\varphi$ of $\calG$
(as an \emph{undirected} graph) such that $G_{\varphi (v)}$ is isomorphic
to $G_v$ for every $v\in\calV$ (as discussed at the beginning of
the proof of Lemma~\ref{diffeo:pieces}, this is equivalent to requiring that $V_{\varphi (v)}$ is diffeomorphic
to $V_v$ for every $v\in\calV$).
We say that a pregraph of groups is \emph{without symmetries}
if it does not admit non-trivial automorphisms. 
We first consider the case when $\calG$ is without symmetries.

Since $\calG$ has at least two vertices, there exists an edge $e\in\calE$
with distinct endpoints $v_1=v_- (e)$, $v_2=v_+ (e)$. We fix this edge for
use in the rest of the proof.

Let  $\Phi$, $\Phi'$ be irreducible gluing patterns for $\calG$. Consider
$\varphi\colon H_{e,v_1}\to H_{e,v_2}$ (resp.~$\varphi'\colon H_{e,v_1}\to H_{e,v_2}$)
the isomorphism of $\Phi$ (resp.~of $\Phi'$) associated to the edge $e$.
We say that $\Phi'$ is \emph{equivalent} to $\Phi$ if there exist an automorphism $\psi_1$ of $G_{v_1}$
and an automorphism $\psi_2$ of $G_{v_2}$ 
such that $\psi_1 (H_{e,v_1})=H_{e,v_1}$, 
$\psi_2 (H_{e,v_2})=H_{e,v_2}$ and 
$\varphi'\circ \psi_1 |_{H_{e,v_1}}=\psi_2|_{H_{e,v_2}}\circ \varphi$.
Note that this notion of equivalence is only sensitive to the behavior of
the gluing along the single edge $e$, and completely ignores what happens
along the remaining edges in $\calG$.

Now, the proof of Theorem \ref{infinitelymany} (in the case of pregraphs of groups without symmetries)
will follow immediately from the 
following two facts:

\vskip 10pt

\noindent {\bf Fact 1:} If $\pi_1 (\calG,\Phi)\cong \pi_1 (\calG,\Phi')$, then $\Phi$ is equivalent to $\Phi'$.

\vskip 5pt

\noindent {\bf Fact 2:} There exist infinitely many pairwise non-equivalent irreducible gluing patterns for $\calG$.

\vskip 10pt

Let us begin by establishing {\bf Fact 1}. Let $\psi\colon \pi_1 (\calG,\Phi)\to\pi_1 (\calG,\Phi')$ be a group isomorphism.
By Theorem~\ref{preserve2:thm}, the isomorphism $\psi$ induces an automorphism of $\calG$. But by hypothesis,
we are in the case where $\calG$ has no symmetries, so the automorphism of $\calG$ must be the identity. 
In particular, we have 
$\psi (G_1)=g_1 G'_1 g_1^{-1}$, $\psi (G_2)=g_2 G'_2 g_2^{-1}$,
where $G_i$ (resp.~$G'_i$) is the image of $G_{v_i}$ in 
$\pi_1 (\calG,\Phi)$ (resp.~in $\pi_1 (\calG,\Phi')$), and $g_1,g_2$ are elements in $\pi_1 (\calG,\Phi')$. 
If $H$ (resp.~$H'$) is the image in  $\pi_1 (\calG,\Phi)$
(resp.~in  $\pi_1 (\calG,\Phi')$) of
$H_{e,v_1}$ and $H_{e,v_2}$ (which are identified by the very definition
of fundamental group of a graph of groups), since $\psi$ induces
the identity of $\calG$ we also have $\psi (H)=g_3 H' g_3^{-1}$
for some $g_3\in \pi_1 (\calG,\Phi')$.

Up to conjugating $\psi$, we can assume $g_1=1$, so that $\psi(G_1)=G_1^\prime$. 
Next note that we have
$g_3 H' g_3^{-1} =\psi (H)\subseteq \psi (G_1)=G'_1$, 
so $H'\subseteq g_3^{-1}G'_1g_3\cap G'_1$. By Lemma~\ref{conj:lemma}-(5), this implies  
that either $g_3\in G'_1$, or $H'$ corresponds to
an edge of $\calG$ having both endpoints on the vertex representing
$G'_1$. But recall that the edge $e$ was chosen to have distinct endpoints, 
ruling out this last possibility. So at the cost of conjugating $\psi$ with 
$g_3^{-1}$, we may further assume that $g_3=1$, and both $\psi (G_1)=G'_1$ and
$\psi (H)=H'$. As a consequence we have
$H'=\psi(H)\subseteq \psi (G_2)=g_2 G'_2g_2^{-1}$, so
$H'\subseteq g_2G_2'g_2^{-1}\cap G_2$, whence $g_2\in G'_2$ as above
and $\psi (G_2)=G_2'$. 

We have thus proved that $\psi$ induces isomorphisms $G_1\cong G_1'$, $G_2\cong G_2'$  
which ``agree'' on $H=G_1\cap G_2$. More precisely, for $i=1,2$ there exists an isomorphism  
$\psi_i\colon G_{v_i}\to G_{v_i}$ such that the following conditions hold: $\psi_i (H_{e,{v_i}})=H_{e,v_i}$
 for $i=1,2$, and $\varphi'\circ \psi_1 |_{H_{e,v_1}}=\psi_2|_{H_{e,v_2}}\circ \varphi$. By definition, 
 this means that $\Phi$ is equivalent to $\Phi'$, and {\bf Fact 1} is proved.

\vskip 10pt

Let us now prove {\bf Fact 2}. Recall that for $i=1,2$ we have an identification $G_{v_i} \cong \Gamma_i\times \mZ^k$,
where $\Gamma_i=\pi_1 (N_{v_i})$.
We also denote
by $L_i$ the subgroup of $\Gamma_i$ such that 
$L_i\times \mZ^k < \Gamma_i\times \mZ^k$ corresponds to $H_{e,v_i}$ 
under the above identification.
As showed in the proof of Lemma~\ref{diffeo:pieces}, every automorphism of $G_{v_i}=\Gamma_i\times \mZ^k$
is of the form $(g,v)\mapsto (\theta_i (g),\alpha_i (g)+\beta_i (v))$,
where $\theta_i\colon \Gamma_i\to\Gamma_i$ and $\beta_i\colon\mZ^k\to\mZ^k$ are isomorphisms,
and $\alpha_i\colon \Gamma_i\to\mZ^k$ is a homomorphism. We now claim that, in a sense to be made precise 
below, if we restrict to automorphisms leaving $L_i$ invariant, then there exist at most a finite number of possibilities 
for the isomorphism $\theta_i$.

Let $\widetilde{\Theta}_i$ be the group of automorphisms of $\Gamma_i$ leaving $L_i$ invariant, and let ${\Theta}_i$ be the group of automorphisms of $L_i$
given by restrictions of elements of $\widetilde{\Theta}_i$. 
For $g\in\Gamma_i$, we denote by $c_g\in {\rm Aut}(\Gamma_i)$ the conjugation by $g$. 
If $\theta,\hat{\theta}\in\widetilde{\Theta}_i$
are such that $\theta=c_g\circ \hat{\theta}$ for some $g\in\Gamma_i$, then $gL_ig^{-1}=L_i$,
whence $g\in L_i$ (see the proof of Lemma~\ref{conj:lemma}-(1)). Since $L_i$ is abelian, this implies that $\theta$ and $\hat\theta$
restrict to the same element of $\Theta_i$. As a consequence, $\Theta_i$
has at most the cardinality of the group of outer automorphisms of $\Gamma_i$,
which is finite by Mostow rigidity (together with the well-known fact
that the group of isometries of a complete finite-volume hyperbolic manifolds is finite). We have thus proved the fact claimed above that
$\Theta_i$ is finite.

For $i=1,2$, 
let us now fix a free basis of $L_i\times \mZ^k\cong\mZ^{n+k-1}$
whose first $n-1$ elements give a basis of $L_i$ and whose last $k$ elements
give a basis of $\mZ^k$. Under the induced identification of $L_i$ with $\mZ^{n-1}$, the group
$\Theta_i$ is identified with a finite subgroup of ${\rm SL} (n-1,\mZ)$ , which will still be denoted by $\Theta_i$. Moreover,
we may identify the group of automorphisms of $H_{e,v_i}\cong L_i\times\mZ^k$ with 
the group of matrices ${\rm SL}(n+k-1,\mZ)$. 
The discussion above shows that under these identifications every automorphism of $L_i\times\mZ^k$ which extends
to an automorphism of $G_{v_i}$ has the form
$$
\left(\begin{array}{ccc}
\theta_i &\vline & 0\\
\hline 
v_i & \vline & w_i \end{array}\right)\in{\rm SL}(n+k-1,\mZ),\quad 
\theta_i\in \Theta_i < {\rm SL} (n-1,\mZ),
$$
and  any isomorphism between $\varphi\colon H_{e,v_1}\to H_{e,v_2}$ may be represented by a matrix
$$
\left(\begin{array}{ccc}
A &\vline & B\\
\hline 
C & \vline & D \end{array}\right)
\in {\rm SL}(n+k-1,\mZ),$$ 
where $A,D$ have order $(n-1)\times (n-1)$ and $k\times k$ respectively.
Moreover, it is readily seen that $\varphi$ can be extended to an irreducible 
gluing pattern if and only if
$\rk (B)=k$. 

Now, since $k<n-1$ and $\Theta_2$ is finite, it is possible to construct an infinite sequence
$\{B_j\}_{j\in\mN}$ of matrices of order $(n-1)\times k$ such that the following conditions hold:
\begin{itemize}
\item
$\rk B_j= k$ for every $j\in\mN$;
\item
if $\Lambda_j$ is the subgroup of $\mZ^{n-1}$ generated by the columns of $B_j$, $j\in\mN$, 
and
$\Lambda_j=\theta (\Lambda_h)$ for some $\theta\in\Theta_2$, then necessarily 
$j=h$.
\end{itemize}

Let $\varphi_j\colon H_{e,v_1}\to H_{e,v_2}$, $j\in\mN$, be the isomorphism represented by the matrix
$$
P_j=\left(\begin{array}{ccc}
{\rm Id}_{n-1} &\vline & B_j\\
\hline 
0 & \vline & {\rm Id}_{k} \end{array}\right),
$$
 and extend $\varphi_j$ to an irreducible gluing pattern $\Phi_j$. We now claim that
 $\Phi_j$ is not equivalent to $\Phi_h$ if $j\neq h$, 
  thus concluding the proof of~(2). In fact, if $\Phi_j$ is equivalent
 to $\Phi_h$, then there exist matrices
 $$
 N_1=\left(\begin{array}{ccc}
\theta_1 &\vline & 0\\
\hline 
v_1 & \vline & w_1 \end{array}\right), \quad
N_2=
\left(\begin{array}{ccc}
\theta_2 &\vline & 0\\
\hline 
v_2 & \vline & w_2 \end{array}\right)
$$
such that $\theta_i\in\Theta_i$, $w_i\in {\rm SL}(k,\mZ)$ for $i=1,2$, and 
$P_j N_1=N_2 P_h$. It is readily seen that this condition implies the equality
$B_j w_1=\theta_2 B_h$. Since $w_1\in{\rm SL}(k,\mZ)$, this implies in turn
$\Lambda_j=\theta_2 (\Lambda_h)$, whence $j=h$ by the properties of the $B_j$'s listed above.
We have thus proved the theorem under the assumption that $\calG$ is without symmetries.

In the general case, the arguments just described ensure that an infinite family $\{\Phi_i\}_{i\in\mN}$
of irreducible gluing patterns exists such that, if $i\neq j$, then $\pi_1 (\calG,\Phi_i)$ is not isomorphic to
$\pi_1 (\calG,\Phi_j)$ via an isomorphism inducing the identity of $\calG$.  Suppose now by contradiction that
the groups $\pi_1 (\calG,\Phi_i)$ fall into finitely many isomorphism classes. Then, up to passing to an infinite subfamily,
we may suppose that for every $i,j\in\mN$ there exists an isomorphism $\psi_{ij}\colon \pi_1 (\calG,\Phi_i)\to
\pi_1 (\calG,\Phi_j)$ inducing the automorphism $\delta_{ij}$ of $\calG$. Since the group of automorphisms of $\calG$ is finite,
there exist $h,k\in\mN\setminus\{0\}$ such that $h\neq k$ and $\delta_{0h}=\delta_{0k}$. Therefore,
the map $\psi_{0k}\circ\psi_{0h}^{-1}$ establishes an isomorphism between $\pi_1 (\calG,\Phi_h)$ and $\pi_1 (\calG,\Phi_k)$
inducing the identity of $\calG$, a contradiction.
\end{proof}

\begin{remark}
The assumption that $\calG$ has at least two vertices is not really necessary. In 
the case that $\calG$ has only one vertex, we could provide a different proof of
Theorem~\ref{infinitelymany} just by replacing our analysis of isomorphisms between
amalgamated products with an analogous analysis of isomorphisms between
HNN-extensions.
\end{remark}

\begin{remark}\label{infinitelymany:rem}
The strategy described in the proof of Theorem~\ref{infinitelymany} can also be applied to the examples discussed
in Remark~\ref{infinite:rem}, where an infinite
family $\{M_i\}_{i\geq 1}$ of irreducible manifolds not supporting any CAT(0) metric
is constructed by gluing two fixed $4$-dimensional pieces $V_1,V_2$
along their unique boundary component. 
With notation as in Corollary~\ref{examples:cor} and Remark~\ref{infinite:rem}, we now show that 
if $V_1$ is not diffeomorphic to $V_2$, then
$M_i$ is not diffeomorphic
to $M_j$ for every $i,j\in\mN$, $i\neq j$. 

Let us choose bases for the fundamental groups of the boundary components of
$V_1,V_2$ (such components
are 3-dimensional tori) in such a way that 
the first vector is null-homologous in $V_i$, $i=1,2$, and the last one belongs to the fiber subgroup (which is isomorphic to $\mZ$).
Then the gluing map defining $M_n$ is encoded by the matrix
$$
 A_n=\left(\begin{array}{ccc}
 1 & \ast & 1\\
 0 & \ast & 0\\
 0 & \ast & n\end{array}\right).
$$
Moreover, every homomorphism of the fundamental group of a piece into the fiber subgroup (which is abelian) vanishes
on null-homologous elements, whence on horizontal slopes.
So any automorphism of the fundamental group of each of the two pieces, when 
restricted to the boundary, gives an automorphism of the form
$$
 \left(\begin{array}{ccc}
 \ast & \ast & 0\\
 \ast & \ast & 0\\
 0 & \ast & \pm 1 \end{array}\right)
$$
 (see the proof of Theorem~\ref{infinitelymany}).
 It is now readily seen that if $N_1,N_2$ are matrices of this form, then
 we have
$$
 N_1 A_n = \left(\begin{array}{ccc}
 \ast & \ast & \ast\\
 \ast & \ast & \ast\\
 \ast & \ast & \pm n \end{array}\right)
\neq
 \left(\begin{array}{ccc}
 \ast & \ast & \ast\\
 \ast & \ast & \ast\\
 \ast & \ast & \pm m \end{array}\right)
 =A_m N_2.
$$
Now, since $V_1$ is not diffeomorphic to $V_2$, the $M_i$'s are associated to a graph
without symmetries. As explained in the proof Theorem~\ref{infinitelymany},
this is now sufficient to conclude that the $M_i$'s are pairwise non-diffeomorphic.
 
Also, observe that by the proof of Theorem~\ref{infinitelymany}, if $V_1$ is diffeomorphic
to $V_2$
we can still conclude that among
the $M_i$'s there exist infinitely many pairwise non-diffeomorphic manifolds.
\end{remark}

\part{Irreducible graph manifolds: coarse geometric properties}

%
%
%

\chapter{Irreducible graph manifolds}\label{strongirr:sec}

In Section~\ref{noncat0-easy:subsec} we proved that 
there exist examples of graph manifolds $M$
with the property that certain walls of $\tilM$ are \emph{not} quasi-isometrically embedded in $M$.
In order to study in detail the quasi-isometric properties
of the fundamental groups of graph manifolds, we would like to find conditions
that prevent this phenomenon to occur. The main result of this Chapter shows that,
if $M$ is irreducible, then walls and chambers of $\tilM$ are quasi-isometrically embedded in
$\tilM$.

By Milnor-Svarc Lemma and Proposition~\ref{irr-acyl},
this fact may be restated as follows. Let us fix the description of $\pi_1(M)$
as the fundamental group of the graph of groups corresponding
to the decomposition of $M$ into pieces; if this graph of groups is acylindrical,
then edge groups and vertex groups are quasi-isometrically embedded in $\pi_1(M)$.

Acylindricity plays a fundamental role in analogous
results for hyperbolic or relatively hyperbolic groups.
For example, in~\cite{Kap} it is shown that
edge groups and vertex groups of an acylindrical graph of hyperbolic groups are quasi-isometrically embedded
in the fundamental group of the graph of groups, provided that edge groups
are quasi-isometrically embedded
in the ``adjacent'' vertex groups (this is always the case in our case of interest).
A similar result in the context of relatively hyperbolic groups
may be deduced from~\cite[Theorem 0.1--(1)]{Dahmani}.

It would be interesting to find less restrictive
conditions under which the walls of $\tilM$ are ensured to be quasi-isometrically
embedded.
In our situation, the fundamental groups of the pieces
are semihyperbolic in the sense of~\cite{alo}. Since every free abelian subgroup
of a semihyperbolic group is quasi-isometrically embedded, an (apparently difficult) strategy 
could be to find conditions on a graph of semihyperbolic groups 
in order to ensure that the fundamental group of the graph is itself semihyperbolic.

Some further discussion of related issues can be found in Section \ref{qi-open:sec}.

\section{The geometry of chambers and walls}

Let $M$ be a graph $n$-manifold.
The boundary of each internal wall $W$ of $\tilM$ decomposes into the union of two connected components
$W_+$, $W_-$, while if $W$ is a boundary wall, we simply set $W_+=W_-=W$. We call $W_+,W_-$ 
the \emph{thin walls} associated to $W$, and we denote 
by $d_{W_{\pm}}$
the path metric on  $W_\pm $ induced by the restriction
of the Riemannian structure of $\tilM$.
If $W$ is an internal wall, then the canonical product structure on the image of $W$ in $M$
induces a canonical product structure $W=\mR^{n-1}\times [-3,3]$
with $ W_\pm=\mR^{n-1}\times\{\pm 3\}$. If  
$p=(x,3)\in W_+$, $q=(y,-3)\in W_-$, we say that $p,q$ are \emph{tied} to each other 
if and only if $x=y$. 
If $W$ is a boundary wall, we say that $p\in W_+=W$ is tied
to $q\in W_-=W$ if and only if $p=q$.
Finally, for every wall $W$
we denote by $s_W\colon W_+\to W_-$
the map that associates to each $p\in W_+$ the point $s_W(p)\in W_-$
tied to $p$. Note that, by the restriction on our gluing maps, the map
$s_W$ is an affine diffeomorphism.

In order to study the quasi-isometry type of $\tilM$ we first need to understand the geometry of its chambers. 
Recall that if
$C\subseteq\tilM$ is a chamber, then there exists an isometry $\varphi\colon C\to B\times\mR^k$, 
where $B\subseteq \matH^{n-k}$ is a neutered space
(such an isometry is unique up to postcomposition
with the product of isometries of $B$ and $\mR^k$). 
Also recall that $B$ is the \emph{base} of $C$, and $F=\mR^k$ the \emph{fiber} of $C$.
If $\pi_B\colon C\to B$, $\pi_F\colon C\to F$ are the natural projections, 
for every $x,y\in C$, we denote by $d_B (x,y)$ the distance (with respect
to the path metric of $B$) between $\pi_B(x)$ and $\pi_B (y)$, and by
$d_F(x,y)$ the distance between $\pi_F(x)$ and $\pi_F(y)$
(so by construction $d_C^2=d_B^2+d_F^2$).

\begin{definition}
We recall that a metric space $X$ is geodesic if for every $x,y\in X$
there exists a rectifiable curve $\gamma\colon [0,1]\to X$ joining $x$ to $y$ whose length
is equal to $d(x,y)$ (the constant speed parameterization of such a curve is called \emph{geodesic}).
Suppose $S$ is a submanifold of the (possibly bounded) simply connected Riemannian manifold $X$, 
and let $d$ be the Riemannian metric 
of $X$. We say that $S$ is \emph{totally geodesic} in $(X,d)$ (in the metric sense) 
if for every $p,q\in S$ there exists a geodesic of $X$ which joins $p$ to $q$ and whose support
is contained in $S$. In this case, the path metric associated to the restriction
of $d$ to $S$ coincides with the restriction of $d$ to $S$.
\end{definition}
 
Let $B$ be a neutered space, endowed with its path metric.
Then, it is well-known (see e.g.~\cite[pgs. 362-366]{bri})
that every component of $\partial B$
is totally geodesic in 
$B$, even if its extrinsic curvature in $B$
does not vanish. 

\begin{lemma}\label{easybil:lem}
For $W$ an arbitrary wall, we have:
\begin{enumerate}
\item
if $C$ is the chamber containing $W_\pm$, then the inclusion
$(W_\pm, d_{W_\pm})\hookrightarrow (C,d_C)$ is isometric;
\item
the inclusion $(W_\pm,d_{W_{\pm}}) \to (W,d_W)$ is a bi-Lipschitz embedding
and a quasi-isometry;
\item 
the map $s_W\colon (W_+,d_{W_+})\to
(W_-,d_{W_-})$ is a bi-Lipschitz homeomorphism.
\end{enumerate}
(Of course, points~(2) and (3) are trivial if $W$ is a boundary wall).
\end{lemma}

\begin{proof}
We have just recalled that the boundary components of a neutered space are totally geodesic
(in the metric sense). Therefore,
if $W_\pm$
is a thin wall contained in the chamber $C$, we have that 
$W_{\pm}$ is a totally geodesic (in the metric sense) hypersurface of $C$. In particular, the path metric
induced on $W_{\pm}$ by the Riemannian structure on $\tilM$
is isometric to the restriction of $d_C$, whence~(1).

Concerning~(2), first observe that, by definition of induced path metric, the inclusion
$i\colon W_\pm\hookrightarrow W$ is $1$-Lipschitz.
The map $i$ is the lift of an embedding which induces an isomorphism
on fundamental groups, so by the Milnor-Svarc Lemma, $i$ is a quasi-isometry. 
This guarantees that $i$ is bi-Lipschitz at large scales,
i.e.~that there exist constants $C'\geq 1$, $R> 0$ such that
$$
d_{W_\pm} (x,y)\leq C' d_W (x,y) \quad {\rm whenever}\ d_{W_\pm } (x,y)\geq R.
$$
We need to control distances within the range $0\leq d_{W_{\pm}} (x,y)\leq R$. Observe that this inequality describes a region
$K\subseteq W_\pm\times W_\pm$ which is invariant under the obvious diagonal $\mZ^{n-1}$-action. 
Moreover, the quotient space $K/\mZ^{n-1}$ is easily seen to be compact.
If $K'=K\setminus \{(x,x),\, x\in W_\pm\}$, then the ratio $d_{W_\pm} / d_W$ defines a positive continuous function on $K'$.
It is not difficult to see that such a function extends to a continuous  $f\colon K\to \mR$
such that $f(x,x)=1$ for every $x\in W_{\pm}$.  Moreover, $f$ is obviously $\mZ^{n-1}$-equivariant,
so compactness of $K /\mZ^{n-1}$ implies that $f$ is bounded above by some constant $C''$. This implies
that $i$ is $\max \{C',C''\}$-bi-Lipschitz, giving~(2).
 
Similarly, $s_W$ is obtained by lifting
to the universal coverings a diffeomorphism between compact manifolds, and is therefore bi-Lipschitz. 
\end{proof}

\section{An important consequence of irreducibility}

The following lemma shows how irreducibility is related to the behaviour of the 
metric of $\tilM$ near the internal walls. Informally, 
it shows that points which almost lie on the same fiber of a thin wall are 
tied to points that are forced to lie on distant fibers of the adjacent chamber. 

\begin{lemma}\label{strong:lem}
Suppose $\psi_l\colon T^+_l\to T^-_l$ is transverse. Let $W\subseteq \tilM$ be a
(necessarily internal) wall
projecting to a regular neighbourhood of $T^+_l=T^-_l$ in $M$, and let $C_+,C_-\subseteq \tilM$
be the chambers adjacent to $W$ with bases $B_+,B_-$. Then 
there exists $k\geq 1$ such that the following holds: 
let $x_+,y_+\in W\cap C_+=W_+$ (resp.~$x_-,y_-\in W\cap C_-=W_-$)
be such that $x_+$ is tied to $x_-$ and $y_+$ is tied to $y_-$; then
$$
d_{C_+}(x_+,y_+)\geq k d_{B_+}(x_+,y_+)\quad \Longrightarrow \quad 
d_{C_-}(x_-,y_-)\leq k d_{B_-}(x_-,y_-).$$
\end{lemma}

\begin{proof}
Suppose by contradiction that there exist sequences $\{x^n_+\}$, $\{y^n_+\}$ of points in $W_+$
such that 
\begin{equation}\label{euclnorms}
d_{C_+} (x^n_+,y^n_+) > n d_{B_+}(x^n_+,y^n_+),\quad  
d_{C_-}(x^n_-,y^n_-) > n d_{B_-}(x^n_-,y^n_-).
\end{equation}
Recall that $W_+$ and $W_-$ are endowed with a canonical affine structure, and the map
$s_W\colon W_+\to W_-$ defined before Lemma~\ref{easybil:lem}
is an affine diffeomorphism. Let $Z_+$, $Z_-$ be the vector spaces underlying the affine spaces $W_+$, $W_-$,
and denote by $\widehat{s}_W\colon Z_+\to Z_-$ the linear map associated to $s_W$.

The product decompositions of $C_+=B_+\times F_+$ and $C_-=B_-\times F_-$ induce direct sum decompositions
$$
Z_+ =\widehat{B}_+\oplus \widehat{F}_+,\quad Z_-=\widehat{B}_-\oplus \widehat{F}_-,
$$
and transversality of $\psi_l$ implies that $\widehat{s}_W (\widehat{F}_+)\cap \widehat{F}_-=\{0\}$.

For every $n\in\mathbb{N}$, we denote by $v^n_+ \in \widehat{F}_+$, $w^n_+\in\widehat{B}_+$ 
(resp.~$v^n_-\in \widehat{F}_-$ ,$w^n_-\in\widehat{B}_-$) the vectors uniquely determined
by the conditions $y_+^n-x^n_+=v^n_++w^n_+$, $y_-^n-x^n_-=v^n_-+w^n_-$.

By Lemma~\ref{easybil:lem}-(1), the restrictions of the distances $d_{C_+}$ and $d_{C_-}$
to $W_+$ and $W_-$ are induced by Euclidean norms  $\|\cdot \|_+$, $\|\cdot\|_-$  on $Z_+$, $Z_-$.
The inequalities~\eqref{euclnorms} may now be rewritten in the following way:
\begin{equation}\label{eucl2}
\frac{\| v^n_+ + w^n_+\|_+}{n} >  \| w^n_+\|_-,\qquad 
\frac{\| v^n_- + w^n_-\|_-}{n} >  \|w^n_-\|_-.
\end{equation}
Up to rescaling, we may suppose that $\| v^n_+ + w^n_+\|_+ =1$ for every $n$. Since $s_W$ is bi-Lipschitz,
there exists $\alpha\geq 1$ such that $\alpha^{-1}\leq \| v^n_- + w^n_-\|_- \leq\alpha$ for every $n$.
In particular, up to passing to subsequences, we may suppose that the sequences 
$\{v^n_+\}$, $\{w^n_+\}$, $\{v^n_-\}$, $\{w^n_-\}$ converge to $v_+\in \widehat{F}_+$, $w_+\in \widehat{B}_+$, 
$v_-\in\widehat{F}_-$, $w_-\in\widehat{B}_-$. Moreover, we have $\widehat{s}_W (v_++w_+)=v_-+w_-$.
As $n$ tends to infinity, inequalities~\eqref{eucl2} imply $w_+=0$, $w_-=0$, so
$\widehat{s}_W (v_+)=v_-$. Since $\| v_+\|_+ = \|v_+ +w_+\|_+=1$, we have that $\widehat{s}_W (v_+)=v_-$
is a non-trivial element in $\widehat{s}_W (\widehat{F}_+)\cap \widehat{F}_-=\{0\}$, and this provides the desired contradiction.
\end{proof}

\section{The geometry of neutered spaces}
The following Proposition provides a useful tool
in the study of neutered spaces, whence of chambers.
It is inspired by~\cite[Lemma 3.2]{osin}:


\begin{proposition}\label{osin:prop}
Let $B$ be a neutered space. Then there exists a constant $Q$ only depending
on $B$ such that the following result holds. 
Let $\gamma\subseteq B$ be a loop obtained by concatenating a finite number of paths
$\alpha_1,\gamma_1,\ldots,\alpha_{n},\gamma_{n}$, where 
\begin{itemize}
\item each $\alpha_{i}$ is a geodesic on a horosphere $O_i\subseteq \partial B$,
\item each $\gamma_{i}$ is any path in $B$ connecting the endpoint of $\alpha_{i}$
with the starting point of $\alpha_{i+1}$, and
\item the endpoints of each $\gamma_i$ lie on distinct walls.
\end{itemize} 
Let $D\subseteq \{1,\ldots,n\}$ be a distinguished subset
of indices such that $O_h\neq O_i$ for every $h\in D$, $i\in\{1,\ldots,n\}$, $i\neq h$.
Then
$$\sum_{h\in D} L(\alpha_{h})\leq Q \sum_{i=1}^n L(\gamma_{i}).$$
\end{proposition}
\begin{proof}
Let $B$ be a neutered space, and 
recall that by the very definitions, the group of isometries of $B$
contains a discrete torsion-free cocompact subgroup $\Gamma$. The quotient $N=B/\Gamma$ is obtained 
by removing horospherical neighbourhoods of the cusps from a finite-volume hyperbolic manifold. 
As a consequence, there exists $R>0$ such that the distance between 
every pair of distinct connected components
of $\partial B$ is at least $R$, so that 
\begin{equation}\label{estimaten}
n\leq \frac{\sum_j L(\gamma_{j})}{R} .
\end{equation}

\begin{figure}
\begin{center}
\includegraphics[width=6cm]{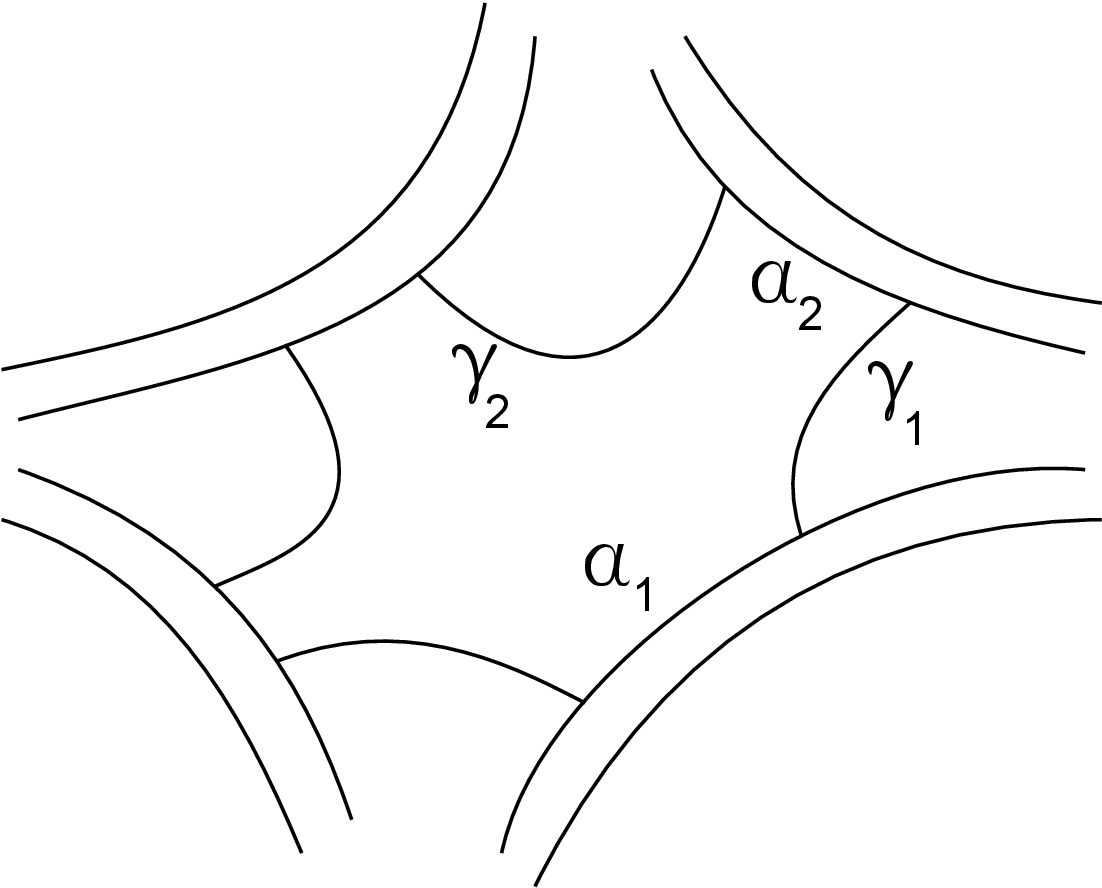}
\caption{Proposition~\ref{osin:prop} provides a bound on the lengths of the 
$\alpha_i$'s in terms of the lengths of the $\gamma_i$'s.}\label{osin:fig}
\end{center}
\end{figure}

Let 
$\{H_1,\ldots,H_l\}$ be the collection of subgroups of $\Gamma$ obtained by choosing a representative in each conjugacy class
of cusp subgroups of $\Gamma$, and
recall that $\Gamma$ is relatively hyperbolic with respect
to the $H_i$'s. Choose
$X$ to be a symmetric set of generators for $\Gamma$ satisfying the assumptions of~\cite[Lemma 3.2]{osin}, 
and let us denote by $\calC_\Gamma$ the corresponding Cayley graph of $\Gamma$ with 
distance $d_\Gamma$.

We denote by $\overline{\calC}_\Gamma$ the Cayley graph of $\Gamma$ with respect
to the (infinite) set of generators $\left( X\cup\left(H_1\cup\ldots\cup H_l\right)\right)\setminus \{1\}$,
and by $\overline{d}_\Gamma$ the path distance on $\overline{\calC}_\Gamma$ (see~\cite{osin}).

More precisely, if $\widetilde{X}$ is a copy of $X$, 
$\widetilde{H}_\lambda$ is a copy of $H_\lambda$ 
and $\calH=\bigsqcup_{\lambda=1}^l \left(\widetilde{H}_\lambda\setminus \{1\}\right)$, 
then $\overline{\calC}_\Gamma$  is the graph having $\Gamma$ as set of vertices and
$\Gamma\times (\widetilde{X}\cup\calH)$ as set of edges, where if $\overline{y}\in\Gamma$
is the element corresponding to $y\in \widetilde{X}\cup\calH$, then the edge 
$(g,y)$ has $g$ and $g\cdot \overline{y}$ as endpoints. We label
the edge $(g,y)$ by the symbol $y$. Note that different labels may represent
the same right multiplication in $\Gamma$: for instance, this is the case if there exist letters
$x\in \widetilde{X}$ and $y\in \widetilde{H}$ representing the same element
$\overline{x}=\overline{y}$ in $\Gamma$, 
i.e.~if $X\cap\left(\bigcup_{\lambda=1}^l H_\lambda\right)\neq\emptyset$. 

Notice that by the very definitions we have a natural inclusion
$\calC_\Gamma\hookrightarrow \overline{\calC}_\Gamma$. 
Let $q$ be a (non-based) loop in $\overline{\calC}_\Gamma$ labelled by the (cyclic) word $w$ with letters in 
$\widetilde{X}\cup\calH$. 
Recall
from~\cite{osin} that a subpath of a loop $q$ in $\overline{\calC}_\Gamma$ is a 
\emph{$H_\lambda$-subpath} if it is labelled by a subword of $w$ with letters in 
$\widetilde{H}_\lambda$. An
\emph{$H_\lambda$-component} of $q$ is a maximal $H_\lambda$-subpath of $q$. 
An $H_\lambda$-component $q'$ of $q$ is \emph{not isolated} if
there exists an $H_\lambda$-component $q''\neq q'$ of $q$ such that 
a vertex in $q'$ and a vertex of $q''$ are joined by an edge labelled by a letter
in $\widetilde{H}_\lambda$
(in algebraic terms this means that such vertices belong to the same left 
coset of $H_\lambda$ in $\Gamma$).

Starting from $\gamma$, we wish to construct a loop $\overline{\gamma}$ in $\overline{\calC}_\Gamma$. Milnor-Svarc's Lemma
provides a $(\mu,\epsilon)$-quasi-isometry $\varphi\colon B\to\calC_\Gamma$.
Up to increasing $\epsilon$,   
we can require that $\varphi$ maps every point of $B$ onto a vertex of $\calC_\Gamma$, i.e.~onto
an element of $\Gamma$, and that every horosphere $O\subseteq \partial B$ is taken by
$\varphi$ onto a lateral class of some $H_\lambda$. It is easy to see
that if $\varphi$ maps the horospheres $O,O'\subseteq \partial B$
onto the same lateral class of the same $H_\lambda$, then $O=O'$. 
Fix $i\in\{1,...,n\}$, suppose that $\gamma_{i}$ is parametrized by arc length,
denote by $m_i$ the least integer number such that 
$L(\gamma_{i})\leq m_i$, and set
$p^j_i=\varphi(\gamma_{i}(j L(\gamma_i)/m_i))\in \Gamma$ for 
$j=0,\ldots,m_i$.
Due to our choices we have 
$p^0_i\in\varphi(O_i)$ and $p^{m_i}_{i}\in\varphi(O_{i+1})$. Now let $\tilde{\gamma}_{i}$ 
be the path in $\calC_\Gamma$ obtained by concatenating the geodesics 
joining $p^j_i$ and $p^{j+1}_i$, $j=0,\ldots, m_{i}-1$,
and let $\overline{\gamma}_{i}$ be the path in $\overline{\calC}_\Gamma$ obtained by
taking the image of $\tilde{\gamma}_{i}$ under the inclusion 
$\calC_\Gamma\hookrightarrow \overline{\calC}_\Gamma$.
Observe that by construction every edge of $\overline{\gamma}_{i}$
is labelled by a symbol in $\widetilde{X}$, so no $\overline{\gamma}_i$ contains any 
$H_\lambda$-subpath. 

As $m_i\leq L(\gamma_i)+1$ by our choice of $m_i$, we have the estimate:
$$
L(\overline{\gamma}_{i})= L(\tilde{\gamma}_{i})= \sum_{j=0}^{m_i-1} d_\Gamma(p^j_i,p^{j+1}_i)
\leq \mu L(\gamma_{i}) +m_i\epsilon\leq (\mu+\epsilon)L(\gamma_{i})+\epsilon . 
$$

Next, observe that $p^{m_{i-1}}_{i-1}$ and $p^0_{i}$ both lie
on $\varphi(O_i)$, and hence belong to the same left coset of
some $H_{\psi (i)}$, $\psi (i)\in \{1,\ldots,l\}$. Thus we can connect $p^{m_{i-1}}_{i-1}$ and $p^0_{i}$
in $\overline{\calC}_\Gamma$ by a path $\overline{\alpha}_i$ which is either constant
(if $p^{m_{i-1}}_{i-1}=p^0_{i}$), or consists of a single edge labelled
by a symbol in $\widetilde{H}_{\psi (i)}$. 
Now define the loop $\overline{\gamma}=\overline{\alpha}_1\ast \overline{\gamma}_1\ast\ldots\ast\overline{\alpha}_n\ast
\overline{\gamma}_n$ in $\overline{\calC}_\Gamma$.
Using~\eqref{estimaten}, we obtain 
\begin{align*}
L(\overline{\gamma}) & \leq \left(\sum_{i=1}^n L(\overline{\gamma}_i)\right) + n\\
& \leq (\mu+\epsilon) \sum_{i=1}^n L(\gamma_{i})+n\epsilon+n \\
& \leq \left(\mu+\epsilon+\frac{\epsilon+ 1}{R}\right) \sum_{i=1}^n L(\gamma_{i}).\\
\end{align*}

Moreover,
due to our assumption on $D$,
for every $h\in D$  
the subpath $\overline{\alpha}_h$ is (either constant or)
an isolated component
of $\overline{\gamma}$, so by~\cite[Lemma 3.2]{osin} 
there exists $Q'$ only depending on (the Cayley graphs $\calC_\Gamma$
and $\overline{\calC}_\Gamma$ of) $\Gamma$ such that 
for 
every $\lambda=1,\ldots,l$
$$
\sum_{h\in D\cap \psi^{-1} (\lambda)} 
d_\Gamma(p^{m_{h-1}}_{h-1},p^0_{m_{h}})\leq Q' L(\overline{\gamma}), 
$$
whence
\begin{equation}\label{estimate:eq}
\sum_{h\in D} d_\Gamma(p^{m_{h-1}}_{h-1},p^0_{m_{h}})\leq lQ' L(\overline{\gamma})\leq
lQ' \left(\mu+\epsilon+\frac{\epsilon+ 1}{R}\right) \sum_{i=1}^n L(\gamma_{i}) . 
\end{equation}

On the other hand we have
\begin{align}\label{estimate2:eq}
\sum_{h\in D} d_\Gamma(p^{m_{h-1}}_{h-1},p^0_{m_{h}}) &\geq \frac{1}{\mu}\sum_{h\in D} L(\alpha_{h})-\epsilon n \\
\nonumber &\geq \frac{1}{\mu}\sum_{h\in D} L(\alpha_{h})-\frac{\epsilon}{R} \sum_{i=1}^n L(\gamma_{i}).
\end{align}
Putting together inequalities~(\ref{estimate:eq}) and~(\ref{estimate2:eq}) we
finally get that 
the 
inequality of the statement holds for some $Q$ only depending on $\mu,\epsilon,Q',R$.
\end{proof}

\section[The coarse geometry of $\tilM$ for an irreducible $M$]{Walls and chambers are quasi-isometrically embedded in the universal
covering of irreducible graph manifolds}

Let us fix the graph manifold $M$ which we are studying. We will now introduce various constants, 
which will be extensively used in the rest of the arguments for this section. Fix the following quantities:

\begin{itemize}
\item the constant $Q$: chosen so that Proposition~\ref{osin:prop} holds for all the bases of the chambers of $\tilM$.
\item the constant $R$: the minimal distance between pairs of thin walls \emph{not} associated to the same internal wall 
(note that $R$ is also the minimal distance between pairs of removed horoballs in the bases of the chambers of $\tilM$).
\item the constant $D$: the maximal distance between pairs of \emph{tied} points on adjacent thin walls (here we refer to the path distance of the corresponding wall).
\item the constant $k$: chosen so that Lemma~\ref{strong:lem} holds for all the internal walls in $\tilM$.
\item the constant $k'$: chosen so that $s_W\colon W_+\to W_-$
is $k'$-bi-Lipschitz for every internal wall $W$ of $\tilM$.
\item the constant $c$: chosen so that all the inclusions 
$W_\pm \hookrightarrow W$ are $c$-bi-Lipschitz (see Lemma~\ref{easybil:lem}). 
\end{itemize}
These constants only depend on the geometry of $M$. In what follows, we will also assume without loss of generality 
that $Q\geq 2$ and $k\geq \sqrt{2}$.

\vskip 10pt

In order to prove that walls and chambers are quasi-isometrically embedded in $\tilM$, we need to show
that the distance between points in the same chamber
can be bounded from below by the distance of the projections of the points
on the base of the chamber. We begin with the following:

\begin{definition}
Let $W_\pm$ be a thin wall, take $x,y\in W_{\pm}$ and let $\gamma$ be a continuous
path in $\tilM$ joining $x$ and $y$. We say that $\gamma$ \emph{does not backtrack} on $W_\pm$ if
$\gamma$ intersects the wall containing $W_{\pm}$ only in its endpoints.
\end{definition}

\begin{lemma}\label{osin:lem}
Let $x,y$ be points on the same thin wall $W_\pm$ 
and let $\gamma$ be a path in $\tilM$ which joins $x$ to $y$ and
does not backtrack on $W_\pm$.
If $C$ is the chamber containing $W_\pm$ and $B$ is the base of $C$,
then $L(\gamma) \geq d_B (x,y)/Q$.
\end{lemma}
\begin{proof}
An easy transversality argument shows that it is not restrictive to assume that the intersection
of $\gamma$ with $C$ consists of a finite number 
of subpaths of $\gamma$. 
Now the sum of the lengths of such subpaths is greater than 
the sum of the lengths of their projections on $B$, which is in turn greater than
$d_B(x,y)/Q$ by Proposition~\ref{osin:prop}.
\end{proof}

If the distance of two points on a thin wall is not suitably bounded
by the distance of their projections on the base of the chamber they belong to,
then Lemma~\ref{osin:lem} does not give an effective estimate.  
The following result can be combined with Lemma~\ref{strong:lem} to show that, in this case, irreducibility
allows us to ``pass to the adjacent chamber'' in order to obtain a better estimate. 

\begin{lemma}\label{osin2:lem}
Let $x_+, y_+\in W_+$ be points on a thin wall, let $C_+$ be the chamber containing
$W_+$, and suppose that $\gamma$ is a rectifiable
path joining $x_+$ and $y_+$ and intersecting $C_+$ only in its endpoints.
Let also $x_-,y_-\in W_-$ be the points tied to $x_+,y_+$, and $C_-$ be the chamber
containing $x_-,y_-$. 
Then
$$
L(\gamma)\geq \frac{d_{B_-} (x_-,y_-)}{cQ}-\frac{2D}{Q},
$$ 
where $B_-$ is the base of $C_-$. 
\end{lemma}
\begin{proof}
An easy transversality argument shows that it is not restrictive to assume that
$\gamma$ intersects the thin walls of $\widetilde{M}$ only in a finite
number of points. Then
our assumptions imply that $\gamma$ decomposes as a concatenation
of curves 
$$
\gamma=\gamma'_1\ast \gamma''_1\ast \gamma'_2\ast \ldots \ast \gamma''_{n}\ast \gamma'_{n+1}
$$
such that 
$\gamma'_i$ is supported in $W$ and $\gamma''_i$ has endpoints $a_i,b_i\in W_-$
and does not backtrack on $W_-$ for every $i$ (see Figure~\ref{osin1:fig}). 
\begin{figure}
\begin{center}
\includegraphics[width=6cm]{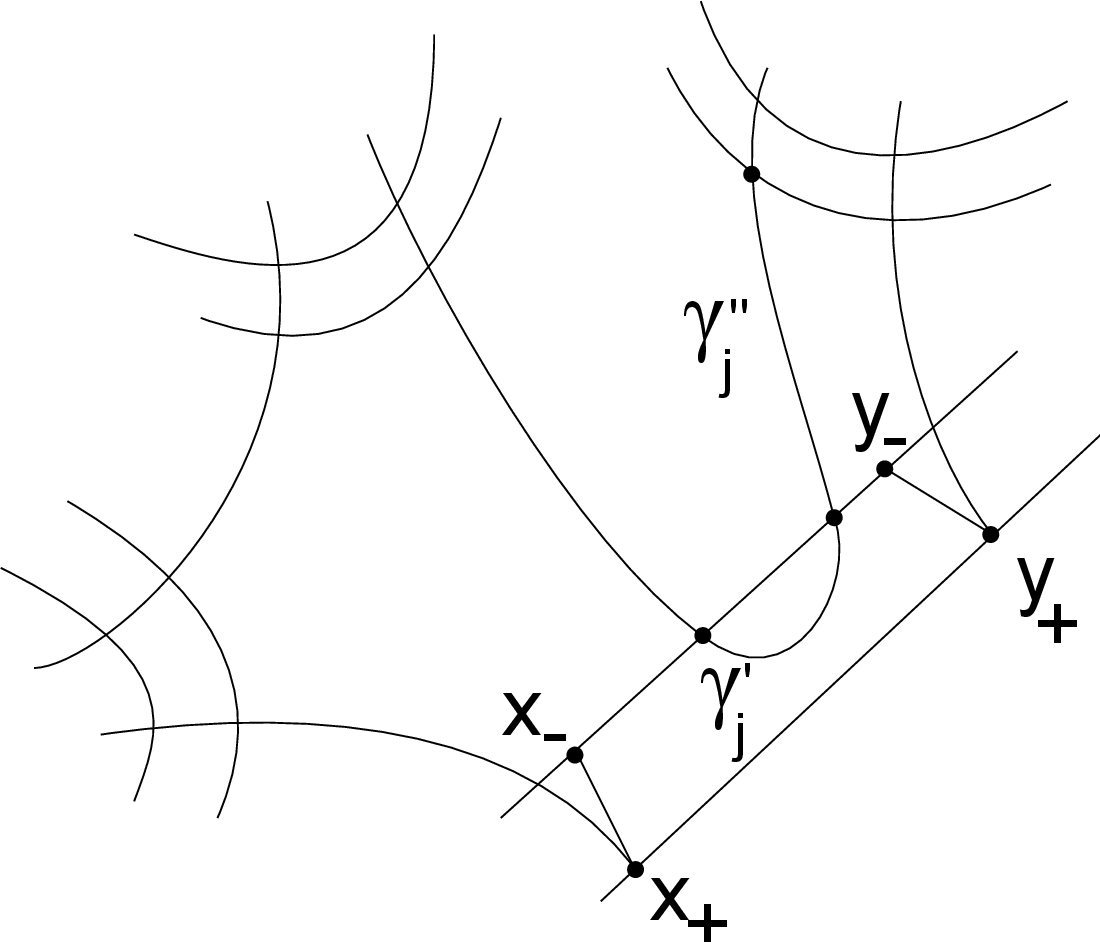}
\caption{Decomposing $\gamma$ in the proof of Lemma~\ref{osin2:lem}.}\label{osin1:fig}
\end{center}
\end{figure}
Let us suppose $n\geq 1$ (the case $n=0$ being
easier).
Since $d_W(x_-,x_+)\leq D$ we have
$$
d_{B_-} (x_-,a_1)\leq d_{W_-} (x_-,a_1)\leq c d_W (x_-,a_1)\leq c(D+L(\gamma'_1)),
$$
and analogously we get $d_{B_-} (y_-,b_{n})\leq c(D+L(\gamma'_{n+1}))$.  
Moreover Lemma~\ref{osin:lem} implies $d_{B_-}(a_i,b_i)\leq Q\cdot L(\gamma''_i)$ for every $i=1,\ldots n$, 
and we also have
$d_{B_-} (b_i,a_{i+1})\leq d_{W_-} (b_i,a_{i+1})\leq c d_W (b_i,a_{i+1})\leq c L(\gamma'_{i+1})$
for every $i=1,\ldots, n-1$.
Putting together all these inequalities we finally get
\begin{align*}
d_{B_-} (x_-,y_-)&\leq d_{B_-} (x_-,a_1)+\sum_{i=1}^n d_{B_-} (a_i,b_i)+\sum_{i=1}^{n-1} 
d_{B_-} (b_i,a_{i+1})+
d_{B_-} (b_{n},y_-)\\ 
&\leq  2cD + c\sum_{i=1}^{n+1} L(\gamma'_i) + Q \sum_{i=1}^n L(\gamma''_i)\\
&\leq 2cD +cQ L(\gamma)
\end{align*}
whence the conclusion.
\end{proof}

In order to proceed to the main argument we finally need the following lemma,
which describes how to get rid of the backtracking of a 
geodesic. 

If $\gamma$ is a path and $r=\gamma (t_0)$, $s=\gamma (t_1)$,
with an abuse we will denote by $[r,s]$ the subpath $\gamma|_{[t_0,t_1]}$ of $\gamma$. We say that $\gamma$ is \emph{minimal} if 
for every chamber $C$, the set $\gamma\cap \mathring{C}$ is a finite collection of paths each of which connects distinct
walls of $C$. Moreover, $\gamma$ is \emph{good} if it is minimal and 
for every thin wall $X$ contained in a chamber $C$ there are at most 2 endpoints of paths in $\gamma\cap\mathring{C}$ 
belonging to $X$. Notice that, since chambers are uniquely geodesic and every thin wall is totally geodesic in the
chamber in which it is contained, every geodesic of $\widetilde{M}$ is minimal.

\begin{lemma}\label{nobktrk:lem}
There exists a constant $\beta\geq 1$ depending only on the geometry of $\tilM$ such that
the following result holds.
Let $x,y$ be points belonging to the same wall
of $\widetilde{M}$. 
Then there exists a good 
path $\gamma$ in $\tilM$ connecting $x$ and $y$ such that 
$L(\gamma)\leq \beta d(x,y)$.
\end{lemma}

\begin{proof}
We first introduce some terminology. 
If $X$ is a thin wall contained in the wall $W$,   
we say that a path $\theta:[t_0,t_1]\to\tilM$ is \emph{external} to 
$X$ if $\theta (t_0)\in X$, $\theta (t_1)\in X$ and 
$\theta|_{(t_0,t_1)}$ is supported
in $\tilM\setminus W$ (this is equivalent to asking that $\theta$ does not backtrack on $X$, but
this new terminology will prove more appropriate here).
Moreover, if $\gamma$ is a minimal path  
and $n$ is the number
of subpaths of $\gamma$ external to $X$, we say that  
the \emph{exceeding number} of $\gamma$ on $X$ is equal to $\max \{0,n-1\}$.
The exceeding number $e(\gamma)$ of $\gamma$ is the sum of the exceeding numbers of $\gamma$
on all the thin walls.
Finally, we denote by $j(\gamma)$ the sum over \emph{all} the chambers
$C$ of $\widetilde{M}$ of the number of connected components
of $\gamma\cap \mathring{C}$.
It is readily seen that
a path $\gamma$ is good if and only if it is minimal and $e(\gamma)=0$. 

Let $\Delta>0$ be a constant, chosen in such
a way that every torus in $M$ obtained as a projection of a thin wall of $\tilM$
has diameter (with respect to its intrinsic path metric) at most $\Delta/2$.
We denote by $\gamma_0$ a geodesic in $\widetilde{M}$ connecting $x$ and $y$. 
As observed above, $\gamma_0$ is minimal, and 
if $\gamma_i$ is a minimal path with $e(\gamma_i)>0$
we will now describe how to modify it in order to get 
a new minimal path $\gamma_{i+1}$ joining $x$ to $y$. The path $\gamma_{i+1}$ will be 
constructed so as to have
$j(\gamma_{i+1})<j(\gamma_i)$ and $L(\gamma_{i+1})\leq  L(\gamma_i)+ 4\Delta +1$.
By the very definitions we have $j(\gamma_0)\leq L(\gamma_0)/R=d(x,y)/R$, so 
after at most $d(x,y)/R$ steps we will end up with a minimal path $\gamma$
which verifies either $e(\gamma)=0$ or $j(\gamma)\leq 1$, whence again $e(\gamma)=0$.
After setting $\beta=1+(4\Delta+1)/R$, such a path satisfies all the conditions required.

So let us suppose that we have some external subpaths $[p_1,p'_1]$, $[p_2,p'_2]$
of $\gamma_i$, with $p_1,p'_1,p_2,p'_2\in X$ for some thin wall $X$ contained in the chamber $C$. 
Consider deck transformations $g,h$ which leave $X$ (and therefore $C$) 
invariant such that $d_X(g(p_2),p'_1)\leq \Delta$,
$d_X(h(p'_1),g(p'_2))\leq\Delta$, and let 
$q_1,q_2\in\gamma\cap \mathring{C}$ be chosen in such a way that
$q_1$ (resp.~$q_2$) slightly precedes (resp.~follows)
$p'_1$ (resp.~$p_2$) on $\gamma_i$: more precisely, we assume that
$L([q_1,p'_1])< 1/2$, $L([p_2,q_2])<1/2$. 
We define a path $\gamma'_{i+1}$ as 
the concatenation of the following paths (see Figure~\ref{nobktrk:fig}):
\begin{enumerate}
\item
the subpath $[x,q_1]$ of $\gamma_i$,
\item
a path $[q_1,g(q_2)]$ in $\mathring{C}$ obtained by slightly pushing
inside $\mathring{C}$ a geodesic in $X$ between $p'_1$ and $g(p_2)$,
in such a way that $L([q_1,g(q_2)])<\Delta+1/2+1/2=\Delta +1$,
\item
$g([q_2,p'_2]),$
\item
a geodesic in $X$ between $g(p'_2)$ and $h(p'_1)$,
\item
$h([p'_1, p_2]),$
\item
a geodesic in $X$ between $h(p_2)$ and $p'_2$,
\item
$[p'_2,y]$,
\end{enumerate}
where geodesics in $X$ are to be considered with respect to its path 
metric.

\begin{figure}
\begin{center}
\includegraphics[width=11cm]{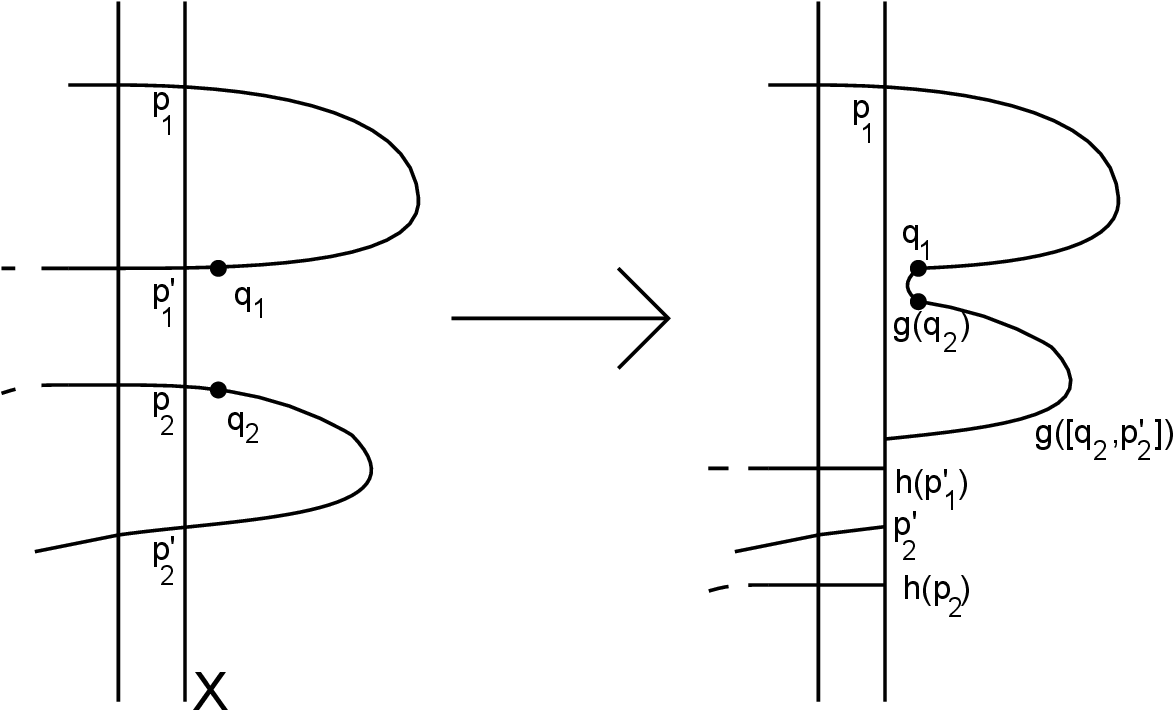}
\caption{Replacing $\gamma_i$ with $\gamma_{i+1}'$ in the proof of Lemma~\ref{nobktrk:lem}.}\label{nobktrk:fig}
\end{center}
\end{figure}

Since $X$ is isometric to $\mR^{n-1}$ and the deck transformations $g,h$
act on $X$ as translations, it follows
that the distance between $h(p_2)$ and $p'_2$ is at most 
$2\Delta$, and this readily yields $L(\gamma'_{i+1})\leq L(\gamma_i)+4\Delta+1$.
Moreover, it is easily checked that $j(\gamma'_{i+1})=j(\gamma_i)-1$.
Now, if $\gamma'_{i+1}$ is minimal we set $\gamma_{i+1}=\gamma'_{i+1}$, and we are done.
On the other hand, the only possible obstruction to $\gamma'_{i+1}$ being minimal is that
its (open) subpath with endpoints $p_1$ and
$g(p'_2)$ may be entirely contained in $\mathring{C}$. 
In this case, since $X$
is totally geodesic in $C$ we can replace the subpath $[p_1,g(p'_2)]$
with a geodesic on $X$, thus obtaining  a minimal path $\gamma_{i+1}$ with
$L(\gamma_{i+1})\leq L(\gamma'_{i+1})$ and 
$j(\gamma_{i+1})= j(\gamma'_{i+1})-1< j(\gamma_i)$, whence the conclusion again.
\end{proof}

\begin{remark}
The strategy described in this Chapter could probably be adapted
in order to study the coarse geometry of the universal coverings of other classes
of manifolds. For example, let $N$ be a cusp-decomposable manifold~\cite{N}, i.e.~a manifold obtained by taking complete, finite volume, negatively curved, locally symmetric
manifolds with deleted cusps, and gluing them along affine diffeomorphisms of their
cuspidal boundary. Then the universal covering $\widetilde{N}$ admits a natural
decomposition
into walls and chambers, and it would be interesting to show that walls are quasi-isometrically embedded also in this context (since cusp-decomposable manifolds
consist only of ``pure pieces'', we don't need to impose any irreducibility
condition here).

By~\cite{farb}, 
the fundamental groups
of the pieces of $N$ are relatively hyperbolic with respect to their cusp subgroups,
so many results proved in this Chapter readily extend to the study of $\widetilde{N}$.
However, the needed generalization of Lemma~\ref{nobktrk:lem} 
could be a challenging task.
In fact,
the proof of Lemma~\ref{nobktrk:lem}
heavily relies on the fact that 
thin walls of $\tilM$ support a \emph{flat} 
metric, a fact which is no longer true in the case of cusp-decomposable manifolds. 
By Bieberbach Theorem, the flatness of thin walls ensures that 
a finite index subgroup of the covering transformations of 
$\tilM$ preserving a thin wall acts on it as a group of translations, 
and this fact was exploited in the proof
of Lemma~\ref{nobktrk:lem}. 
\end{remark}

\begin{lemma}\label{back:lem} 
Fix a wall $W\subseteq \tilM$ and suppose that $\alpha\geq 1$ exists such that
the following holds: if $x,y\in W_\pm$ 
are points joined by a good path $\gamma$ in $\tilM$ which does not backtrack on $W_\pm$, then
$$d_C(x,y)\leq\alpha \cdot L(\gamma),$$
where $C$ is the chamber containing $x,y$. 
Then $W$ is bi-Lipschitz embedded in $\tilM$.
\end{lemma}

\begin{proof}
The inclusion  $(W,d_W)\hookrightarrow \tilM$ is clearly $1$-Lipschitz,
so we have to check that $d_W$ is linearly bounded
below by the distance $d$ on $\tilM$. 
More precisely, we have to show that there exists
$\lambda\geq 1$ such that
\begin{equation}\label{fund:eq}
d_W (p,q)\leq \lambda d(p,q) \qquad {\rm for\ all}\ p,q\in W.
\end{equation}
Let $\gamma$ be the path provided by Lemma~\ref{nobktrk:lem} such that
$L(\gamma)\leq \beta d(p,q)$, and let
$m$ be the number of the chambers adjacent
to $W$ whose internal parts intersect $\gamma$ (so $m=0,1$ or $2$). 
It is readily seen that $\gamma$ splits
as a concatenation
$$\gamma_1\ast\gamma'_1\ast\dots\ast\gamma_m\ast\gamma'_m\ast\gamma_{m+1},$$
where the $\gamma_i$'s are contained in $W$ and each $\gamma'_i$ is a good
path with endpoints on $W_\pm$ which does not backtrack on $W_\pm$.
Due to our assumptions and to the fact that $W_\pm$ 
are totally geodesic in the chambers in which they are contained,
the $\gamma'_i$'s can be replaced by curves contained in $W$
in such a way that the total length of the curve so obtained does not exceed
$\alpha \cdot L(\gamma)$. So
$$d_W (p,q)\leq \alpha L(\gamma),$$
and hence inequality~(\ref{fund:eq}) holds with
$\lambda=\alpha\cdot\beta$.
\end{proof}

\begin{theorem}\label{quasi-isom:thm}
If $M$ is irreducible and $W\subseteq \tilM$ is a wall, then
the inclusion $(W,d_W)\hookrightarrow \tilM$ is a bi-Lipschitz embedding. In particular, 
it is a quasi-isometric embedding. Moreover, the bi-Lipschitz constant of the embedding 
only depends on the geometry of $\tilM$ (i.e.~it does not depend
on the fixed wall $W$).
\end{theorem}

\begin{proof}
Take $x,y\in W_+$, let $C$ be the chamber containing $W_+$  
and let $\gamma$ be  a good path in $\tilM$ which joins $x$ to $y$ and does not backtrack on $W_+$.
By Lemma~\ref{back:lem}, in order to conclude it is sufficient to show 
that the inequality 
\begin{equation}\label{fund2:eq}
d_C(x,y)\leq \alpha \cdot L(\gamma).
\end{equation}
holds for some $\alpha\geq 1$ only depending on $\tilM$ 
(via the constants $D,R,Q,k,k',c$). 
We will have to analyze several different cases, and we
will take $\alpha$ to be the maximum among the constants we will find in each case.
 
Let $B$, $F$ be the base and the fiber of $C$.
We first distinguish the case when the distance between $x$ and $y$
is controlled (up to a suitable constant factor) by $d_B(x,y)$ from the case
when $d_C (x,y)$ is controlled by $d_F (x,y)$.

\smallskip

So let us suppose $d_C(x,y)\leq k d_B(x,y)$. In this case
by Lemma~\ref{osin:lem} we have
$L(\gamma)\geq d_B(x,y)/Q\geq d_C (x,y)/(kQ)$, so
$$
d_C (x,y)\leq k Q L(\gamma), 
$$
and we are done.

\smallskip

Let us now consider the other case and assume that $d_C(x,y)>k d_B(x,y)$.
Since $d_C^2=d_B^2+d_F^2$ and $k> \sqrt{2}$ an easy computation shows that  
$$
d_F(x,y)>  \frac{d_C (x,y)}{\sqrt{2}}, 
\qquad d_F (x,y)> d_B (x,y).
$$  
Write $\gamma\cap \mathring{C}=\gamma_1\cup...\cup\gamma_m$ where each $\gamma_i=(x_i,y_i)$ is a path 
in the (open) chamber $\mathring{C}$, 
let $W_i$ be the wall containing $y_i$ and $x_{i+1}$, and 
let $l_i$ be the length of the projection of $\gamma_i$ on the fiber $F$. 
Observe that since $\gamma$ is minimal we have $m\leq L(\gamma)/R$.
Of course, we have $\sum l_i + \sum d_F(y_i,x_{i+1}) \geq d_F (x,y)$, so either
$\sum l_i\geq  d_F(x,y)/2$
or $\sum d_F(y_i,x_{i+1})\geq d_F(x,y)/2$. In the first case we have
$$
L(\gamma)\geq \sum L(\gamma_i) \geq \sum l_i\geq \frac{d_F (x,y)}{2}> \frac{d_C (x,y)}{2\sqrt{2}},
$$
and we are done. 
Otherwise let us define
$$
I_1=\{i\in \{1,\ldots,m-1\}\, |\, 
k d_B(y_i,x_{i+1})\leq d_C(y_i, x_{i+1})\},\quad 
I_2=\{1,\ldots, m-1\}\setminus I_1.
$$ 
Since $\sum d_C (y_i,x_{i+1})\geq \sum d_F (y_i,x_{i+1})\geq d_F (x,y)/2$, we have two
possibilities: either
$\sum_{i\in I_1} d_C(y_i,x_{i+1})\geq d_F (x,y)/4$, or 
$\sum_{i\in I_2} d_C(y_i,x_{i+1})\geq d_F (x,y)/4$. 

\vskip 10pt

We begin by dealing with the first case. Let $W^i_+$ be the thin wall
containing $x_{i+1},y_i$,  
denote by $x_{i+1}^-\in W^i_-$ (resp.~$y_i^-\in W^i_-$) 
the point tied to $x_{i+1}$ (resp.~to $y_i$), let $C_i$ be the chamber containing
$x_{i+1}^-,y_i^-$, and $B_i$ the base of $C_i$. 

Recall that we have the estimate:
$$d_{C_i} (y_i^-,x_{i+1}^-)=d_{W^i_-} (y_i^-,x_{i+1}^-) 
\geq d_{W^i_+} (y_i,x_{i+1})/k'= d_C (y_i,x_{i+1})/k'.$$ 
So applying Lemma~\ref{osin2:lem} and Lemma~\ref{strong:lem} 
(which gives the inequality $d_{B_i} (y_i^-,x_{i+1}^-)\geq d_{C_i} (y_i^-,x_{i+1}^-)/k$) we obtain
the estimates:
\begin{align*}
L(\gamma) &\geq \frac{1}{cQ}
\sum_{i\in I_1}d_{B_i}(y_i^-,x_{i+1}^-)-\frac{2(m-1)D}{Q}\\
&\geq \frac{1}{kcQ} \sum_{i\in I_1}d_{C_i}(y_i^-,x_{i+1}^-)-\frac{2(m-1)D}{Q} \vspace{.2cm}\\
&\geq \frac{1}{kcQk'} \sum_{i\in I_1}d_{C}(y_i,x_{i+1})-\frac{2D}{RQ} L(\gamma)\\
&\geq \frac{1}{4kcQk'} d_F (x,y) -\frac{2D}{RQ}L(\gamma)\vspace{.2cm}\\
& \geq \frac{1}{4\sqrt{2}kcQk'} d_C(x,y)-\frac{2D}{RQ}L(\gamma)
\end{align*}
Isolating the $d_C(x,y)$ term, this gives us
$$
d_C(x,y)\leq 
\frac{4\sqrt{2}kck'(RQ+2D)}{R}\cdot L(\gamma).
$$
which gives us the requisite estimate in the first case.

\vskip 10pt

We are now left to deal with the second case, $\sum_{i\in I_2} d_C(y_i, x_{i+1})\geq d_F(x,y)/4$. 
In this case we have that:
$$\sum_{i\in I_2} d_B(y_i,x_{i+1})\geq d_F(x,y)/(4k)\geq d_C(x,y)/(4\sqrt{2}k).$$
Let $\gamma'$ be the loop in $C$ obtained by concatenating the geodesic
in $W_+$ joining $y$ with $x$, 
the paths of the form $\gamma\cap \mathring{C}$ and
the geodesics in the $W^i_+$'s joining  $y_i$ with $x_{i+1}$, and set 
$\overline{\gamma}=\pi_B\circ \gamma'$. If $\eta$ is the sum of the lengths of the subpaths
of $\overline{\gamma}$ 
obtained by projecting the paths in $\gamma\cap \mathring{C}$
we obviously have $L(\gamma)\geq \eta$. 
Moreover, the properties  of $\gamma$ described in Lemma~\ref{nobktrk:lem} 
ensure that 
the $W_i$'s are pairwise distinct, and distinct from $W$. 
As such, we can apply Proposition~\ref{osin:prop} to $\overline{\gamma}$ thus getting 
$$
L(\gamma)\geq \eta\geq \frac{1}{Q} \sum_{i\in I_2} d_B(y_i, x_{i+1}),
$$
whence 
$$
d_C (x,y)\leq 4\sqrt{2}k
\sum_{i\in I_2} d_B(y_i, x_{i+1})
\leq 4\sqrt{2} kQ L(\gamma) .
$$
\smallskip

This completes the last case, establishing that inequality~(\ref{fund2:eq}) holds with constant 
$$
\alpha=\max \left\{kQ,\, 2\sqrt{2},\, \frac{4\sqrt{2}kck'(RQ+2D)}{R},\, 4\sqrt{2}k Q\right\}=\frac{4\sqrt{2}kck'(RQ+2D)}{R},
$$
thus proving the proposition.
\end{proof}

\begin{corollary}
If $M$ is irreducible, then the inclusion of a chamber in $\widetilde{M}$ is a bi-Lipschitz embedding.
\end{corollary}
\begin{proof}
Let $p,q$ be points in a chamber $C$ and let $\delta$ be a geodesic of 
$\widetilde{M}$ joining $p$ to $q$. Then $\delta$ splits as a concatenation
$$
\delta=\delta_1\ast\eta_1\ast\ldots\ast\eta_n\ast\delta_{n+1},
$$
where $\delta_i$ is a geodesic segment (with respect to the metric $d$
on $\widetilde{M}$) supported in $C$ and the endpoints $p_i,q_i$
of $\eta_i$ belong to a thin wall $W_i^+$ adjacent to $C$. By Theorem~\ref{quasi-isom:thm} there exists $\alpha\geq 1$ such that
$d_{W_i} (p_i,q_i)\leq \alpha d(p_i,q_i)$, and this implies in turn that $d_{W_i^+}(p_i,q_i)\leq c\alpha d(p_i,q_i)$, so we may replace every $\eta_i$ with a path $\eta_i'\subseteq W_i^+$ having the same endpoints
as $\eta_i$ and length that does not exceed $c\alpha d(p_i,q_i)$. The path
$$
\delta'=\delta_1\ast\eta'_1\ast\ldots\ast\eta'_n\ast\delta_{n+1}
$$
is supported in $C$ and has length at most $c\alpha d(p,q)$, so $d_C(p,q)\leq c\alpha d(p,q)$, 
and we are done.
\end{proof}

\begin{corollary}\label{fibre:cor}
Suppose that $M$ is irreducible. Then,
the inclusion of chambers, walls and fibers (with their path metrics)
in $\tilM$ are quasi-isometric embeddings. In particular:
\begin{itemize}
\item
If $C\subseteq \tilM$ is a chamber, then $C$ is quasi-isometric (with the metric
induced by $\tilM$) to a product $B\times \mR^k$, where $B$ is a neutered space.
\item
If $W\subseteq \tilM$ is a wall, then $W$ is quasi-isometric (with the metric induced
by $\tilM$) to $\mR^{n-1}$.
\item
If $F\subseteq \tilM$ is a fiber, then $F$ is quasi-isometric (with the metric induced
by $\tilM$) to $\mR^h$, $h\leq n-3$.
\end{itemize}
\end{corollary}

\chapter[Pieces of irreducible graph manifolds are quasi-preserved]{Pieces of irreducible graph manifolds\\ are quasi-preserved}\label{preserve:sec}

In this chapter, we prove Theorem~\ref{qi-preserve:thm}, which we recall here for the convenience of the reader:

\begin{Thm2}[Pieces of irreducible manifolds are preserved]
Let $M_1$, $M_2$ be a pair of irreducible graph manifolds, and $\G_i=\pi_1(M_i)$ their respective
fundamental groups.  Let $\La_1 \leq \G_1$ be a subgroup conjugate to the fundamental
group of a piece in $M_1$, and $\varphi: \G_1\rightarrow \G_2$ be a quasi-isometry. 
Then, the set $\varphi(\La_1)$ is within finite Hausdorff distance from
a conjugate of $\La_2 \leq \G_2$, where $\La_2$ is the fundamental group of a piece  in $M_2$.
\end{Thm2}

So, let us fix graph manifolds $M_1,M_2$ with fundamental
groups $\Gamma_i=\pi_1 (M_1)$ and suppose $\psi\colon \Gamma_1\to\Gamma_2$ is a quasi-isometry.
Due to Milnor-Svarc Lemma (see Theorem~\ref{milsv}), $\psi$ induces a quasi-isometry between
$\widetilde{M}_1$ and $\widetilde{M}_2$, which we will still denote by $\psi$.
The statement of Theorem~\ref{qi-preserve:thm} is equivalent to the fact that
$\psi$ sends, up to a finite distance, chambers of $\tilM_1$ into chambers
of $\tilM_2$. In order to prove this fact, we will use the technology of \emph{asymptotic cones},
which we now briefly describe. 

\section{The asymptotic cone of a geodesic metric space}

Roughly speaking, 
the asymptotic cone of a metric space gives a picture of the metric
space as ``seen from infinitely far away''. It was introduced by Gromov in~\cite{gropol},
and formally defined in~\cite{wilkie}. 

A \emph{filter} on $\mN$ is a set $\omega\subseteq \calP(\mN)$ satisfying the following conditions:
\begin{enumerate}
\item
$\emptyset\notin \omega$;
\item
$A,B\in \omega\ \Longrightarrow\  A\cap B\in \omega$;
\item 
$A\in \omega,\ B\supseteq A\ \Longrightarrow\ B\in\omega$.
\end{enumerate}
For example, the set of complements of finite subsets of $\mN$
is a filter on $\mN$, known as the \emph{Fr\'echet filter} on $\mN$.

A filter $\omega$ is a \emph{ultrafilter} if for every $A\subseteq\mN$
we have either $A\in\omega$ or $A^c\in\omega$, where $A^c :=\mN\setminus A$.
For example, fixing an element $a\subset \mN$, we can take the associated
\emph{principal ultrafilter} to consist of all subsets of $\mathbb N$ which 
contain $a$. An ultrafilter is \emph{non-principal} if it does not contain any 
finite subset of $\mN$. 

It is readily seen that a filter is an ultrafilter if and only if it is maximal
with respect to inclusion. Moreover, an easy application of Zorn's Lemma
shows that any filter is contained in a maximal one. Thus, non-principal
ultrafilters exist (just take any maximal filter containing the Fr\'echet filter). 

From this point on, let us fix a non-principal ultrafilter $\omega$ on $\mN$.
As usual, we say that a statement $\mathcal{P}_i$ depending on $i\in\mathbb{N}$ holds 
$\omega$-a.e.~if the set of indices such that $\mathcal{P}_i$ holds belongs to $\omega$. 
If $X$ is a topological space, and $(x_i)\subseteq X$ is a  
sequence in $X$, 
we say that $\omega$-$\lim x_i =x_\infty$ if 
$x_i\in U$ $\omega$-a.e. for every neighbourhood $U$ of $x_\infty$.
When $X$ is Hausdorff, an $\omega$-limit of a sequence, if it
exists, is unique. Moreover, any sequence in any compact space admits
an $\omega$-limit. For example, any sequence $(a_i)$ in
$[0,+\infty]$ admits a unique $\omega$-limit.

Now let $(X_i,x_i,d_i)$, $i\in\mN$, be a sequence of pointed metric spaces. 
Let $\calC$ be the set of sequences $(y_i), y_i\in X_i$,
such that $\omega$-$\lim {d_i(x_i,y_i)}<+\infty$, and consider the 
following equivalence
relation on $\calC$:
$$
(y_i)\sim (z_i)\quad \Longleftrightarrow \quad
\omega\text{-}\lim {d_i(y_i,z_i)}=0.
$$
We set $\omega$-$\lim (X_i,x_i,d_i)=\calC/_\sim$, and we endow 
$\omega$-$\lim (X_i,x_i,d_i)$ with the well-defined distance given by
$d_\omega \big([(y_i)],[(z_i)]\big)=\omega$-$\lim {d_i(y_i,z_i)}$. 
The pointed metric space $(\omega$-$\lim (X_i,x_i,d_i),d_\omega)$ is called the 
\emph{$\omega$-limit} of the pointed metric spaces $X_i$.

Let $(X,d)$ be a metric space, $(x_i)\subseteq X$ a sequence of
base-points, and $(r_i)\subset \mR^+$ a sequence of rescaling factors
diverging to infinity. We introduce the notation $(X_\omega ((x_i),(r_i)),d_\omega):=
\omega$-$\lim (X_i,x_i,{d}/{r_i})$.

\begin{definition}\label{cone:def}
The metric space $\big(X_\omega\big((x_i),(r_i)\big),d_\omega\big)$ 
is the \emph{asymptotic cone}
of $X$ with respect to the ultrafilter $\omega$, 
the basepoints $(x_i)$ and the rescaling factors $(r_i)$. 
For conciseness, we will occasionally just write $X_\omega\big((x_i),(r_i)\big)$
for the asymptotic cone, the distance being implicitly understood to
be $d_\omega$.
\end{definition}

If $\omega$ is fixed and $(a_i)\subseteq \mR$ is any sequence, we say that $(a_i)$
is $o(r_i)$ (resp.~$O(r_i)$) if $\omega$-$\lim {a_i}/{r_i}=0$ (resp.~$\omega$-$\lim {|a_i|}/{r_i}<\infty$).

Let $(x_i)\subseteq X$, $(r_i)\subseteq \mR$ be fixed sequences of basepoints
and rescaling factors, and set $X_\omega=(X_\omega((x_i),(r_i)),d_\omega)$.
Sequences of subsets in $X$ give rise
to subsets of $X_\omega$: if for every $i\in\mN$ we are given a subset
$\emptyset\neq A_i\subseteq X$, we set
$$
\omega\text{-}\lim A_i=\{[(p_i)]\in X_\omega\, |\, p_i\in A_i\ {\rm for\ every}\ i\in\mN\}.
$$
It is easily seen that for any choice of the $A_i$'s, the set $\omega$-$\lim A_i$
is closed in $X_\omega$. Moreover, $\omega$-$\lim A_i\neq\emptyset$ if and only
if the sequence $(d(x_i,A_i))$ is $O(r_i)$.

\section{Quasi-isometries and asymptotic cones}
We are interested in describing how quasi-isometries asymptotically define
bi-Lipschitz homeomorphisms. In order to do this, and to fix some notations, 
we recall some basic results about $\omega$-limits of quasi-isometries and quasi-geodesics.

Suppose that $(Y_i,y_i,d_i)$, $i\in\mN$ are pointed metric spaces, and that $(X,d)$
is a metric space. Let $(x_i)\subseteq X$ be a 
sequence of basepoints and $(r_i)\subset\mR$ a sequence of rescaling factors.
Until the end of the section, to simplify the notation, we set $X_\omega:=(X_\omega,(x_i),(r_i))$.
The following result is well-known (and very easy):

\begin{lemma}\label{quasiasymp:lem}
Suppose $(k_i)\subseteq \mR^+$, $(c_i)\subseteq \mR^+$ are sequences
satisfying $k=\omega$-$\lim k_i<\infty$, and $c_i=o(r_i)$.
For each $i\in\mN$, let $f_i\colon Y_i\to X$ be a map
with the property that for every $y,y'\in Y_i$, the inequality
$$
d(f_i (y),f_i(y'))\leq k_i d_i (y,y')+c_i
$$
holds. If $d (f_i (y_i),x_i)=O(r_i)$, then the formula $[(p_i)]\mapsto
[f_i(p_i)]$ provides a well-defined map $f_\omega\colon \omega$-$\lim (Y_i,y_i,d_i/r_i)\to
X_\omega$. Moreover, $f_\omega$ is $k$-Lipschitz,
whence continuous. If $k >0$ and 
$$
d(f_i (y),f_i (y'))\geq \frac{d_i (y,y')}{k_i}-c_i
$$
is also satisfied (i.e.~if $f_i$ is a $(k_i,c_i)$-quasi-isometric embedding), then
$f_\omega$ is a $k$-bi-Lipschitz embedding.
\end{lemma}

As a corollary, quasi-isometric metric spaces have bi-Lipschitz homeomorphic asymptotic cones.
We recall that a $(k,c)$-quasi-geodesic in $X$
is a $(k,c)$-quasi-isometric embedding of a (possibly unbounded) interval
in $X$. 

\begin{lemma}\label{quasigeod1:lem}
Suppose $(k_i)\subseteq \mR^+$, $(c_i)\subseteq \mR^+$ are sequences
satisfying $k=\omega$-$\lim k_i<\infty$, and $c_i=o(r_i)$.
For each $i\in\mN$,
let $\gamma_i\colon [a_i,b_i]\to X$ be a $(k_i,c_i)$-quasi-geodesic
with image $H_i={\rm Im}\, \gamma_i$, and assume that $d(x_i,H_i)=O(r_i)$. 
Then up to precomposing $\gamma_i$ with a translation of $\mR$,
we may suppose that $0$ is the basepoint of $[a_i,b_i]$, and that the sequence $(\gamma_i)$ induces
a $k$-bi-Lipschitz path 
$$
\gamma_\omega\colon [\omega\text{-}\lim (a_i/r_i),\omega\text{-}\lim (b_i/r_i)] \to X_\omega.
$$
Moreover, we have ${\rm Im}\, \gamma_\omega=\omega$-$\lim H_i$.
\end{lemma}
\begin{proof}
The only non-trivial (but easy) assertion is the last one, which we leave to the reader.
\end{proof}

The following result extends the previous lemma to the case of
Lipschitz loops. For every $r>0$ we denote by $r\cdot S^1$ the circle of length $2\pi r$. Using that $S^1$ is compact,
it is immediate to check that $\omega{\textrm -}\lim \frac{1}{r_i} (r_i\cdot S^1)$ may be identified with $S^1$ independently of the choice of the basepoints involved in the definition
of the $\omega$-limit. 

\begin{lemma}\label{quasigeod2:lem}
Suppose that the sequence $(k_i)\subseteq \mR^+$
satisfies $k=\omega$-$\lim k_i<\infty$.
For each $i\in\mN$,
let $\gamma_i\colon r_i\cdot S^1 \to X$ be a $k_i$-Lipschitz loop
with image $H_i={\rm Im}\, \gamma_i$, and assume that $d(x_i,H_i)=O(r_i)$. 
Then 
the sequence $(\gamma_i)$ induces
a $k$-Lipschitz loop 
$$
\gamma_\omega\colon S^1 \to X_\omega.
$$
such that ${\rm Im}\, \gamma_\omega=\omega$-$\lim H_i$.
\end{lemma}
\begin{proof}
 The proof is left to the reader.
\end{proof}

The previous results assert that coarsely Lipschitz paths and Lipschitz loops give rise to Lipschitz paths and loops
in the asymptotic cone (in fact, in the case of loops the Lipschitz
condition could be replaced by the analogous coarse Lipschitz condition, but this is not relevant to our purposes). The next lemma shows a type of converse to this result.

\begin{lemma}\label{lift2:lem}
Assume $X$ is a geodesic space, let $Y=[0,1]$ or $Y=S^1$, and let $\gamma_\omega\colon Y\to X_\omega$ be a $k$-Lipschitz 
path. Let also $Y_i=[0,r_i]$ (if $Y=[0,1]$) or $Y_i=r_i\cdot S^1$ (if $Y=S^1$).
Then, for every $\varepsilon>0$ there exists a sequence of $(k+\varepsilon)$-Lipschitz paths
$\gamma_i\colon Y_i\to X$ with the following properties:
\begin{itemize}
\item 
$d(x_i,\gamma_i (0))=O(r_i)$, so if
$Y=\omega$-$\lim \frac{1}{r_i}Y_i$ then $(\gamma_i)$ defines a $(k+\varepsilon)$-Lipschitz
path $\omega$-$\lim \gamma_i\colon Y\to X_\omega$;
\item
$\omega$-$\lim \gamma_i=\gamma_\omega$.
\end{itemize}
\end{lemma}
\begin{proof}
We prove the statement under the assumption that $Y=[0,1]$, the case when $Y=S^1$ being
analogous.

For every $t\in [0,1]$ set $p^t=\gamma_\omega (t)=[(p^t_i)]$, and for every $j\in\mN$ let
$A_j\subseteq \mN$ be the set of indices $i\in\mN$ such that
$$ d(p_i^t,p_i^{t'})\leq \left(1+\frac{\varepsilon}{k}\right)r_i d_\omega (p^t,p^{t'})$$ 
for every
$t=h\cdot 2^{-j}, t'=h'\cdot 2^{-j},\, h,h'\in\mZ,\, 0\leq h,h'\leq 2^j$.
By construction we have $A_{j+1}\subseteq A_j$ and $A_j\in \omega$ for every $j\in\mN$.
For every $i\in\mN$, let 
$$
j(i)=\sup \{j\in\mN\, |\, i\in A_j\}\in \mN\cup\{\infty\},
$$ 
and set $j'(i)=i$ if $j(i)=\infty$ and $j'(i)=j(i)$ otherwise. 
By the nature of the construction, we have $i\in A_{j'(i)}$.
For every $i\in\mN$, 
we define the curve $\gamma_i\colon [0,r_i]\to X$ as follows:
if $h\in\{0,1,\ldots,2^{j'(i)}-1\}$, then the restriction of $\gamma_i$
to the interval $[hr_i 2^{-j'(i)}, (h+1)r_i 2^{-j'(i)}]$ 
is a linear parameterization of a geodesic joining 
$p_i^{h 2^{-j'(i)}}$ with $p_i^{(h+1) 2^{-j'(i)}}$. Since $i\in A_{j'(i)}$ 
each such restriction is $(k+\varepsilon)$-Lipschitz, so $\gamma_i$ is 
$(k+\varepsilon)$-Lipschitz. It readily follows that $\omega$-$\lim \gamma_i$ is 
$(k+\varepsilon)$-Lipschitz, and in particular continuous. Thus, in order to show
that $\omega$-$\lim \gamma_i=\gamma_\omega$ it is sufficient to show that 
$(\omega$-$\lim \gamma_i) (t)=\gamma_\omega (t)=p^t$ for every $t$ of the form
$h2^{-j}$, $h,j\in\mN$. However, if $t=h2^{-j}$ by construction we have
$$
\begin{array}{lll}
\{i\in\mN\, |\, \gamma_i (t)=p_i^t\} &\supseteq&
\{i\, |\, j\leq j(i)<\infty\}\cup\big( \{i\, |\, j(i)=\infty\}\cap \{i\, |\, i\geq j\}\big) \\
&\supseteq& \{i\, |\, j\leq j(i)\}\cap \{i\, |\, i\geq j\} \\
&\supseteq& A_j\cap \{i\, |\, i\geq j\} \in \omega\ .
\end{array}
$$ 
As a result, for each $t=h2^{-j}$, we have that $\omega$-$\lim \gamma_i (t)=[(\gamma_i (t))]=[(p_i^t)]=
\gamma_\omega (t)$, whence the conclusion. 
\end{proof}

\section{Tree-graded spaces}\label{treegr:subsec}
We are going to need some results about the asymptotic cones 
of complete hyperbolic manifolds of finite volume. The following definitions are taken from~\cite{dru}. If $X$ is a set, then we denote by $|X|$ the cardinality of $X$.

\begin{definition}
A geodesic metric space $X$ is said to be \emph{tree-graded} 
with respect to a collection of closed subsets $\{P_i\}_{i\in I}$, 
called \emph{pieces}, if
\begin{enumerate}
\item
$\bigcup P_i=X$,
\item
$|P_i\cap P_j|\leq 1$ if $i\neq j$,
\item
any simple geodesic triangle in $X$  is contained in a single piece.
\end{enumerate}
\end{definition}

\begin{definition}\label{as-tree:def}
A geodesic metric space $X$ is \emph{asymptotically tree-graded} with respect to a collection of subsets
$A=\{H_i\}_{i\in I}$ if the following conditions hold: 
\begin{enumerate}
\item 
for each choice of basepoints $(x_i)\subseteq X$ and rescaling factors $(r_i)$, 
the associated asymptotic cone $X_\omega=X_\omega((x_i),(r_i))$ is tree-graded
with respect to the collection of subsets 
$\mathcal{P}=\{\omega$-$\lim H_{i(n)}\, |\, H_{i(n)}\in A\}$, and
\item if $\omega$-$\lim H_{i(n)}= 
\omega$-$\lim H_{j(n)}$, where $i(n),j(n)\in I$, then $i(n)=j(n)$ $\omega$-a.e.
\end{enumerate}
\end{definition}

We summarize in the following lemmas some properties of tree-graded spaces
which are proved in~\cite{dru}  and will be useful later.

\begin{lemma}\label{treegr:lem}
Let $P,P'$ be distinct pieces of a tree-graded space $Y$. Then there exist $p\in P$, $p'\in P'$
such that,
for any continuous path $\gamma\colon [0,1]\to Y$ with $\gamma (0)\in P$ and $\gamma (1)\in P'$,
we have $p, p'\in {\rm Im}\, \gamma$. 
\end{lemma}
\begin{proof}
If $Q$ is a piece of $Y$, then a projection
$Y\to Q$ is defined
in~\cite[Definition 2.7]{dru}. By~\cite[Lemma 2.6]{dru}
the piece $P'$ is connected. Since $|P'\cap P|\leq 1$,  
by~\cite[Corollary 2.11]{dru}
the projection of $P'$ onto $P$ consists of a single point $p$. In the same
way, the projection of $P$ onto $P'$ consists of a single point $p'$.
Now the conclusion follows from~\cite[Corollary 2.11]{dru}.
\end{proof}

\begin{lemma}[Lemma 2.15 in \cite{dru}]\label{newdrutu:lem}
Let $A$ be a path-connected subset of $Y$ without a cut-point. Then $A$ is contained in a piece.
\end{lemma}

The following theorem 
establishes an important bridge between 
the study of tree-graded spaces and the analysis of the asympotic
cones of universal coverings of irreducible graph manifolds.
It is a consequence of 
the fundamental work of Farb on relatively hyperbolic groups~\cite{farb},
and of the characterization of relative hyperbolicity provided in~\cite{dru}.

\begin{theorem}\label{neutered:relhyp}
Let $B$ be a neutered space obtained as the complement in $\mathbb{H}^n$ of an equivariant family
of pairwise disjoint open horoballs, and let $H$ be the collection of the boundary
components of $B$. Then $B$, endowed with its path metric, is asymptotically tree-graded with respect to $H$. Moreover, each piece of any asympotic cone of $B$ is isometric
to $\mR^{n-1}$.
\end{theorem}
\begin{proof}
 The main result of~\cite{farb} ensures that the fundamental group of a complete Riemannian manifold of finite volume with pinched negative curvature is relatively hyperbolic with respect to cusp subgroups. Moreover, by~\cite[Theorem 1.11]{dru}, a finitely generated
group is relatively hyperbolic (with respect to a family of subgroups) if and only
if it is asymptotically tree-graded (with respect to the corresponding family of 
left cosets
of subgroups). The first statement now follows from the fact that, if $B$ is the universal covering of $N$, then Milnor-Svarc Lemma implies that $B$ is quasi-isometric
to the fundamental group of $N$ via a quasi-isometry inducing a bijection
between the components
of $\partial B$ and the cosets of the cusp subgroups of $\pi_1(N)$. The second statement
follows from the fact that each component of $\partial B$ is isometric
to $\mR^{n-1}$.
\end{proof}

\section{The asymptotic cone of $\widetilde{M}$}

Let $M$ be an \emph{irreducible} graph manifold with universal covering $\tilM$. 
Let $\omega$ be any non-principal ultrafilter on $\mN$, let $(x_i)\subseteq \tilM$,
$(r_i)\subseteq \mR$ be fixed sequences of basepoints and rescaling factors,
and set $\tilM_\omega=(\tilM_\omega, (x_i),(r_i))$.

\begin{definition}
An $\omega$-\emph{chamber} (resp.~$\omega$-wall, $\omega$-fiber)
in $\tilM_\omega$ is a subset $X_\omega\subseteq \tilM_\omega$
of the form $X_\omega=\omega\text{-}\lim X^i$,
where each $X^i\subseteq \tilM$ is a chamber (resp.~a wall, a fiber). 
\end{definition}

We say that an $\omega$-wall $W_\omega=\omega{\textrm -}\lim W_i$ is a \emph{boundary}
(resp.~\emph{internal}) $\omega$-wall if $W_i$ is a boundary (resp.~internal)
wall $\omega$-a.e. The following lemma ensures that these notions are indeed
well-defined:

\begin{lemma}\label{unique:sequence:lem}
Let $C_\omega$ (resp.~$W_\omega$) be an $\omega$-chamber (resp.~an $\omega$-wall) of $\tilM_\omega$, and suppose
that $C_\omega=\omega{\textrm -}\lim C_i=\omega{\textrm -}\lim C'_i$ (resp.~$W_\omega=\omega{\textrm -}\lim W_i=\omega{\textrm -}\lim W'_i$). Then 
$C_i=C'_i$ $\omega$-a.e. (resp~$W_i=W'_i$ $\omega$-a.e.).
\end{lemma}
\begin{proof}
The conclusion follows from the fact that distinct chambers (resp.~walls) of $\tilM$ lie at infinite Hausdorff
distance one from the other (see Lemma~\ref{prefacile2:lem} and Corollary~\ref{prefacile3:lem}).
\end{proof}

In the next sections we will describe some analogies between the decomposition of a tree-graded space into its pieces
and the decomposition of $\tilM_\omega$ into its $\omega$-walls. We first observe that 
a constant $k$ exists such that each point of $\tilM$ has distance
at most $k$ from some wall, so every point of $\tilM_\omega$ lies in some $\omega$-wall.
Lemma~\ref{newdrutu:lem} implies that, in a tree-graded space, subspaces
homeomorphic to Euclidean spaces of dimension bigger than one are contained in pieces.
The main result of this section shows that a similar phenomenon occurs in our context: in fact,
in Proposition~\ref{wall:char:prop}  we prove that $\omega$-walls can be characterized as the only subspaces of
$\tilM_\omega$ which are bi-Lipschitz homeomorphic to $\mathbb{R}^{n-1}$. As a consequence, 
every bi-Lipschitz homeomorphism of $\tilM_\omega$ preserves the decomposition
of $\tilM_\omega$ into $\omega$-walls. Together
with an argument which allows us to recover quasi-isometries
of the original spaces from homeomorphisms of asymptotic cones, 
this will allow us to prove  
Theorem~\ref{qi-preserve:thm}. We will prove Proposition~\ref{wall:char:prop}
by contradiction: with some effort we will show that any bi-Lipschitz copy of $\mR^{n-1}$
in $\tilM_\omega$ which is not contained in an $\omega$-wall is disconnected
by a suitably chosen $\omega$-fiber. This will provide the required contradiction,
since $\omega$-fibers are too small to disconnect bi-Lipschitz copies of $\mR^{n-1}$ (see Lemma~\ref{fundamental:lem}).

The following lemma is a direct consequence of the description of the quasi-isometry type of walls and fibers given in Corollary~\ref{fibre:cor}.

\begin{lemma}\label{omegapezzi:lem}
There exists $k\geq 1$ such that
 every $\omega$-wall of $\tilM$ is $k$-bi-Lipschitz homeomorphic to
$\mR^{n-1}$, and
every $\omega$-fiber of $\tilM$ is $k$-bi-Lipschitz homeomorphic to
$\mR^{h}$, $h\leq n-3$.
\end{lemma}

Theorem~\ref{neutered:relhyp} and Corollary~\ref{fibre:cor} imply that every chamber is quasi-isometric to the product
of an asymptotically tree-graded space with a Euclidean fiber, and this implies
that every $\omega$-chamber is bi-Lipschitz homeomorphic to the product of a tree-graded space with a Euclidean factor: 

\begin{lemma}\label{omegapezzi2:lem}
There exists $k\geq 1$ such that
for any $\omega$-chamber $C_\omega$ there exists a $k$-bi-Lipschitz homeomorphism
$\varphi\colon C_\omega\to Y\times 
\mR^{l}$, where $Y$ is a tree-graded space whose pieces are $k$-bi-Lipschitz homeomorphic
to $\mR^{n-l-1}$, such that the following conditions hold:
\begin{enumerate}
 \item 
For every $p\in Y$, the subset
$\varphi^{-1}(\{p\}\times \mR^l)$ is an $\omega$-fiber of $\tilM_\omega$.
\item
For every piece $P$ of $Y$, the set
$\varphi^{-1}(P\times \mR^l)$ is an $\omega$-wall of $\tilM_\omega$. 
\end{enumerate}
\end{lemma}
\begin{proof}
 Suppose that $C_\omega=\omega$-$\lim C_i \subseteq \tilM_\omega$. For every $i$ we denote by $d_{C_i}$ the intrinsic distance on $C_i$, i.e.~the path distance induced on $C_i$
by the global distance $d$ of $\tilM$. Since $M$ is irreducible, there exist
constants $k\geq 1$, $\varepsilon\geq 0$ such that the identity of $C_i$
induces a $(k,\varepsilon)$-quasi-isometry between $(C_i,d_{C_i})$ and
$(C_i,d)$. Therefore, it is sufficient to prove the lemma in the case
when $C_\omega$ is replaced by the $\omega$-limit $C'_\omega$
of the rescaled spaces $(C_i,d_{C_i}/r_i)$ (with respect to any
choice of basepoints).

Recall that every $(C_i,d_{C_i})$ is isometrically identified
with the product $B_i\times \mR^{l_i}$, where $B_i$ is a neutered space.
Since $M$ is obtained by gluing finitely many pieces, the pairs $(C_i,d_{C_i})$
fall into finitely many isometry classes of metric spaces. As a consequence,
$C'_\omega$ is isometric to an asympotic cone of a fixed $(C_{i},d_{C_{i}})$.
Moreover, since $(C_{i},d_{C_{i}})$ is the product of a
neutered space with a Euclidean space, 
$C'_\omega$ is isometric to 
a product $Y\times 
\mR^{l}$, where $Y$ is a tree-graded space whose pieces are 
isometric to $\mR^{n-l-1}$ (see Theorem~\ref{neutered:relhyp}). This proves the first part of the lemma. In order
to prove points (1) and (2) it is now sufficient to observe that 
for every $p_i\in B_i$ the set $\{p_i\}\times \mR^{l_i}$
is a fiber of $\tilM$, and for every boundary component $P_i$ of $B_i$ 
the set $P_i\times \mR^{l_i}$ is at finite Hausdorff distance (with respect to $d$)
from a wall of $\tilM$.
\end{proof}

\begin{definition}\label{fiber-of}
Let $C_\omega$ be an $\omega$-chamber. 
A \emph{fiber of $C_\omega$} is an $\omega$-fiber 
of $\tilM_\omega$ of the form described
in point (1) of Lemma~\ref{omegapezzi2:lem}.  
A \emph{wall of $C_\omega$} is an $\omega$-wall of $\tilM_\omega$ of the form described
in point (2) of Lemma~\ref{omegapezzi2:lem}.  
If $W_\omega$ is a wall of $C_\omega$, then we also
say that $C_\omega$ is \emph{adjacent} to $W_\omega$.
\end{definition}

It is not difficult to show that $\omega$-walls of $C_\omega$ are exactly the $\omega$-walls
of $\tilM_\omega$ which are contained in $C_\omega$ (however, this fact won't be used
later). On the contrary, an internal $\omega$-wall $W_\omega$ is adjacent to two 
$\omega$-chambers
$C_\omega,C'_\omega$ (see Lemma~\ref{twochambers}), and an $\omega$-fiber of $C'_\omega$ contained in $W_\omega$ is not
an $\omega$-fiber of $C_\omega$ in general (see Lemma~\ref{fiber:intersection:lem} below). Therefore, not every $\omega$-fiber
contained in an $\omega$-chamber $C_\omega$ is an $\omega$-fiber of $C_\omega$ in the sense
of Definition~\ref{fiber-of}.

\begin{lemma}\label{twochambers}
 Let $W_\omega$ be an $\omega$-wall. 
\begin{enumerate}
 \item 
If $W_\omega$ is internal, then it is adjacent to exactly two $\omega$-chambers. 
\item 
If $W_\omega$ is boundary, then it is adjacent exactly
to one $\omega$-chamber.
\end{enumerate}
\end{lemma}
\begin{proof}
Point (2) is obvious, so we may suppose that
$W_\omega$ is internal. Then $W_\omega=\omega$-$\lim W_i$,
where $W_i$ is an internal wall of $\tilM$ for every $i$. Let 
$C_i^+$, $C_i^-$ be the chambers adjacent to $W_i$, and let us set
$C_\omega^\pm=\omega$-$\lim C_i^\pm$. Of course each $C_\omega^\pm$ is adjacent
to $W_\omega$, and
by Lemma~\ref{unique:sequence:lem}
we have $C_\omega^+\neq C_\omega^-$. 
Finally, if $C'_\omega=\omega$-$\lim C'_i$ is any $\omega$-chamber
adjacent to $W_\omega$, then we have $C'_i=C^+_i$ $\omega$-a.e., or
 $C'_i=C^-_i$ $\omega$-a.e., so either $C'_\omega=C^+_\omega$
or $C'_\omega=C^-_\omega$. 
\end{proof}

\begin{definition}\label{side}
 Let $W_\omega=\omega{\textrm -}\lim W_i$ be an $\omega$-wall. A \emph{side} $S(W_\omega)$
of $W_\omega$ is a subset $S(W_\omega)\subseteq \tilM_\omega$ which is defined as follows.
For every $i$, let $\Omega_i$ be a connected component of $\tilM\setminus W_i$.
Then
$$
S(W_\omega)=\left(\omega{\textrm -}\lim \Omega_i\right)\setminus W_\omega\ .
$$
\end{definition}

The proof of the following easy lemma is left to the reader (points (1) and (2)
may be proved by the very same argument exploited for Lemma~\ref{twochambers}).

\begin{lemma}\label{sides:lem}
 Let $W_\omega$ be an $\omega$-wall. Then:
\begin{enumerate}
 \item If $W_\omega$ is internal, then $W_\omega$ has exactly
two sides $S(W_\omega)$, $S'(W_\omega)$. Moreover, $S(W_\omega)\cap S'(W_\omega)=\emptyset$,
$S(W_\omega)\cup S'(W_\omega)=\tilM_\omega\setminus W_\omega$, and
any Lipschitz path joining points contained in distinct
sides of $W_\omega$ must pass through $W_\omega$.
\item
If $W_\omega$ is boundary, then $W_\omega$ has only one side $S(W_\omega)$, and
$S(W_\omega)=\tilM_\omega\setminus W_\omega$.
\item
If $W'_\omega\neq W_\omega$ is an $\omega$-wall, then $W'_\omega\setminus W_\omega$
is contained in one side of $W_\omega$. 
\item 
If $C_\omega$ is an $\omega$-chamber, then 
$C_\omega\setminus W_\omega$ is contained in one side of $W_\omega$. Moreover, every side
of $W_\omega$ intersects exactly one $\omega$-chamber which is adjacent to $W_\omega$.
\end{enumerate}
\end{lemma}

We have already mentioned the fact that an internal $\omega$-wall
admits two fibrations by $\omega$-fibers which are in general different one from the other. 

\begin{definition}
Let $S(W_\omega)$ be a side of $W_\omega$, and let $C_\omega$
be the unique $\omega$-chamber of $\tilM_\omega$ which intersects $S(W_\omega)$
and is adjacent to $W_\omega$. 
 A fiber of $W_\omega$ associated
to $S(W_\omega)$ is a fiber of $C_\omega$ (in the sense of Definition~\ref{fiber-of})
that is contained in $W_\omega$. 
\end{definition}

\begin{lemma}\label{fiber:intersection:lem}
Let $S^+(W_\omega)$ and $S^-(W_\omega)$ be the sides of 
the internal $\omega$-wall $W_\omega$,
and let $F^+_\omega$, $F^-_\omega$ be  fibers of $W_\omega$ associated respectively to
$S^+(W_\omega)$, $S^-(W_\omega)$. Then $|F^+_\omega\cap F^-_\omega|\leq 1$.
\end{lemma}
\begin{proof}
Let $W_\omega=\omega{\textrm -}\lim W_i$, let $C_i^+$ and $C_i^-$ be the chambers adjacent to $W_i$,
and set $W_i^\pm=W_i\cap C_i^\pm$.
If $X$ is a subspace of $\tilM$, then we denote by
$d_X$ the intrinsic distance on $X$, i.e.~the path distance induced
on $X$ by the global distance $d$ of $\tilM$. Recall that $(C_i^\pm,d_{C_i^\pm})$
is isometric to a product $B_i^\pm\times \mR^{l_i}$, where $B_i^\pm$
is a neutered space, and denote by $d_{B_i}^\pm$ the pseudo-distance on $C_i^\pm$
obtained by composing the path-distance of $B_i^\pm$ with the projection $C_i^\pm\to B_i^\pm$.

Let $k>0$ be a constant chosen so that Lemma~\ref{strong:lem} holds for all the internal walls in $\tilM$.
Let $p=[(p_i)]$, $q=[(q_i)]$ be distinct points in $F^+_\omega$. 
We may assume that $p_i,q_i\in W_i^+$ 
for every $i$. We also denote by $p_i^-$ and $q_i^-$ the points
of $W_i^-$ tied respectively to $p_i$ and $q_i$. Since $p,q\in F^+_\omega$
we have $(d_{B_i^+}(p_i,q_i))=o(r_i)$, while $p\neq q$ implies
that $(d_{C_i^+}(p_i,q_i))$ is not
$o(r_i)$. Therefore, Lemma~\ref{strong:lem} implies that 
\begin{equation}\label{eqeq}
d_{C_i^-}(p_i^-,q_i^-)\leq k\cdot d_{B_i^-}(p_i^-,q_i^-)\qquad  \omega{\textrm -a.e.}
\end{equation}
But tied points lie at a universally bounded distance
one from the other, 
so $p=[(p_i^-)]$, $q=[(q_i^-)]$ and $(d_{C_i^-}(p_i^-,q_i^-))$ is not $o(r_i)$.
Therefore, by Equation~\eqref{eqeq} 
also the sequence $(d_{B_i^-}(p_i^-,q_i^-))$ is not $o(r_i)$, and
the set $\{p,q\}$ cannot be contained in 
$F^-_\omega$.
\end{proof}

\begin{lemma}\label{wall-chambers}
 Let $C_\omega$ be an $\omega$-chamber and let $p\in \tilM_\omega\setminus C_\omega$.
Then there exists an $\omega$-wall $W_\omega$ of $C_\omega$ such that
 every Lipschitz path in $\tilM$ joining $p$ with a point in $C_\omega$  intersects
$W_\omega$.
Moreover, $W_\omega$ is internal and $C_\omega\setminus W_\omega$ and $p$ lie on different sides of $W_\omega$.
\end{lemma}
\begin{proof}
 Let $C_\omega=\omega{\textrm -}\lim C_i$ and $p=[(p_i)]$. 
Since $p\notin C_\omega$, we may suppose that $p_i\notin C_i$ for every $i$.
Therefore, for every $i$ there exists a wall $W_i$ adjacent to $C_i$ such that
every continuous path joining $p_i$ to $C_i$ has to intersect $W_i$. 
We claim
that the $\omega$-wall $W_\omega=\omega{\textrm -}\lim W_i$ satisfies the properties stated in the lemma. 

Let $\gamma_\omega$ be a Lipschitz path in $\tilM$ joining $p$ with a point in $C_\omega$.
By Lemma~\ref{lift2:lem} we may choose
a sequence $\gamma_i\colon [0,r_i]\to \tilM$
of Lipschitz paths such that $\gamma_i(0)=p_i$, $\gamma_i(r_i)\in C_i$
and $\omega{\textrm -}\lim \gamma_i=\gamma_\infty$. Our choices imply that the image
of $\gamma_i$ intersects $W_i$ for every $i$, so by Lemma~\ref{quasigeod1:lem}
the image of $\gamma_\omega$ intersects $W_\omega$. This proves the first statement. 
The second statement is an immediate consequence of the description of $W_\omega$.
\end{proof}

We have seen in Lemma~\ref{omegapezzi2:lem} that $\omega$-chambers are products
of tree-graded spaces with Euclidean factors, so
the following results are immediate consequences of the results about tree-graded spaces described in
Lemma~\ref{treegr:lem}.

\begin{lemma}\label{omegapezzi2_bis:lem}
Let $W_\omega$ and $W'_\omega$ be distinct $\omega$-walls of the $\omega$-chamber
$C_\omega$. Then there exists an $\omega$-fiber $F_\omega\subseteq W_\omega$ 
of $C_\omega$ such that
every continuous path in $C_\omega$ joining a point in $W_\omega$ with a point
in $W'_\omega$ has to pass through $F_\omega$.
\end{lemma}

\begin{corollary}\label{omegapezzi2_tris:lem}
 Let $W_\omega$ and $W'_\omega$ be distinct $\omega$-walls of the $\omega$-chamber
$C_\omega$, and let $\gamma\colon [0,1]\to C_\omega$ be a continuous path such that
$\gamma(0)\in W_\omega$, $\gamma(1)\in W'_\omega$ and $W_\omega\cap {\rm Im}\,\gamma=\{\gamma (0)\}$. If $F_\omega$ is the fiber of $C_\omega$
containing $\gamma(0)$, then every continuous path joining a point in $W_\omega$
with a point in $W'_\omega$ intersects $F_\omega$.
\end{corollary}

The following result 
extends Lemma~\ref{omegapezzi2_bis:lem}
to pairs of $\omega$-walls
which are not contained in the same $\omega$-chamber.

\begin{lemma}\label{inters:lem}
Let $W_\omega,W'_\omega$ be distinct $\omega$-walls, and
let
$S(W_\omega)$ be the side of $W_\omega$ containing $W'_\omega\setminus W_\omega$.
Then
there exists an $\omega$-fiber $F_\omega$ of $W_\omega$ such that
\begin{enumerate}
 \item $F_\omega$ is associated to $S(W_\omega)$, and
 \item every Lipschitz path joining a point in $W'_\omega$ with
a point in $W_\omega$ passes through $F_\omega$. 
\end{enumerate}
\end{lemma}

\begin{proof}
Let $\gamma\colon [0,1]\to \tilM_\omega$ be a Lipschitz path with $\gamma(0)\in W_\omega$,
$\gamma(1)\in W'_\omega$, and let $W_i,W'_i\subseteq \tilM$, $i\in\mathbb{N}$, be walls such that $\omega$-$\lim W_i=W_\omega$,
$\omega$-$\lim W'_i=W'_\omega$. Since $W_\omega\neq W'_\omega$,  we may suppose
$W_i\neq W'_i$ for every $i\in\mathbb{N}$.

Let us take $\varepsilon>0$. By Lemma~\ref{lift2:lem}, $\gamma=\omega$-$\lim \gamma_i$
where $\gamma_i\colon [0,a_i]\to \tilM$ is a $(k+\epsilon)$-Lipschitz path and $(a_i)$ is $O(r_i)$. Of course
(see the proof of Lemma~\ref{lift2:lem}) we may suppose $\gamma_i(0)\in W_i$, $\gamma_i (a_i)\in W'_i$
$\omega$-a.e. 

For every $i\in\mathbb{N}$, 
let us define a wall $L_i$ and a chamber $C_i$ as follows: 
if both $W_i$ and $W'_i$ are adjacent to the same chamber, then
$L_i=W'_i$ and $C_i$ is the chamber adjacent both to $W_i$ and to $L_i$;
if $W_i,W'_i$ do not intersect the same chamber, then $L_i\neq W_i$ and $C_i$ are such that
$W_i\cap C_i\neq\emptyset$, $L_i\cap C_i\neq \emptyset$,
and every path connecting $W_i$ and $W'_i$ intersects $L_i$
(the existence of such $L_i$, $C_i$ is an obvious consequence of the realization of
$\tilM$ as a tree of spaces). 
We would 
like to associate to $\gamma_i$ a path $\alpha_i$ joining $W_i$ with $L_i$ which does not intersect
any chamber different from $C_i$. This can be done in the following way. 
Let
$z_i$ be the last point of $\gamma_i$ which lies on $W_i$, let $p_i$ be the first point
of $\gamma_i$
which follows $z_i$ and lies on $L_i$ and 
call 
$\gamma'_i$ the subpath of $\gamma_i$ with endpoints $z_i, p_i$. 
We have that $\gamma'_i\cap C_i$ is a collection of paths in $C_i$, and, 
since $\gamma'_i$ is rectifiable and the distance between walls is bounded from below, 
only finitely many of them, say
$\delta^i_1,\ldots,\delta^i_m$, have endpoints in different walls. 
By concatenating the $\delta^i_j$'s with suitable
geodesics $\psi^i_j$ contained in the appropriate thin walls
we obtain the desired $\alpha_i$. 
By construction,  $\alpha_i$ intersects $W_i$ only in its initial point.
Also note that because thin walls
are quasi-isometrically embedded in $\tilM$ the length of $\alpha_i$ is 
uniformly linearly bounded by the length of $\gamma'_i$,
whence of $\gamma_i$. Therefore,
we can suppose that there exists $k'>0$ such that $\alpha_i$ is defined on the same
interval as $\gamma_i$, and $\alpha_i$ is $k'$-Lipschitz
$\omega$-a.e.

Now consider $C_\omega=\omega$-$\lim C_i$ and  $L_\omega=\omega$-$\lim L_i$.
We find ourselves in the context of Lemma~\ref{omegapezzi2_bis:lem}, which implies that there
exists an $\omega$-fiber $F_\omega\subseteq W_\omega$ with the property that every path 
joining $W_\omega$ and $L_\omega$ passes through $F_\omega$. 
Observe that $F_\omega$ satisfies property (1) of the statement by construction.
Now,
by Lemma~\ref{quasiasymp:lem}, $\alpha=\omega$-$\lim \alpha_i$ is a continuous path joining
$W_\omega$ and $L_\omega$, 
so $\alpha$ necessarily passes through $F_\omega$.
Then, in order to prove~(2)
it is sufficient to show
that $\gamma$ must also pass through $F_\omega$.

Choose the points $q_i\in {\rm Im}\, \alpha_i$ so that the corresponding
$q=[(q_i)]\in {\rm Im}\, \alpha$ is the first point along $\alpha$ which 
belongs to $F_\omega$. By the definition of $\gamma_i$ and $\psi^i_j$, 
at least one of the following possibilities must hold:
\begin{enumerate}[(i)]
\item
$q_i\in\gamma_i$ $\omega-$a.e.
\item
$q_i\in\psi^i_{j(i)}$ $\omega-$a.e. and $l_{B_i}(\psi^i_{j(i)}|_{q_i})=o(r_i)$, where 
$\psi^i_{j(i)}|_{q_i}$ denotes the initial subpath of $\psi^i_{j(i)}$ ending in $q_i$ 
and $l_{B_i}$ denotes the length of the projection of such a path on the base of $C_i$,
\item
$q_i\in\psi^i_{j(i)}$ $\omega-$a.e. and $\omega$-$\lim l_{B_i}(\psi^i_{j(i)}|_{q_i})/r_i > 0$.
\end{enumerate}
In cases (i) and (ii), it is clear that there is a point on $\gamma\cap F_\omega$. So 
let us now prove that case (iii) cannot occur. 
Indeed, the sequence 
of the starting points of the $\psi^i_{j(i)}$'s gives a point $q'\neq q$ 
which comes before $q$ along $\alpha$. 
Since $\alpha_i$ intersects $W_i$ only in its initial point, by Lemma~\ref{omegapezzi2_bis:lem}
the initial subpath of $\alpha$ ending in $q'$ 
joins a point on $W_\omega$ with a point on an 
$\omega-$wall $Q_\omega$ such that $Q_\omega \neq W_\omega$, 
and, by our hypothesis on $q$, it does not pass through the fiber $F_\omega$.
But the portion of $\omega$-$\lim \psi^i_{j(i)}$
between $q'$ and $q$ provides
a path starting on $Q_\omega$ and intersecting $W_\omega$ only in $q\in F_\omega$.
By Corollary~\ref{omegapezzi2_tris:lem}, this implies that \emph{every} continuous path joining
a point on $Q_\omega$ to a point in $W_\omega$ has to 
intersect $F_\omega$, a contradiction. This completes the proof of 
the Lemma.
\end{proof}

As every point in $\tilM_\omega$ is contained in an $\omega$-wall, we get the following.

\begin{corollary}\label{inters:cor}
Let $W_\omega$ be an $\omega$-wall, 
let $p\in\tilM_\omega\setminus W_\omega$, and let 
$S(W_\omega)$ be the side of $W_\omega$ containing 
$p$.
Then
there exists an $\omega$-fiber $F_\omega$ of $W_\omega$ associated to $S(W_\omega)$
such that 
every Lipschitz path joining $p$ with
$W_\omega$ passes through $F_\omega$. 
\end{corollary}

The fact that $\tilM$ is a tree of spaces suggests that the $\omega$-chambers
of $\tilM_\omega$ should be arranged in $\tilM_\omega$ following a sort of tree-like pattern.
We can formalize this fact as follows.

\begin{lemma}\label{omegatree}
Let $k>0$ and $\gamma_\omega\colon S^1\to \tilM_\omega$ be a Lipschitz path. 
Then there exists an $\omega$-chamber $C_\omega$ such that
$\gamma_\omega(p)\in C_\omega$ and $\gamma_\omega(q)\in C_\omega$, where
$p,q$ are distinct points of $S^1$.
\end{lemma}
\begin{proof}
For every $n$ we set $Y_i=r_i\cdot S^1$, and we denote by
$\gamma_i\colon Y_i\to \tilM$ the $(k+\varepsilon)$-Lipschitz path approximating
$\gamma_\omega$ in the sense of Lemma~\ref{lift2:lem}. 
We now make the following:

\vskip 10pt

\noindent {\bf Assertion:} 
There exists $H>0$ such that,
for every $n$, there exist points $p_i$ and $q_i$ in $Y_i$ and a chamber
$C_i\subseteq \tilM$ such that
$d(p_i,q_i)\geq (2\pi r_i)/3$ and $d(\gamma_i(p_i),C_i)\leq H$.

\vskip 5pt

\noindent Let $\Gamma$ be the Bass-Serre tree corresponding to the decomposition
of $\tilM$ as a tree of spaces, let $\pi\colon \tilM\to\Gamma$ be the canonical
projection, and let $\gamma_i'\colon Y_i\to \Gamma$ be defined by $\gamma_i'=\pi\circ\gamma_i$. Since the distance
between a point in a wall and any chamber adjacent to the wall is bounded
from above by a universal constant $H$, it is sufficient to show that
there exist points $p_i,q_i\in Y_i$ such that $d(p_i,q_i)\geq (2\pi r_i)/3$
and $\gamma'_i(p_i)=\gamma'(q_i)$. 

Pick three points $a^1_i,a^2_i,a^3_i$ on $Y_i$ such that $d(a^j_i,a^l_i)=(2\pi r_i)/3$
for $j\neq l$, and let $Y_i^1,Y_i^2,Y_i^3$ be the subarcs
of $Y_i$ with endpoints $a^1_i,a^2_i,a^3_i$.
We may suppose that  $\gamma'(a^j_i)\neq \gamma'(a^k_i)$ for $j\neq k$, 
otherwise we are done. Since $\Gamma$ is a tree,
there exists a point $v_i\in \Gamma$ such that, if $j\neq l$, then any path joining $\gamma'(a^j_i)$ with 
$\gamma'(a^l_i)$ passes through $v_i$. Therefore, every $Y_i^j$, $j=1,2,3$,
contains a point which is taken by $\gamma_i'$ onto $v_i$. As a consequence, 
the preimage $(\gamma'_i)^{-1}(v_i)$ contains two points $p_i,q_i$ with the desired properties, and the assertion is proved.

\vskip 10pt

Let now $C_\omega=\omega{\textrm -}\lim C_i$, and set $p=[(p_i)]\in S^1$, $q=[(q_i)]\in S^1$.
By construction we have $d(p,q)\geq (2\pi)/3$, and $\gamma_\omega (p)\in C_\omega$,
$\gamma_\omega(q)\in C_\omega$, whence the conclusion.
\end{proof}

\section{A characterization of bi-Lipschitz $(n-1)$-flats in $\tilM_\omega$}
A \emph{bi-Lipschitz $m$-flat} in $\tilM_\omega$ is the image of a 
bi-Lipschitz embedding $f\colon \mR^m\to \tilM_\omega$.
This section is aimed at proving that $\omega$-walls are the only
bi-Lipschitz $(n-1)$-flats in $\tilM_\omega$. 

We say that a metric space is L.-p.-connected if
any two points in it may be joined by a Lipschitz path.

\begin{lemma}\label{fundamental:lem}
Let $A\subseteq \tilM_\omega$ be a bi-Lipschitz $(n-1)$-flat. Then for every fiber $F_\omega$
the set $A\setminus F_\omega$ is L.-p.-connected.
\end{lemma}
\begin{proof}
Let $f\colon \mR^{n-1}\to C_\omega$ be a bi-Lipschitz embedding
such that $f(\mR^{n-1})=A$, 
and let $l\leq n-3$ be such that $F_\omega$ is bi-Lipschitz homeomorphic
to $\mR^l$.
The set $f^{-1}(F_\omega)$ is a closed subset
of $\mR^{n-1}$ which is bi-Lipschitz homeomorphic to a subset
of $\mR^l$. But it is known that the complements
of two homeomorphic closed subsets of $\mR^{n-1}$ have the same singular homology
(see e.g.~\cite{dold}), so $\mR^{n-1}\setminus f^{-1}(F_\omega)$ is
path-connected. It is immediate to check that any two points in a connected
open subset of $\mR^{n-1}$ are joined by a piecewise linear path,
so $\mR^{n-1}\setminus f^{-1}(F_\omega)$ is L.-p.-connected.
The conclusion follows from the fact that $f$ takes Lipschitz paths
into Lipschitz paths.
\end{proof}

We can already characterize bi-Lipschitz $(n-1)$-flats which are contained in a single 
$\omega$-chamber.

\begin{proposition}\label{reduction}
Let $A$ be a bi-Lipschitz $(n-1)$-flat contained
in the $\omega$-chamber $C_\omega$.
Then $A$ is equal to a wall of $C_\omega$.
\end{proposition}
\begin{proof}
 Recall that $C_\omega$ is homeomorphic to a product $Y\times \mR^l$,
where $Y$ is a tree-graded space and $l\leq n-3$. We denote by $\pi\colon C_\omega\cong Y\times \mR^l\to Y$ the projection on the first factor, and we set $A'=\pi (A)$.

We show that $A'$ has no cut points. In fact, if $p\in A'$, then
Lemma~\ref{fundamental:lem}
implies  that
the set $B=A\setminus \pi^{-1}(p)$ is L.-p.-connected, whence connected,
so $A'\setminus \{p\}=\pi(B)$ is also connected.
Now Lemma~\ref{newdrutu:lem} implies that $A'$ is contained in a piece of the tree-graded space $Y$,
so $A$ is contained in an $\omega$-wall $W_\omega$
of $C_\omega$. 
Being the image of a bi-Lipschitz embedding of a complete space, the set $A$ is closed
in $W_\omega$. Moreover, $A$ is open in $W_\omega$ by invariance of domain, so we finally
get $A=W_\omega$.
\end{proof}

We are now left to show that any bi-Lipschitz $(n-1)$-flat is contained
in an $\omega$-chamber.

\begin{lemma}\label{no-two-sides}
 Let $A$ be a bi-Lipschitz $(n-1)$-flat in $\tilM_\omega$ and let $W_\omega$
be an $\omega$-wall. Then $A\setminus W_\omega$ is contained in one side
of $W_\omega$.
\end{lemma}
\begin{proof}
 Suppose by contradiction that $p,q\in A$ are points on opposite sides of $W_\omega$
(in particular, $p\notin W_\omega$ and $q\notin W_\omega$), 
and let 
$F_\omega$ be the fiber of $W_\omega$ such that every path joining $p$ with
$W_\omega$ passes through $F_\omega$ (see Corollary~\ref{inters:cor}).

Let $\gamma$
be a Lipschitz path in $A$ joining $p$ with $q$.
By Lemma~\ref{sides:lem}, $\gamma$ must intersect $W_\omega$, whence $F_\omega$.
Since neither $p$ nor $q$ are contained in $F_\omega$, this implies
that $A\setminus F_\omega$ is not L.-p.-connected, and contradicts
Lemma~\ref{fundamental:lem}.
\end{proof}

\begin{problem}
Let $A$ be a subset of $\tilM_\omega$ such that $A\setminus W_\omega$ lies
on a definite side of $W_\omega$ for every $\omega$-wall $W_\omega$. Is it true
that $A$ is contained in an $\omega$-chamber? 
\end{problem}

By Lemma~\ref{no-two-sides}
and Proposition~\ref{reduction}, an affirmative answer to the above question
would readily imply that every bi-Lipschitz $(n-1)$-flat in $\tilM_\omega$
is equal to an $\omega$-wall.

\begin{lemma}\label{2points}
Let $A\subseteq \tilM_\omega$ be a bi-Lipschitz $(n-1)$-flat.
There there exists an $\omega$-wall $W_\omega$ such that $|A\cap W_\omega|\geq 2$.
\end{lemma}
\begin{proof}
Since $n\geq 3$, there exists
an injective Lipschitz loop contained in $A$. By Lemma~\ref{omegatree}, two distinct
points of this loop are contained in the same $\omega$-chamber. Therefore, 
there exists an $\omega$-chamber $C_\omega$ containing two distinct points $p,q$ of $A$.

If $A\subseteq C_\omega$, then we are done by Proposition~\ref{reduction}.
So we may suppose that $A$ contains a point $r\notin C_\omega$. Since $A$ is bi-Lipschitz homeomorphic to $\mR^{n-1}$, we may choose two Lipschitz paths $\gamma_p\colon [0,1]\to\tilM_\omega$,
$\gamma_q\colon[0,1]\to \tilM_\omega$ such that the following conditions hold: 
\begin{enumerate}
 \item 
the images
of $\gamma_p$ and of $\gamma_q$ are contained in $A$; 
\item 
$\gamma_p(0)=\gamma_q(0)=r$,
$\gamma_p(1)=p$, $\gamma_q(1)=q$; 
\item 
the image of $\gamma_p$ intersects
the image of $\gamma_q$ only in $r$. 
\end{enumerate}
By Lemma~\ref{wall-chambers}, point (2) implies that there exists an $\omega$-wall $W_\omega$
of $C_\omega$ intersecting the image of $\gamma_p$ in a point $p'$
and the image of $\gamma_q$ in a point $q'$. Since $r\notin C_\omega$, point (3)
ensures that $p'\neq q'$, while point (1) implies that $p'$ and $q'$ are contained in $A$.
\end{proof}

\begin{proposition}\label{wall:char:prop}
Let $A$ be a bi-Lipschitz $(n-1)$-flat in $C_\omega$.
Then
$A$ is an $\omega$-wall.
\end{proposition}
\begin{proof}
By Proposition~\ref{reduction}, it is sufficient to show that
$A$ is contained in an $\omega$-chamber.

By Lemma~\ref{2points} we may find an $\omega$-wall $W'_\omega$ 
and distinct points $p_1,p_2$ in $W'_\omega$ such that
$\{p_1,p_2\}\subseteq A\cap W'_\omega$. If $A\subseteq W'_\omega$ we are done, otherwise
Lemma~\ref{no-two-sides} ensures that
$A\setminus W'_\omega$ lies on one side $S(W'_\omega)$ of $W'_\omega$.
Moreover, if $r$ is any point in $A\setminus W'_\omega$, then Corollary~\ref{inters:cor}
implies that there exists a fiber $F'_\omega$ of $W'_\omega$ 
associated to $S(W'_\omega)$ such that 
any Lipschitz path joining $r$ with $p_1$ or $p_2$ passes through $F'_\omega$.
If $p_i\notin F'_\omega$ for some $i$, there would not be any Lipschitz path in $A\setminus F'_\omega$
joining $r$ to $p_i$, and this would contradict Lemma~\ref{fundamental:lem}. Therefore
we have
\begin{equation}\label{contained1}
 \{p_1,p_2\}\subseteq F'_\omega\ .
\end{equation}

Let $C_\omega$ be the unique $\omega$-chamber
which is adjacent to
$W'_\omega$ and intersects $S(W'_\omega)$. Since $F'_\omega$ is associated
to $S(W'_\omega)$, the fiber
$F'_\omega$ is a fiber of $C_\omega$.
We will prove that $A$ is contained in $C_\omega$.

Suppose by contradiction that there exists $q\in A\setminus C_\omega$,
and let $W_\omega$ be the wall of $C_\omega$ such that every Lipschitz
path joining $q$ with $C_\omega$ passes through $W_\omega$ 
(see Lemma~\ref{wall-chambers}). 
By Lemma~\ref{wall-chambers} the point $q\in A$ and the subset $C_\omega\setminus W_\omega$
lie on opposite sides of $W_\omega$. However, recall from
Lemma~\ref{no-two-sides} that we cannot have points of $A$ on 
opposite sides of $W_\omega$, 
so $\{p_1,p_2\}\subseteq A\cap C_\omega\subseteq W_\omega$.
Therefore, the
the fiber $F'_\omega$ is not disjoint from $W_\omega$, and this implies
at once that $F'_\omega$ is contained not only in $W'_\omega$, but also in $W_\omega$. More precisely, $F'_\omega$ is a fiber of $W_\omega$ associated to the side 
of $S(W_\omega)$ of $W_\omega$ containing $C_\omega\setminus W_\omega$. 

We have already observed that $q$ belongs to the side
$\overline{S}(W_\omega)$ of $W_\omega$ opposite to $S(W_\omega)$.
Moreover, by Corollary~\ref{inters:cor}
there exists a fiber $F_\omega$ of $W_\omega$ associated to $\overline{S}(W_\omega)$
such that every Lipschitz path joining $q$ with $p_i$, $i=1,2$, passes through $F_\omega$.
But $\{p_1,p_2\}\subseteq F'_\omega$, and $F_\omega$, $F'_\omega$ are fibers
of $W_\omega$ associated to opposite sides of $W_\omega$. By Lemma~\ref{fiber:intersection:lem},
this implies that $p_i\notin F_\omega$ for at least one $i\in \{1,2\}$.
It follows that $q$ and $p_i$ cannot be joined by any Lipschitz path in
$A\setminus F_\omega$, and this contradicts Lemma~\ref{fundamental:lem}.
\end{proof}

\begin{remark}
 Let $m\geq n-1$ and let $A$ be an $m$-bi-Lipschitz flat in $\tilM_\omega$. The arguments developed in this section
show that $A$ is contained
in an $\omega$-wall, so $m=n-1$ and $A$ is in fact an $\omega$-wall.
Therefore, $\omega$-walls are exactly the bi-Lipschitz flats of
$\tilM_\omega$ of maximal dimension. 
\end{remark}

\section{A characterization of quasi-flats of maximal dimension in $\tilM$}
Our characterization of bi-Lipschitz $(n-1)$-flats 
in $\tilM_\omega$ yields the following result:

\begin{corollary}\label{wall:cor}
For each $k,c$, there exists $\beta\geq 0$ 
(only depending on $k,c$ and the geometry of $\tilM$)
such that the image of $\mR^{n-1}$ under a 
$(k,c)-$quasi-isometric embedding in 
$\widetilde{M}$ is contained in the $\beta-$neighborhood of a wall.
\end{corollary}

\begin{proof}
By contradiction,
take a sequence of $(k,c)$-quasi-isometric embeddings $f_m:\mR^{n-1}\to\tilM$ such that 
for each $m\in\mN$ and wall $W\subseteq \tilM$ we have $f_m (\mR^n)\nsubseteq N_m (W)$,
where $N_m (W)$ is the $m$-neighbourhood of $W$.
Fix a point $p\in\mR^{n-1}$. The $f_m$'s induce a bi-Lipschitz embedding 
$f$ from the asymptotic cone $\mR^{n-1} = \mR^{n-1}_\omega ((p),(m))$ 
to the asymptotic cone $\tilM_\omega (f_m (p),(m))$. 
(Recall that, if $X$ is a metric space, we denote by $X_\omega ((x_m), (r_m))$ the asymptotic 
cone of $X$ associated to the sequence of basepoints $(x_m)$ and the sequence of rescaling factors $(r_m)$.)
By the previous proposition, there is an $\omega-$wall $W_\omega=\omega$-$\lim W_m$ such that 
$f(\mR^{n-1})=W_\omega$. 
By hypothesis, for each $m$ there is a point $p_m\in\mR^{n-1}$ with $d(f_m (p_m),W_m)\geq m$.
Set $r_m = d(p_m,p)$. By choosing $p_m$ as close to $p$ as possible, 
we may assume that no point $q$ such that $d(p,q)\leq r_m-1$ satisfies $d(f_m (q),W_m)\geq m$, 
so 
\begin{equation}\label{cond1:eqn}
d(f_m (q),W_m)\leq m +k+c\qquad  {\rm for\ every}\ q\in\mR^{n-1}\ {\rm s.t.}\ d(p,q)\leq r_m.
\end{equation} 
Notice that $\omega$-$\lim r_m/m=\infty$,
for otherwise $[(p_m)]$ should belong to $\mR^{n-1}_\omega ((p),(m))$,
$[f_m (p_m)]$ should belong to $\tilM_\omega ((f_m(p)), (m))$, and, since
$f(\mR^{n-1})=W_\omega$, we would
have $d(f_m(p_m),W_m)=o(m)$. 

Let us now change basepoints, and consider instead the pair of asymptotic cones
$\mR^{n-1}_\omega( (p_m), (m))$ and $\tilM_\omega ((f_m(p_m)), (m) )$.
The sequence $(f_m)$ induces a bi-Lipschitz embedding $f'$ between these asymptotic cones
(note that $f\neq f'$, simply because due to the change of basepoints, 
$f$ and $f'$ are defined on different spaces
with values in different spaces!).
Let $A_m=\{q\in\mR^{n-1}\, |\, d(q,p)\leq r_m\}$ and $A_\omega=
\omega$-$\lim A_m\subseteq \mR^{n-1}_\omega ((p_m), (m))$.
Since $\omega$-$\lim r_m/m=\infty$, it is easy to see that $A_\omega$ is bi-Lipschitz
homeomorphic to a half-space in $\mR^{n-1}$. Moreover, 
by (\ref{cond1:eqn}) each point in $f'(A_\omega)$ is at a distance at most 1 from 
$W'_\omega=\omega$-$\lim W_i$ (as before, observe that the sets $W_\omega$ and $W'_\omega$
live in different spaces). 
Again by Proposition~\ref{wall:char:prop}
we have that $f'(A_\omega)\subseteq f'(\mR^{n-1}_\omega((p_m), (m)) )= W''_\omega$ 
for some $\omega-$wall $W''_\omega$. Moreover, since $[(f_m (p_m))]\in W''_\omega\setminus
W'_\omega$, we have $W'_\omega \neq W''_\omega$. 

By Lemma~\ref{inters:lem} there exists a fiber $F_\omega\subseteq W'_\omega\cap W''_\omega$ such that
every path joining a point in $W''_\omega$
with a point in $W'_\omega$ has to pass through $F_\omega$. Now, if $a\in f'(A_\omega)$  
we have $d(a, W'_\omega)\leq 1$, so there exists a geodesic of length at most one 
joining $a\in W''_\omega$ with some point in $W'_\omega$. Such a geodesic must pass
through $F_\omega$, so 
every point of $f'(A_\omega)$ must be at a distance at most 1 from $F_\omega$.
If $h\colon f'(A_\omega)\to F_\omega$ is such that $d(b,h(b))\leq 1$ for every $b\in f'(A_\omega)$,
then $h$ is a $(1,2)$-quasi-isometric embedding. Therefore the map
$g=h\circ f'\colon A_\omega\to F_\omega$ is a quasi-isometric embedding. But this is not possible,
since if $n-1>l$ there are 
no quasi-isometric embeddings from a half space in $\mR^{n-1}$ to $\mR^l$ (as, taking asymptotic cones, 
such an embedding would provide an injective continuous function 
from an open set in $\mR^{n-1}$ to $\mR^l$). This completes the proof of the corollary.
\end{proof}

\section{Walls and chambers are quasi-preserved by quasi-isometries}
We are now ready to conclude the proof of Theorem~\ref{qi-preserve:thm}.
We come back to our original situation, i.e.~we take irreducible graph $n$-manifolds
$M_1,M_2$ and we suppose that $f\colon \tilM_1\to\tilM_2$ is a given $(k,c)$-quasi-isometry. We will
say that a constant is \emph{universal} if it only depends on $k,c$ and on the geometry
of $M_1,M_2$. We begin by recalling the following well-known result (see \emph{e.g.}~\cite[Corollary 2.6]{kapleenew}):

\begin{lemma}\label{rn}
Let $f\colon \mR^{n-1}\to \mR^{n-1}$ be an $(a,b)$-quasi-isometric embedding.
Then $f$ is an $(a',b')$-quasi-isometry, where $a',b'$ only depend on $a,b$.
\end{lemma}

\begin{proposition}\label{wall2:prop}
A universal constant $\lambda$ exists such that for every wall 
$W_1\subseteq \tilM_1$, there exists a wall
$W_2\subseteq \tilM_2$ with the property that the Hausdorff distance between $f(W_1)$ and $W_2$  
is $\leq \lambda$. Moreover, $W_2$ is the unique wall in $\tilM_2$
at finite Hausdorff distance from $f(W_1)$.
\end{proposition}
\begin{proof}
Since $M_1$ is irreducible, there exists a $(k',c')$-quasi-isometry
$i\colon \mR^{n-1}\to W_1$ (where $k',c'$ only depend
on the geometry of $M_1$), and 
Corollary~\ref{wall:cor} (applied to the quasi-isometric embedding $f\circ i$) ensures that
$f(W_1)$ is contained in the $\beta$-neighbourhood of $W_2$
for some wall $W_2$, where $\beta$ is universal.
For every $y\in f(W_1)$ let $p(y)\in W_2$ be a point such that $d(y,p(y))\leq\beta$.
It follows easily from Lemma~\ref{rn} that the map
$p\circ f|_{W_1}\colon W_1\to W_2$ is a $(k'',c'')$-quasi-isometry,
where $k'',c''$ are universal. This in turn implies that $W_2$
is contained in the $\beta'$-neighbourhood of $f(W_1)$, where $\beta'$ is universal.
The first statement follows, with
$\lambda=\max \{\beta,\beta'\}$. The uniqueness of $W_2$ is an immediate consequence
of Lemma~\ref{prefacile2:lem}.
\end{proof}

Putting together Propositions~\ref{wall2:prop}
and~\ref{useful:prop} we now get the following result, which concludes the proof
of Theorem~\ref{qi-preserve:thm}:

\begin{proposition}\label{qi-useful:prop}
There exists a universal constant $H$ such that for every chamber $C_1\subseteq \tilM_1$ there exists a unique chamber
$C_2\subseteq \tilM_2$ such that the Hausdorff distance between $f(C_1)$ and $C_2$ 
is bounded by $H$. Moreover, if $W_1$ is a wall adjacent to $C_1$ then $f(W_1)$
lies at finite Hausdorff distance from a wall $W_2$ adjacent to $C_2$.
\end{proposition}

\section{Thickness and relative hyperbolicity}\label{thickness:sec}

For an irreducible graph manifold $M$, we may exploit the study of the coarse geometry of $\pi_1(M)$
to answer the question whether $\pi_1(M)$ is relatively hyperbolic with respect to some finite family of proper subgroups.
Recall from Proposition~\ref{exception} that (even when $M$ is not necessarily irreducible) $\pi_1(M)$ is relatively hyperbolic provided that
at least one piece of $M$ is purely hyperbolic. In this Section we show that this sufficient condition
is also necessary if $M$ is irreducible.

As mentioned in the Introduction,
there are several equivalent definitions of the notion of relative hyperbolicity
of $G$ with respect to $H_1,\ldots,H_n$. Since we are going to describe obstructions to relative hyperbolicity
coming from the study of asymptotic cones, we recall the characterization
of relative hyperbolicity provided by the following result:

\begin{theorem}[\cite{dru}]
Let $C(G)$ be any Cayley graph of $G$.
Then, the group $G$ is relatively hyperbolic with respect
to $H_1,\ldots, H_n$ if and only if $C(G)$ is asymptotically tree-graded (see Definition~\ref{as-tree:def})
with respect to the left cosets of $H_1,\ldots,H_n$
(considered as subsets of $C(G)$).
\end{theorem}


The notion of thickness was introduced by Behrstock, Dru\c{t}u and
Mosher in \cite{BDM} as an obstruction for a metric space to be
asymptotically tree-graded, and hence, for a group to be relatively
hyperbolic. The simplest such obstruction is being unconstricted, 
i.e.~having no cut-points in any asymptotic cone (by definition, a metric
space is thick of order 0 if it is unconstricted). 
It is readily seen that the product of two unbounded geodesic metric spaces
(e.g.~a graph manifold consisting of a single piece with non-trivial fiber and without internal walls)
is unconstricted. 
Notable thick metric
spaces and groups which are not unconstricted include the mapping class
group and Teichm\"uller space (equipped with the Weil-Petersson metric)
of most surfaces (see \cite{BDM}), fundamental groups of classical $3$-dimensional
graph manifolds (see again \cite{BDM}), and the group $\out(F_n)$
for $n\geq 3$ (see Algom-Kfir \cite{algom}).

\par
Let us briefly describe what it means for a metric space $X$ to be thick
of order at most 1 with respect to a collection of subsets
$\mathcal{L}$. First of all, the family
$\mathcal{L}$ is required to ``fill'' $X$, that is there must exist
a positive constant $\tau$
such that the union of the sets in $\mathcal{L}$ is $\tau-$dense in $X$ (property $(N1)$).
Secondly, a certain coarse connectivity property (denoted by $(N2)$) must be satisfied: for each
$L,L'\in\mathcal{L}$ we can find elements $L_0=L,L_1\dots,L_n=L'$ of $\mathcal{L}$
such that $N_\tau(L_i)\cap N_\tau(L_{i+1})$ has infinite diameter, where the constant
$\tau$ is independent of $L,L'$. The space $X$ is said to be a $\tau$-network
with respect to the family of subspaces $\mathcal {L}$ if conditions
$(N1)$ and $(N2)$ hold (with respect to the constant $\tau$). For $X$ to be thick
of order at most 1, we need $X$ to be a $\tau$-network with respect a 
family $\mathcal {L}$, where each $L\in\mathcal{L}$ is
unconstricted (actually the stricter condition that the family $\mathcal{L}$
is uniformly unconstricted is required to hold). 
\par
Notice that property $(N2)$ fails if $X$ is asymptotically tree-graded
with respect to $\mathcal{L}$ as in that case there are uniform bounds
on the diameter of $N_k(L)\cap N_k(L')$ for $L,L'\in \mathcal{L}$ with
$L\neq L'$.

 \begin{proposition}
Let $M$ be an irreducible graph manifold, with at least one internal wall, and with the property that all
pieces have non-trivial fibers. Then $\tilM$ and $\pi_1(M)$ are both thick of order $1$. 
\end{proposition}

\begin{proof}
Let us first argue that $\tilM$ is thick of order $\leq 1$. We show that $\tilM$ is a 
$\tau-$network with respect to the collection $\mathcal{H}$ of its chambers 
(for $\tau$ large enough). In fact, every point in $\tilM$ is clearly
uniformly close to a chamber (property $(N1)$). Furthermore, if $\tau$ is large enough, 
then the intersection of two adjacent chambers contains a wall. 
As walls have infinite diameter, we easily obtain property $(N2)$ as well.

To complete the proof that $\tilM$ is thick of order $\leq 1$ we are only left with proving
that $\calH$ is uniformly unconstricted. This is true because there exists
a uniform constant $k\geq 1$ such that any $\omega-$chamber is 
$k$-biLipschitz homeomorphic to the product of a geodesic metric space and some $\mathbb{R}^n$, $n>0$.

Finally, we note that, by a result of Drutu, Mozes, and Sapir \cite[Theorem 4.1]{DMS},
any group which supports an acylindrical action on a tree has the property that every 
asymptotic cone has a cut point. In view of Proposition \ref{irr-acyl}, we conclude that $\pi_1(M)$ has
cut points in every asymptotic cones, hence cannot be thick of order $0$. This concludes
the proof of the Proposition.
\end{proof}

Therefore, if we assume that every piece of our irreducible graph manifold $M$ 
has non-trivial torus factor, then $\pi_1 (M)$ is either thick of order $0$
(when $M$ consists of a single piece without internal walls), or thick of order $1$
(when $M$ has at least one internal wall). Therefore, from Proposition~\ref{exception} and~\cite[Corollary 7.9]{BDM}
we deduce the following:

\begin{proposition}\label{thick:prop}
Let $M$ be an irreducible graph manifold. Then $\pi_1(M)$ is relatively hyperbolic with respect to a finite
family of proper subgroups if and only if $M$ contains at least one purely hyperbolic piece.
\end{proposition}

\chapter{Quasi isometry rigidity, I}\label{product:sec}
This chapter is devoted to the proof of Theorem~\ref{product:thm}.
We recall the statement for the convenience of the reader:

\begin{Thm2}
Let $N$ be a complete finite-volume hyperbolic $m$-manifold, $m\geq 3$, and
let $\Gamma$ be a finitely generated group quasi-isometric to
$\pi_1 (N)\times\mZ^d$, $d\geq 0$. 
Then there exist a finite-index subgroup $\Gamma'$ of $\Gamma$,
a finite-sheeted covering $N'$ of $N$, a group $\Delta$  and a finite group $F$ 
such that the following short exact sequences hold:
$$
\xymatrix{
1\ar[r] &\mZ^d \ar[r]^j & \Gamma' \ar[r] & \Delta \ar[r] & 1,\\
}
$$
$$
\xymatrix{
1\ar[r] & F \ar[r] & \Delta\ar[r] & \pi_1 (N')\ar[r] & 1 .
}
$$
Moreover,
$j(\mZ^d)$ is contained in the center of $\Gamma'$.
In other words, $\Gamma'$ is a central extension by $\mZ^d$ 
of a finite extension of $\pi_1 (N')$.
\end{Thm2}

In what follows we will give a proof of Theorem~\ref{product:thm} under the additional
assumption that the cusps of $N$ are toric. However, the attentive reader will observe that
all the results needed in the proofs below also hold in the case where $N$ is not assumed to have toric cusps,
provided that the walls of the universal covering of $\overline{N}\times T^d$ are quasi-isometrically embedded
in the universal covering $B\times \mR^d$, where $B$ is the neutered space covering $\overline N$. 
But this last fact is obvious, since the boundary components
of  $B\times \mR^d$ are totally geodesic (in the metric sense). 

So, let us consider the graph manifold with boundary 
$M=\overline{N}\times T^d$, and  observe that $\Gamma$ is 
quasi-isometric to $\pi_1 (M)$. Moreover, $M$ is obviously irreducible, and
the universal covering $\tilM$ of $M$ is isometric to the Riemannian
product $B\times \mR^d$, where $B$ is a neutered space. The walls of $\tilM$ coincide
with the boundary components of $\tilM$.

\section{The quasi-action of $\Gamma$ on $\tilM$}
As discussed  
in Section~\ref{quasiact:sub}, a quasi-isometry between $\Gamma$ and $\pi_1 (M)$ induces a $k$-cobounded
$k$-quasi-action $h$ of $\Gamma$ on $\tilM$ for some $k\geq 1$. From this point on, we will fix such a quasi-action. 
Henceforth, for every $\gamma\in\Gamma$, we will abuse notation, and also denote by $\gamma$ the corresponding
quasi-isometry $h(\gamma)\colon \tilM\to\tilM$.

We want to prove that every quasi-isometry $\gamma\colon \tilM\to \tilM$, $\gamma\in \Gamma$ 
can be coarsely projected on $B$ to obtain a quasi-isometry of $B$. 
We say that a constant is universal if it depends only on $k, H$ and the geometry of
$B$, where $H$ is such that for every $\gamma\in \Gamma$ and every wall
$W\subseteq \tilM$, the set $\gamma (W)$ is at Hausdorff distance at most $H$ from
a wall of $\tilM$ (see Proposition~\ref{wall2:prop}).

\begin{lemma}
There exists a universal constant $H'$ such that, for each fiber 
$F=\{b\}\times\mR^d\subseteq \tilM$ and each $\gamma\in \Gamma$, the set
$\gamma(F)$ is at Hausdorff distance bounded by $H'$ from a fiber
$\overline F=\{\overline b\}\times \mR^d\subseteq \tilM$. 
\end{lemma}

\begin{proof}
Let $K\subseteq \overline N$ be the cut-locus of $\overline{N}$ relative to
$\partial\overline N$, i.e.~the set of points of $\overline{N}$ whose distance from $\partial\overline{N}$
is realized by at least two distinct geodesics, and let $R'=2\sup \{d_{\overline{N}} (p,q)\, |\, p\in K,\, q\in\partial
\overline N\}$. Since $\overline N$ is compact, $R'$ is finite, and it is easily seen
that for each $p\in \overline N$ there exist (at least) two distinct components
of $\partial\overline N$ whose distance from $p$ is at most $R'$.  
This implies that
for each fiber $F$ there exist two walls $W, W'$ such that 
$F\subseteq A_{R'}(W,W')=\{x\in \tilM\,|\, d(x,W)\leq R', d(x, W')\leq R'\}$.

Moreover, if $O,O'$ are disjoint horospheres in $\partial B$, it is easy to see
that the diameter of the set $\{b\in B\, |\, d(b,O)\leq R', d(b,O')\leq R'\}$ is bounded
by a constant which only depends on $R'$.
As a consequence, if $F$ is a fiber contained in $A_{R'}(W,W')$ then
there exists a universal constant $D$ 
such that $A_{R'}(W,W')\subseteq N_D(F)$.
As quasi-isometries almost preserve walls, there exist a universal constant $R''\geq R '$
and walls 
$\overline{W},\overline{W'}$ such that $\gamma(A_{R'}(W,W'))\subseteq  A_{R''}(\overline{W},\overline{W'})$.
It follows that $\gamma$ restricts to a $(k', k')$-quasi-isometric embedding of $F$
into $A_{R''}(\overline{W},\overline{W'})$, where $k'$ is a universal constant.
But both $F$ and $A_{R''}(\overline{W},\overline{W'})$ are quasi-isometric to $\mR^d$,
so by Lemma~\ref{rn} the restriction of $\gamma$ to $F$ defines a quasi-isometry
(with universal constants) between $F$ and $A_{R''}(\overline{W},\overline{W'})$, and this
forces the Hausdorff distance between $\gamma (F)$ and
a fiber in $A_{R''}(\overline{W},\overline{W'})$ 
to be bounded by a universal $H'$.
\end{proof}

The above Lemma can be used to define a quasi-action of $\Gamma$ on $B$. 
Recall that $\tilM$ is isometric to $B\times \mR^d$, 
and fix $\gamma\in \Gamma$. 
We define a map $\psi (\gamma)\colon B\to B$ by setting 
$\psi(\gamma)(b)=\pi_B (\gamma( (b,0) ))$ for every $b\in B$, where $\pi_B\colon \tilM\cong B\times \mR^d \to B$
is the natural projection, and for $(b,f), (b',f')\in B\times \mR^d \cong \tilM$ 
we denote by $d_B ((b,f),(b',f'))$ the distance in $B$ between $b$ and $b'$
(see Section~\ref{construction:sec}). With a slight abuse of notation, we also 
denote by $d_B$ the distance on $B$.

We now show that every $\psi (\gamma)$ is a quasi-isometry (with universal constants).
Let $b,b'\in B$ and set $F=\{b\}\times \mR^d$ and $F'=\{b'\}\times \mR^d$. 
The Hausdorff distance between $\gamma (F)$ and $\gamma (F')$ is bounded
from below by $d_B (b,b')/k -k$, so 
if $\overline{F},\overline{F'}$ are fibers with Hausdorff distance bounded
by $H'$ from $\gamma(F),\gamma(F')$ respectively, then the Hausdorff distance between
$\overline{F}$ and $\overline{F'}$ is at least $d_B (b,b')/k -k -2H'$. 
We have therefore
\begin{align*}
d_B \big(\psi(\gamma)(b),&\psi(\gamma)(b') \big) \\ 
& \geq d_B \big(\pi_B (\overline{F}),\pi_B (\overline{F'}) \big)- d_B \big(\psi(\gamma)(b),\pi_B (\overline{F}) \big)-
d_B \big(\psi(\gamma)(b'), \pi_B (\overline{F'}) \big) \\ 
& \geq \big(d_B (b,b')/k -k -2H'\big)-2H' \\
& = d_B (b,b')/k -k -4H'.
\end{align*}
On the other hand, we also have
\begin{align*}
d_B \big(\psi(\gamma)(b),\psi(\gamma)(b') \big) &= d_B \Big(\gamma \big((b,0) \big),\gamma \big((b',0) \big)\Big)\\
&\leq kd \big((b,0),(b',0) \big)+k\\
&\leq k d_B (b,b')+k.
\end{align*}
Having $(k+2H')$-dense image, the map
$\psi(\gamma)\colon B\to B$ is 
therefore
a $(k',k')$-quasi-isometry with $k'$-dense image, where $k'$ is a universal constant. 
It is now easy to show that the map $\gamma\mapsto \psi (\gamma)$ defines a 
quasi-action of $\Gamma$ on $B$. Moreover, up to increasing $k'$ we may assume that
such a quasi-action is $k'$-cobounded.
From the way the action of $\Gamma$ on
$B$ was defined, we also have that, for every $\gamma\in\Gamma$ and every component $O$ of $\partial B$,
there exists a component $O'$ of $\partial B$ such that
the Hausdorff distance between $\psi(\gamma) (O)$ and $O'$ is 
bounded by $H$. In order to simplify notations, we will as usual denote $\psi (\gamma)$ simply by $\gamma$.

Recall that $m=n-d$ is the dimension of the neutered space $B$, and let $G$
be the isometry group of $(B,d_B)$. Every element of $G$ is the restriction to $B$ of an isometry
of the whole hyperbolic space $\matH^m$ containing $B$.
We will denote by ${\rm Comm} (G)$ the \emph{commensurator} of 
$G$ in ${\rm Isom} (\matH^m)$, i.e.~the group
of those elements $h\in {\rm Isom} (\matH^m)$ such that
the intersection $G\cap (h G h^{-1})$
has finite index both in $G$ and in $h G h^{-1}$.

We are now in a position to use a deep result due to Schwartz
(see~\cite[Lemma 6.1]{schw}), which in our
context can be stated as follows:

\begin{theorem}[\cite{schw}]\label{schwartz1}
There exists a universal constant $\beta$ such that the following condition holds:
for every $\gamma\in\Gamma$ a unique isometry
$\theta (\gamma)\in {\rm Isom}(\matH^m)$ exists such that $d_\matH (\gamma(x),\theta(\gamma)(x))\leq \beta$
for every $x\in B$, where $d_\matH$ denotes the hyperbolic distance on
$\matH^m$. Moreover, for every $\gamma\in\Gamma$ the isometry
$\theta (\gamma)$ belongs to ${\rm Comm} (G)$, and 
the resulting map $\theta\colon \Gamma\to {\rm Comm}(G)$ is
a group homomorphism.
\end{theorem} 

In the next few sections, we will analyze the kernel and image of the morphism
$\theta$, in order to extract information about the structure of $\Gamma$.

\section{The image of $\theta$}

From now on we denote by $\Lambda< {\rm Isom}(\matH^m)$ the image of the
homomorphism $\theta$. Our next goal is to show that $\Lambda$ is 
commensurable with $\pi_1 (N)$. 
It is a result of Margulis that a non-uniform lattice in ${\rm Isom} (\matH^m)$ is arithmetic
if and only if it has infinite index in its commensurator (see~\cite{zim}). As a result,
things would be quite a bit easier if $N$ were assumed to be non-arithmetic. 
To deal with the general case, we will again use results (and techniques) from~\cite{schw}.
Note that, at this stage, we don't even know that $\Lambda$ is a discrete subgroup of 
${\rm Isom}(\matH^m)$. 

From now on, unless otherwise stated,
we will consider the Hausdorff distance of subsets of $\matH^m$ 
with respect to the hyperbolic metric $d_\matH$ on $\matH^m$. 
We denote by $P\subseteq \partial\matH^m$
the set of all the basepoints of horospheres in $\partial B$. 
As an immediate corollary of Theorem~\ref{schwartz1} we get the following:

\begin{lemma}\label{easyfinite:lem}
For every $\alpha\in\Lambda$ and every horosphere $O\subseteq \partial B$
there exists a unique horosphere $O'\subseteq \partial B$ such 
that the Hausdorff distance between $\alpha (O)$ and $O'$ is at most
the universal constant $H+\beta$. In particular, the group $\Lambda$ acts on $P$.
\end{lemma}

\begin{lemma}\label{finite-orbits:lem}
The action of $\Lambda$ on the set $P$ has a finite number of orbits, 
and every element of $\Lambda$ which fixes a point in $P$ is parabolic.
\end{lemma}

\begin{proof}
Fix a point $b\in B$. Let $A$ be the set of boundary components of $B$  
whose hyperbolic distance from $b$ is $\leq k'(H +k') + k'$. The set $A$ is finite,
and define $P_0$ to be the (finite) set of basepoints corresponding to the horospheres 
in the set $A$. We will prove that $P_0$ contains a set of representatives for the action of 
$\Lambda$ on $P$.

So taking an arbitrary $p\in P$, let
$O$ be the corresponding component of $\partial B$, and fix a point $y\in O$. 
Since the quasi-action of $\Gamma$ on $B$ is $k'$-cobounded,
there exists $\gamma\in\Gamma$ such that $d_\matH (\gamma(b),y)\leq d_B (\gamma (b),y)\leq k'$. 
We know that there exists a component $O'$ of $\partial B$ based at $p'\in P$ 
such that $\gamma(O')$ is at Hausdorff distance bounded by $H$ from $O$. 
It follows that $\gamma(O')$ contains a point at distance at most $H$ from $y$, 
and this in turn implies that $O'$ belongs to $A$, so
$p'$ belongs to $P_0$. 
Moreover, the horosphere $\theta (\gamma) (O')$ is at bounded Hausdorff distance from $O$,
giving us $ \theta(\gamma)(p')=p$. So $p$ belongs to the $\Lambda$-orbit of a point in $P_0$, 
completing the first part of the Lemma.

Now assume $p\in P$ is fixed by an element $\alpha\in\Lambda$, and let $O$ be the connected
component of $\partial B$ corresponding to $p$. 
Since $\alpha (p)=p$, the horosphere $\alpha (O)$ is also based at the point $p$. It easily follows
that the Hausdorff distance between $O$ and $\alpha^n (O)$ equals $n$ times 
the Hausdorff distance between $O$ and $\alpha (O)$. 
Since $\alpha^n\in\Lambda$ for every $n\in\mathbb{N}$, 
if such a distance were positive, then for sufficiently large $n$ the Hausdorff distance 
from $O$ to $\alpha^n (O)$ would exceed the uniform constant $H+\beta$, contradicting
Lemma~\ref{easyfinite:lem}. We conclude $\alpha (O)=O$, so $\alpha$ is parabolic.
\end{proof}

Now let $P_0=\{p_1,\ldots,p_j\}\subseteq P$ as in Lemma \ref{finite-orbits:lem} 
be a finite set of representatives for the action of $\Lambda$
on $P$. For every $i=1,\ldots,j$ let $O_i$ be the component of $\partial B$
based at $p_i$, and let $\widehat{O}_i$ be the horosphere contained in the horoball
bounded by $O_i$ and having Hausdorff distance $H+\beta$ from $O_i$. We let
$\widehat{\mathcal{O}}$ be the set of horospheres obtained by translating
$\widehat{O}_1,\ldots,\widehat{O}_j$ by all the elements of $\Lambda$, and we denote
by $\widehat{B}$ the complement in $\matH^m$ of the union of the horoballs bounded
by elements in $\widehat{\mathcal{O}}$. By construction the set
$\widehat{B}$ is $\Lambda$-invariant, and since all the stabilizers
of points in $P$ are parabolic, for every $p\in P$ there exists exactly one
horosphere in $\widehat{O}$ based at $p$. 
Let $R>0$ be the minimal distance between distinct connected components
of $\partial B$.
Take $\widehat{O}\in\widehat{\mathcal{O}}$ and let $O$ be the corresponding
boundary component of $B$. By definition there exist $i\in\{1,\ldots,j\}$
and an element $\alpha\in\Lambda$ such that 
$\widehat{O}=\alpha (\widehat{O}_i)$. Recall now that the Hausdorff distance 
between $\alpha (O_i)$ and $O$ is bounded by $H+\beta$. Together with our choice 
for the construction of $\widehat{O}_i$, this implies that
$\widehat{O}$ is contained in the horoball bounded by $O$, and the Hausdorff distance
between $\widehat{O}$ and $O$ is bounded by $2(H+\beta)$. As a consequence we 
easily deduce the following:

\begin{lemma}\label{neuteredmodified:lem}
The set $\widehat{B}$ is $\Lambda$-invariant and is such that
$$
B\subseteq \widehat{B}\subseteq N_{2(H+\beta)} (B)
$$
(where regular neighbourhoods are considered with respect to the hyperbolic metric
$d_\matH$).
Moreover, 
if $\widehat{O},\widehat{O'}$ are distinct elements of 
$\widehat{\mathcal{O}}$, then the distance between the horoballs bounded by
$\widehat{O}$ and $\widehat{O'}$ is at least $R$ (in particular, such horoballs are
disjoint).
\end{lemma}

We are now ready to prove the following:

\begin{proposition}
The group $\Lambda$ is a non-uniform lattice in
${\rm Isom} (\matH^m)$, and admits $\widehat{B}$ as
associated neutered space.
\end{proposition}

\begin{proof}
We begin by showing that $\Lambda$ is discrete. Since $N$ has finite volume,
the set $P$ is dense in $\partial \matH^m$, so we may find 
horospheres $\widehat{O}_1,\ldots,\widehat{O}_{m+1}$  in $\partial \widehat{B}$ with basepoints 
$p_1,\ldots,p_{m+1}$ such that $\{p_1,\ldots,p_{m+1}\}$ is not contained in the trace at infinity of any
hyperbolic hyperplane of $\matH^m$. In particular, if $\alpha\in {\rm Isom}(\matH^m)$
is such that $\alpha (p_i)=p_i$ for every $i=1,\ldots,m+1$, then $\alpha={\rm Id}$. 

Recall that the minimal distance between distinct connected components of $\partial \widehat{B}$ is 
bounded from below by the constant $R>0$. Choose $x_i\in O_i$ for $i=1,\ldots,m+1$ and set
$$
U=\left\{\alpha\in{\rm Isom} (\matH^m)\, |\, d_\matH (\alpha (x_i),x_i)< R\ {\rm for\ every}\ i=1,\ldots,m+1
\right\}.
$$
Then $U$ is an open neighbourhood of the identity in ${\rm Isom}(\matH^m)$; let us compute the
intersection $\Lambda \cap U$.
If $\alpha\in \Lambda$, we have that $\alpha$ permutes the component of $\partial \widehat{B}$.
If we also assume $\alpha \in U$, then $\alpha$ moves each of the horospheres $O_i$ at most $R$,
which forces $\alpha (O_i)=O_i$, whence $\alpha (p_i)=p_i$, for
each $i=1,\ldots, m+1$. As noted above, this implies $\alpha={\rm Id}$, and $\Lambda \cap U = \{{\rm Id}\}$.
But this implies $\Lambda$ is a discrete subgroup.

Next we verify that $\Lambda$ has finite co-volume.
Since $\widehat{B}$ is contained in the $2(H+\beta)$-neighbourhood of $B$, there exists
a $\Gamma$-orbit which is $(k'+2H+2\beta)$-dense in $\widehat{B}$, and this immediately implies
that there exists a $\Lambda$-orbit which is $(k'+2H+3\beta)$-dense in $\widehat{B}$. It follows that
the quotient orbifold $\widehat{B}/\Lambda$ is compact. By Lemma~\ref{finite-orbits:lem},  
such an orbifold has a finite number $V_1,\ldots,V_j$ of boundary components. Let $\widehat{O}_j$
be the boundary component of $\widehat{B}$ projecting onto $V_j$. 
Since elements of $\Lambda$ permute
the boundary components of $\widehat{B}$, if $\alpha\in\Lambda$ is such that $\alpha (\widehat{O}_j)\cap 
\widehat{O}_j\neq\emptyset$,
then $\alpha (\widehat{O}_j)=\widehat{O}_j$, so $\alpha$ belongs to the stabilizer $\Lambda_j$ of the basepoint of 
$\widehat{O}_j$.
Being a closed subset of the compact quotient $\widehat{B}/\Lambda$, the set
$V_j=\widehat{O}_j/\Lambda_j$ is also compact. If $W_j\subseteq \matH^m$ is the horoball 
bounded by $\widehat{O}_j$, it follows that 
the quotient $W_j/\Lambda_j$ has finite volume. Since $\left(\bigcup_{i=1}^j W_j\right)\cup\widehat{B}$
projects surjectively onto $\matH^m/\Lambda$, we conclude that
$\matH^m/\Lambda$ has finite volume, and we have verified that $\Lambda$ is a non-uniform lattice.
\end{proof}

\begin{corollary}\label{schwartz2:cor}
The group $\Lambda$ is commensurable with $\pi_1 (N)$. 
\end{corollary}

\begin{proof} Since $B\subseteq\widehat{B}\subseteq N_{2(H+\beta)} (B)$, 
the spaces $B$ and $\widehat{B}$, when endowed with their path distances, are quasi-isometric.
Since $\pi_1 (N)$ acts properly and cocompactly on $B$ and $\Lambda$ acts properly and cocompactly
on $\widehat{B}$, by Milnor-Svarc's Lemma this ensures that $\Lambda$ is quasi-isometric
to $\pi_1 (N)$. The conclusion now follows from~\cite[Corollary 1.3]{schw}, since both $\pi_1 (N)$ and 
$\Lambda$ are non-uniform lattices in ${\rm Isom}(\matH^m)$.
\end{proof} 

\section{The kernel of $\theta$}

Having obtained an understanding of the image of $\theta$, we now turn to studying the kernel.

\begin{lemma}\label{kertheta}
The group
$\ker \theta$ is finitely generated and quasi-isometric to $\mZ^d$.
Moreover, it is quasi-isometrically embedded in $\Gamma$.
\end{lemma}

\begin{proof}
Let $F=\{b\}\times \mR^d \subseteq \tilM$ be a fixed fiber of $\tilM$, set $x_0=(b,0)\in F$ 
and observe that there exists $\beta'>0$ such that
if $\gamma\in\ker \theta$
then $\gamma (x_0)\in N_{\beta'} (F)$ 
(we may take as $\beta'$ the smallest number
such that in the base $B$ every $d_\matH$-ball of radius $\beta$ is contained in
a $d_B$-ball of radius $\beta'$). For $\gamma\in\ker\theta$, $x\in F$, we denote by
$\alpha (\gamma,x)\in F$ a point such that $d(\alpha (\gamma,x), \gamma (x))\leq\beta'$.
It is not difficult to see that the resulting map
$\alpha\colon \ker\theta \times F\to F$ defines a quasi-action. Since the fiber
$F$ is isometric to $\mR^d$ (and hence quasi-isometric to $\mZ^d$), Lemma~\ref{milsv+:lem}
tells us the first statement would follow provided we can 
show that $\alpha$ is cobounded, i.e.~that 
the orbit of $x_0$ is 
$Q$-dense in $F$ for some $Q$. 

First observe that if $\gamma\in\Lambda$ is such that $\gamma (x_0)\in N_{\beta'} (F)$, then 
$\theta (\gamma)$ moves
$b$ a universally bounded distance from itself, so discreteness of $\Lambda$ implies that 
$\theta (\gamma)$ belongs to a fixed finite subset $A\subseteq \Lambda$. For every $a\in A$
we choose an element $\gamma_a\in \Gamma$ such that $\theta (\gamma_a)=a$ and we
set $M=\max \{d (x_0,\gamma_a^{-1} (x_0)), \ a\in A\}$.
Now, for each point $p\in F$
there exists $\gamma\in \Gamma$ such that 
$d(\gamma(x_0),p)\leq k$. Then, if $\theta (\gamma)=a\in A$ we have that $\gamma\gamma_a^{-1}\in\ker\theta$ and
\begin{align*}
d((\gamma \gamma_a^{-1})(x_0),p) &\leq d(\gamma (\gamma_a^{-1} (x_0)),p)+k \\
&\leq d(\gamma (\gamma_a^{-1} (x_0)),\gamma (x_0))+d(\gamma (x_0),p)+k \\ 
&\leq k d (\gamma_a^{-1} (x_0),x_0)+3k \\
&\leq M+3k
\end{align*}
so $d(\alpha (\gamma\gamma_a^{-1},x_0),p)\leq M+3k+\beta'$.
We have thus proved that $\alpha$ is cobounded, and
from Lemma~\ref{milsv+:lem} we can now deduce that $\ker\theta$
is finitely generated and quasi-isometric to $F$ (whence to
$\mZ^d$) via the map
$$
j_{x_0}\colon \ker\theta\to F,\qquad
j_{x_0} (\gamma)=\alpha (\gamma,x_0)\ .
$$

Let us now prove that $\ker\theta$ is quasi-isometrically embedded in $\Gamma$.
Let $\varphi\colon \Gamma\to\tilM$, $\psi\colon \tilM\to\Gamma$ be the quasi-isometries 
introduced in Section~\ref{quasiact:sub}, and let $i\colon F \to \tilM$ be the inclusion. 
Also choose $k''$ large enough, so that $\psi$ is a $(k'',k'')$-quasi-isometry and
$d(\psi(\varphi(\gamma)),\gamma)\leq k''$ for every $\gamma\in\Gamma$. Since $F$ 
is totally geodesic in $\tilM$, the inclusion $i$ 
defines an isometric embedding of $F$ into $M$, hence the composition of quasi-isometric
embeddings $\psi\circ i\circ j_{x_0}\colon \ker\theta\to\Gamma$ is also a quasi-isometric embedding. In order to conclude,
it is now sufficient to show that
the inclusion of $\ker \theta$ into $\Gamma$ stays at bounded distance from
$\psi\circ i\circ j_{x_0}$.

Keeping the notation from Section~\ref{quasiact:sub} (and recalling that, in the proof above, we denoted 
by $\gamma(x_0)$ the point $\varphi(\gamma\cdot \psi (x_0))$), 
for
every $\gamma\in\ker\theta$ we have the series of inequalities:
\begin{align*}
d(\psi(i(j_{x_0} (\gamma))),\gamma) &= d(\psi(\alpha(\gamma,x_0)),\gamma) \\
& \leq d(\psi (\alpha(\gamma,x_0)),\psi(\gamma(x_0)))+ d(\psi(\gamma (x_0)),\gamma)\\
& \leq k''\beta'+k'' + d(\psi(\varphi(\gamma\cdot \psi (x_0))),\gamma)\\
& \leq  k''\beta'+ 2k'' + d(\gamma\cdot \psi(x_0),\gamma)\\ 
& = k''\beta'+2k'' +d(\psi (x_0),1_\Gamma), 
\end{align*}
where the last equality is due to the $\Gamma$-invariance of any word metric on $\Gamma$,
and this concludes the proof.
\end{proof}

We now need the following fundamental
result by Gromov:

\begin{theorem}[\cite{gropol}]\label{virtab:thm}
A finitely generated group quasi-isometric to $\mZ^d$ contains 
a finite index subgroup isomorphic to $\mZ^d$.
\end{theorem}

By Theorem~\ref{virtab:thm}, 
$\ker\theta $ contains a finite index subgroup $K$ isomorphic to $\mZ^d$.
Being finitely generated, $\ker \theta$ contains only a finite number of subgroups
having the same index as $K$. The intersection of all such subgroups has finite index in $K$ and
is characteristic in $\ker\theta$. Therefore, up to replacing $K$ with one of its finite index subgroups,
we can assume 
that $K$ is characteristic in $\ker\theta$, 
hence normal in $\Gamma$. By construction, the quotient $\Gamma/ K$ is a finite extension
of $\Lambda=\Gamma/\ker\theta$. By Corollary~\ref{schwartz2:cor}, 
there exists a finite index subgroup $\Lambda'$ of $\Lambda$ such that
$\Lambda'\cong \pi_1 (N')$ for some finite-sheeted covering $N'$ of $N$.
Let us set $\Gamma'=\theta^{-1}(\Lambda')$ and $\Delta=\Gamma'/K$. Then,
we have the following exact sequences:
\begin{equation}\label{exseq:eq}
\xymatrix{
1\ar[r] & \mZ^d\ar[r]^j & \Gamma' \ar[r]^>>>>>{\theta} & \Gamma'/K=\Delta \ar[r] & 1,\\
}
\end{equation}
\begin{equation}\label{exseq2:eq}
\xymatrix{
1\ar[r] & F \ar[r] & \Delta \ar[r] & \pi_1 (N')\ar[r] & 1,
}
\end{equation}
where $K=j(\mZ^d)$, and $F$ is finite.

\section{Abelian undistorted normal subgroups are virtually central}
In order to conclude the proof of Theorem~\ref{product:thm}, it is sufficient 
to show that the sequence~\eqref{exseq:eq} is virtually central, i.e.~that 
$K=j(\mZ^d)$ is contained in the center of a finite-index subgroup
of $\Gamma'$. In fact, in this case we can replace $\Gamma'$
with this finite-index subgroup, and, up to replacing $\Delta$, $F$ and $\pi_1(N')$
with suitable finite-index subgroups, the exact sequences \eqref{exseq:eq}, \eqref{exseq2:eq}
satisfy all the properties stated in Theorem~\ref{product:thm}.

Since $K$ is a finite-index subgroup
of $\ker\theta$ and $\Gamma'$ is a finite-index subgroup
of $\Gamma$, by Lemma~\ref{kertheta} the inclusion of
$K$ in $\Gamma'$ is a quasi-isometric embedding.
Therefore, in order to conclude the proof of Theorem~\ref{product:thm}
we just need to apply the following result to the case
$\overline\Gamma=\Gamma'$, $\overline{K}=K$.

\begin{proposition}\label{virtualcentral}
Let $\overline\Gamma$ be a finitely generated group, and let
$\overline{K}$ be a free abelian normal subgroup of $\overline{\Gamma}$. Also
suppose that $\overline{K}$ is quasi-isometrically embedded in $\overline\Gamma$.
Then $\overline K$ is contained in the center of a finite-index subgroup of $\overline\Gamma$.
\end{proposition}
\begin{proof}
In the proof of this Proposition we exploit 
the notion of
\emph{translation number}, and follow a strategy already described 
in~\cite{gro1,gersten} (see also~\cite{alo,klelee}). 

Let $G$ be a finitely generated group with finite set of generators $A$,
and for every $g\in G$ let us denote by $|g|_A$ the distance
between $g$ and the identity of $G$ in the Cayley graph of $G$
relative to $A$. The \emph{translation number} of $g$ is then
given by the non-negative number
$$
\tau_{G,A} (g)=\lim\limits_{n\to\infty} \frac{|g^n|_A}{n}
$$
(the fact that such a limit exists follows from the inequality 
$|g^{m+n}|_A\leq |g^m|_A+|g^n|_A$, which holds for every $g\in G$, $m,n\in\mN$). 
We recall the following well-known properties of the translation number:
\begin{enumerate}
\item \label{conj:pro}
$\tau_{G,A} (ghg^{-1})=\tau_{G,A} (h)$
for every $g,h\in G$;
\item \label{abelian:pro}
if $G$ is free abelian and $A$ is a basis of $G$, then
$\tau_{G,A} (g)=|g|_A$ for every $g\in G$;
\item \label{quasi:pro}
let
$G$ be a subgroup of $G'$ and $A,A'$ be finite set of generators
for $G,G'$; if the inclusion $i\colon G\to G'$ is a $(\lambda,\varepsilon)$-quasi-isometric
embedding (with respect to the metrics defined on $G,G'$ by $A,A'$), then
for every $g\in G$ we have
$$
\lambda^{-1}\tau_{G,A} (g)\leq \tau_{G',A'} (g)\leq \lambda \tau_{G,A} (g).
$$
\end{enumerate}

For every $x\in \overline{\Gamma}$
we consider
$\alpha (x)\colon \overline{K}\to \overline{K}$ 
defined by $\alpha (x) (k)=x \cdot k \cdot x^{-1}$. Of course,
the map
$\alpha\colon \overline{\Delta}\to {\rm Aut} (\overline{K})$ is a well-defined 
homomorphism of groups.

Now let $\overline{A} \subseteq \overline{\Gamma}$ be a finite set of generators and
let $A=\{k_1,\ldots,k_d\}$ be a free basis
of $\overline{K}$. 
For every $x\in \overline{\Gamma}$, $i=1,\ldots,d$, the element
$\alpha (x)(k_i)$ is conjugate to $k_i$ in $\overline{\Gamma}$,
so by property~\eqref{conj:pro} above we have
\begin{equation}\label{conj:eq}
\tau_{\overline{\Gamma},\overline{A}} (\alpha (x) (k_i))= 
\tau_{\overline{\Gamma},\overline{A}} (k_i).
\end{equation} 
Since $\overline K$ is quasi-isometrically embedded in $\overline\Gamma$,
by property~\eqref{quasi:pro} of the translation number 
there exists $\lambda>0$ such that
\begin{equation}\label{quasi:eq}
\tau_{\overline{K},A} (\alpha(x)(k_i))\leq \lambda
\tau_{\overline{\Gamma},\overline{A}} (\alpha(x)(k_i)), \qquad \tau_{\overline{\Gamma},\overline{A}} (k_i)
\leq \lambda \tau_{\overline{K},A} (k_i) = \lambda . 
\end{equation}
Putting together property~\eqref{abelian:pro} of the translation number
with equations~\eqref{conj:eq} and~\eqref{quasi:eq} we finally obtain
$$
|\alpha (x)(k_i)|_{A}=
\tau_{\overline{K},A} (\alpha(x)(k_i))\leq \lambda
\tau_{\overline{\Gamma},\overline{A}} (\alpha(x)(k_i))= \lambda \tau_{\overline{\Gamma},\overline{A}} (k_i)
\leq \lambda^2 
$$
for every $x\in\overline{\Gamma}$, $i=1,\ldots,d$. This implies that the orbit of each
$k_i$ under the action of $\alpha (\overline{\Gamma})$ is finite, so the homomorphism 
$\alpha\colon \overline{\Gamma}\to{\rm Aut} (\overline{K})$ has finite image, and 
$\ker\alpha$ has finite index in 
$\overline{\Gamma}$. 
Moreover, $\overline{K}$ is contained in the center of $\ker \alpha$,
so $\ker\alpha$ provides the required finite-index subgroup of $\overline{\Gamma}$.
\end{proof}

\begin{remark}
Let us analyze further the short exact sequence
$$
\xymatrix{
1 \ar[r] & \overline{K} \ar[r] & 
\overline{\Gamma} \ar[r]^\pi & \overline{\Delta}=\overline{\Gamma}/\overline{K}\ar[r] & 1\ ,
}
$$
studied in Proposition~\ref{virtualcentral}. 
In our case of interest,
i.e.~when $\overline{K}=K$, $\overline{\Gamma}=\Gamma'$, and $\overline{\Delta}=\Delta$, we also know
that the following condition holds:
\begin{itemize}
 \item[(*)] 
there exists a quasi-isometry $q\colon \overline{\Gamma}\to \overline{K}\times \overline{\Delta}$ which makes the following diagram commute:
$$
\xymatrix{
\overline{\Gamma}\ar[r]^\pi \ar[d]_q & \overline{\Delta}\ar[d]^{\rm Id}\\
\overline{K}\times\overline{\Delta} \ar[r] & \overline{\Delta}
}
$$
where the horizontal arrow on the bottom represents the obvious projection.
\end{itemize}
Moreover, the group $\overline\Delta$ is a finite extension of the fundamental group of a cusped
hyperbolic manifold.
One may wonder whether these extra assumptions 
could be exploited to show that the
sequence~\eqref{exseq:eq} above virtually splits. In this remark  we show
that this is not true in general.

Condition (*) is equivalent to the existence of a Lipschitz section
$s\colon \overline{\Delta}\to \overline{\Gamma}$ such that $\pi\circ s={\rm Id}_{\overline{\Delta}}$ (see {e.g.}~\cite[Proposition 8.2]{klelee}).
Recall that a central extension of $\overline{\Delta}$ by $\overline{K}$ is 
classified by its characteristic coclass in $H^2(\overline{\Delta},\overline{K})$. In the case
when $\overline{K}\cong \mZ$, Gersten proved that a sufficient condition for a
central extension
to satisfy condition (*) is that its characteristic coclass admits a \emph{bounded}
representative (see \cite[Theorem 3.1]{Ger:last}). Therefore, in order to construct
an exact sequence  that satisfies condition (*) but does not virtually
split, it is sufficient to find an element of $H^2(\overline{\Delta},\mZ)$
of infinite order that admits a bounded representative.

Let us set 
$\overline{\Delta}=\pi_1(N)$, where $N$ is a hyperbolic $3$-manifold
with $k\geq 1$ cusps and second Betti number $b_2>k$ (it is not difficult
to construct such a manifold, for example by considering suitable link complements
in the connected sum of several copies of $S^2\times S^1$). We denote by
$\widehat{N}$
a closed hyperbolic $3$-manifold obtained by Dehn filling all the cusps of $N$. An easy argument using a Mayer-Vietoris sequence 
shows that a $2$-class $c_N\in H_2 (N;\mZ)$ exists such
that the element $i_\ast (c_N)\in H_2 (\widehat{N};\mZ)$ has infinite order,
where $i\colon N\to \widehat{N}$ is the natural inclusion.
Thanks to the Universal Coefficient Theorem, a coclass
$\omega\in H^2 (\widehat{N};\mZ)$ exists
such that $\omega (i_\ast (c_N))=1$ (here and henceforth we denote
by $\omega (i_\ast (c_N))$ the number $\langle \omega, i_\ast (c_N)\rangle$,
where $\langle \cdot \, , \, \cdot \rangle\colon H^2(\widehat{N};\mZ)\times 
H_2(\widehat{N};\mZ) \to \mZ$ is the Kronecker pairing).

Since $N$ and $\widehat{N}$ have contractible universal coverings,
we have natural isomorphisms 
$H_2 (N;\mZ)\cong H_2 (\pi_1 (N);\mZ)$, 
 $H_2 (\widehat{N};\mZ)\cong H_2 (\pi_1 (\widehat N);\mZ)$,
 $H^2 (N;\mZ)\cong H^2 (\pi_1 (N);\mZ)$, 
 $H^2 (\widehat{N};\mZ)\cong H^2 (\pi_1 (\widehat N);\mZ)$. Abusing
 notation, we will denote by $c_N\in H_2 (\pi_1 (N);\mZ)$,
$i_\ast(c_N)\in H_2 (\pi_1 (\widehat{N});\mZ)$,
$\omega\in H^2 (\pi_1 (\widehat{N});\mZ)$ the elements corresponding
to the (co)classes introduced above. The inclusion 
$i:N \hookrightarrow \widehat{N}$ induces a morphism
 $i^\ast \colon H^2 (\pi_1(\widehat{N});\mZ)\to 
H^2 (\pi_1({N});\mZ)$. 

Recall now that $\overline{\Delta}=\pi_1(N)$, and consider the central extension
\begin{equation}\label{example:seq}
1\to \mZ \to \overline{\Gamma} \to \overline{\Delta}\to 1
\end{equation}
associated to the coclass $i^\ast(\omega)\in H^2 (\overline{\Delta};\mZ)$.  
On one hand, since $\pi_1 (\widehat{N})$
is Gromov-hyperbolic, by~\cite{neumann} the coclass $\omega\in H^2 (\pi_1 (\widehat{N});\mZ)$ 
admits a bounded representative, so $i^\ast (\omega)$ is also bounded,
and the sequence~\eqref{example:seq} satisfies condition (*).
On the other hand,
we have $i^\ast (\omega) (c_N)=\omega (i_\ast (c_N))=1$,
so $i^\ast (\omega)$ has infinite order in $H^2 (N;\mZ)$,
and this proves that the sequence~\eqref{example:seq} does not virtually split.
\end{remark}

\section{Pieces with quasi-isometric fundamental groups}
The following proposition provides a necessary and sufficient condition for
two pieces of graph manifolds to have quasi-isometric fundamental groups.

\begin{proposition}\label{qi-pieces:prop}
Let $n\geq 3$ be fixed, and,
for $i=1,2$,
let $N_i$ be a complete finite-volume hyperbolic $n_i$-manifold
with toric cusps, $n_i\geq 3$. If $\pi_1 (N_1\times T^{n-n_1})=\pi_1 (N_1)\times\mZ^{n-n_1}$
is quasi-isometric to $\pi_1 (N_2\times T^{n-n_2})=\pi_1 (N_2)\times \mZ^{n-n_2}$, then $n_1=n_2$
and $N_1$ is commensurable with $N_2$. 
\end{proposition}
\begin{proof}
Let us set $G_i=\pi_1 (N_i)\times \mZ^{n-n_i}$. 
By Theorem~\ref{product:thm}, since $G_1$ is quasi-isometric to
$\pi_1 (N_2)\times \mZ^{n-n_2}$, 
there exist a finite index subgroup $G'_1$ of
$G_1$, a group $\Delta$ and a finite group $F$ which fit in the following
short exact sequences:
$$
\xymatrix{
1\ar[r] & \mZ^{n-n_2} \ar[r]^j & G'_1 \ar[r] & \Delta \ar[r] & 1,\\
}
$$
$$
\xymatrix{
1\ar[r] & F \ar[r] & \Delta\ar[r] & \pi_1 (N_2')\ar[r] & 1,
}
$$
where $N'_2$ is a finite-sheeted covering of $N_2$.
Moreover, $j(\mZ^{n-n_2})$ lies in the center of $G'_1$.
 
Let $Z(G_1)$ (resp.~$Z(G'_1)$) be the center of $G_1$ (resp.~of $G'_1$).
We claim that $Z(G'_1)=Z(G_1)\cap G'_1$. The inclusion $\supseteq$ is obvious.
Moreover, if $p_1\colon G_1\to \pi_1 (N_1)$ is the projection on the first factor,
then $p_1 (G'_1)$ is a finite-index subgroup of $\pi_1 (N_1)$. 
Since any finite-index subgroup of $\pi_1 (N_1)$ has trivial center, this implies
that any element $(\gamma,w)\in G'_1\subseteq G_1= \pi_1 (N_1)\times\mZ^{n-n_1}$
which commutes with all the elements of $G'_1$ must satisfy $\gamma=1$ in $\pi_1 (N_1)$.
We conclude that $(\gamma,w)\in Z(G_1)$, as claimed. 

This implies that $j(\mZ^{n-n_2})\subseteq Z(G'_1)\subseteq Z(G_1)\cong \mZ^{n-n_1}$, so 
$n_1\leq n_2$ by injectivity of $j$. Interchanging the roles
of $G_1$ and $G_2$ we also get $n_2\leq n_1$, forcing $n_1=n_2$.

Since $Z(G_1')=Z(G_1)\cap G'_1$, the quotient $G'_1 /Z(G_1')$ is isomorphic
to a finite-index subgroup of $G_1 / Z(G_1)$, which is in turn isomorphic
to $\pi_1 (N_1)$. In particular, $G'_1 /Z(G_1')$ is quasi-isometric
to $\pi_1 (N_1)$. Moreover, since $n_1=n_2$ the groups
$j(\mZ^{n-n_2})$ and $Z(G'_1)$ share the same rank, and this implies
that $j(\mZ^{n-n_2})$ is a finite-index subgroup of $Z(G'_1)$, 
so that $\Delta\cong G'_1 /j(\mZ^{n-n_2})$ is quasi-isometric to $G'_1/Z(G'_1)$, whence to $\pi_1 (N_1)$. 
On the other hand, since $\Delta$
is a finite extension of
$\pi_1 (N'_2)$ and $\pi_1 (N'_2)$ is of finite index in $\pi_1 (N_2)$,
the group $\Delta$ is quasi-isometric to $\pi_1 (N_2)$ too, so
$\pi_1 (N_1)$ and 
$\pi_1 (N_2)$ are quasi-isometric to each other.
The conclusion now follows from~\cite{schw}. 
\end{proof}

%
%
%

\chapter{Quasi isometry rigidity, II}\label{qirigidity:sec}
The aim of this chapter is the proof of Theorem~\ref{qirigidity:thm},
which we recall here:

\begin{Thm2}
Let $M$ be an irreducible graph $n$-manifold obtained by gluing the pieces
$V_i=\overline{N}_i\times T^{d_i}$, 
$i=1,\ldots, k$. Let $\Gamma$ be a group quasi-isometric
to $\pi_1 (M)$. Then either $\Gamma$ itself, or a subgroup of $\Gamma$ of index two,
is isomorphic to the fundamental group of a graph of groups satisfying the following 
conditions:
\begin{itemize}
\item
every edge group contains $\mZ^{n-1}$ as a subgroup
of finite index;
\item
for every vertex group $\G_v$ there exist $i\in\{1,\ldots, k\}$,
a finite-sheeted covering $N'$ of $N_i$ and a finite-index subgroup
$\G'_v$ of $\G_v$ that fits into the exact sequences
$$
\xymatrix{
1\ar[r] &\mZ^{d_i} \ar[r]^j & \Gamma_v' \ar[r] & \Delta \ar[r]& 1,\\}
$$
$$
\xymatrix{
1\ar[r] & F \ar[r] & \Delta\ar[r] & \pi_1 (N')\ar[r] & 1 ,
}
$$
 where $F$ is a finite group, and $j(\mZ^{d_i})$ is contained
 in the center of $\Gamma'_v$.
\end{itemize}
\end{Thm2}

Throughout this chapter we denote by $M$ an \emph{irreducible} graph manifold
with universal covering $\widetilde{M}$, and by $\Gamma$ a finitely generated group
quasi-isometric to $\pi_1 (M)$. As discussed  
in Section~\ref{quasiact:sub}, a quasi-isometry between $\Gamma$ and $\pi_1 (M)$ induces a $k$-cobounded
$k$-quasi-action $h$ of $\Gamma$ on $\tilM$ for some $k\geq 1$, which will from now on be fixed. Henceforth, 
for every $\gamma\in\Gamma$ we will denote simply by $\gamma$ the quasi-isometry $h(\gamma)\colon \tilM\to\tilM$.

\section{From quasi-actions to actions on trees}\label{quasiact2:sub}
Let $(\tilM, p, T)$ be the triple which endows $\tilM$ with 
the structure of a tree of spaces (see Section~\ref{univ:subsec}). 
Building on the results proved in Chapter~\ref{preserve:sec},
we wish to define an action of $\Gamma$ on $T$.
Fix $\gamma\in\Gamma$. By Propositions~\ref{wall2:prop} and~\ref{qi-useful:prop}, if
$v_1$, $e_1$ are a vertex and an edge 
corresponding respectively to a chamber $C_1$ and a wall $W_1$, 
then there exist a unique chamber $C_2$ at finite Hausdorff distance
from $\gamma (C_1)$ and a unique wall $W_2$ at finite Hausdorff distance from $\gamma (W_1)$. 
We will denote by $\gamma (v_1)$, $\gamma (e_1)$
the vertex and the edge corresponding respectively to $C_2$ and $W_2$. Again, by Proposition~\ref{qi-useful:prop},
if $W_1$ is adjacent to $C_1$ then $W_2$ is adjacent to $C_2$, which gives us the following:

\begin{proposition}\label{act-on-tree:prop}
The map $\gamma\colon T\to T$ just defined provides a simplicial automorphism of $T$.
\end{proposition}

In what follows, when saying that a group $G$ acts on a tree $T'$ we will always mean that
$G$ acts on $T'$ by simplicial automorphisms.
Recall that $G$ acts on  $T'$ \emph{without inversions} if no
element of $G$ switches the endpoints of an edge of $T'$. 
We wish to apply the following fundamental result from Bass-Serre theory 
(see~\cite{serre}):

\begin{theorem}\label{serre:thm}
Suppose $G$ acts on a tree $T'$ without inversions. 
Then $G$ is isomorphic to the fundamental group of a graph of groups supported by the 
graph $\mathcal{G}$ with set of vertices $V$ and set of edges $E$. If $G_v$, $v\in V$,
and $G_e$, $e\in E$, are the vertex and edge groups of the graph of groups, then:
\begin{enumerate}
\item
$\mathcal{G}$ is the quotient of $T'$ by the action of $G$.
\item
For each $v\in V$, the group $G_v$ is isomorphic to the stabilizer of a vertex of $T'$
projecting to $v$.
\item
For each $e\in E$, the group $G_e$ is isomorphic to the stabilizer of an edge of $T'$
projecting to $e$.
\end{enumerate}
\end{theorem}

Now the action of $\Gamma$ on $T$ described in Proposition \ref{act-on-tree:prop} might include
some inversions. However, every tree is a bipartite graph in a canonical way. 
The group $Aut(T)$ 
of all simplicial automorphisms of $T$ contains a subgroup $Aut^0(T)$, of index at most two, 
which preserves both parts of that bi-partition. This subgroup consists solely of elements that act without inversions. We may now set $\Gamma^0=\Gamma\cap Aut^0(T)$, and conclude
that $\Gamma^0$ is a subgroup of $\Gamma$ of index at most two that acts on $T$ without inversions.



\section{The action of $\Gamma^0$ on $T$} 

Recall that $\Gamma^0$ quasi-acts via $(k,k)$-quasi-isometries
with $k$-dense image
on $\tilM$, and, up to increasing the constant $k$,
we may also assume that every $\Gamma^0$-orbit is $k$-dense in $\tilM$. We denote by $E$ the set of edges
of $T$, and we suppose that for every wall $W$ (resp.~chamber $C$) and every $\gamma\in\Gamma^0$
the set $f(W)$ (resp.~$f(C)$) has Hausdorff distance bounded by $H$ from a wall (resp.~a chamber) (see Propositions~\ref{wall2:prop} and~\ref{qi-useful:prop}).
We first show that the quotient of $T$ by the action of $\Gamma^0$ is a finite graph.

\begin{lemma}
The action of $\Gamma^0$ on $E$ has a finite number of orbits.
\end{lemma}

\begin{proof}
Fix a point $p\in\tilM$. The set $A$ of those walls whose distance from $p$ is less than $k(H +k) + 3k$ is finite. 
Let $W$ be any wall, and fix a point $w\in W$. There exists $\gamma\in\Gamma^0$ such that $d(\gamma(p),w)\leq k$. We know that there exists a wall $W'$ such that $\gamma(W')$ is at Hausdorff distance bounded by $H$ from $W$. This implies that $\gamma(W')$ contains a point $\gamma(w')$, $w'\in W'$, at distance less than $H$ from $w$. We can use this to estimate:
\begin{align*}
d(w', p) & \leq d(\gamma^{-1}(\gamma(w')),\gamma^{-1}(\gamma(p)))+2k\\
& \leq k d(\gamma(w'),\gamma(p))+3k \\
& \leq k(H+k)+3k,
\end{align*}
so $W'\in A$. As a result, the finite set of edges corresponding to walls in $A$ contains a set of representatives 
for the action of $\Gamma^0$ on $E$.
\end{proof}

\section{Stabilizers of edges and vertices}
If $e$ (resp.~$v$) is an edge (resp.~a vertex) of $T$, then
we denote by $\Gamma^0_e$ (resp.~$\Gamma^0_v$) the stabilizer
of $e$ (resp.~of $v$) in $\Gamma^0$.

\begin{lemma}\label{stab:lemma}
For every edge $e$ of $T$, 
the stabilizer $\Gamma^0_e$ is quasi-isometric to a wall. 
The stabilizer $\Gamma^0_v$ of a vertex $v$ is quasi-isometric to the chamber corresponding to $v$.
\end{lemma}

\begin{proof}
Let us focus on proving the first statement, as the second statement follows from a very similar argument.
Let $N_H (W)$ be the $H$-neighbourhood of the wall $W$ corresponding to the edge $e\subseteq T$, and let
$\varphi_e\colon \Gamma^0_e \to N_H (W)$ be defined by 
$\varphi_e (\gamma)=\gamma(w)$, where $w\in W$ is a fixed basepoint. 
Let us first prove that $\varphi_e (\Gamma^0_e)$ is $p-$dense in $N_H (W)$ for some $p$. 
For each wall $W_i$, $i=1,\ldots, m$, in the orbit of $W$ and having distance less than $k^2 + 2k+H$ from $w$,
we choose $\gamma_i\in \Gamma^0$ such that $\gamma_i(W)$ has Hausdorff distance 
from $W_i$ bounded by $H$. Let $L$ be large enough so that $d(w,\gamma_i(w))\leq L$
for every $i=1,\ldots,m$. Now pick any point $w'\in N_H(W)$. 
We know that there is $\gamma\in\Gamma^0$ (but not necessarily in $\Gamma^0_e$) 
such that $d(\gamma(w), w')\leq k$. It is not difficult to show 
that $\gamma^{-1}(W)$ has finite Hausdorff distance from one of the $W_i$'s, 
so there exists $j$ such that $\gamma(W_j)$ is at finite Hausdorff distance from $W$.
Then $\gamma\cdot\gamma_j\in \Gamma^0_e$, 
and we have the estimate:
\begin{align*}
d\big((\gamma\gamma_j)(w), w'\big) & \leq d \big((\gamma\gamma_j) (w),\gamma (w) \big)+d(\gamma(w),w')\\
&\leq \big(d(\gamma (\gamma_j (w)),\gamma (w)) + k \big) + k \\ 
&\leq \big(k d(\gamma_j (w),w)+2k \big) +k \\
&\leq kL+3k .
\end{align*}
This implies that $\varphi_e (\Gamma^0_e)$ is $(kL+3k)$-dense in $N_H (W)$. 

In order to 
apply Lemma~\ref{milsv+:lem} we now need to construct a quasi-action of $\Gamma^0_e$
on $(W,d_W)$, where $d_W$ is the path-distance of $W$. 
With this goal in mind, 
for every $\gamma\in \Gamma^0_e$ and $x\in W$, we let $h_e(\gamma)(x)$ be a point in $W$ such that 
$d(\gamma (x),h_e(\gamma)(x))\leq H$.
It is easily checked that the map $\gamma\mapsto h_e(\gamma)$ indeed defines a quasi-action of
$\Gamma^0_e$ on $(W,d)$, where $d$ is the restriction to $W$ of the distance on $\tilM$.
Moreover, the orbit of $w$ under this quasi-action is $(kL+3k+2H)$-dense in $(W,d)$.
But since $M$ is irreducible the identity map on $W$ provides a quasi-isometry
between $(W,d)$ and the path metric space $(W,d_W)$, so $h$ provides a quasi-action
 of $\Gamma^0_e$ on $(W,d_W)$, and the orbit of $w$ is $p$-dense in $(W,d_W)$ for some 
 $p$.
By Lemma~\ref{milsv+:lem}, this implies that $\Gamma^0_e$ is finitely generated
and quasi-isometric to $(W,d_W)$.
\end{proof}

\begin{remark}
Arguing as in the proof of Lemma~\ref{kertheta}, it is possible to prove
that the stabilizers $\Gamma^0_e$, $\Gamma^0_v$ are quasi-isometrically embedded
in $\Gamma^0$. 
\end{remark}

Putting together Lemma~\ref{stab:lemma} and Gromov's Theorem~\ref{virtab:thm} we immediately get the 
following:

\begin{proposition}\label{edgestab:prop}
If $\Gamma^0_e$ is the stabilizer of an edge $e\subseteq T$, then
$\Gamma^0_e$ contains $\mZ^{n-1}$ as a subgroup of finite index.
\end{proposition}

\vskip 10pt

Theorem~\ref{qirigidity:thm} is now a direct consequence of
Theorem~\ref{serre:thm}, Proposition~\ref{edgestab:prop} and Theorem~\ref{product:thm}.

\vskip 10pt

\section{Graph manifolds with quasi-isometric fundamental groups}\label{qirig2:subsec}
We are now interested in analyzing when irreducible graph manifolds have quasi-isometric 
fundamental groups.

For $i=1,2$, 
let $M_i$ be an irreducible graph manifold, and let
us denote by $T_i$ the tree corresponding to the decomposition
of $\tilM_i$ into chambers. We can label each vertex $v$ of $T_i$ as follows:
if $v$ corresponds to a chamber projecting in $M$ onto a piece 
of the form $N\times T^{d}$, where $N$ is a cusped hyperbolic manifold,
then we label $v$ with the commensurability class of $N$.
The following result gives a necessary condition for $M_1,M_2$ to have quasi-isometric fundamental groups:

\begin{theorem}\label{qi-mnpc}
Suppose the fundamental groups of $M_1$ and $M_2$ are quasi-isometric.
Then $T_1$ and $T_2$ are isomorphic as labelled trees. 
\end{theorem}
\begin{proof}
By Milnor-Svarc's Lemma, a quasi-isometry between $\pi_1 (M_1)$ and
$\pi_1 (M_2)$ induces a quasi-isometry, say $\psi$, between the universal coverings
$\tilM_1$ and $\tilM_2$. By Proposition~\ref{qi-useful:prop}
(see also Subsection~\ref{quasiact:sub}),
such a quasi-isometry induces a simplicial isomorphism $f_\psi$ between
$T_1$ and $T_2$. We will now show that such isomorphism preserve labels,
thus proving the theorem.

Let $v_1$ be a vertex of $T_1$ corresponding
to the chamber $C_1$, and
suppose that $C_1$ is the universal covering
of $N_1\times T^{d_1}$, where $N_1$ is a cusped hyperbolic
manifold.
Let $C_2$ be the chamber of $\tilM_2$
staying at finite Hausdorff distance from 
$\psi (C_1)$, let $v_2$ be the vertex of $T_2$ 
corresponding to $C_2$, and suppose that $C_2$ projects into $M_2$ onto
a piece of the form $N_2\times T^{d_2}$, where 
$N_2$ is a cusped hyperbolic manifold. By construction, $f_\psi$ takes $v_1$ onto $v_2$,
so we only need to check that the labels
of $v_1$ and $v_2$ are equal, i.e. that $N_1$ is commensurable with $N_2$. However, 
since $M_1,M_2$ are irreducible, the chamber $C_i$ is quasi-isometrically embedded in $M_i$, 
and this implies that $\psi|_{C_1}$ stays at bounded distance 
from a quasi-isometry between $C_1$ and $C_2$. By Milnor-Svarc's Lemma,
it follows
that $\pi_1 (N_1)\times \mZ^{d_1}$ is quasi-isometric to $\pi_1 (N_2)\times \mZ^{d_2}$, so $N_1$ is commensurable with $N_2$ by Proposition~\ref{qi-pieces:prop}.
\end{proof}

Observe that, in each dimension, there exist infinitely many commensurability classes of complete finite-volume
hyperbolic manifolds with toric cusps (see \cite{MRS}). Along with Theorem~\ref{qi-mnpc}, this immediately allows us to deduce:

\begin{corollary}
Suppose $n\geq 3$. Then, there exist infinitely many quasi-isometry classes of fundamental groups
of irreducible graph $n$-manifolds.
\end{corollary}

\begin{remark}\label{not-suff-rem}
Let us fix the notation as in Theorem~\ref{qi-mnpc}.
The following construction shows that the fact that $T_1$ and $T_2$ are isomorphic as labelled
trees is {\bf not} sufficient for ensuring that $\pi_1 (M_1)$ and $\pi_1 (M_2)$ are quasi-isometric.

Let $N$ be a cusped hyperbolic $3$-manifold with two toric cusps, let $\partial_1 \overline N$, $\partial_2\overline N$
be the boundary components of the truncated manifold $\overline N$, and assume that the Euclidean structures
induced by $N$ on $\partial_1\overline N$, $\partial_2\overline N$ are not commensurable with each other. 
The fact that such a manifold exists is proved in~\cite{GHH}
(we may take for example the manifold $7c\ 3548$ in the census available at the
address~\cite{web}). Furthermore, let $N',N''$ be non-commensurable $1$-cusped hyperbolic $3$-manifolds
(for example, suitable hyperbolic knot complements), and consider the (obviously irreducible) graph manifolds
$M_1,M_2$ defined as follows: $M_1$ is obtained by gluing $\overline{N}$ with $\overline{N'}$ along $\partial_1 \overline N$,
and with $\overline{N''}$ along $\partial_2 \overline N$;
$M_2$ is obtained by gluing $\overline{N}$ with $\overline{N'}$ along $\partial_2 \overline N$,
and with $\overline{N''}$ along $\partial_1 \overline N$. 
Of course, the labelled trees associated to $M_1$ and $M_2$ are isomorphic. 

On the other hand, a hypothetical quasi-isometry
between $\pi_1 (M_1)$ and $\pi_1 (M_2)$ should induce a quasi-isometry of $\pi_1 (N)$ into itself taking
the cusp subgroup $\pi_1 (\partial_1 \overline N)$ to a set at finite Hausdorff distance from $\pi_1 (\partial_2 \overline{N})$.
By~\cite{schw}, this would imply that $\pi_1 (\partial_1 \overline N)$ and $\pi_1 (\partial_2 \overline{N})$ admit finite index subgroups
that are conjugated by an isometry of $\matH^3$. 
As a consequence,  the Euclidean structures
induced by $N$ on $\partial_1\overline N$, $\partial_2\overline N$ should be commensurable with each other, which would
contradict our choices.
\end{remark}

\begin{remark}
In~\cite{behr}, Behrstock and Neumann proved that the fundamental groups of any two closed $3$-dimensional
irreducible 
graph manifolds are quasi-isometric. This result could seem in contrast with the phenomenon exhibited by the previous construction.
However, hyperbolic bases, in dimensions $\geq 3$, are much more rigid than
hyperbolic surfaces with boundary. As a consequence, 
in higher dimensions there is no obvious counterpart for
all the ``strechings'' performed on thickened graphs in~\cite{behr}.
\end{remark}

\part{Concluding remarks}

%
%
%

\chapter{Examples not supporting locally CAT(0) metrics}\label{construction2:sec}

We already saw a method in Section \ref{noncat0-easy:subsec} for constructing
graph manifolds which do not support any locally CAT(0) metric. The idea was
to take a finite volume hyperbolic manifold $N$ with at least two toric cusps,
and glue together two copies of $N\times T^2$ in such a way that the 
fundamental group of the resulting
graph manifold contains a non quasi-isometrically embedded abelian subgroup (see
Proposition \ref{notqi}). This method could be used to produce infinitely many 
such examples in all dimensions $\geq 5$.

In this Chapter we provide some additional methods for constructing graph manifolds
which do \emph{not} support any locally CAT(0) metric. In Section~\ref{obstructions:sub}
we show that certain $S^1$-fiber-bundles over the double of cusped hyperbolic 
manifolds do {\it not} support locally CAT(0) metrics. This allows us to construct infinitely many new examples in each dimension $\geq 4$.

Section~\ref{strirr:sub} is devoted to the construction of irreducible examples. 
We can
produce infinitely many such examples in each dimension $\geq 4$.

For ease of notation, we will omit
the coefficient ring in our cohomology groups, with the understanding that all homology
and cohomology in this chapter is taken with coefficients in $\mZ$.

\section{Fiber bundles}\label{obstructions:sub}

In this section, we describe a construction providing
graph manifolds which do not support any locally CAT(0) metrics.  
We start by recalling that
principal $S^1$-bundles over a manifold $K$ are classified (topologically) by their Euler
class in $H^2(K)$.  The Euler class is the ``primary obstruction'' to the existence of a
section, and satisfies the following two key properties:

\vskip 5pt

\noindent{\bf Fact 1:} The Euler class of a principal $S^1$-bundle $S^1\rightarrow K^\p
\rightarrow K$ is zero if and only if $K^\p \cong K\times S^1$ (i.e. $K^\p$ is the trivial
$S^1$-bundle).

\vskip 5pt

\noindent{\bf Fact 2:} If $f:L\rightarrow K$ is continuous, and $S^1\rightarrow K^\p\rightarrow
K$ is a principal $S^1$-bundle, let $S^1\rightarrow L^\p\rightarrow L$ be the pullback
principal $S^1$-bundle.  Then $e(L^\p) = f^*(e(K^\p))$, where $e(L^\p), e(K^\p)$
denote the Euler classes of the respective $S^1$-bundles, and $f^*:H^2(K; \mZ)\rightarrow
H^2(L;\mZ)$ is the induced map on the second cohomology.

\vskip 5pt

Since the manifolds we will be considering arise as principal $S^1$-bundles,
we now identify a cohomological obstruction for 
certain principal $S^1$-bundles to support a locally CAT(0) metric.

\begin{lemma}\label{Euler-class}
Let $K$ be a compact topological manifold supporting a locally CAT(0) metric,
and let $S^1\rightarrow K^\p \rightarrow K$ be a principal $S^1$-bundle over $K$ (so that $K^\p$
is also compact).  If $K^\p$
supports a locally CAT(0) metric, then $e(K^\p)$ has finite order in $H^2(K)$.
\end{lemma}

\begin{proof}
Since all spaces in the fibration are aspherical,
the associated long exact sequence in homotopy degenerates to a single short
exact sequence:
$$0\rightarrow \mZ \rightarrow \pi_1(K^\p) \rightarrow \pi_1(K) \rightarrow 0.$$
As $K^\p$ is compact, the action of $\pi_1(K^\p)$
on the CAT(0) universal cover $\tilde {K^\p}$ is by semi-simple isometries (i.e. for every
$g\in \pi_1(K^\p)$, there exists a $x\in \tilde {K^\p}$ satisfying $d(x, gx)\leq d(y,gy)$ for
all $y\in \tilde {K^\p}$).  Furthermore, $\pi_1(K^\p)$ contains $\mZ$ as a normal subgroup.
A well-known consequence of the Flat Torus theorem (see the discussion in \cite[pgs. 244-245]{bri})
implies that there exists a finite index subgroup $\La \leq \pi_1(K^\p)$ that centralizes the
$\mathbb Z$-subgroup, i.e. we have:

\[
\xymatrix@C=20pt@R=30pt{
\mZ \ar[r] & \pi_1(K^\p) \ar[r] & \pi_1(K)\\
\mZ \ar[r] \ar[u]^{=}& \La \ar[u]_{\text{Finite Index}}  \ar[r]& \La /\mZ}
\]

It is easy to see (by chasing the diagram) that there is an induced inclusion $\La /\mZ \hookrightarrow
\pi_1(K)$ which is also of finite index.  Let $L\rightarrow K$ be the finite cover corresponding to
$\La /\mZ \hookrightarrow \pi_1(K)$, and $L^\p \rightarrow K^\p$ the cover corresponding to
$\La \hookrightarrow \pi_1(K^\p)$.  We now obtain the commutative diagram of principal bundles
(see~\cite[Theorem II.7.1-(5)]{bri}):

\[
\xymatrix@C=20pt@R=30pt{
S^1 \ar[r] & K^\p \ar[r] & K\\
S^1 \ar[r] \ar[u]^{=}& L^\p \ar[u]  \ar[r]& L \ar[u]_{\text{Finite Cover}}}
\]
where both the ``top row'' and the ``bottom row'' are principal $S^1$-bundles.  Now observe that
the bottom row splits as a product, i.e. $L^\p  \cong L\times S^1$.  Indeed, this follows from the
fact that $\La$ centralizes the $\mZ$-factor, and splits as $\mZ \oplus \La / \mZ$, while acting
on the CAT(0) space $\tilde {K^\p}$. From {\bf Fact 1}, this implies that $e(L^\p)=0\in H^2(L)$.
From {\bf Fact 2}, and commutativity of the diagram, we get that $p^*(e(K^\p))=e(L^\p)=0$, where
$p^*: H^2(K)\rightarrow H^2(L)$ is the map induced by the covering projection $p:L\rightarrow K$.

On the other hand, recall that there is a transfer map on cohomology $T:H^*(L)\rightarrow H^*(K)$
associated with any finite covering $p:L\rightarrow K$.  This map has the property that $T\circ p^*
: H^*(K)\rightarrow H^*(K)$ is just multiplication by the degree of the covering map.  Hence if
$d$ denotes the degree of the covering map, we have that:
$$d\cdot e(K^\p) = (T\circ p^*)(e(K^\p)) = T (0) = 0 \in H^2(K)$$
implying that $e(K^\p) \in H^2(K)$ is a torsion element, and completing the proof of the Lemma.
\end{proof}

Keeping the notation from Section~\ref{initial:subsec}, 
let $N$ be a finite volume, non-compact, hyperbolic manifold, with all cusps diffeomorphic
to a torus times $[0, \infty)$, and let $\overline{N}$ be the compact manifold obtained by ``truncating
the cusps''.  Note that the boundary $\partial \overline{N}$ consists of a finite number of codimension
one tori, and the inclusion $i:\partial \overline{N} \hookrightarrow \overline{N}$ induces
the map $i^*:H^1(\overline{N})\rightarrow H^1(\partial \overline{N})$ on the first cohomology.
We will consider principal $S^1$-bundles over the double $D\overline{N}$.  

\begin{proposition}\label{cohom-cond}
Assume there exists a non-trivial cohomology class $\alpha \in H^1(\partial \overline{N})$ having
the property that $\langle \alpha \rangle \cap i^*(H^1(\overline{N}))=0 \subset H^1(\partial \overline{N})$.
Then there exists a manifold $M$, which is topologically a principal $S^1$-bundle over $D\overline{N}$, 
having the properties:
\begin{enumerate}
\item $M$ does {\bf not} support any locally CAT(0) metric.
\item $M$ is a graph manifold.
\end{enumerate}
\end{proposition}

\begin{proof}
It is well-known that the double $D\overline{N}$ supports a Riemannian metric of non-positive sectional 
curvature (see for example \cite[Theorem 1]{ArFa}). In view of Lemma \ref{Euler-class}, any principal $S^1$-bundle
whose Euler class has infinite order will {\bf not} support any locally CAT(0) metric. Since every class
in $H^2(D\overline{N})$ is realized as the Euler class of some principal $S^1$-bundle, we just need to
find a cohomology class of infinite order.

Consider the Mayer-Vietoris sequence in cohomology for the decomposition $D\overline{N}= \overline{N}_1
\cup _{\partial \overline{N}} \overline{N}_2$, where the $\overline{N}_i$ are the two copies of $\overline{N}$. 
We have:
$$
\xymatrix{
H^1(\overline{N}_1)\oplus H^1(\overline{N}_2)\ar[r]^<<<<<{i} & H^1(\partial \overline{N}) \ar[r]^j & H^2(D\overline{N})
\ar[r] & H^2(\overline{N}_1) \oplus H^2(\overline{N}_1)}
$$ Now by hypothesis
there exists an element $\alpha \in H^1(\partial \overline{N})$ having the
property that $\langle \alpha \rangle \cap i^*(H^1(\overline{N})) = 0$. If
$i_1, i_2$ denotes the inclusions of $\partial \overline{N}$ into $\overline{N}_1, \overline{N}_2$, we have that the first map in the Mayer-Vietoris
sequence above is given by $i:= i_1^* - i_2^*$, and hence the
non-trivial element $\alpha \in H^1(\partial \overline{N})$ has the property that
$\langle \alpha \rangle \cap i(H^1(\overline{N}_1)\oplus H^1(\overline{N}_2))=
\{0\}$. In particular, 
since $H^1(\partial \overline{N})$ is torsion-free, the subgroup
$j(H^1(\partial \overline{N})) \leq H^2(D\overline{N})$
contains an element of infinite order, namely $j(\alpha)$.  Let $M$
be the associated principal $S^1$-bundle over $D\overline{N}$; from the
discussion above, $M$ cannot support any locally CAT(0) metric.

So to conclude, we just need to argue that $M$ is a graph manifold. To see this, 
observe that $M$ naturally decomposes as a union $M= M_1\cup M_2$, where
each $M_i$ is the preimage of the respective $\overline{N}_i$ under the canonical map
$S^1\rightarrow M \rightarrow D\overline{N} = \overline{N}_1\cup _{\partial \overline{N}} \overline{N}_2$.
We now show that the $M_i$ are the pieces for the decomposition of $M$ as a graph manifold.
To do this, we need to understand the topology of the $M_i$.

From {\bf Fact 2}, we can compute the Euler class of the bundles
$S^1\rightarrow M_i\rightarrow \overline{N}_i$ by looking at the image of $\alpha \in H^2(D\overline{N})$
under the maps $H^2(D\overline{N})\rightarrow H^2(\overline{N}_i)$ induced by the inclusions
$\overline{N}_i\hookrightarrow D\overline{N}$.  But observe that these maps are exactly the ones
appearing in the Mayer-Vietoris sequence:
$$H^1(\partial \overline{N})\rightarrow H^2(D\overline{N}) \rightarrow H^2(\overline{N}_1)\oplus H^2(\overline{N}_2)$$
By exactness of the sequence, we immediately
obtain that $\rho (j(\alpha)) =0 \in H^2(\overline{N}_1)\oplus H^2(\overline{N}_2)$, and so the Euler
class of both $M_i$ is zero in the corresponding $H^2(\overline{N}_i)$.  Applying {\bf Fact 1}, we conclude that
each $M_i$ is the trivial $S^1$-bundle over $\overline{N}_i$, i.e.~each $M_i$ is homeomorphic to $\overline{N}_i \times S^1$.
Let us now endow each $M_i$ with the smooth structure induced by the product $\overline{N}_i \times S^1$ of smooth manifolds.
Now the only possible obstruction to $M$ being a graph manifold lies in the  
gluing map between
$M_1$ and $M_2$ being affine. 
However, if the gluing map is not affine, we can replace it by a 
homotopic affine diffeomorphism without affecting the Euler class of the corresponding principal $S^1$-bundle
(actually, if $n > 5$, we can replace the given gluing map by a $C^0$-isotopic affine diffeomorphism
without changing the topological type of the manifold $M$ -- see the discussion in Remark \ref{aff-diff:rem}). 
Then $M$ is indeed a graph manifold,
and this concludes the proof of the Proposition.
\end{proof}

In order to obtain the desired examples, we need to produce finite volume hyperbolic
manifolds $N$ so that the associated truncated $\overline{N}$ satisfies:
\begin{enumerate}[(1)]
\item all the boundary components of $\overline{N}$ are diffeomorphic to tori, and
\item there exists a non-trivial element $\alpha \in H^1(\partial \overline{N})$ which
satisfies $$\langle \alpha\rangle \cap i^*(H^1(\overline{N}))=\{0\} \subset
H^1(\partial \overline{N}).$$  
\end{enumerate}
The next step towards achieving this is to
turn the cohomological condition (2) to a homological condition, as explained
in the following Lemma.

\begin{lemma}\label{hom-cond}
Let $N$ be a finite volume hyperbolic manifold, so that the associated $\overline{N}$ satisfies
condition (1) above.  
Then $N$ also satisfies condition
(2) above if and only if 
$i_*\colon H_1(\partial\overline{N})\to H_1(\overline{N})$ is \emph{not} injective.
\end{lemma}

\begin{proof}
Since $H^1(\partial \overline{N})$ is a finitely generated torsion-free abelian group, property~(2) above is equivalent
to the fact that the index of $i^* (H^1(\overline{N}))$ in $H^1(\partial\overline{N})$ is infinite, so we need to prove
that this last condition is in turn equivalent to the 
fact that $\ker i_*\neq \{0\}$.

For a torus $T^k$, the Kronecker pairing induces an
isomorphism between $H^1(T^k)$ and ${\rm Hom}(H_1(T^k),\mZ)$.  Property (1) ensures
that this duality extends to an isomorphism between $H^1(\partial
\overline{N})$ and ${\rm Hom}(H^1(\partial \overline{N}),\mZ)$.
Moreover, it is easily seen that a subgroup $H$ of  ${\rm Hom}(H^1(\partial \overline{N}),\mZ)$
has infinite index if and only if there exists a non-trivial
element $\alpha'\in H_1(\partial\overline{N})$ such that $\varphi(\alpha')=0$ for every
$\varphi\in H$. Therefore, 
 the index of $i^* (H^1(\overline{N}))$ in $H^1(\partial\overline{N})$ is infinite
 if and only if there exists a non-trivial element $\alpha'\in H_1(\partial\overline{N})$ 
 such that
 \begin{equation}\label{injsurj}
 0=\langle i^*(\beta),\alpha'\rangle=\langle \beta,i_*(\alpha')\rangle \qquad {\rm for\ every}\ \beta\in H^1(\overline{N})\ .
 \end{equation}
An easy application of the Universal Coefficient Theorem shows that 
the Kronecker pairing between $H_1(\overline{N})$ and $H^1(\overline N)$ induces an epimorphism 
$H^1(\overline{N})\to {\rm Hom}(H_1(\overline{N}),\mathbb{Z})$,
so the condition described in Equation~\eqref{injsurj}
is equivalent to the fact that $\varphi (i_*(\alpha'))=0$ for every $\varphi\in {\rm Hom}(H_1(\overline{N}),\mZ)$, whence to the
fact that $i_*(\alpha')$ has finite order in $H_1(\overline{N})$.

We have thus shown that property~(2) above is equivalent to the existence of a non-trivial
element $\alpha'\in H_1(\partial\overline{N})$ such that $i_*(\alpha')$ has finite order in $H_1(\overline{N})$.
Since $H_1(\partial \overline{N})$ is torsion-free, this last condition 
holds if and only if the kernel of $i_*$ is non-trivial, concluding the proof.
\end{proof}

Now the advantage in changing to a homological criterion is that it
is easier to achieve geometrically. One
needs to find examples of finite volume, non-compact, hyperbolic
manifolds $N$ having the property that they contain an embedded
$S\hookrightarrow N$, where $S$ is non-compact surface with
finitely many cusps, and the embedding is proper. After truncation, 
this yields an element in
$H_1(\partial \overline{N})$, namely the element corresponding to
$\partial \bar S \hookrightarrow \partial \overline{N}$, having the
property that $i_*([\partial \bar S])=0 \in H_1(\overline{N})$.
Moreover, if $S$ is suitably chosen one may also ensures that 
$[\partial \bar S]\neq 0$ in $H_1 (\overline{N})$. 

One approach to finding such examples would be to construct
$N$ so as to contain a properly embedded totally geodesic non-compact
finite volume hyperbolic surface $\Sigma$. We refer the reader to the paper of McReynolds,
Stover, and Reid \cite{MRS} for arithmetical constructions
of such pairs $(N, \Sigma)$ in all dimensions.

An alternate approach is to ignore the geometry and to try to argue purely
topologically. Fixing a single boundary torus $T$ inside one of these truncated 
hyperbolic manifolds $\overline{N}$, we let $x_1, \ldots , x _{n-1}$ be a basis for the
first cohomology $H^1(T) \cong \mZ ^{n-1}$.
The following proposition was suggested to us by Juan Souto:

\begin{proposition}\label{juan-prop}
Assume that the cohomology classes $x_i$ for $1\leq i\leq n-2$ have the property that
$\langle x_i \rangle \cap i^*(H^1(\overline{N}))\neq \{0\}$. 
Then there exists an embedded
smooth surface with boundary $(\Sigma, \partial \Sigma) \hookrightarrow (\overline{N}, \partial \overline{N})$, having the 
following properties:
\begin{enumerate}[i)]
\item $\Sigma \cap \partial \overline{N} = \partial \Sigma$ is entirely contained in the boundary component $T$, and
\item the collection of curves $\partial \Sigma$ represent a non-zero class in $H_1(T)$
(and in particular, $\partial \Sigma \neq \emptyset$).
\end{enumerate}
\end{proposition}

\begin{proof}
Since each $\langle x_i \rangle \cap i^*(H^1(\overline{N}))\neq \{0\}$, we can find 
non-zero integers $r_1, \ldots , r_{n-2}$ with the property
that $r_i\cdot x_i \in i^*(H^1(\overline{N}))$ for $1\leq i \leq n-2$. 
Let $y_i\in H^1(\overline{N})$ be chosen so that $i^*(y_i) = r_i \cdot x_i$.
We will be considering elements in four (co)-homology groups, related via the 
commutative diagram:
\[
\xymatrix@C=20pt@R=30pt{
 H^1(\overline{N}) \ar[r]^{i^*} \ar[d]^{\cong}  & H^1(\partial \overline{N}) \ar[d]^{\cong}\\
 H_{n-1}(\overline{N}, \partial \overline{N}) \ar[r]^\partial & H_{n-2}(\partial \overline{N}) }
\]
where the vertical maps are isomorphisms given by Poincar\'e-Lefschetz duality, the top map
is induced by inclusion, and the bottom map is the boundary map.
We now proceed to use the cohomology classes $y_i$ to construct the surface $\Sigma$.

First, recall that for a smooth $k$-manifold $M$ (possibly with boundary), the Poincar\'e-Lefschetz dual of a 
$1$-dimensional cohomology class $x \in H^1(M)$ has a 
simple geometric interpretation. One can think of the element $x$ as a homotopy class of maps
into the classifying space $K(\mathbb Z, 1) = S^1$, with the trivial element corresponding to a constant map. 
Fixing a reference point $p\in S^1$, we can find a smooth map $f$ within the homotopy class with the 
property that $f$ is transverse to $p$. Then $f^{-1}(p)$ defines a smooth submanifold, which represents 
the Poincar\'e-Lefschetz dual to $x$. This will represent a class in either $H_{k-1}(M)$ or in 
$H_{k-1}(M, \partial M)$, according to whether $\partial M=\emptyset$ or $\partial M\neq \emptyset$. 
For example, in the special case consisting of the trivial cohomology class, one can perturb the constant
map to not contain $p$ in the image, so that the dual class is represented by the ``vacuous'' submanifold.

Let us apply this procedure to each of the cohomology classes $y_i\in H^1(\overline{N})$, obtaining corresponding
smooth maps $f_i: \overline{N} \rightarrow S^1$ transverse to $p$. Now the restriction of $f_i$ to $\partial \overline{N}$
will yield the Poincar\'e-Lefschetz dual to the cohomology class $i^*(y_i) = r_i \cdot x_i \in H^1(\partial \overline{N})$.
The cohomology $H^1(\partial \overline{N})$ decomposes as a direct sum of the cohomology of the individual 
boundary components, and by construction the class $i^*(y_i) = r_i \cdot x_i$ is purely supported on the 
$H^1(T)$ summand. Geometrically, this just says that the restriction of $f_i$ to any of the remaining 
boundary components is homotopic to a point, which we can take to be distinct from $p$. Using a collared 
neighborhood of each of the boundary components, we can effect such a homotopy, allowing us to replace
$f_i$ by a homotopic map which has the additional property that {\it $T$ is the only boundary component of
$\overline{N}$ whose image intersects $p$}. 

Taking pre-images of $p$ under these maps, we obtain a collection
of $(n-1)$-dimensional manifolds $W_1, \ldots ,W_{n-2}$ representing the dual homology classes in
$H_{n-1}(\overline{N}, \partial \overline{N})$. Moreover, each $W_i$ intersects $\partial \overline{N}$ in a collection of 
$(n-2)$-dimensional submanifolds $\partial W_i \subset T$, which represent the duals to the cohomology 
classes $r_i \cdot x_i \in H^1(T)$. Perturbing the pairs $(W_i, \partial W_i) \subset (\overline{N}, T)$ slightly, 
we may assume they are all pairwise transverse. This in turn ensures that the intersection 
$\Sigma= \cap_{i=1}^{n-2} W_i$ is a smooth submanifold. Since $\Sigma$ is the intersection of $n-2$
manifolds each of which has codimension one, we see that $\Sigma$ has codimension $n-2$ in the
$n$-dimensional manifold $\overline{N}$, i.e. $\Sigma$ is a surface. 
Since $T$ is the only boundary component which intersects any of the $W_i$, we have that $\partial \Sigma \subset T$ 
giving us (i).

So to conclude, we need to verify property (ii): that the family of curves defined by $\partial \Sigma$ represent a 
non-zero class in $H^1(T)$. But recall that $\partial \Sigma = \cap _{i=1}^{n-2} \partial W_i$, where each $\partial W_i$
is an $(n-2)$-dimensional submanifold of the $(n-1)$-dimensional torus $T$, representing the Poincar\'e dual to the
cohomology class $i^*(y_i) = r_i \cdot x_i \in H^1(T)$. Under Poincar\'e duality, the geometric intersection of cycles
corresponds to the cup product of the dual cocycles. As such, the collection of curves $\partial \Sigma$ represents
the Poincar\'e dual of the cup product 
$$\cup _{i=1}^{n-2} (r_i \cdot x_i) = \big(\prod r_i \big) \cdot \big( \cup _{i=1}^{n-2} x_i\big) \in H^{n-2}(T) \cong \mathbb Z^{n-1}.$$ 
We know that the cohomology ring $H^*(T)$ is an exterior algebra over the $x_i$, hence
the cup product $\cup _{i=1}^{n-2} x_i$ is non-zero. Since the coefficient $\prod r_i$ is a non-zero integer, the Poincar\'e 
dual of $[\partial \Sigma] \in H_1(T)$ is non-trivial. This implies that the homology class $[\partial \Sigma]$ is likewise
non-zero, establishing (ii), and concluding the proof of the Proposition.
\end{proof}

\begin{corollary}\label{not-inj}
The map $i_*\colon H_1(\partial \overline{N})\to H_1(\overline N)$ is not injective.
\end{corollary}
\begin{proof}
Fix a boundary component $T$ of $\overline{N}$, and choose a basis $x_1, \ldots , x_{n-1}$ for the
first cohomology $H^1(T) \cong \mZ ^{n-1}$. If any of the elements $x_1, \ldots ,x _{n-2}$ 
has the property that $\langle x_i \rangle \cap i^*(H^1(\overline{N}))=\{0\}$, then we are done by
Lemma \ref{hom-cond}. So we can assume that $\langle x_i \rangle \cap i^*(H^1(\overline{N}))\neq \{0\}$ 
for each $1\leq i \leq n-2$, allowing us to apply Proposition \ref{juan-prop}, whence the conclusion again.
\end{proof}

Putting together Proposition~\ref{cohom-cond}, Lemma~\ref{hom-cond}
and Corollary~\ref{not-inj}, we can now establish:

\begin{theorem}\label{bundle-example}
Let $N$ be any finite volume, non-compact, hyperbolic manifold, with all cusps diffeomorphic
to a torus times $[0, \infty)$, and let $\overline{N}$ be the compact manifold obtained by ``truncating
the cusps''. Then one can find a graph manifold, arising as a principal $S^1$-bundle over the 
double $D\overline{N}$, which does \emph{not} support a locally CAT(0) metric.
\end{theorem}


To conclude, we recall that there exist examples, in all dimensions
$\geq 3$, of non-compact finite volume hyperbolic manifolds with toric cusps 
(see \cite{MRS}).
From Theorem \ref{bundle-example}, we immediately deduce:

\begin{corollary}
There are examples, in all dimensions $\geq 4$, of principal $S^1$-bundles which are
graph manifolds, but do {\bf not} support any locally CAT(0) metric. 
\end{corollary}

\section{Irreducible examples}\label{strirr:sub}
In this Section we prove that in any dimension $\geq 4$
there exist irreducible graph manifolds
which do not support any locally CAT(0) metric. In fact, we provide 
examples of irreducible graph manifolds whose fundamental groups
are not CAT(0). Usually, a group is defined to be CAT(0)
if it acts properly, cocompactly and isometrically on a CAT(0) space (see e.g.~\cite{ger4,Swenson,Swenson-Papa,Alibegovich-Best,Geo-Ont,Rouane}).
Our Definition~\ref{CAT0} below is 
slightly less restrictive.

Let us briefly recall some definitions and results from~\cite[Chapter II.6]{bri}.
Let $G$ be a group acting by isometries on the complete geodesic metric space
$X$.
For every $g\in G$ the \emph{translation length} of $g$ is defined by setting
$$\tau (g)=\inf \{d(x,g (x))\, |\, x\in X\}\ .$$ 
We also set
$$\mn (g)=\{x\in X\, |\, d(x,g(x))=\tau (g)\}\subseteq X\ .$$
If $H$ is a subgroup of $\Gamma$, then we set 
$\mn (H)=\bigcap_{\gamma \in H}\mn (\gamma)\subseteq X$.
An element $g\in G$ is semisimple if $\mn (g)$ is non-empty, i.e.~if
the infimum in the definition of $\tau(g)$ is a minimum. It is well-known that, if $G$ acts
cocompactly on $X$, then every element of $G$ is semisimple.
Following~\cite[Chapter I.8]{bri}, we say that
the action of $G$ on $X$ is \emph{proper} if every point $x\in X$ has a neighbourhood
$U\subseteq X$ such that the set $\{g\in G\, |\, g(U)\cap U\neq \emptyset\}$
is finite. As observed in~\cite{bri}, it is probably more usual 
to say that $G$ acts properly on $X$
if the set $\{g\in G\, |\, g(K)\cap K\neq \emptyset\}$ is finite for every compact
set $K\subseteq X$. The definition we are adopting here implies
that every compact subset $K\subseteq X$ has a neighbourhood $U_K$ such that 
$\{g\in G\, |\, g(U_K)\cap U_K\neq \emptyset\}$
is finite, so the two definitions coincide if $X$ is a proper metric space
(i.e.~if $X$ is locally compact or, equivalently, if every bounded subset
is relatively compact in $X$).

\begin{definition}\label{CAT0}
 Let $G$ be a group. Then $G$ is CAT(0) if it acts properly via semisimple
isometries on a complete CAT(0) space. 
\end{definition}

By Cartan-Hadamard Theorem for metric spaces (see~\cite[Chapter II.4]{bri}),
the universal covering of a complete locally CAT(0) space is complete
and globally CAT(0),
so if a compact topological space $M$ supports a locally CAT(0) metric, then
$\pi_1(M)$ is a CAT(0) group.

\medskip

Let us now come to our construction.
Let $N$ be a complete finite-volume hyperbolic $n$-manifold
with toric cusps, $n\geq 3$, and set $V=\overline{N}\times S^1$, where $\overline{N}$ is as usual
the natural compactification of $N$. We denote by $n$ the dimension of $V$.
We are going to show that one may always choose affine gluing maps between the boundary components
of two copies of $V$ in such a way that the 
resulting graph manifold $M$ is irreducible, and the
fundamental group $\pi_1(M)$ 
is not CAT(0).
As a consequence, irreducible graph manifolds which do not support any locally CAT(0) metric exist in every 
dimension $\geq 4$. 

Let $T^*_1,\ldots,T^*_r$ be the boundary components of $V$.
We denote by $V^+$, $V^-$ two copies of $V$, and by $T_i^+$ (resp.~by $T_i^-$) the boundary component
of $V^+$ (resp.~of $V^-$) corresponding to $T^*_i$, $i=1,\ldots,r$. For every $i=1,\ldots,r$
we fix an affine diffeomorphism $\psi_i\colon T^+_i\to T^-_i$, we denote by 
$M$ the graph manifold
obtained by gluing $V^+$ and $V^-$ along the $\psi_i$, and by $T_i\subseteq M$ the torus
corresponding to $T_i^+\subseteq \partial V^+$ and $T_i^-\subseteq \partial V^-$.

We denote by $\Gamma$ the fundamental group $\pi_1(M)$ of $M$, and we suppose
that $\Gamma$ acts properly by semisimple isometries on the complete CAT(0) space $X$.
For every $i=1,\ldots,r$ we also denote by $A_i$ (a representative of the conjugacy class of) the subgroup $\pi_1(T_i)<\Gamma$. 
Following~\cite{leeb}, we briefly describe
the Euclidean scalar product induced by the metric of $X$ on each $H_1(A_i)\cong A_i\cong \mZ^{n-1}$, $i=1,\ldots,r$. 

Since $A_i\cong \mZ^{n-1}$, 
by the Flat Torus Theorem the subset $\mn (A_i)$ splits as a metric  product 
$\mn (A_i)=Y_i\times E^{n-1}$, where $E^k$ is the Euclidean $k$-dimensional space
(see \emph{e.g.}~\cite[Chapter II.7]{bri}).
Moreover, $A_i$ leaves $\mn (A_i)$ invariant, and the action of every $a\in A_i$ on $\mn (A_i)$ 
splits as the product of the identity on $Y_i$ and a non-trivial translation $v\mapsto v+v_a$ on $E^{n-1}$.
If $l_1,l_2$ are elements of $H_1 (A_i)$ we set 
$$
\langle l_1,l_2\rangle_i =\langle v_{a_1},v_{a_2}\rangle\ ,
$$
where $a_j$ is the element of $A_i\cong H_1 (A_i)$ corresponding to
$l_j$, and $\langle \cdot,\cdot \rangle$ denotes the standard scalar product of $E^{n-1}$.
It is readily seen that $\langle \cdot ,\cdot\rangle_i$
is indeed well-defined. Moreover,
the norm $\| l \|_i=\sqrt{\langle l,l\rangle_i}$ of any element $l\in H_1(A_i)$
coincides with the translation length of the corresponding element $a\in A_i<\Gamma$,
so if $l_1,l_2\in H_1(T_i)$ correspond to the elements $a_1,a_2\in A_i$ we have
$$
2\langle l_1,l_2\rangle_i={\tau(a_1\circ a_2)^2-\tau (a_1)^2-\tau(a_2)^2}\ .
$$

Let us fix a representative $\Gamma^\pm$ of the conjugacy class of
the subgroup $\pi_1 (V^\pm)$ of $\pi_1 (M)\cong \Gamma$. We also choose
the subgroups $A_i$ corresponding to the tori $T_i$ in such a way that
$A_i<\Gamma^\pm$ for every $i=1,\ldots,r$. 
We denote by $f^\pm\in H_1(\Gamma^\pm)$ the class represented by the fiber
of $V^\pm$, i.e.~the element of $H_1 (\Gamma^\pm)=H_1(\pi_1(N))\oplus H_1(\pi_1(S^1))$ 
 corresponding to the positive generator of $H_1(\pi_1(S^1))=\mZ$. If $i^\pm_*\colon \bigoplus_{i=1}^r H_1(A_i)\to H_1 (\Gamma^\pm)$ is 
 the map induced by the inclusions $A_i\hookrightarrow \Gamma^\pm$,
then for every $i=1,\ldots,r$ 
 there exists a unique element  $f_i^\pm\in H_1 (A_i)$ such that 
 $i_*^\pm (f_i^\pm)=f^\pm$. Observe that our definitions imply that $M$ is irreducible if and only if
 $f_i^+\neq \pm f_i^-$ for every $i=1,\ldots,r$.
Lemma~\ref{orto:lem} and Proposition~\ref{nonCAT0-bis} below are inspired 
by the proof of~\cite[Theorem 3.7]{kaplee1}:

\begin{lemma}\label{orto:lem}
For every $i=1,\ldots,r$ let $b_i$ be an element of $H_1 (A_i)$ such that
$$
i^\pm_* (b_1+\ldots+b_r)=0\ .
$$
Then
$$
\sum_{i=1}^r \langle b_i,f_i^\pm\rangle_i =0\ .
$$ 
\end{lemma}
\begin{proof}
Let $\phi^\pm\in\Gamma^\pm$ be the element corresponding to $({\rm Id},1)$ under the identification
$$
\Gamma^\pm=\pi_1 (V^\pm)=\pi_1 (\overline{N})\times \pi_1 (S^1)=\pi_1 (\overline{N})\times\mathbb{Z}\ .
$$
By construction we have 
$\phi^\pm\in \bigcap_{i=1}^r A_i\subseteq \Gamma^\pm$, and
the image of $\phi^\pm$ under the Hurewicz homomorphism
$\Gamma^\pm\to H_1(\Gamma^\pm)$ coincides with $f^\pm$.
   
Since $\phi^\pm$ lies in the center of $\Gamma^\pm$ the set $\mn (\phi^\pm)\subseteq X$ is
$\Gamma^\pm$-invariant. Moreover, the action of $\Gamma^\pm$ preserves 
the isometric splitting $\mn (\phi^\pm)=W\times E^1$, so
we have an induced representation 
$\rho\colon \Gamma^\pm\to {\rm Isom} (W)\times {\rm Isom} (E^1)$. If $\rho_0\colon \Gamma^\pm\to
{\rm Isom} (W)$, $\rho_1\colon \Gamma^\pm\to
{\rm Isom} (E^1)$ are the components of $\rho$, then $\rho_0 (\phi^\pm)$ is the identity of $W$, 
while $\rho_1 (\phi^\pm)$
is a non-trivial translation. As a consequence, since for every $\gamma\in \Gamma^\pm$ 
the isometries $\rho_1 (\gamma)$
and $\rho_1 (\phi^\pm)$ commute, the representation $\rho_1$ takes values in the abelian group
of translations of $E^1$, which can be canonically identified with $\mR$. 
Therefore, 
the homomorphism $\rho_1$ factors through
$H_1 (\Gamma^\pm)$, thus defining a homomorphism $\overline{\rho}_1\colon H_1(\Gamma^\pm)\to \mR$.

Let us now observe that, since $\phi^\pm\in A_i$, we have $\mn (A_i)\subseteq \mn (\phi^\pm)=W\times E^1$,
so in order compute the translation length of elements of $A_i$ it is sufficient to consider their 
action on $W\times E^1$. Therefore, for every $a\in A_i$ we have
$\tau (a)^2=\tau_W(\rho_0(a))^2+\rho_1(a)^2$, where we denote by $\tau_W$ 
the translation length of elements of ${\rm Isom} (W)$, and we recall
that we are identifying the group of translations of $E^1$ with $\mR$.
We now let $\beta_i\in A_i$ be a representative of $b_i\in H_1(A_i)$, and proceed
to evaluate the scalar product $\langle b_i,f_i^\pm\rangle_i$. We know that:
$$
2\langle b_i,f_i^\pm\rangle_i=\tau(\phi^\pm\circ\beta_i)^2-\tau(\phi^\pm)^2-\tau(\beta_i)^2\ .
$$
Considering the terms on the right hand side, we recall that
$\phi^\pm\in A_i$ is a representative of $f_i^\pm\in H_1(A_i)$, and hence 
we have $\tau(\phi^\pm)^2=\rho_1(\phi^\pm)^2$. Using the product structure on $W\times E^1$,
the remaining two terms are $\tau(\beta_i)^2=\tau_W(\rho_0 (\beta_i))^2+\rho_1(\beta_i)^2$, and
$\tau (\phi^\pm	\circ\beta_i)^2=\tau_W(\rho_0 (\beta_i))^2+ (\rho_1(\phi^\pm)+\rho_1(\beta_i))^2$.
Substituting these into the expression and simplifying, we obtain that

$$
2\langle b_i,f_i^\pm\rangle_i=2\rho_1(\phi^\pm)\rho_1(\beta_i)=
2\rho_1(\phi^\pm)\overline{\rho}_1(i^\pm_*(b_i))\ .
$$
Summing over all $i$, we deduce that
$$
\sum_{i=1}^r \langle b_i,f_i^\pm\rangle=\rho_1 (\phi^\pm)\cdot \sum_{i=1}^r \overline{\rho}_1 (i^\pm_* (b_i)) 
=\rho_1(\phi^\pm)\cdot \overline{\rho}_1 \left(i^\pm_* \left(\sum_{i=1}^r b_i\right)\right)=0\ ,
$$
whence the conclusion. 
\end{proof}

\begin{proposition}\label{nonCAT0-bis}
There exists a choice for the gluing maps $\psi_i\colon T_i^+\to T_i^-$ such that the following conditions hold:
\begin{enumerate}
\item
the graph manifold $M$ obtained by gluing $V^+$ and $V^-$ along
the $\psi_i$'s is irreducible;
\item
the group $\Gamma=\pi_1 (M)$ is not CAT(0)
(in particular, $M$ does not admit any locally CAT(0) metric).
\end{enumerate}
\end{proposition}
\begin{proof}
Let $Y_1,\ldots,Y_r$ be the boundary components of $\overline{N}$.
By Corollary~\ref{not-inj}, there exist elements $b'_i\in H_1(Y_i)$, $i=1,\ldots,r$, such that
$0\neq b_1'+\ldots+b_r'\in H_1(Y_1)\oplus\ldots\oplus H_1(Y_r)=H_1(\partial \overline{N})$, and
$j_*(b_1'+\ldots+b_r')=0$ in $H_1(\overline{N})$, where $j_*$ is induced by the inclusion
$\partial\overline{N}\hookrightarrow \overline{N}$.
Recall that $V^\pm=\overline{N}\times S^1$, so that we have natural identifications
$T_i^\pm=Y_i\times S^1$ and 
$H_1(T^\pm_i)=H_1(Y_i\times S^1)\cong H_1(Y_i)\oplus H_1(S^1)$, 
$i=1,\ldots,r$. Under these identifications, every affine diffeomorphism $\psi_i\colon T_i^+\to T_i^-$ induces an isomorphism
$$
 (\psi_i)_*\colon H_1(Y_i)\oplus H_1(S^1)\to
H_1(Y_i)\oplus H_1(S^1)\ .
$$

Let us denote by $\lambda$ the positive generator of $H_1 (S^1)$.
For every $i=1,\ldots,r,$ we choose the diffeomorphism 
$\psi_i\colon T_i^+\to T_i^-$ as follows. Let $I=\{i\, |\, b_i'\neq 0\} \subset \{1,\ldots, r\}$,
and observe that our assumptions ensure that $I$ is non-empty. Then:
\begin{enumerate}
 \item 
if $i\notin I$, we only ask that the gluing $\psi_i$ is transverse, i.e.~that
$(\psi_i)_* (0,\lambda)\neq (0,\pm \lambda)$,
\item
if $i\in I$, we choose a positive integer $n_i$ and we let $\psi_i$
be an affine diffeomorphism such that $(\psi_i)_* (v,0)=(v,0)$
for every $v\in H_1(Y_i)$ and $(\psi_i)_* (0,\lambda)=(n_i b_i',\lambda)$.
Also in this case, our choice ensures that $\psi_i$ is transverse.
\end{enumerate}
Recall that $T_i$ is the
toric hypersurface corresponding to $T_i^+$ and $T_i^-$ in the resulting graph manifold $M$, and that we fixed a representative $A_i$ in the conjugacy class
of $\pi_1(T_i)$ in $\pi_1(M)$. 
We denote by $b_i\in H_1(A_i)$ the unique element 
corresponding to the elements
$(b'_i,0)\in H_1(T_i^+)$ and $(b'_i,0)=(\psi_i)_* (b'_i,0)\in H_1(T^-_i)$ under the 
canonical identifications
$H_1(T_i^+)\cong H_1(T_i)\cong H_1(A_i)$ and 
$H_1(T_i^-)\cong H_1(T_i)\cong H_1(A_i)$. Observe that $b_i=0$ if and only if 
$b'_i=0$, i.e.~if and only if $i\notin I$. Moreover, for every $i\in I$ we have
$f_i^+=f_i^-+n_i b_i$.

Let $M$ be the graph manifold obtained by gluing $V^+$ and
$V^-$ along the $\psi_i$'s. By construction, $M$ is irreducible.
Let us suppose by contradiction that $\pi_1(M)$ acts properly by semisimple isometries
on the complete CAT(0) space $X$. We denote by $\langle \cdot , \cdot \rangle_i$
the scalar product induced on $H_1(A_i)$ by the metric of $X$.  
Since $j_*(b_1'+\ldots+b_r')=0$ in $H_1(\overline{N})$, we have
$i^{\pm}_*(\sum_{i=1}^r b_i )=0$ in $H_1(\Gamma^\pm)\cong H_1(V^\pm)$, so
Lemma~\ref{orto:lem} implies that 
\begin{align*}
0&=\sum_{i=1}^r \langle f_i^+,b_i\rangle_i =\sum_{i=1}^r \langle f_i^-+n_ib_i,b_i\rangle_i\\
&=\sum_{i=1}^r \langle f_i^-,b_i\rangle_i+\sum_{i=1}^r n_i\langle b_i,b_i\rangle_i \\
&=\sum_{i=1}^r n_i\|b_i\|^2_i \, ,
\end{align*}
a contradiction since $n_i>0$ and $b_i\neq 0$ for every $i\in I$. 
We have thus shown that $\pi_1(M)$ cannot 
act properly via semisimple isometries on a complete CAT(0) space,
and this concludes the proof.
\end{proof}

\begin{corollary}\label{examples:cor}
Let $n\geq 4$. Then,
there exist infinitely many closed irreducible graph $n$-manifolds having a non-CAT(0)
fundamental group.
In particular, there exist infinitely many closed irreducible graph $n$-manifolds
which do {\bf not} support any locally CAT(0) metric.
\end{corollary}

\begin{proof}
Let us fix an integer $m\geq 3$.
It is proved in~\cite{MRS} that there exist infinitely many complete finite-volume hyperbolic
$m$-manifolds with toric cusps. If $N$ is any such manifold, Proposition~\ref{nonCAT0-bis}
shows that there exists an irreducible graph manifold $M$ which does not support any locally CAT(0) metric
and decomposes as the union of two pieces $V^+$ and $V^-$, each of which is diffeomorphic to $\overline{N}\times S^1$.

In order to conclude it is sufficient to show that the diffeomorphism type of $M$ completely determines
the hyperbolic manifold $N$, so that the infinite family of hyperbolic manifolds provided by~\cite{MRS}
gives rise to the infinite family of desired examples. 
However, 
Theorem~\ref{iso-preserve:thm} implies that the diffeomorphism type of $M$ determines the isomorphism type of the fundamental
group of $V^\pm$.
Since $\pi_1 (N)$ is equal to the quotient
of $\pi_1 (V^\pm)$ by its center
(see Remark~\ref{center:rmk}), 
the conclusion follows 
by Mostow rigidity. 
\end{proof}

\begin{remark}\label{infinite:rem}
Even when starting with a fixed pair of pieces, one can still obtain an infinite family of 
irreducible graph manifolds with non-CAT(0) fundamental group. 
For example, let $N$ be a hyperbolic knot complement in $S^3$, set $V^+=V^-=\overline{N}\times S^1$
and denote by $T^+$ (resp.~$T^-$) the unique boundary component of $V^+$ (resp.~of $V^-$).
The boundary of a Seifert surface for $K$ defines an element $b'\in H_1(\partial\overline{N})$
which bounds in $\overline{N}$, whence an element $b\in H_1(T^\pm)$ such that
$i_*(b)=0\in H_1(V^\pm)$. Let $M(n)$ be the irreducible graph manifold obtained 
by gluing the base of $V^+$ to the base of $V^-$ via the identity of $\partial\overline{N}$,
and by gluing the fibers of $V^+$ and $V^-$ in such a way that
$f^+=f^-+nb$ in 
$H_1(T)$, where $T$ is the internal wall in $M(n)$ corresponding
to $T^+$ and $T^-$.
As described in the proof of Proposition~\ref{nonCAT0-bis}, 
for every positive integer $n$ the group $\pi_1(M(n))$
is not CAT(0).
Moreover, as explained in Remark~\ref{infinitelymany:rem}, the proof of 
Theorem~\ref{infinitelymany}
can be adapted to show that among the fundamental groups of the $M(n)$'s,
there are infinitely many non-isomorphic groups. 
\end{remark}

\begin{remark}
Let $N$ be a complete finite-volume hyperbolic manifold
with toric cusps. We have proved in Proposition~\ref{nonCAT0-bis} that
there exist ``twisted doubles'' of $\overline{N}\times S^1$ which provide
examples of {\it closed} irreducible graph manifolds not admitting any locally CAT(0) metric.
However, in principle one can use a similar construction to also get examples with non-empty boundary. 

Indeed, if $T_1\cup\ldots\cup T_k\subseteq \partial\overline{N}\times S^1$
is a family of boundary tori such that the map $i_*\colon H_1(T_1\cup\ldots\cup T_k)\to 
H_1(\overline{N}\times S^1)$ is not injective, then
the proof of Proposition~\ref{nonCAT0-bis} shows that the obstruction to putting a global
nonpositively curved metric on such twisted doubles is concentrated near the 
gluing tori $T_1,\ldots, T_k$. In other words, if $\partial(\overline{N}\times S^1)$
contains some boundary component other than $T_1,\ldots,T_k$, we can easily
construct irreducible graph manifolds just by gluing two copies
of $\overline{N}\times S^1$ along the corresponding copies of $T_1,\ldots, T_k$,
thus obtaining examples of irreducible graph manifolds, with non-empty boundary,
and which do not support any locally CAT(0) metric.
\end{remark}


%
%
%

\chapter{Directions for future research}\label{open:sec}

Our purpose in this monograph was to initiate the study of the class of 
high-dimensional graph manifolds. In this final chapter, we collate 
various problems that came up naturally in the course of this work,
and could serve as directions for future research.

\section{Further algebraic properties}

In Chapter \ref{groups:sec}, we established various algebraic properties of the fundamental
groups of high dimensional graph manifolds. Most of the results followed fairly easily from
the structure of such groups, expressed as a graph of groups. In contrast, there
are a number of interesting properties of groups whose behavior under amalgamations
is less predictable. It would be interesting to see which of these properties hold for the
class of graph manifold groups. For concreteness, we identify some properties which
we think would be of most interest:

\begin{problem} Are fundamental groups of high dimensional graph manifolds Hopfian? Are
they residually finite? Are they linear? What if one additionally assumes the graph manifold
is irreducible?
\end{problem}

A slightly different flavor of problems come from the algorithmic viewpoint. We showed that
the word problem is solvable for the $\pi_1(M)$ of irreducible graph manifolds. Some other
algorithmic problems one can consider include:

\begin{problem}
Is the conjugacy problem solvable for fundamental groups of high dimensional graph manifolds?
Is the isomorphism problem solvable within the class of graph manifold groups?
\end{problem}

Finally, one can also ask for a better understanding of the outer automorphism group
$\out (\pi_1(M))$, and of how it relates to the topology of $M$. For instance:

\begin{problem}
Is the group $\out(\pi_1(M))$ always infinite? What can be said about the structure of $\out(\pi_1(M))$?
\end{problem}

\begin{problem}
If we have a finite subgroup in $\out(\pi_1(M))$, can we lift it back to a finite subgroup of $\diff(M)$?
\end{problem}

This last problem is an analogue of the classic Nielson realization problem. Note that,
by Theorem \ref{smrigidity:thm}, the natural map $\diff(M) \rightarrow \out(\pi_1(M))$ is surjective. 
So we can always lift back individual elements from $\out(\pi_1(M))$ to $\diff(M)$, and the problem
asks whether we can choose the lifts in a compatible manner.

\section{Studying quasi-isometries}\label{qi-open:sec}

One of our main results, Theorem \ref{qirigidity:thm}, gives us some structure theory for groups which
are quasi-isometric to the fundamental group of an irreducible graph manifold. Specializing
to the class of graph manifold groups, this result gives us a necessary condition for deciding
whether two such groups $\pi_1(M_1)$ and $\pi_1(M_2)$ are quasi-isometric to each other: loosely
speaking, the two graph manifolds $M_i$ must essentially be built up from the same collection of 
pieces (up to commensurability), with the same patterns of gluings. The only
distinguishing feature between $M_1$ and $M_2$ would then be in the actual gluing maps
used to attach pieces together. This brings us to the interesting:

\begin{problem}\label{glueing-QI}
To what extent do the gluing maps influence the quasi-isometry type of the resulting graph 
manifold group? More concretely, take pieces $V_1$ and $V_2$ each having exactly one 
boundary component, and let $M_1,M_2$ be a pair of irreducible graph manifolds obtained 
by gluing $V_1$ with $V_2$. Must the the fundamental groups of $M_1$ and $M_2$ be 
quasi-isometric? 
\end{problem}

In order to prove that the answer is positive,
one could try to follow the strategy described in~\cite{behr}, as follows:
\begin{enumerate}
 \item 
Define a \emph{flip manifold} as a graph manifold whose gluing maps
are such that fibers are glued to parallel copies of the traces
at the toric boundaries
of the adjacent base
(this definition generalizes the one given in~\cite{kaplee}). 
\item
Observe that since $V_1$ and $V_2$ can be glued to provide irreducible graph manifolds,
they can also be glued to obtain a flip manifold $M$. Note however that such a manifold
is not uniquely determined by $V_1$ and $V_2$.
\item
Prove that the universal covering of $M_i$, $i=1,2$, 
is quasi-isometric to the universal covering of $M$.
\end{enumerate}
The analogue of Step (3) for pieces with $2$-dimensional bases 
is proved in Section~2 of~\cite{kaplee}. However,
the argument given there does not apply
in our case, since our bases are not negatively curved. 

In Theorem \ref{qi-mnpc}, we argued that a {\it labelled} version of the Bass-Serre tree 
associated to an irreducible graph manifold (with each vertex labelled by the commensurability
class of the hyperbolic factor in the corresponding vertex group) provides a
quasi-isometric invariant. However, it is shown in Remark~\ref{not-suff-rem}
that this is {\it not} a complete invariant,
i.e. that there exist a pair of irreducible graph manifolds with the same invariant, but 
which are nevertheless not quasi-isometric.
We can ask:

\begin{problem}
Can one devise a more sophisticated labeling in order to get a \emph{complete} quasi-isometric invariant?
\end{problem}

It would be interesting to see how the quasi-isometry classes behave with respect
to curvature conditions. For instance, we could ask:

\begin{problem}
Is there a pair of irreducible graph manifolds with quasi-isometric fundamental groups, with the
property that one of them supports a locally CAT(0) metric, but the other
one cannot support any locally CAT(0) metric? 
\end{problem}

Note that if the quasi-isometry class ends up being independent of the gluing maps used
(among the ones giving irreducible graph manifolds),
then by varying the gluing maps, one can give an affirmative answer to this last question.

Now all the quasi-isometry results we have are for the class of irreducible graph manifolds.
The key result we use is that, for this class of graph manifolds, all the walls are {\it undistorted}
in the universal cover (see Chapter \ref{strongirr:sec}). This in turn can be used to show that 
quasi-isometries must send walls to walls (up to finite distance), and hence chambers to 
chambers (see Chapter \ref{preserve:sec}). Trying to generalize these, we can formulate the
following question, which was suggested to us by C.~Drutu and P.~Papasoglu:

\begin{problem}
For a graph manifold $M$, assume that a wall $W$ in the universal cover $\tilde M$ is not 
too distorted (say, polynomially distorted). What additional hypotheses are sufficient to ensure
that quasi-isometries send walls to (bounded distance from) walls? And how can we choose
gluings in order to ensure these hypotheses are satisfied?
\end{problem}

For example, one possibility is to assume that all fibers have dimension which is small relative 
to the degree of polynomial growth. It seems like this constraint might be enough to show that
walls are rigid under quasi-isometries. Finally, we have the most general (and consequently, the
most difficult):

\begin{problem}
Develop methods to analyze quasi-isometries of general graph manifolds (i.e. without 
the assumption of irreducibility).
\end{problem}

Notice that in the proof of Theorem~\ref{qirigidity:thm} we studied each vertex stabilizer separately.
It might be possible to obtain additional information by studying
the interaction between vertex stabilizers of adjacent vertices.

\begin{problem}
 Is it possible, under additional hypotheses, to obtain a better description of the vertex stabilizers?
\end{problem}

A possible strategy to achieve this is to use the fact that walls admit ``foliations'' which are coarsely invariant under quasi-isometries,
namely those given by fibers of the adjacent chambers.
In order to obtain additional information out of this, one probably has to assume that the dimension of the fibers is half
that of the walls.

Finally, it is natural to ask whether versions of our quasi-isometric rigidity results hold for extended graph manifolds as well. The very first step in this direction would be the quasi-isometric rigidity of the fundamental group of a single piece. But the fundamental group of a single piece is just the product of a free group and an abelian group, which leads to the following natural question.

\begin{problem}
What can one say about a group $G$ quasi-isometric to $F_k\times \mathbb{Z}^d$, where $F_k$ is the free group on $k\geq 2$ generators? Is it true that $G$ is virtually of the form $F_{k'}\times \mathbb{Z}^d$?
\end{problem}

Notice that quasi-isometric rigidity is known for both abelian groups (see \cite{gropol}) and for free groups (see~\cite{Sta1} and ~\cite{Dun}).

\section{Non-positive curvature and differential geometry}

We have already given three different constructions of high dimensional graph
manifolds which cannot support a locally CAT(0) metric (see 
Section \ref{noncat0-easy:subsec} and Chapter \ref{construction2:sec}),
and hence no Riemannian metric of non-positive sectional curvature. It would
be interesting to identify precise conditions for such metrics to exist:

\begin{problem}
Find necessary and sufficient conditions for a graph manifold $M$ to
\begin{enumerate}[(i)]
\item support a Riemannian metric of non-positive sectional curvature, or
\item support a locally CAT(0)-metric.
\end{enumerate}
\end{problem}

It is not even clear whether or not items (i) and (ii) above are really
distinct:

\begin{problem}
Assume the high dimensional graph manifold $M$ supports a locally CAT(0) metric.
Does it follow that $M$ supports a Riemannian metric of non-positive sectional
curvature?
\end{problem}

Note that, for the classical $3$-dimensional graph manifolds, Buyalo and Svetlov \cite{BS}
have a complete criterion for deciding whether or not such a manifold supports a 
non-positively curved Riemannian metric (see also \cite{leeb}). 
Some partial results in dimension $=4$ appear in \cite{BK}.

Concerning the second problem, in the 
$3$-dimensional setting, there is no difference between Riemannian and metric 
non-positive curvature (see for instance \cite[Section 2]{DJL}). However, in all
dimensions $\geq 4$, there exist manifolds supporting locally CAT(0) metrics which
do {\it not} support Riemannian metrics of non-positive curvature (see the discussion
in \cite[Section 3]{DJL}). For the class of graph manifolds, the situation is relatively
tame, and one might expect the two classes to coincide.

\vskip 10pt

Next, we discuss a question about ordinary hyperbolic manifolds. One can ask
whether examples exist satisfying a strong form of the cohomological condition
appearing in Proposition \ref{juan-prop}. More precisely:

\begin{problem}
Can one find, in each dimension $n\geq 4$, an example of a
truncated finite volume hyperbolic $n$-manifold $N$, with all boundary components
consisting of tori, such that at least one boundary component 
$T$ has the property that the map $i_*:H_1(T) \rightarrow H_1(N)$ induced 
by inclusion has a non-trivial kernel?
\end{problem}

Note that such examples clearly exist in dimensions $=2, 3$. 
A recent result by Kolpakov and Martelli ensures that one-cusped hyperbolic manifolds with toric cusp
exist also in dimension 4~\cite{MarKo}.
Moreover, if one could
construct a finite volume hyperbolic $n$-manifold having a single cusp with toral
cross section, then Proposition \ref{juan-prop} could be used to show that the
corresponding $\ker (i_*)$ is non-trivial. The problem of constructing hyperbolic 
manifolds with a single cusp is, however, still open.

\vskip 10pt

We have already discussed the behaviour of fundamental groups of graph manifolds with respect
to several conditions encoding nonpositive curvature for groups: for example, we showed that
our groups are often non-relatively hyperbolic, and that, in general, they do not act properly via
semisimple isometries on CAT(0) spaces. An interesting question, which was suggested to the authors
by the anonymous referee, is the following:

\begin{problem}
 Does there exist an (irreducible) graph manifold whose fundamental group does not admit any proper action on a proper CAT(0) space?
\end{problem}

A positive answer to this question would support the feeling that fundamental groups of graph manifolds are genuinely outside of the world of non-positively
curved groups.

\vskip 10pt

Our next question comes from a differential geometric direction.
Intuitively, one can think of high dimensional graph manifolds as being ``mostly''
non-positively curved: the difficulties in putting a global metric of non-positive
curvature is concentrated in the vicinity of the gluing tori, which are a collection
of smooth, pairwise disjoint, codimension one submanifolds. Gromov has 
formulated the notion of {\it almost non-positively curved manifolds}: these are
manifolds with the property that for each $\epsilon >0$, one can find a 
Riemannian metric with the property that the diameters $d$ and maximal 
sectional curvature $K$ satisfy the inequality $K\cdot d^2 \leq \epsilon$ 
(see \cite{gro-flat}). It would be interesting to study graph manifolds from this
viewpoint. In particular:

\begin{problem}
Are graph manifolds almost non-positively curved?
\end{problem}

We note that the class of almost non-positively curved manifolds is very 
mysterious. The only known examples of manifold which are known to {\it not} 
be almost non-positively curved are the sphere $S^2$ and the projective
plane $\mathbb R P^2$ (by Gauss-Bonnet). Aside from manifolds supporting
non-positive curvature, the only additional known examples of almost 
non-positively curved manifolds occur in dimension =3 (all $3$-manifolds
are non-positively curved, see Bavard \cite{Bav}) and in dimension =4 
(a family of examples was constructed by Galaz-Garcia \cite{G-G}).

\vskip 10pt

Keeping on the theme of differential geometry, we recall that the minimal 
volume of a smooth manifold is defined to be the infimum of the volume functional,
over the space of all Riemannian metrics whose curvature is bounded between 
$-1$ and $1$.  Gromov \cite{gro-bddcohom} showed that manifolds with positive simplicial volume have
positive minimal volume and have positive minimal entropy. In view of our Proposition
\ref{simp-vol}, one can ask the following:

\begin{problem}
Let $M$ be a graph manifold with at least one purely hyperbolic piece (i.e. a piece with 
trivial fiber). Can one
compute the minimal volume of $M$? Does it equal the sum of the hyperbolic volumes of the 
purely hyperbolic pieces? Does the choice of gluing maps between tori affect this invariant?
If there are some pieces with non-trivial fiber, can the minimal volume ever be attained
by an actual metric on $M$?
\end{problem}

Similarly, minimal entropy is defined to be the infimum of the topological
entropy of the geodesic flow, over the space of all Riemannian metrics whose 
volume is equal to one. Gromov \cite{gro-bddcohom} also showed that positive simplicial 
volume implies
positive minimal entropy. One could formulate the same types of questions concerning
the minimal entropy.


\backmatter
%

\bibliographystyle{smfalpha}

\begin{thebibliography}{9999999}

\bibitem[A-K]{algom}
Y.~Algom-Kfir,
\emph{Strongly contracting geodesics in outer space},
preprint available on the arxiv:math/0812.1555.

\bibitem[AlBe]{Alibegovich-Best}
E.~Alibegovi\'c, M.~Bestvina,
\emph{Limit groups are $\rm CAT(0)$},
J. London Math. Soc. (2) \textbf{74} (2006), 259-272.


\bibitem[AlBr]{alo}
J.\ M.\ Alonso \& M.\ R.\ Bridson,
\emph{Semihyperbolic groups} 
Proc. Lond. Math. Soc. \textbf{70} (1995), 56-114.

\bibitem[ArFa]{ArFa}
C.\ S.\ Aravinda \& F.\ T.\ Farrell,
\emph{Twisted doubles and nonpositive curvature},
Bull. Lond. Math. Soc. \textbf{41} (2009), 1053-1059. 

\bibitem[AM]{AM}
G.\ Arzhantseva \& A.\ Minasyan,
\emph{Relatively hyperbolic groups are $C\sp \ast$-simple},
J. Funct. Anal., \textbf{243} (2007), 345-351.

\bibitem[AMO]{AMO}
G.\ Arzhantseva, A.\ Minasyan, \& D.\ Osin,
\emph{The SQ-universality and residual properties of relatively hyperbolic groups},
J. Algebra, \textbf{315} (2007), 165-177.

\bibitem[AO]{AO}
G.\ Arzhantseva \& D.\ Osin,
\emph{Solvable groups with polynomial Dehn functions},
Trans. Amer. Math. Soc. \textbf{354} (2002), 3329-3348.

\bibitem[BaLu]{BL}
A.\ Bartels \& W.\ L\"uck, 
\emph{Isomorphism conjecture for homotopy $K$-theory and groups acting on trees},
J. Pure Appl. Algebra \textbf{205} (2006), 660-696. 

\bibitem[Ba1]{Ba1}
H. Bass, \emph{Some remarks on group actions on trees}, Commun. Alg. 
\textbf{4} (1976), 1091-1126.


\bibitem[Ba2]{Ba}
H.\ Bass,
\emph{Covering theory for graphs of groups},
J. Pure Appl. Algebra \textbf{89} (1993), 3-47.

\bibitem[BHS]{BHS}
H.\ Bass, A.\ Heller, \& R.\ G.\ Swan, 
\emph{The Whitehead group of a polynomial extension},
Inst. Hautes \'Etudes Sci. Publ. Math. \textbf{22} (1964), 61-79. 

\bibitem[Ba]{Bav}
C.\ Bavard,
\emph{Courbure presque n\'egative en dimension 3}, 
Compos. Math. \textbf{63} (1987), 223-236.

\bibitem[BdlHV]{BdlHV}
B.\ Bekka, P.\ de la Harpe, \& A.\ Valette,
\emph{Kazhdan's property (T)}. 
New Mathematical Monographs, \textbf{11}. 
Cambridge University Press, Cambridge, 2008. 


\bibitem[BDM]{BDM}
J.\ A.\ Behrstock, C.\ Dru\c{t}u, \& L.\ Mosher
 \emph{Thick metric spaces, relative hyperbolicity, and quasi-isometric rigidity},
Math. Ann. \textbf{344} (2009), 543-595. 

\bibitem[BDS]{BDS}
J.\ A.\ Behrstock, C.\ Dru\c{t}u, \& M.\ Sapir,
\emph{Median structures on asymptotic cones and homomorphisms into mapping class groups},
to appear in Proc. Lond. Math. Soc.

\bibitem[BJN]{BJN}
J.\ A.\ Behrstock, T. Januszkiewicz, \& W.\ D.\ Neumann,
\emph{Quasi-isometric classification of some high dimensional right-angled Artin groups},
Groups Geom. Dyn. 4 (2010), no. 4, 681-692.

\bibitem[BeNe]{behr}
J.\ A.\ Behrstock \& W.\ D.\ Neumann, 
\emph{Quasi-isometric classification of graph manifold groups},
Duke Math. J. \textbf{141} (2008), 217-240. 

\bibitem[BePe]{BePe}
R.\ Benedetti, C.\ Petronio, 
\emph{Lectures on hyperbolic geometry}.
Universitext. Springer-Verlag, Berlin, 1992.

\bibitem[Be]{Be}
A.\ A.\ Bernasconi,
\emph{On HNN-extensions and the complexity of the word problem for one relator groups}, Ph.D. Thesis, University
of Utah, June 1994. Available online at 
http://www.math.utah.edu/$\sim$sg/Papers/bernasconi-thesis.pdf

\bibitem[BF]{BF}
M.~Bestvina \& K.~Fujiwara, 
\emph{Bounded cohomology of subgroups of mapping class groups},
Geom. Topol. \textbf{6} (2002), 69-89.

\bibitem[Bow]{Bow}
B.H.~Bowditch, \emph{Tight geodesics in the curve complex},
 Invent. Math. \textbf{171} (2008), 
281-300.

\bibitem[Br]{Br}
N.\ Brady,
\emph{Branched coverings of cubical complexes and subgroups of hyperbolic groups},
J. Lond. Math. Soc. \textbf{60} (1999), 461-480. 

\bibitem[Br-dlH]{Br-dlH1}
M.\ Bridson \& P.\ de la Harpe, 
\emph{Mapping class groups and outer automorphism groups of free groups are $C^*$-simple},  
J. Funct. Anal. \textbf{212} (2004), 195-205.


\bibitem[BGHM]{BGHM}
M.\ Bridson, D.\ Groves, J.\ A.\ Hillman \& G.\ J.\ Martin,
\emph{Cofinitely Hopfian groups, open mappings 
and knot complements}, Groups Geom. Dyn. \textbf{4} (2010), 693-707.

\bibitem[BrHa]{bri}
M.\ R.\ Bridson \& A.\ Haefliger,
\emph{Metric spaces of non-positive curvature}.
Grundlehren der Mathematischen Wissenschaften \textbf{319}. 
Springer-Verlag, Berlin, 1999.

\bibitem[Bu-dlH]{Bu-dlH}
M.\ Bucher \& P.\ de la Harpe, 
\emph{Mapping class groups and outer automorphism groups of free groups are $C^*$-simple}, 
Math. Notes \textbf{67} (2000), 686-689.

\bibitem[BBFIPP]{BBFIPP}
M.~Bucher, M.~Burger, R.~Frigerio, A.~Iozzi, C.~Pagliantini, M.~B.~Pozzetti,
\emph{Isometric Embeddings in Bounded Cohomology},
arXiv:1305.2612.


\bibitem[BuTa]{BT}
J.\ Burillo \& J. Taback, 
\emph{Equivalence of geometric and combinatorial Dehn functions},
New York J. Math. \textbf{8} (2002), 169-179. 

\bibitem[BuKo]{BK}
S.\ V.\ Buyalo \& V.\ L.\ Kobelski\v i,
\emph{Generalized graph-manifolds of nonpositive curvature}, 
St. Petersburg Math. J. \textbf{11} (2000), 251-268.

\bibitem[BuSv]{BS}
S.\ V.\ Buyalo \& P.\ Svetlov, 
\emph{Topological and geometric properties of graph-manifolds},
St. Petersburg Math. J. \textbf{16} (2005), 297-340.

\bibitem[Ca]{Ca}
S.\ Cappell, \emph{A splitting theorem for manifolds}, Invent. Math. \textbf{33} (1976), pp. 69-170.

\bibitem[CCJJV]{CCJJV}
P.-A. Cherix, M. Cowling, P. Jolissaint, P. Julg \& A. Valette, 
\emph{Groups with the Haagerup property}. {Progress in Mathematics},
\textbf{197}, {Birkh\"auser Verlag}, {Basel}, {2001}. {viii+126} pp.

\bibitem[CoPr]{CP}
F.\ X.\ Connolly \& S.\ Prassidis, \emph{On the exponent of the cokernel of the forget-control map on 
$K_0$-groups}, Fund. Math. \textbf{172} (2003), pp. 201-216.

\bibitem[Da]{Dahmani}
F.\ Dahmani, 
\emph{Combination of convergence groups},
Geom. Topol. \textbf{7} (2003), 933-963.

\bibitem[DGO]{DGO}
F.~Dahmani, V.~Guirardel \& D.~Osin,
\emph{Hyperbolically embedded subgroups and rotating families in
groups acting on hyperbolic spaces}, arXiv:1111.7048

\bibitem[DJL]{DJL}
M.\ Davis, T.\ Januszkiewicz \& J.-F.\ Lafont, 
\emph{4-dimensional locally CAT(0)-manifolds with no Riemannian smoothings},
preprint available on the  arXiv:1002.4235

\bibitem[De]{Del}
T.\ Delzant, 
\emph{Sur l'accessibilit\'e acylindrique des groupes de pr\'esentation finie}, 
Ann. Inst. Fourier \textbf{49} (1999), 1215-1224.

\bibitem[Do]{dold} 
A.\ Dold, 
\emph{A simple proof of the Jordan-Alexander complement theorem},
Amer. Math. Monthly \textbf{100} (1993), 856-857.

\bibitem[DMS]{DMS}
C.\ Dru\c{t}u, S.\ Mozes \& M.\ Sapir,
\emph{Divergence in lattices in semisimple Lie groups and graphs of groups}, 
Trans. Amer. Math. Soc. \textbf{362} (2010), 2451-2505. 

\bibitem[DrSa]{dru}
C.\ Dru\c{t}u \& M.\ Sapir, 
\emph{Tree-graded spaces and asymptotic cones of groups. 
With an appendix by D. Osin and M. Sapir}, 
Topology  \textbf{44} (2005), 959-1058.

\bibitem[Du]{Dun}
M.\ J.\ Dunwoody,
\emph{The accessibility of finitely presented groups}, 
Invent. Math. \textbf{81} (1985), no. 3, 449-457.

\bibitem[DS]{DS}
M.~J.~Dunwoody, M.~Sageev, 
\emph{JSJ-splittings for finitely presented groups over slender groups},
Invent. Math. 135 (1999), no. 1, 25--44.


\bibitem[EMO]{EMO}
A.\ Eskin, S.\ Mozes, \& H.\ Oh,
\emph{On Uniform Exponential Growth for Linear Groups}, 
Int. Math. Res. Not. (2002), no. 31, 1675-1683

\bibitem[Fa1]{farb1}
B.\ Farb,
\emph{The Extrinsic geometry of subgroups and the generalized word problem},
Proc. Lond. Math. Soc. \textbf{68} (1994), 577--593.

\bibitem[Fa2]{farb}
B.\ Farb, 
\emph{Relatively hyperbolic groups},
Geom. Funct. Anal. \textbf{8} (1998), 810-840.

\bibitem[FaMa]{farbbook}
B.\ Farb \& D.\ Margalit,
\emph{A primer on mapping class groups}, available
at http://www.math.uchicago.edu/~margalit/mcg/mcgv50.pdf

\bibitem[Fa]{Fa}
F.\ T.\ Farrell,
\emph{Surgical methods in rigidity}. Springer-Verlag, Berlin, 1996. iv + 98 pp.


\bibitem[FaJo1]{FJ1}
F.\ T.\ Farrell \& L.\ E.\ Jones, 
\emph{Topological rigidity for compact non-positively curved manifolds},
in ``Differential geometry: Riemannian geometry (Los Angeles, CA, 1990)'', Proc. Sympos. Pure Math.
\textbf{54} (1993), 229-274.

\bibitem[FaJo2]{FJ2}
F.\ T.\ Farrell \& L.\ E.\ Jones, 
\emph{Rigidity for aspherical manifolds with $\pi_1\subset{\rm GL}_m({\bf R})$}, 
Asian J. Math. \textbf{2} (1998), 215-262.


\bibitem[Fo]{Forester}
M.~Forester, \emph{On uniqueness of JSJ decompositions of finitely generated groups}, Comment. Math. Helv.  \textbf{78}  (2003),  no. 4, 740--751.

\bibitem[FP]{FP}
K.~Fujiwara, P.~Papasoglu, 
\emph{JSJ-decompositions of finitely presented groups and complexes of groups},
Geom. Funct. Anal. \textbf{16} (2006), no. 1, 70--125.


\bibitem[G-G]{G-G}
F.\ Galaz-Garcia,
\emph{Examples of 4-manifolds with almost nonpositive curvature}, 
Differential Geom. Appl. \textbf{26} (2008), 697-703.

\bibitem[GeOn]{Geo-Ont}
R.~Geoghegan, P.~Ontaneda, 
\emph{Boundaries of cocompact proper ${\rm CAT}(0)$ spaces},
Topology \textbf{46} (2007), 129-137.



\bibitem[Ge1]{Ger:last}
A.~M.~Gersten, 
\emph{Bounded cocycles and combings of groups},
Internat. J. Algebra Comput. \textbf{2} (1992), 307-326.

\bibitem[Ge2]{Ger}
S.\ M.\ Gersten,
\emph{Isoperimetric and isodiametric functions}, in 
``Geometric Group Theory I'', ed. by G. Niblo and M. Roller. 
LMS Lecture notes {\bf 181}. Cambridge Univ. Press, 1993.

\bibitem[Ge3]{ger2}
S.\ M.\ Gersten,
\emph{Quadratic divergence of geodesics in CAT(0) spaces}, 
Geom. Funct. Anal. \textbf{4} (1994), 37-51.

\bibitem[Ge4]{ger4}
S.\ M.\ Gersten, 
\emph{The automorphism group of a free group is not a ${\rm CAT}(0)$  group},
Proc. Amer. Math. Soc.  \textbf{121}  (1994),  999-1002.

\bibitem[GHH]{GHH}
O.\ Goodman, D.\ Heard, \& C.\ Hodgson, 
\emph{Commensurators of cusped hyperbolic manifolds}, 
Experiment. Math. \textbf{17} (2008), 283-306.

\bibitem[Gr1]{gropol}
M.\ Gromov, \emph{Groups of polynomial growth and expanding maps},
Inst. Hautes \'Etudes Sci. Publ. Math. \textbf{53} (1981), 53-73.

\bibitem[Gr2]{gro1} 
M.\ Gromov,
\emph{Hyperbolic groups}, in ``Essays in group theory'',
75-263, Math. Sci. Res. Inst. Publ. \textbf{8} (1987), Springer, New York.


\bibitem[Gr3]{gro-flat}
M.\ Gromov, 
\emph{Almost flat manifolds}, 
J. Differential Geom. \textbf{13} (1978) 231-241.

\bibitem[Gr4]{gro-bddcohom}
M.\ Gromov,
\emph{Volume and bounded cohomology},
Inst. Hautes \'Etudes Sci. Publ. Math. \textbf{56} (1982), 5-99.

\bibitem[GeSh]{gersten}
S.\ M.\ Gersten \& H.\ Short, 
\emph{Rational subgroups of biautomatic groups},
Ann. Math. \textbf{134} (1991), 125-158.

\bibitem[GL]{GL}
V.~Guirardel, G.~Levitt, 
\emph{JSJ decompositions: definitions, existence, uniqueness. II. Compatibility and acylindricity},
arXiv:1002.4564.

\bibitem[Hag]{Hag}
F.~Haglund, 
\emph{Isometries of CAT(0) cube complexes are semi-simple}, 
 arXiv:0705.3386.

\bibitem[Ha]{Ha}
M.\ Hall,
\emph{Finiteness conditions for soluble groups},
Proc. Lond. Math. Soc. \textbf{4} (1954), 419-436.

\bibitem[dlH1]{dlH-new}
P.~de la Harpe, \emph{Reduced $C^\ast$-algebras of discrete groups which are simple with a unique trace},  Operator algebras and their connections with topology and ergodic theory (Busteni, 1983),  230-253, Lecture Notes in Math., \textbf{1132}, Springer, Berlin, 1985.

\bibitem[dlH2]{dlH}
P.\ de la Harpe,
\emph{On simplicity of reduced $C^*$-algebras of groups},
Bull. Lond. Math. Soc. \textbf{39} (2007), 1-26.

\bibitem[dlH-Pr]{dlH-Pr}
P.~de la Harpe, J.P.~Pr\'eaux, 
\emph{$C^*$-simple groups: amalgamated free products, HNN extensions, and fundamental groups of 3-manifolds},
J. Topol. Anal. \textbf{3} (2011), 451-489.

\bibitem[HK]{HK}
N. Higson \& G. Kasparov, 
\emph{$E$-theory and $KK$-theory for groups which act properly and isometrically on Hilbert space},
{Invent. Math.} \textbf{144} (2001), {23-74}.

\bibitem[HsWa]{HsWa}
W.\ C.\ Hsiang \& C.\ T.\ C.\ Wall,
\emph{On homotopy tori II},
Bull. Lond. Math. Soc. \textbf{1} (1969), 341-342.

\bibitem[Ka]{Kap}
I.\ Kapovich, 
\emph{The combination theorem and quasiconvexity},
Internat. J. Algebra Comput. \textbf{11} (2001), 185-216.

\bibitem[KaLe1]{kaplee0}
M.\ Kapovich \& B.\ Leeb,
\emph{On asymptotic cones and quasi-isometry classes of fundamental groups of 3-manifolds},
Geom. Funct. Anal. \textbf{5} (1995), no. 3, 582-603.

\bibitem[KaLe2]{kaplee1}
M.\ Kapovich \& B.\ Leeb,
\emph{Actions of discrete groups on nonpositively curved spaces},
Math.~Ann.~\textbf{306} (1996), 341-352. 

\bibitem[KaLe3]{kapleenew}
M.\ Kapovich \& B.\ Leeb,
\emph{Quasi-isometries preserve the geometric decomposition of Haken manifolds},
Invent.~Math. \textbf{128}  (1997),   393-416.

\bibitem[KaLe4]{kaplee}
M.\ Kapovich \& B.\ Leeb,
\emph{3-manifold groups and non-positive curvature},
Geom. Funct. Anal \textbf{8} (1998), 841-852.

\bibitem[KlLe]{klelee}
B.\ Kleiner \& B.\ Leeb, 
\emph{Groups quasi-isometric to symmetric spaces}, 
Comm. Anal. Geom. \textbf{9} (2001), 239-260.

\bibitem[KoMa]{MarKo}
A.~Kolpakov, B.~Martelli, \emph{Hyperbolic four-manifolds with one cusp}, Geom. Funct. Anal. \textbf{23} (2013), 1903--1933.


\bibitem[KS]{kirby-siebenmann}
R.\ C.\ Kirby \& L.\ C.\ Siebenmann, \emph{Foundational essays on topological manifolds,
smoothings, and triangulations}, Annals of Math. Studies \textbf{88}, Princeton Univ. Press, 
Princeton, 1977. 355pp.

\bibitem[Ko]{Kou}
M.\ Koubi, 
\emph{Croissance uniforme dans les groupes hyperboliques}, 
Ann. Inst. Fourier \textbf{48} (1998), 1441-1453.

\bibitem[Ku]{ku}
T. Kuessner, 
\emph{Multicomplexes, bounded cohomology and additivity of the simplicial volume},
preprint available on the arxiv:math/0109057v2.

\bibitem[La]{L}
J.-F.\ Lafont, 
\emph{A boundary version of Cartan-Hadamard and applications to rigidity},
J. Topol. Anal. \textbf{1} (2009), 431-459.

\bibitem[LoRe]{LR}
D.\ D.\ Long \& A.\ W.\ Reid,
\emph{All flat manifolds are cusps of hyperbolic orbifolds},
Algebr. Geom. Topol. \textbf{2} (2002), 285--296.

\bibitem[Le]{leeb}
B.\ Leeb,  
\emph{$3$-manifolds with(out) metrics of nonpositive curvature},
Invent. Math. \textbf{122} (1995), 277-289.

\bibitem[LuR]{LuR}
W. L{\"u}ck \& H. Reich, 
\emph{The Baum-Connes and the Farrell-Jones conjectures in $K$- and $L$-theory},
in ``{Handbook of $K$-theory. Vol. 1, 2}'', pp. 703-842.
{Springer}, {Berlin}, {2005}.

\bibitem[LySc]{LySc}
R.\ C.\ Lyndon \& P.\ E.\ Schupp,
\emph{Combinatorial group theory.}
Reprint of the 1977 edition. Classics in Mathematics. Springer-Verlag, Berlin, 2001. xiv+339 pp. 

\bibitem[Ma]{Ma}
A.\ Manning, 
\emph{Topological entropy for geodesic flows},
Ann. of Math. \textbf{110} (1979), 567-573. 

\bibitem[MRS]{MRS}
D.\ B.\ McReynolds, A.\ Reid, \& M.\ Stover,
\emph{Collisions at infinity in hyperbolic manifolds}, 
Math. Proc. Cambridge Philos. Soc. \textbf{155} (2013), no. 3, 459-463. 

\bibitem[MNS]{MNS}
C.\ F.\ Miller III, W.\ D.\ Neumann, \& G.\ Swarup, 
\emph{Some examples of hyperbolic groups}, in ``Geometric group theory down under (Canberra, 1996)'', 195--202, 
de Gruyter, Berlin, 1999. 


\bibitem[Mi]{Mi}
J. Milnor,
\emph{Microbundles}, 
Topology \textbf{3} (1964) Suppl. 1, 53--80.


\bibitem[MV]{MV}
G. Mislin \& A. Valette, 
\emph{Proper group actions and the Baum-Connes conjecture},
Advanced Courses in Mathematics - CRM Barcelona,
Birkh\"auser Verlag, Basel, 2003.
viii+131 pp.

\bibitem[MSW1]{MSW1}
L.\ Mosher, M.\ Sageev, \& K.\ Whyte,
\emph{Quasi-actions on trees I. Bounded Valence},
Ann. of Math. (2) \textbf{158} (2003), no. 1, 115-164.

\bibitem[MSW2]{MSW2}
L.\ Mosher, M.\ Sageev, \& K.\ Whyte,
\emph{Quasi-actions on trees II: Finite depth Bass-Serre trees},
 Mem. Amer. Math. Soc.  \textbf{214}  (2011),  no. 1008.

\bibitem[Ng]{N}
T.\ T.\ Nguyen Phan, 
\emph{Smooth (non)rigidity of cusp-decomposable manifolds},
to appear in Comment. Math. Helv.

\bibitem[NeRe]{neumann}
W.D. Neumann \& L. Reeves, 
\emph{Central extensions of word hyperbolic groups},
Ann. Math. \textbf{145}  (1997), 183-192.

\bibitem[Ol]{Ols}
A.\ Yu.\ Olshanskii, 
\emph{SQ-universality of hyperbolic groups}, Mat. Sb. \textbf{186} (1995) 119-132 (in Russian); Sb.
Math. \textbf{186} (1995) 1199-1211.

\bibitem[On]{On}
P.\ Ontaneda, \emph{The double of a hyperbolic manifold and
non-positively curved exotic PL structures}, Trans. Amer. Math. Soc.
\textbf{355} (2002), 935-965.

\bibitem[Os]{osin} 
D.\ V.\ Osin, 
\emph{Relatively hyperbolic groups: intrinsic geometry, 
algebraic properties, and algorithmic problems},
Mem. Amer. Math. Soc. \textbf{179} (2006), no. 843.

\bibitem[Os2]{osin-pre} 
D.\ V.\ Osin, 
\emph{Acylindrically hyperbolic groups},
arXiv:1304.1246.

\bibitem[O-O]{O-O}
H. Oyono-Oyono, 
\emph{Baum-Connes conjecture and group actions on trees},
$K$-Theory \textbf{24} (2001), 115-134.

\bibitem[Pa]{Pap}
P.\ Papasoglu, 
\emph{Group splittings and asymptotic topology}, 
J. Reine Angew. Math. \textbf{602} (2007), 1-16. 

\bibitem[PaSw]{Swenson-Papa}
P.~Papasoglu, Panos, E.~L.~Swenson, \emph{Boundaries and JSJ decompositions of CAT(0)-groups},  
Geom. Funct. Anal.  \textbf{19}  (2009),  559-590.

\bibitem[PaVa]{PV}
I.\ Pays \& A.\ Valette, 
\emph{Sous-groupes libres dans les groupes d'automorphismes d'arbres},
Enseign. Math. \textbf{37} (1991), 151-174. 


\bibitem[Po]{Pow}
R.\ T.\ Powers, 
\emph{Simplicity of the $C^*$-algebra associated with the free group on two generators},
Duke Math. J. \textbf{42} (1975), 151-156.


\bibitem[Q]{Q}
F. Quinn,
\emph{Ends of Maps. III: Dimensions 4 and 5},
J. Diff. Geom. \textbf{17} (1982), 503--521.


\bibitem[Rat]{Ratcliffe}
J.~G.~Ratcliffe, \emph{Foundations of hyperbolic manifolds, second edition}, Graduate Texts in Mathematics \textbf{149},
Springer, 2006.


\bibitem[Ru1]{R}
K.\ Ruane, 
\emph{CAT(0) boundaries of truncated hyperbolic space},
Spring Topology and Dynamical Systems Conference. 
Topology Proc. \textbf{29} (2005), 317-331. 

\bibitem[Ru2]{Rouane}
K.~Ruane,
\emph{CAT(0) groups with specified boundary},
Algebr. Geom. Topol. \textbf{6} (2006) 633-649.

\bibitem[Sc]{schw}
R.\ E.\ Schwartz, 
\emph{The quasi-isometry classification of rank-one lattices},
Inst. Hautes \'Etudes Sci. Publ. Math. \textbf{82} (1995), 133-168.

\bibitem[SW]{SW}
P.~Scott, T.~Wall \emph{Topological methods in group theory},  
Homological group theory (Proc. Sympos., Durham, 1977),  pp. 137--203, London Math. Soc. Lecture Note Ser. \textbf{36}, Cambridge Univ. Press, Cambridge-New York, 1979-

\bibitem[Se]{serre}
\emph{Trees}. 
Springer monographs in Mathematics, Springer-Verlag, Berlin and New York, 1980.

\bibitem[Si]{sisto}
A.\ Sisto, 
\emph{Projections and relative hyperbolicity}, 
preprint available on the arXiv:1010.4552.

\bibitem[St]{Sta1}
J.\ R.\ Stallings,
\emph{On torsion-free groups with infinitely many ends},
Ann. of Math. (2) \textbf{88} (1968), 312-334.

\bibitem[Sw]{Swenson}
E.~L.~Swenson, \emph{A cut point theorem for ${\rm CAT}(0)$ groups},
J. Differential Geom.  \textbf{53}  (1999),  327-358. 


\bibitem[Wa1]{Wa3}
F.\ Waldhausen,
\emph{On irreducible $3$-manifolds which are sufficiently large},
Ann. of Math. \textbf{87} (1968), 56-88.

\bibitem[Wa2]{Wa}
F.\ Waldhausen, 
\emph{ Whitehead groups of generalized free products}, 
in ``Algebraic K-theory, II,'' pp. 155-179. Lecture Notes in Math., Vol. \textbf{342},
Springer, Berlin, 1973.

\bibitem[Wa3]{Wa1}
F.\ Waldhausen,
\emph{Algebraic $K$-theory of generalized free products. I, II}, 
Ann. Math. \textbf{108} (1978), 135-204. 

\bibitem[Wa4]{Wa2}
F.\ Waldhausen,
\emph{Algebraic $K$-theory of generalized free products. III, IV}, 
Ann. Math. \textbf{108} (1978), 205-256. 


\bibitem[Wal1]{W}
C.\ T.\ C.\ Wall, \emph{Surgery on compact manifolds}. 
Amer. Math. Soc., Providence, 1999.

\bibitem[Wal2]{W2}
C.\ T.\ C.\ Wall,
\emph{The geometry of abstract groups and their splittings}, 
Rev. Mat. Complut. \textbf{16} (2003), 5-101. 

\bibitem[Wh]{Wh}
J.\ H.\ C.\ Whitehead, 
\emph{On the asphericity of regions in a 3-sphere}, Fund. Math. \textbf{32} (1939), 149-166.


\bibitem[vdDWi]{wilkie}
L.\ van den Dries \& J.\ Wilkie,
\emph{Gromov's theorem on groups of polynomial growth and elementary logic},
J. Algebra \textbf{89} (1984), 349-374.


\bibitem[Xie]{Xie}
X. Xie, \emph{Growth of relatively hyperbolic groups}, Proc. Amer. Math. Soc. \textbf{135} (3) (2007) 695--704 (electronic). 

\bibitem[Zi]{zim}
R.\ Zimmer, 
\emph{Ergodic theory and semi-simple Lie groups}. 
Birkhauser, Boston, 1984.


\bibitem[www]{web}
http://www.ms.unimelb.edu.au/$\sim$snap/commens/comm\_grouped


\end{thebibliography}


\printindex

\end{document}